\documentclass[11pt,a4paper]{article}
\usepackage{amsmath}
\usepackage{amsfonts}
\usepackage{amssymb}
\usepackage{amsthm}
\usepackage{bm}
\usepackage[left=2.5cm,right=2.5cm,top=2cm,bottom=2cm]{geometry}
\usepackage[colorlinks=true,linkcolor=black,urlcolor=black,citecolor=black]{hyperref}
\usepackage{authblk}
\usepackage{apptools}
\usepackage{graphics}
\usepackage{type1cm}

\theoremstyle{plain}
\newtheorem{theorem}{Theorem}
\newtheorem{assumption}[theorem]{Assumption}

\newtheorem{proposition}[theorem]{Proposition}
\newtheorem{definition}[theorem]{Definition}
\newtheorem{lemma}[theorem]{Lemma}
\newtheorem{remark}[theorem]{Remark}
\newtheorem{corollary}[theorem]{Corollary}
\newtheorem{notation}[theorem]{Notation}

\newcommand{\e}{\varepsilon}
\newcommand{\indiq}{{\bm 1}}
\newcommand{\R}{\mathbb{R}}
\newcommand{\N}{\mathbb{N}}
\newcommand{\HH}{\mathbb{H}}
\newcommand{\QQ}{\mathbb{Q}}
\newcommand{\PP}{\mathbb{P}}
\newcommand{\DD}{\mathbb{D}}
\newcommand{\Sp}{\mathbb{S}}
\newcommand{\E}{\mathbb{E}}
\newcommand{\JS}{{\bm{\mathrm{J}}_1}}
\newcommand{\MS}{{\bm{\mathrm{M}}_1}}
\newcommand{\cA}{\mathcal{A}}
\newcommand{\cF}{\mathcal{F}}
\newcommand{\cB}{\mathcal{B}}
\newcommand{\cH}{\mathcal{H}}
\newcommand{\cG}{\mathcal{G}}
\newcommand{\cL}{\mathcal{L}}
\newcommand{\cK}{\mathcal{K}}
\newcommand{\cR}{\mathcal{R}}
\newcommand{\cP}{\mathcal{P}}
\newcommand{\cE}{\mathcal{E}}
\newcommand{\cI}{\mathcal{I}}
\newcommand{\cZ}{\mathcal{Z}}
\newcommand{\cEc}{\cE_0}
\newcommand{\cEd}{\cE_+}
\newcommand{\be}{\mathbf{e}}
\newcommand{\bbf}{\mathbf{f}}
\newcommand{\bbg}{\mathbf{g}}
\newcommand{\bn}{\mathbf{n}}
\newcommand{\cl}{\bar \ell}
\newcommand{\cu}{\mathfrak{u}}
\newcommand{\co}{\mathfrak{o}}
\newcommand{\bg}{\bar g}
\newcommand{\bd}{\bar d}
\newcommand{\tell}{\tilde \ell}
\newcommand{\tx}{{\tilde x}}
\newcommand{\bv}{{\bar v}}
\newcommand{\bmm}{{\bar m}}
\newcommand{\dr}{\mathrm{d}}
\newcommand{\II}{\mathrm{I}}
\newcommand{\JJ}{\mathrm{J}}
\newcommand{\nn}{\mathfrak{n}}
\newcommand{\Dd}{\mathcal{D}}
\newcommand{\closure}[2][3]{{}\mkern#1mu\overline{\mkern-#1mu#2}}
\newcommand{\cDd}{\closure{\Dd}}
\newcommand{\pDd}{\partial \Dd}
\newcommand{\sm}{{s-}}
\newcommand{\um}{{u-}}
\newcommand{\vm}{{v-}}
\newcommand{\bE}{\bm{E}}
\newcommand{\bX}{\bm{X}}
\newcommand{\bM}{\bm{M}}
\newcommand{\bV}{\bm{V}}
\newcommand{\brv}{\bm{v}}
\newcommand{\brw}{\bm{w}}
\newcommand{\brx}{\bm{x}}
\newcommand{\rF}{\mathrm{F}}
\newcommand{\rG}{\mathrm{G}}
\newcommand{\rH}{\mathrm{H}}
\newcommand{\rM}{\mathrm{M}}
\newcommand{\rN}{\mathrm{N}}
\newcommand{\rK}{\mathrm{K}}
\newcommand{\rP}{\mathrm{P}}
\newcommand{\bw}{\mathbf{w}}
\newcommand{\bov}{\mathbf{v}}
\newcommand{\vip}{\vskip5pt}

\usepackage[dvipsnames]{xcolor}
\definecolor{applegreen}{rgb}{0.55, 0.71, 0.0}

\author[3]{Lo\"ic B\'ethencourt\textsuperscript{1} and Nicolas Fournier\textsuperscript{2}\\
Appendix~\ref{ageo} by Laurent Mazet}
\affil[1]{\footnotesize Laboratoire d'Analyse, Géométrie et Applications, UMR CNRS 7539, 
Université Sorbonne Paris Nord, France. }
\affil[2]{\footnotesize Sorbonne Université, CNRS, Laboratoire de Probabilités, Statistiques et 
Modélisation, 75005 Paris, France.}
\affil[3]{\footnotesize Université de Tours, Université d’Orléans, CNRS, IDP, UMR 7013, Tours, France.}

\title{{\bf Fractional diffusion in convex domains and reflected isotropic stable processes}}

\date{}

\begin{document}

\maketitle

\begin{abstract}
We establish the fractional diffusion limit of the kinetic scattering equation 
with diffusive boundary condition in a strongly convex bounded domain $\Dd\subset\R^d$. 
According to the nature of the 
boundary condition, two types of fractional heat equations may arise at the limit,
corresponding to two types of isotropic 
stable processes reflected in $\Dd$.
In both cases, when the process tries to jump across the boundary, it is stopped 
at the unique point where $\pDd$ intersects the line segment defined by the attempted jump.
It then leaves the boundary either continuously (for the first type) or by a power-law distributed jump
(for the second type).
The construction of these processes is done via an Itô synthesis: we concatenate 
their excursions in the domain,
which are obtained by translating, rotating and stopping the excursions of some stable processes reflected in
the half-space. The key ingredient in this procedure is the construction
of the boundary processes, \textit{i.e.} the processes time-changed by their local time on the 
boundary,
which solve stochastic differential equations driven by some Poisson measures 
of excursions. The well-posedness of these boundary processes relies on delicate estimates involving 
some geometric inequalities and the
laws of the undershoot and overshoot of the excursion when it leaves the domain.
We show that these reflected Markov processes are Markov and Feller, we
study their infinitesimal generator and we write down the reflected fractional heat equations 
satisfied by their time-marginals. \\[5pt]
\textit{Keywords:}  Kinetic scattering equation in a domain,
Reflected processes, Markov processes in domains, Isotropic stable processes, Excursion theory,
Itô synthesis in a domain, Fractional heat equation in a domain. \\[5pt]
\textit{MSC2020 AMS classification:} 60J25, 60J50, 60G52, 35R11.
\end{abstract}

\tableofcontents

\section{Introduction}

We study the fractional diffusion limit of some {\it scattering processes} reflected in a strongly
convex domain $\Dd \subset \mathbb{R}^d$, $d \geq2$, with diffusive boundary conditions. Two 
types of stable processes reflected in $\Dd$ may arise at the limit, 
with the following behavior: (i) within the domain $\Dd$, the 
processes follow the dynamics of an isotropic $\alpha$-stable process for some $\alpha \in(0,2)$, 
(ii) when they try to jump across the boundary, they are stopped at the intersection of the boundary 
$\pDd$ and of the line segment defined by the attempted jump and (iii) they either leave continuously 
(first type) the boundary or jump into the domain (second type) according to the measure 
$|x|^{-\beta - d}\dr x$ for some $\beta \in(0,\alpha/2)$.
In this introduction, we call $(R_t^*)_{t\geq0}$ (resp. $(R_t^{(\beta)})_{t\geq0}$) a process 
of the first (resp. second) type. We will
see right after~\eqref{se1} why $\beta$ is restricted to $(0,\alpha/2)$.

\subsection{Fractional diffusion limit of kinetic equations}\label{ssmm}

We are mainly motivated by the study 
of linear kinetic equations in domains with heavy-tailed equilibria. Such 
equations describe the position and velocity of a random particle subject
to some environment. In the last two decades, many works from both the probabilistic and the P.D.E. communities 
have
established \textit{fractional diffusion limits} for these equations. In a nutshell, such results 
state that the density of the position of the particle can be approximated, when the collision rate increases 
to infinity, by the solution of the fractional heat equation, \textit{i.e.} by the law of an $\alpha$-stable 
process. The works of Mellet~\cite{mellet2010fractional}, Ben Abdallah, Mellet and Puel~\cite{ben2011anomalous}, 
Mellet, Mischler and Mouhot~\cite{mellet2011fractional} and Komorowski, Olla and Ryzhik~\cite{komo_olla} 
concern some toy
linear Boltzmann equations, 
whereas Cattiaux, Nasreddine and Puel~\cite{cattiaux2019diffusion}, Lebeau and Puel~\cite{lebeau2019diffusion},
Fournier and Tardif~\cite{fournier2018one,fournier_dimension}, Bouin and Mouhot~\cite{bouin2020quantitative},
Dechicha and Puel~\cite{dechicha2023construction, dechicha2023fractional} and Dechicha~\cite{dechicha2024spectral}
deal with the kinetic Fokker-Planck equation. All these papers assume that $\Dd=\R$ or $\Dd = \mathbb{R}^d$.

\vip
Let us now focus on the kinetic scattering equation, which is the simplest toy linear Boltzmann equation,
in the asymptotic of frequent collisions. 
From a probabilistic point of view, the solution to this P.D.E. describes  
the law of the stochastic process 
$(\bX_t^\varepsilon, \bV_t^\varepsilon)_{t\geq0}$ representing the position and velocity of a particle. 
The velocity process is defined as $\bV_t^\varepsilon = W_{N_t^\varepsilon}$
where $(N_t^\varepsilon)_{t\geq0}$ is a Poisson process of parameter $\varepsilon^{-1}$ and $(W_n)_{n\geq0}$ 
is an i.i.d. sequence of $\R^d$-valued random variables distributed according 
to some rotationally invariant equilibrium $\mathrm{F}$ 
with heavy tail, \textit{i.e.} such that $\mathrm{F}(v)\sim |v|^{-\alpha - d}$ as $|v|\to\infty$ 
for some $\alpha \in(0,2)$. 
The position process is then defined as $\bX_t^\varepsilon 
= x_0 + \varepsilon^{1/\alpha - 1}\int_0^t \bV_s^\varepsilon \dr s$. 
It is almost immediate to prove, 
using the stable central limit theorem, that $(\bX_t^\varepsilon)_{t\geq 0}$ converges in law, in
the sense of finite-dimensional distributions, to an
isotropic $\alpha$-stable process as $\varepsilon\to0$.
\vip
Let us now consider the same kinetic equation in some domain $\Dd$ with diffusive boundary conditions. 
These conditions state that when the position process $\bX_t^\varepsilon$ hits the boundary 
(at some point $x\in \pDd$),
it is restarted with a new velocity, distributed according to $2\mathrm{G}(v)\indiq_{\{v\cdot  \bn_x >0\}}$, 
where $\mathrm{G}$ 
is a given velocity distribution and where $\bn_x$ is the inward normal vector to $\Dd$ at $x$. 
When in the interior of $\Dd$,
the process has the same dynamics as previously described, with the equilibrium $\mathrm{F}$.
In the limit $\varepsilon \to 0$, one should see an isotropic $\alpha$-stable process, but each jump 
(which corresponds to a line segment when $\e>0$) that makes it leave $\Dd$ is cut at the boundary.
We thus should expect
the limit process (when $\varepsilon\to0$) to satisfy (i) and (ii). When the velocities 
distributed according to $\mathrm{G}$ are relatively small compared to the ones distributed as $\mathrm{F}$, 
there is no reason for the limit process to leave the boundary by a jump and it should come out continuously. 
But if this is not the case, for instance if $\mathrm{G}$ is also heavy-tailed, we can expect a 
jump in the limit, which would necessarily have a power law and this explains (iii).

\vip
This particular problem is treated from a P.D.E. point of view in the works of 
Cesbron, Mellet and Puel~\cite{cesbron2020fractional, cesbron2021fractional} when $\Dd=(0,1)$ or 
$\Dd = \HH:=(0,\infty) \times \mathbb{R}^{d-1}$, when 
$\alpha > 1$ and when $\mathrm{G}(v) \sim |v|^{1-\alpha - d}$. 
They find a limit P.D.E., which corresponds
to the law of the reflected stable process leaving the boundary with a jump distributed as $|x|^{-\beta - d}\dr x$ 
for $\beta = \alpha -1$, \textit{i.e.} to the law of $(R_t^{\alpha - 1})_{t\geq0}$.
As mentioned in~\cite{cesbron2020fractional, cesbron2021fractional}, this special value of $\beta$
is particularly natural from a scaling point of view.

\vip
In the present paper, see Theorem~\ref{mr} and Corollary~\ref{pdev},  we extend the results 
of~\cite{cesbron2020fractional, cesbron2021fractional} to general convex domains (with some assumptions) and
to a wide range of boundary distributions $\mathrm{G}$:  when $\mathrm{G}$ has a moment of 
order $\alpha /2$, the position process $(\bX_t^\varepsilon)_{t\geq 0}$ 
converges in law to $(R_t^*)_{t\geq0}$ while when 
$\mathrm{G}(v) \sim |v|^{-\beta - d}$ for some 
$\beta \in (0, \alpha/2)$, it converges in law to $(R_t^{(\beta)})_{t\geq0}$.

\vip
Although it might seem surprising, the scaling exponent is the same for 
$(R_t^*)_{t\geq0}$ and $(R_t^{(\beta)})_{t\geq0}$ 
for all $\beta\in (0,\alpha/2)$, see Proposition~\ref{scal}.
When $\beta$ is small (resp. large), the process jumps far away from (resp. close to) the boundary,
but consequently the return time to the boundary is large (resp. small), resulting in a perfect overall balance.
This fact is rather classical: Lamperti~\cite{MR307358} (see also Yano~\cite{MR2543582}) has shown 
that any continuous positive self-similar Markov process behaves as a Bessel process in $(0,\infty)$, 
and is either reflected continuously, or by a power law jump according to $x^{-\beta-1}\dr x$ with 
some $\beta \in(0,1)$. We also refer to the work of Vuolle-Apiala~\cite{MR1288123}.

\vip

It is likely that similar results should hold for the kinetic Fokker-Planck equation, although the study 
of this equation is more difficult. In the half-line, it was recently shown in~\cite{bethencourt2022fractional} 
that, for this equation (with a diffusive boundary condition involving some distribution $\mathrm{G}$ with
a finite moment of suitable order),
the limit process is the stable process reflected on its infimum, corresponding 
in our situation to the process $(R^*_t)_{t\geq 0}$.

\vip

Let us also mention other works for the scattering equation with different boundary conditions. 
Cesbron~\cite{cs2020refelctive} extends the works~\cite{cesbron2020fractional, cesbron2021fractional} 
in the half-space to treat the case of Maxwell boundary conditions, which mix diffusive and 
specular reflections. In dimension $1$, and from a probabilistic perspective, Komorowski, 
Olla and Ryzhik~\cite{komo_olla} and Bogdan, Komorowski and Marino~\cite{bogdan2022anomalous} 
considered reflective/transmissive/absorbing boundary conditions.

\vip

The aim of the present paper is twofold. First, we construct the limiting reflected stable processes, 
which is a difficult task, as we shall see. We believe that the construction method is of 
independent interest. Second, we show the convergence in law of the position of the scattering process towards 
one of these limiting processes, according to the boundary conditions.

\subsection{Overview of the literature about reflected processes}

The study of stochastic processes constrained in a domain by a reflection has a long history, and one 
of the main methods to construct such processes is to solve a Skorokhod type problem, starting with the work of 
Skorokhod~\cite{MR145598} in the half-line. Tanaka~\cite{MR529332} extended this work for the 
\textit{normal reflection} to convex regions in higher dimensions, and Lions and Sznitman~\cite{MR745330} were 
then able to treat general smooth domains, allowing for \textit{oblique reflections}. See also 
Costantini~\cite{MR1142761} and Dupuis and Ishii~\cite{MR1207237} for non-smooth domains, as well as
the contributions of Stroock and Varadhan~\cite{MR277037} through submartingale problems. While this 
overview is far from exhaustive, and many works have been handled since then, most of these 
focus on continuous processes. 

\vip 
The situation for jump processes is richer than for continuous processes. 
When a jump process $(Z_t)_{t\geq0}$ exits a domain $\Dd$, say at time $\sigma$, it may in most of the 
situations do it by a jump, so that $Z_{\sigma-}\in \Dd$ and $Z_{\sigma}\in\Dd^c$. To reflect this process, 
one needs to choose the point $R_{\sigma}$ at which it is restarted, and one can make this point 
depend on $Z_{\sigma-}$, on $Z_{\sigma}$ or on both. This dependence can be deterministic or random. 
In~\cite{anulova1991diffusional}, Anulova and Lipster extend Tanaka's work~\cite{MR529332} to handle 
càdlàg paths in convex domains for the normal reflection, corresponding to $R_{\sigma} = p_\Dd(Z_{\sigma})$, 
where $p_\Dd$ is the orthogonal projection on $\cDd$. In this case, the construction of the reflected process is 
more or less straightforward because, as noticed by Tanaka, the normal reflection enjoys some  
contraction property,  namely, $|p_\Dd(z)- p_\Dd(z')|\leq |z-z'|$, see~\cite[Lemma~2.2]{MR529332}. 
In our case, we have
$R_\sigma = [Z_{\sigma-}, Z_\sigma] \cap \pDd$, and such a contraction property does not seem to hold true, 
see Subsection~\ref{sdebof}.
In some recent works, Bogdan and Kunze~\cite{MR4779596, bogdan2024stable} constructed a reflected stable 
process in a Lipschitz domain, in the case where $R_\sigma$ is reset randomly in 
$\Dd$ through a kernel $\mu(Z_\sigma, \dr y)$.
This reflection can be considered as a special case of concatenation of Markov process, 
since each time the process is reflected, it requires a positive amount of time to be reflected again. 
Note that they suppose a condition on the kernel $\mu$ ensuring that the resulting process
never hits the boundary. We refer to Meyer~\cite{MR415784} and more 
recently to Werner~\cite{MR4247975} about concatenation of Markov processes. In the present paper,
the reflected stable process indeed hits the boundary and, like the Brownian motion, hits it uncountably
many times, which makes the situation much more involved. When $d=1$, Bogdan, Fafu{\l}a
and Sztonyk~\cite{bogdan2024nonlocal} modify the reflection mechanism of~\cite{MR4779596, bogdan2024stable}, 
so that the process
may reach the boundary in finite time, and this indeed happens when $\alpha \in(1, 2)$.
But the process is not extended 
beyond this time.
\vip
Let us also mention the works~\cite{MR2006232,MR2214908} on censored stable processes, where the jumps 
of the stable process which make it leave the domain are removed. In~\cite{MR2006232}, Bogdan, Burdzy and Chen  
construct this process through 
Dirichlet forms and show that the censored stable process hits the boundary 
if and only if $\alpha > 1$. In this case, they extend the process beyond its lifetime 
using the theory of  \textit{actively reflected} Dirichlet forms.
In~\cite{MR2214908}, 
Guan and Ma compute the generator of this process and decompose it as a semimartingale.

\vip

A recent emphasis has been put on the study of fractional P.D.E.s in domains.
Barles, Chasseigne, Georgelin and Jakobsen~\cite{MR3217703} solve a non-local P.D.E. with Neumann boundary 
conditions in the half-space, including our reflection mechanism (ii), 
which they call \textit{fleas on the window}.
They construct the resolvent of the underlying Markov process, but they do not build
the associated Markov semigroup, nor the Markov process, which is known to be a hard task. 
It seems that we obtain different boundary
conditions. Moreover, as we shall see, studying the half-space is much 
easier than dealing with general convex domains.
We also refer to the works of Defterli, D'Elia, Du,
Gunzburger, Lehoucq and Meerschaert~\cite{MR3323906}, of Baeumer, Kov\'{a}cs, 
Meerschaert and Sankaranarayanan~\cite{MR3787705}, of Baeumer, Kov\'{a}cs, Meerschaert, 
Schilling and Straka~\cite{MR3413862} and finally, to the recent work from Bogdan, Fafu{\l}a, 
Sztonyk~\cite{bogdan2024nonlocal}. The laws of the processes we construct in this article 
provide some new examples of such P.D.E.

\subsection{About Dirichlet forms}\label{adf}

As seen a few lines above, reflected Markov processes may be built \textit{via} the powerful theory of 
Dirichlet forms. It can either be done with the traditional theory of Dirichlet forms, or with the 
theory of \textit{actively reflected} Dirichlet forms which was initiated by Silverstein~\cite{silverstein}, 
Fukushima~\cite{MR236998} and Le Jan~\cite{lejan1,lejan2}, see also Chen and Fukushima~\cite{MR2849840}. 
Very roughly, this theory enables to extend a process beyond its lifetime, and to reflect 
it in a unique manner. In the case of the Brownian motion, this gives the classical reflected 
Brownian motion. The reflection mechanisms treated in those works do not allow one 
to take into account the value of $Z_\sigma$, with the notation of the previous subsection, 
and thus does not contain our reflection mechanisms. However, we believe that when $\alpha>1$ and 
$\beta=\alpha-1$, the Markov process $(R^{\alpha-1}_t)_{t\geq 0}$ may be symmetric, see 
Cesbron, Mellet and Puel~\cite[Proposition 3.4]{cesbron2020fractional} when $\Dd=\HH$.
The well-posedness result~\cite[Theorem 1.2]{cesbron2020fractional}, based on the ellipticity 
estimate~\cite[Proposition 4.1]{cesbron2020fractional}, strongly suggests that one could build
the process $R^{\alpha-1}$ \text{via} the classical theory of 
Dirichlet forms. This would not give us the Feller property, 
which is an important property to obtain the process
as limit of \textit{e.g.} the scattering equation. Although stated in a different way, this is mentioned 
in~\cite{cesbron2020fractional} just after Theorem 2.1, where they explain why they only get convergence of a 
subsequence (this has been fixed in dimension $1$ in~\cite{cesbron2021fractional}).
Finally, we believe that a construction based on Dirichlet forms would really use that $\beta=\alpha-1$: 
in other cases, it seems the
process cannot be symmetric. Le Jan~\cite{lejan1,lejan2} does not consider only symmetric processes,
but it seems that a strong asymmetry is an issue. At least, $(R^*_t)_{t\geq 0}$ is 
highly non-symmetric: it hits the boundary by a jump and leaves it continuously.

\subsection{A stochastic differential equation}\label{sdebof}

Although we do not show it explicitly, because it would be quite tedious, 
we believe that, when $\alpha \in (0,1)$, $(R_t^*)_{t\geq 0}$ should solve the following 
S.D.E.: for some Poisson measure on 
$\mathbb{R}_+ \times \mathbb{R}^d$ with intensity $\dr s |z|^{-\alpha - d}\dr z$ (describing the jumps 
of an isotropic $\alpha$-stable process),
\begin{equation}\label{eq:sde_etoile}
 R_t^* = R_0^* + \int_0^t\int_{\mathbb{R}^d}\big[\Lambda(R_\sm^*, R_\sm^* + z) - R_\sm^*\big]N(\dr s, \dr z),
\end{equation}
where for $r \in \cDd$ and $z \in \R^d$, $\Lambda(r,z)=r+z$ if $r+z\in \cDd$ and $\Lambda(r,z)=[r,r+z]\cap\pDd$
otherwise. A compensated version of this S.D.E. should be considered when $\alpha \in [1,2)$.
Concerning $(R^{(\beta)}_t)_{t\geq 0}$, an additional complicated term should be added to 
take into account the jumps outside $\pDd$.
\vip
While weak existence of a solution to~\eqref{eq:sde_etoile} is not hard, we did not manage to show uniqueness
in law. Hence we cannot show that the resulting process is Markov, and even less that it is Feller,
which is a crucial property to derive this equation from a {\it discrete} process such as the kinetic scattering process described in Subsection~\ref{ssmm}.
\vip
In the simplest case, concerning $(R_t^*)_{t\geq 0}$ when $\alpha\in (0,1)$, the well-posedness (and Feller property)
of the solution to~\eqref{eq:sde_etoile} should require a Lipschitz estimate like: for all $r,r'\in \cDd$,
$$
\int_{\R^d} (|\Lambda(r,r+z)-\Lambda(r',r'+z)|- |r-r'|) |z|^{-\alpha - d}\dr z \leq C  |r-r'|.
$$
Such a property obviously holds true when $d=1$ and $\Dd=(0,1)$ or $\Dd=(0,\infty)$, but one can show that
it fails to be true when $d\geq 2$ and $\Dd=\HH$ is the half-space. It also seems to fail in a smooth 
strongly convex domain.
One can however not 
exclude that a similar inequality holds true, using another notion of distance or any other trick,
but we did not succeed. Let us mention that in the special case where $\Dd=\HH$,
we can show the well-posedness of~\eqref{eq:sde_etoile}, studying first the first coordinate and exploiting that
the other coordinates are determined by the first one.
\vip
We will rather use the following approach: we first build the excursions of the processes reflected
in $\HH$ and then use these excursions (that we translate, rotate and stop) to build our processes reflected in $\Dd$.

\subsection{Itô's program}
We construct our processes $(R_t^*)_{t\geq0}$ and $(R_t^{(\beta)})_{t\geq0}$ by gluing  their excursions 
inside the domain $\closure{\Dd}$, 
following Itô's original idea and completing what is known as \textit{Itô's program}. 
Given a standard Markov process $(X_t)_{t\geq0}$ living in a space $\mathrm{E}$ and some point 
$b\in\mathrm{E}$, it is now well-known how to extract from $(X_t)_{t\geq0}$ a countable family of 
pieces of trajectories away from the point $b$, called \textit{excursions} away from $b$. 
Thanks to the theory of local times for Markov processes, this \textit{extraction} can be done in 
such a way that the collection of excursions, indexed by the inverse of the local time, 
forms a Poisson point process. Its intensity $\nn(\dr e)$, called \textit{excursion measure}, 
corresponds to the law of an excursion. In most cases, the measure $\nn$ has infinite mass,
reflecting the fact that, starting from $b$, the process $(X_t)_{t\geq0}$ visits $b$ infinitely 
many times immediately. We refer to the book of Blumenthal~\cite{blu1992} for a detailed 
account on excursion theory. 

\vip

One can also go the other way: given an excursion measure 
$\nn$, one can build the Poisson point process of excursions, the local time of 
the process at the point $b$, and finally concatenate the excursions and construct the process, 
see for instance~\cite[Chapter V]{blu1992}. This is called \textit{Itô's synthesis theorem}. 
In the case of excursions away from a point, this procedure is straightforward, but
to show that the resulting process is indeed a Markov process is more challenging. 

\vip
Maisonneuve~\cite{MR400417} studied excursions away from a set. His theory of 
\textit{exit systems} explains how to extract excursions of a Markov process away from a set 
in a suitable way. Let $(X_t)_{t\geq0}$ be a standard Markov process 
living in some domain $\closure{\Dd} \subset \mathbb{R}^d$. Maisonneuve first constructs 
the local time of the process on the boundary $\pDd$, and then extracts the collection of 
excursions away from the boundary (\textit{i.e.} pieces of trajectories between two 
\textit{successive} returns to the boundary), indexed again by the inverse of the local time. 
Of course, the resulting point process is not a Poisson point process, but only a point process of Markov type.
Nevertheless, Maisonneuve shows the existence of a family of measures $(\nn_x)_{x\in\pDd}$ 
corresponding to the law of an excursion starting from $x\in\pDd$, which are related by 
the \textit{excursion formula}, see~\cite[Theorem 4.1]{MR400417}.

\vip

Extending Itô's synthesis theorem to sets seems much more involved, as it is not clear how to 
construct the point process of excursions and to glue the excursions together. As far as we know, 
the only general result in this direction is that of Motoo~\cite{MR220349}, also detailed 
in~\cite[Chapter VII]{blu1992}. Without going into details, Motoo's theory is rather analytical and 
requires lots of technical assumptions which may be hard to check in practice, but it yields a 
construction theorem. However, quoting Blumenthal~\cite[page 258]{blu1992}, his result 
\textit{is not a synthesis theorem, in the sense that the desired $X$ is not constructed by 
hooking together paths}. More importantly, Motoo's theory requires the existence of the 
boundary process, \textit{i.e.} the \textit{successive} points on the boundary visited by 
the process, which is obtained with a time-change of the process $(X_t)_{t\geq0}$ by the inverse 
of its local time on the boundary. The existence of this process in our case is far from being easy,
except when $\Dd$ is a Euclidean ball or the half-space.

\subsection{Main ideas of the construction of the reflected stable processes}
In our setting, the point process of excursions is not Poisson, but we may use the following facts: 

\vip\noindent
(i) an excursion of an isotropic stable process in the half-space 
immediately enters a smooth domain tangent to the half-space at the starting point, see Lemma~\ref{imp}; 

\vip\noindent
(ii) the law of an isotropic stable process is invariant by any isometry.

\vip
These two facts imply that the law of an excursion starting from a point $x\in \pDd$
can be obtained by translating and rotating an excursion from the half-space starting at $0$, 
which is then stopped when leaving the domain. In other words, an excursion issued from $x$ 
can be obtained by stopping an excursion in the half-space $\HH_x$ tangent to $\pDd$ at $x$
when leaving 
the domain; and the latter one is itself obtained by translating and rotating an excursion 
in $\HH$. 

\vip
We give ourselves a family $(A_x)_{x\in\pDd}$ of isometries such that 
for every $x\in\pDd$, $A_x$ sends $\be_1$ to $\bn_x$, where $\be_1 = (1, 0, \ldots, 0)$ and $\bn_x$ 
is the inward unit normal vector to $\Dd$ at $x$. We now introduce the space $\cE$ of half-space excursions, 
\textit{i.e.} of càdlàg paths which take values in the closure of 
$\mathbb{H}$, and which are stopped when leaving it. We also give ourselves a  Poisson measure 
$\Pi = \sum_{u\in\mathrm{J}}\delta_{(u,e_u)}$
of stable excursions in $\HH$ starting from $0$. This is a random measure on
$\mathbb{R}_+ \times \cE$ with intensity $\dr u \nn(\dr e)$ for some excursion measure $\nn$
in the half-space. Then, for every $x\in\pDd$ and every $u\in\mathrm{J}$, 
the excursion in $\HH_x$ is $e^x(s) = x+ A_x e(s)$. Its law does 
not depend on the choice of $A_x$ because $\nn$ is left invariant by any isometry sending $\be_1$ to $\be_1$.
\vip
 
We will consider two different kinds of excursion measures. 
\vip
\noindent $\bullet$ A first one, $\nn_*$,
under which the excursion always starts at $0$ and with which we 
will build $(R_t^*)_{t\geq0}$. This measure $\nn_*$ is defined as the intensity of the Poisson measure of excursions
of the isotropic $\alpha$-stable process reflected in the half-space. Although we do not prove it because
it would be useless for our study, 
we believe that for $Z$ the isotropic $\alpha$-stable process starting from $x$ under $\PP_x$,
for some constant $a_*\in (0,\infty)$, in a sense to be precised,
\begin{equation}\label{newformula}
\nn_*=a_*\lim_{x\in \HH,x\to 0} \frac{\PP_x((Z_{t\land \ell(Z)})_{t\geq 0}\in \cdot)}{(x\cdot\be_1)^{\alpha/2}},
\quad \text{where} \quad \ell(Z)=\inf\{t>0 : Z_t \notin \HH\}.
\end{equation}
Such a formula holds true in dimension $1$, see Chaumont and Doney~\cite[Corollary 1]{chdo}.

\vip
\noindent $\bullet$ A second one, $\nn_\beta$, indexed by
$\beta\in(0,\alpha/2)$ and with which we will build $(R_t^{(\beta)})_{t\geq0}$. Under $\nn_\beta$,
the excursion is simply an isotropic $\alpha$-stable process starting from a 
$|x|^{-\beta-d}\indiq_{\{x\in \HH\}} \dr x$-distributed point, stopped when it leaves $\HH$.

\vip
One crucial point for piecing together excursions, and this is actually the only deep issue, 
is to construct the \textit{boundary process} $(b_u)_{u\geq0}$, \textit{i.e.} the \textit{successive} 
points on the boundary visited by the process. For $x\in\pDd$ and $e\in\cE$, we set 
$\bar{\ell}_x(e) = \inf\{t \geq0, \: x+ A_xe(t) \notin \Dd \}$ and  
$\cu(x, e) = x + A_xe(\bar{\ell}_x(e)-)$ and $\co(x,e) = x + A_xe(\bar{\ell}_x(e))$.
Then $\cu(x, e)$ (resp. $\co(x, e)$) is the position of the excursion just 
before (resp. after) leaving the domain. We now introduce 
$g_x(e) = \Lambda(\cu(x,e), \co(x,e))$, which is the point intersecting $\pDd$ and $[\cu(x,e), \co(x,e)]$.
Knowing $b_{u-}$, the next position should be $g_{b_{u-}}(e_u)$, so that the boundary process should solve
(if $R_0=x\in \pDd$)
\begin{equation}\label{eq:boundary_intro}
 b_u = x + \int_0^u\int_{\mathcal{E}}\big[g_{b_{v-}}(e) - b_{v-}\big]\Pi(\dr v, \dr e).
\end{equation}
One of the main achievements of this paper is to show that, at least when $d=2$ (a weaker but sufficient
result holds when $d\geq 3$), this S.D.E. admits
a pathwise unique solution, continuous with respect to its initial condition, 
when the domain $\Dd$ is smooth and strongly convex. 

\vip
To prove that~\eqref{eq:boundary_intro} is well-posed, we show, see 
Proposition~\ref{ttoopp}, the following Lipschitz estimate:
\begin{equation}\label{lipes}
\text{for all $x,x'\in \pDd$,}\quad 
 \int_\cE \Big||g_x(e)-g_{x'}(e)|- |x-x'|\Big| 
\nn (\dr e) \leq C (|x-x'|+||A_x-A_{x'}||)
\end{equation}
for some $C >0$. Here are the two main ingredients which enable us to obtain this inequality.

\vip\noindent
(a) Recalling that $g_x(e) = \Lambda(\cu(x,e), \co(x,e))$, we first obtain estimates 
on the joint law of $(\cu(x,e), \co(x,e))$ under $\nn$, see Proposition~\ref{underover},
obtained from sharp known estimates on the Green function of 
isotropic stable processes found in Chen~\cite{chen1999}.

\vip\noindent
(b) We establish some geometric inequalities (see Proposition~~\ref{tyvmmm} and Appendix~\ref{ageo}) 
that allow us to bound~$||g_x(e)-g_{x'}(e)|- |x-x'||$ by a quantity that we can then integrate 
with respect to the previous estimates. 

\vip
To conclude that~\eqref{eq:boundary_intro} is well-posed, it would be convenient to choose the 
family $(A_x)_{x\in\pDd}$ so that $x \mapsto A_x$ is Lipschitz. This is possible when $d=2$, but
not when e.g. $d=3$, due to the hairy-ball theorem. However, for any $y \in \pDd$, one can find a 
family of isometries $(A_x^y)_{x\in\pDd}$ such $x \mapsto A_x^y$ is locally Lipschitz on $\pDd\setminus \{y\}$. 
Therefore, for each $y\in\pDd$,~\eqref{eq:boundary_intro} is well-posed 
(with the choice $(A_x^y)_{x\in\pDd}$) until the hitting time of $y$ of the boundary process. 
Using that the law of $(b_u)_{u\geq0}$ actually does not depend on the isometries
and that  $(b_u)_{u\geq0}$ a.s. has finite variations (so that the dimension of its image is smaller than $1$), 
we are able to show that for almost every $y \in \pDd$, this hitting time is infinite. Hence, 
the equation is well-posed in law for a.e. $y \in \pDd$.

\vip

From this boundary process, one then builds the inverse of the local time at the boundary,
which is obtained by summing the lengths of the excursions:
\begin{equation}\label{eq:inv_local}
 \tau_u = \int_0^u\int_\cE\bar{\ell}_{b_{v-}}(e)\Pi(\dr v, \dr e).
\end{equation}
As we shall see, when the considered excursion measure is
$\nn_\beta$, it is well-defined only for $\beta < \alpha /2$. We then define the local
time on the boundary $(L_t)_{t\geq0}$ as the right-continuous inverse of $(\tau_u)_{u\geq0}$. 
We can now define the process $(R_t)_{t\geq0}$ starting from $x\in\pDd$ by setting
\[
 R_t = b_{L_t-} + A_{b_{L_t-}}e_{L_t}(t- \tau_{L_t-})
\]
when $\tau_{L_t} > t$, \textit{i.e.} when $t$ lies inside the excursion which starts at time $\tau_{L_t-}$.
Otherwise, 
this means that $R_t$ is on the boundary and we set $R_t = b_{L_t}$. 

\vip

When the starting point 
$x \in \Dd$ is not on the boundary, we give ourselves a stable process $(Z_t)_{t\geq0}$ starting
from $x$, then set $\sigma = \inf\{t> 0, \: Z_t \notin \closure{\Dd}\}$, $R_t = Z_t$ for 
$t\in[0, \sigma)$ and $R_\sigma = \Lambda(Z_{\sigma-}, Z_\sigma) \in \pDd$. Finally $(R_{t+\sigma})_{t\geq0}$ 
is constructed as above, starting from $R_{\sigma}\in \pDd$.

\vip
We show that the law of $(R_t)_{t\geq0}$ starting from $x\in\closure{\Dd}$, that we denote 
$\mathbb{Q}_x$, is uniquely defined and does not depend on the choice of $(A_x)_{x\in\pDd}$. 
Our main result states that the family of laws $(\mathbb{Q}_x)_{x\in\closure{\Dd}}$ defines a strong Markov 
process on the space of $\closure{\Dd}$-valued càdlàg functions, which is Feller. We study 
its generator and the P.D.E. satisfied by its semigroup.

\subsection{Other related works on reflected Markov processes}

In~\cite{wata}, Watanabe uses a similar procedure to build some continuous 
diffusion processes with Wentzell boundary conditions in the half-space. 
He is able, under a few assumptions, to reduce to the case where the first coordinate (the one to be reflected)
is a Brownian motion. 
The reflection is governed by a linear operator, involving 
first and second order derivatives, as well as a non-local part. The boundary process then solves
some S.D.E. with diffusion, drift and jumps. The well-posedness of this boundary process
is quickly checked, since he works in the half-space and since the first coordinate 
consists of a Brownian motion, so that the law of the length of the excursion does not depend on
the starting point in $\partial \HH$.

\vip
In~\cite{hsup}, Hsu starts from a
Brownian motion $(X_t)_{t\geq 0}$ reflected in a bounded smooth domain, 
and describes the conditional law of $(X_t)_{t\geq 0}$ knowing its boundary process $(b_u)_{u\geq 0}$.
This allows him to build $(X_t)_{t\geq 0}$ from $(b_u)_{u\geq 0}$, by 
sampling some excursions, conditionally on $(b_u)_{u\geq 0}$.

\subsection{Main ideas of the anomalous diffusion limit}

Let us quickly explain the main steps we use to prove that the position $(\bX^\e_t)_{t\geq 0}$ 
of the scattering process converges to $(R^*_t)_{t\geq 0}$ or $(R^{(\beta)}_t)_{t\geq 0}$,
depending on the boundary conditions.
\vip

We first show that up to a small time change, $(\bX^\e_t)_{t\geq 0}$ is the linear interpolation of a
Markov process $(R^\e_t)_{t\geq 0}$
(the {\it couple} $(\bX^\e_t,\bV^\e_t)_{t\geq 0}$ is Markov, but not $(\bX^\e_t)_{t\geq 0}$
{\it alone}). It thus suffices to study the convergence of $(R^\e_t)_{t\geq 0}$.
\vip
The process $(R^\e_t)_{t\geq 0}$ actually solves a S.D.E. which is close to~\eqref{eq:sde_etoile}. 
We show that it shares the same structure as  $(R^*_t)_{t\geq 0}$ or $(R^{(\beta)}_t)_{t\geq 0}$: 
it can be built by translating/rotating/stopping the excursions from the half-space 
of a continuous-time random walk $(Z^\e_t)_{t\geq 0}$.
This random walk lies in the domain of attraction of the isotropic $\alpha$-stable process and starts, roughly,
from $Z^\e_0=\e^{1/\alpha}O$, where $O\sim2\rG(v)\indiq_{\{v\cdot \be_1>0\}}\dr v$, recall that 
$\rG$ appears in the boundary condition of the scattering equation. We then proceed as follows.
\vip
\noindent (i) We study the excursion measure of $(Z^\e_t)_{t\geq 0}$, which is proportional to the
 law of $(Z^\e_t)_{t\geq 0}$ stopped when exiting $\HH$. We show that, by choosing correctly the 
multiplicative coefficient, it converges in some sense to $\nn_*$
if $\rG$ has a finite moment of order $\alpha/2$ or to $\nn_\beta$ 
if $\mathrm{G}(v) \sim |v|^{-\beta - d}$ for some $\beta \in (0, \alpha/2)$. 
This is done in Section~\ref{stech} and is particularly delicate in the case of $\nn_*$. 
We have to adapt many ideas found in Doney~\cite{doney1985conditional} to our 
situation where $d\geq 2$ and where
the random walk does not start from $0$.

\vip
\noindent (ii) We next prove that the boundary process $(b^\e_t)_{t\geq 0}$, 
representing the successive points of $\pDd$ visited by 
$(R^\e_t)_{t\geq 0}$, converges in law to the boundary process $(b_t)_{t\geq 0}$ of 
$(R^*_t)_{t\geq 0}$ or $(R^{(\beta)}_t)_{t\geq 0}$. For this, we proceed by tightness/uniqueness, because
we are far from being able to prove an $\e$-version of the Lipschitz estimate \eqref{lipes}. 
Of course, the well-posedness of the S.D.E. \eqref{eq:boundary_intro} satisfied by 
$(b_t)_{t\geq 0}$ is crucial in this step.

\vip
\noindent (iii) Finally, we prove the convergence of the whole process $(R^\e_t)_{t\geq 0}$
by studying what happens inside the excursions.

\vip

It might be possible, with more work, to treat the critical case 
$\mathrm{G}(v) \sim |v|^{-\alpha/2 - d}$
or even $\mathrm{G}(v) \sim |v|^{-\alpha/2 - d}\ell(|v|)$ for some slowly varying function $\ell$: 
we expect that $(\bX^\e_t)_{t\geq 0}\to (R^*_t)_{t\geq 0}$ in such a case.
It might also be possible to replace the condition  $\mathrm{G}(v) \sim |v|^{-\beta - d}$ by  
$\mathrm{G}(v) \sim |v|^{-\beta - d}\ell(|v|)$ when 
$\beta \in (0,\alpha/2)$, without affecting the normalization nor the limiting process $(R^{(\beta)}_t)_{t\geq 0}$.
One might finally assume that $\rF(v)\sim |v|^{-\alpha-d}\ell(|v|)$ instead of 
$\rF(v)\sim |v|^{-\alpha-d}$, if one modifies suitably the normalization.

\subsection{Last comments}

We do not treat the case of the half-space:
since $\HH$ is unbounded, this would add some small technical issues, but
this case would be considerably easier from many other points of view.

\vip
Although our proofs rely on the strong convexity of 
the domain, we believe that this assumption is not essential, and the construction
might be carried out for arbitrary smooth domains. However, there is at least one place where we deeply use 
the strong convexity of $\Dd$: the proof of the geometric inequalities stated in Proposition~\ref{ttoopp}.
As shown in Remark~\ref{strange},
these inequalities fail to be true for general (non strongly) convex domains, 
although they are obviously satisfied
in the (flat) half-space. The situation is thus rather intricate.

\subsection{Summary}

Let us summarize the main achievements of this work. We build two kinds of isotropic 
$\alpha$-stable processes reflected 
in strongly convex domains and show that these processes arise as scaling limits
of the kinetic scattering model with diffusive 
boundary conditions, as the collision rate tends to infinity. This convergence result
extends the work of Cesbron, 
Mellet and Puel~\cite{cesbron2020fractional, cesbron2021fractional} to general convex domains 
(not only $\Dd=\HH$),  and to a broader range of boundary velocity distributions $\mathrm{G}$. 
We establish the  existence of a \textit{subcritical} regime, when $\mathrm{G}$ has a moment of order $\alpha /2$, 
and a \textit{supercritical} regime, where all the values  of $\beta \in(0,\alpha /2)$ can be taken. 

\vip
Beyond its connection to kinetic models, our method of constructing these processes by concatenating 
translated and rotated excursions appears to be novel. To our knowledge, the only 
results in this direction are those of Motoo~\cite{MR220349}, which is not really an Itô synthesis, 
and of Watanabe~\cite{wata}, which treats the case of the half-space for continuous processes.  
It seems that we handle the first Itô synthesis for a Markov process inside a domain which 
is not the half-space. This might give 
a new perspective on the construction of reflected jump Markov processes.

\subsection{Plan of the paper}
In Section~\ref{sec:main_result}, we introduce some notations and state our main 
results. In Section~\ref{sec:excursion}, we establish a few 
properties of the excursion measures (and of the stable process), that will be used in the whole paper.
Section~\ref{sec:crucial} is dedicated to the proof of our main Lipschitz estimate, which is crucial 
to show the well-posedness (and Feller property) of the limiting boundary processes. As outlined above, 
this estimate is straightforward when the domain $\Dd$ is a Euclidean ball, so that this 
delicate section may be skipped at first reading. In Section~\ref{sec:start_bounda}, 
we show the existence and uniqueness of the reflected stable processes, when starting from a 
point on the boundary, and establish their continuity with respect to their initial position. 
In Section~\ref{any}, we introduce the reflected stable processes when starting from anywhere,
show that they are Markov and Feller, and prove a few other properties. We study in
Section~\ref{sec:pde} the infinitesimal generators of our processes and establish some P.D.E.s
satisfied by their laws. In Section~\ref{scatM}, we establish a few properties of the scattering process, 
introduce a modified Markov
scattering process, and show that it is sufficient to prove 
the convergence 
of this modified process. The convergence of the modified process is shown in Section~\ref{sconvmark}, 
admitting a few 
results on conditioned random walks, that are established in Section~\ref{stech}.
Appendix~\ref{ageo} is dedicated 
to the proofs of the geometric inequalities used in Section~\ref{sec:crucial}, of two lemmas concerning the
parameterization of the domain, of a result about regular families of isometries, 
and of the continuity of the cutoff function $\Lambda$.
In Appendix~\ref{tf}, we discuss our sets of test functions. We recall some more or less well-known facts about 
Skorokhod's $\JS$ and $\MS$ topologies in Appendix~\ref{sko}. Finally, we show the link between the 
scattering process and the scattering P.D.E. in Appendix~\ref{deredp}.

\subsection*{Acknowledgements}

We warmly thank Quentin Berger, Thomas Duquesne and Camille Tardif for many fruitful discussions.

\section{Main results}\label{sec:main_result}

\subsection{Notation}\label{nono}

We fix $d\geq2$ and $\alpha \in (0,2)$ for the whole paper and introduce some notation
of constant use.

\vip

\noindent {\bf Stable process.} 
A (normalized) $d$-dimensional isotropic $\alpha$-stable
process $(Z_t)_{t\geq 0}$ issued from $x\in \R^d$, an ISP$_{\alpha,x}$ in short, is given by
\begin{equation}\label{eqs}
Z_t=x+ \int_0^t\int_{\{|z|\leq 1\}} z \tilde N(\dr s,\dr z) + \int_0^t\int_{\{|z|>1\}}z 
N(\dr s,\dr z),
\end{equation}
where $N$ is a Poisson measure on $\R_+\times\R^d\setminus\{0\}$ with intensity measure 
$\dr s |z|^{-d-\alpha} \dr z$ and where $\tilde N$ is the associated compensated
Poisson measure. We call $\PP_x=\cL((Z_t)_{t\geq 0})$ its law.
\vip

\noindent {\bf Excursion spaces.} We denote by $\HH = (0,\infty) \times \R^{d-1}=\{ x \in \R^d : x\cdot \be_1>0\}$ 
the upper half-space. Here $\be_1$ is the first vector of the canonical basis.
For $e\in\DD(\R_+,\R^d)$, the space of càdlàg functions from $\R_+$ to $\R^d$, we set 
$\ell(e) = \inf\{t> 0, \: e(t) \notin {\HH}\}$. 
We introduce
$$
 \mathcal{E} = \big\{e \in \DD(\R_+,\R^d) \; : \; \ell(e)\in(0,\infty) \text{ and for any }t\geq\ell(e), \: 
e(t) = e(\ell(e))\big\}.
$$
This space is endowed with the usual Skorokhod $\JS$-topology, see Appendix~\ref{sko},
and the associated Borel 
$\sigma$-field. We consider the two subspaces 
$$
\cEc=\{e \in \cE \; : \; e(0)=0\} \qquad \hbox{and} \qquad \cEd=\{e \in \cE \; : \; e(0)\in \HH\}.
$$

\noindent {\bf A first excursion measure: $\nn_*$.}
We consider, on a filtered probability space $(\Omega, \mathcal{F}, (\mathcal{F}_t)_{t\geq0}, \PP)$, 
an ISP$_{\alpha,0}$ $(Z_t)_{t\geq0}$ and define its excursions in $\HH$.
We denote by $Z^1_t=Z_t\cdot \be_1$ the first coordinate of $Z_t \in \R^d$. It is a classical 
fact, see Lemma~\ref{oncemore}, that the process 
$(Y_t)_{t\geq0}=(Z_t^1 - \inf_{s\in[0,t]}Z_s^1)_{t\geq0}$ is Markov, possesses a local time $(\xi_t)_{t\geq0}$ 
at $0$, 
and that this local time is uniquely defined if we impose $\E[\int_0^\infty e^{-t} \dr \xi_t]=1$.
Its right-continuous inverse $(\gamma_u=\inf\{s\geq 0 : \xi_s > u\})_{u\geq0}$ is a $(1/2)$-stable 
$\R_+$-valued subordinator, 
see Lemma~\ref{oncemore}. We introduce $\JJ = \{u\geq0, \: \Delta\gamma_u > 0\}$ and set, for $u \in \JJ$,
\begin{equation}\label{et}
e_u = (Z_{(\gamma_{u-} + s)\wedge\gamma_u} - Z_{\gamma_{u-}})_{s\geq0}.
\end{equation}
We will see in Lemma~\ref{oncemore} that a.s., $e_u\in \cEc$ and $\ell(e_u)=\Delta \gamma_u$ 
for all $u \in \JJ$.
The strong Markov property and the Lévy character of $(Z_t)_{t\geq0}$
classically imply that $\Pi_* = \sum_{u\in \JJ}\delta_{(u, e_u)}$
is a time-homogeneous $(\mathcal{F}_{\gamma_u})_{u\geq0}$-Poisson measure on $\R_+ \times \cE$.
Its intensity measure is thus of the form $\dr u \nn_*(\dr e)$, for some $\sigma$-finite measure $\nn_*$ on $\cE$
(carried by $\cE_0$).  
\vip
Note that in the definition of $e_u$, we keep track of the last jump of $Z$,
{\it i.e.} we set $e_u(\ell(e_u))=Z_{\gamma_u}-Z_{\gamma_u-}$, unlike what is usually done for positive 
excursions of  one-dimensional Lévy processes. This will be important for our purpose, see~\eqref{g1} below.

\vip

\noindent {\bf A second excursion measure: $\nn_\beta$.}  
For all $x \in \HH$, under $\PP_x$, $(Z_{t\land \ell(Z)})_{t\geq 0}$
a.s. belongs to $\cEd$: $Z_0=x \in \HH$, $\ell(Z)>0$ by right-continuity of the paths, and
$\ell(Z)<\infty$ a.s. by recurrence of the first coordinate of $(Z_t)_{t\geq 0}$. 
For $\beta>0$, we introduce the measure $\nn_\beta$
on $\cE$ (carried by $\cEd$) defined, for all Borel subset $A$ of $\cE$, by 
\begin{equation}\label{nnb}
\nn_\beta(A) = \int_{x \in \HH} |x|^{-d-\beta} \PP_x\big((Z_{t\land \ell(Z)})_{t\geq 0} \in A\big)\dr x.
\end{equation}
Admitting~\eqref{newformula}, one could show that
$\nn_*=\lim_{\beta \to (\alpha/2)-} c_{\beta,\alpha}^{-1} a_* \nn_\beta$ in a sense to be precised, where
$c_{\beta,\alpha}=\int_{x\in\HH, |x|<1}|x|^{-d-\beta}x_1^{\alpha/2}\dr x$ blows up as $\beta$ increases to $\alpha/2$.

\vip

\noindent {\bf Some estimates}. For $e \in \cE$, we set $M(e)=\sup_{t \in [0,\ell(e)]} |e(t)|$.
By Lemmas~\ref{qdist} and~\ref{qdist2} below,
\begin{equation}\label{se1}
\text{if $\beta \in \{*\}\cup(0,\alpha/2)$,} \qquad
\int_\cE  [\ell(e)\land 1 + M(e)\land 1] \nn_\beta(\dr e) <\infty,
\end{equation}
but $\nn_\beta(\ell>1)=\infty$ when $\beta \geq \alpha/2$.

\vip

\noindent {\bf Notation about the domain.} We consider an open strictly
convex domain $\Dd \subset \R^d$ at least 
of class $C^2$. For $x \in \pDd$, let $\bn_x$ be the inward unit normal vector.
For $y \in \cDd$ and $z\in \R^d$, we set
\begin{equation}\label{Lambda}
\Lambda(y,z)= \left\{\begin{array}{ll}
z & \hbox{if }\; z \in \cDd \\[4pt]
(y,z] \cap \pDd & \hbox{if } \; z \notin \cDd \hbox{ and } (y,z] \cap \pDd \neq \emptyset\\[4pt]
y & \hbox{if }\; z \notin \cDd \hbox{ and } (y,z] \cap \pDd = \emptyset
\end{array} \right\} \in \cDd,
\end{equation}
where $(y,z]=\{y+ \theta (z-y) : \theta \in (0,1]\}$.
This definition makes sense: since $\Dd$ is strictly convex, $(y,z]\cap \pDd$ has at most one element. We will see in Lemma~\ref{Lambdacon} that $\Lambda$ is continuous on $\cDd\times\R^d$.

\vip

{\it The main idea is that $\Lambda(y,z)$ is the {\it post-jump} position of our process when it tries 
to jump from $y\in \cDd$ to $z$, as if such a try was done by
crossing the segment $(y,z]$ at infinite speed: if it hits the boundary, it actually jumps to the point
$(y, z]\cap \pDd$.
Such a point does not exist if $y \in \pDd$ and if $z-y$ is unfavorably oriented,
in which case the process does not jump.}

\vip

\noindent{\bf Some isometries.} For $x \in \pDd$, we call $\cI_x$ the set of linear isometries
of $\R^d$ sending $\be_1$ to $\bn_x$. 
Observe that $\be_1$ is the unit inward normal vector of the domain $\HH$, at any point of its boundary.
For $x \in \pDd$, $A\in \cI_x$ and $u \in \R^d$, we set $h_x(A,u)=x+Au$.
For $x \in \pDd$, we introduce $\HH_x=h_x(A,\HH)$: it is the half-space tangent 
to $\pDd$ at $x$ and it does not depend on 
$A\in \cI_x$. Since $\Dd$ is convex, it holds that $\Dd \subset \HH_x$.

\vip
 
{\it We will build our process from excursions in the half-space $\HH$. We thus need, for each $x\in \pDd$, 
to map $\HH$ to $\HH_x$, which is done
through the map $h_x$. When $d=2$,
one can find a family $(A_x)_{x\in \pDd}$ in such a way that
$A_x \in \cI_x$ and $x\mapsto A_x$ is Lipschitz continuous.
However, if e.g. $d=3$ and if $\Dd$ is a ball (or any open and bounded convex domain), 
it is impossible to find a family $(A_x)_{x\in \pDd}$ such that $x\mapsto A_x$ is continuous: 
this follows from the hairy ball theorem.}

\vip

\noindent{\bf More notation about the domain.}
For all $x \in \pDd$, $A \in \cI_x$ and $e \in \cE$, we introduce
\begin{equation}\label{lx}
\cl_x(A,e)=\inf \{t> 0, \: h_x(A,e(t)) \notin {\Dd}\}.
\end{equation}
Since $\{y \in \R^d : h_x(A,y)\in\Dd\} \subset \HH$ by convexity of $\Dd$, we have
$\cl_x(A,e) \leq \ell(e)$. We also set
\begin{equation}\label{g1}
g_x(A,e) = \Lambda\Big(h_x(A,e(\cl_x(A, e)-), h_x(A,e(\cl_x(A,e)))\Big) \in \pDd,
\end{equation}
with the convention that $e(0-)=0$, so that $h_x(A,e(0-))=x$.
\vip
{\it The point $g_x(A,e)$ is built as follows. First, $h_x(A,e)$ is the image of $e$ by the isometry  
$h_x(A,\cdot)$ mapping $\HH$ to $\HH_x$. 
Then $\cl_x(A,e)$ and $g_x(A,e)$ are the instant and position at which $h_x(A,e)$ exits
from $\Dd$, doing
as if the jump at time $\cl_x(A,e)$ was performed by crossing at infinite speed the segment
$[h_x(A,e(\cl_x(A, e)-), h_x(A,e(\cl_x(A,e)))]$. If $h_x(A,e(0)) \notin \Dd$, 
we naturally have $\cl_x(A,e)=0$ and $g_x(A,e)=\Lambda(x,h_x(A,e(0)))$.}

\vip

\noindent{\bf Final notation.} We denote by $(\be_1,\dots,\be_d)$ the canonical basis of $\R^d$, 
by $|\cdot|$ the Euclidean norm on $\R^d$ and by
$\Sp_{d-1}=\{w \in \R^d : |w|=1\}$ the Euclidean sphere.
For $n\in \N_*$, $u\in \R^n$ and $r>0$, 
let $B_n(u,r)=\{w \in \R^n : |w-u|<r\}$.
For $A:\R^d \to \R^d$ a linear mapping, we set $||A||=\sup_{u \in \Sp_{d-1}} |Au|$.

\subsection{The reflected stable processes}\label{secrsp}
We will work under the following conditions on the domain.

\begin{assumption}\label{as}
The set $\Dd\subset \R^d$ is $C^3$, non-empty, open, bounded and strongly convex.
\end{assumption}

\begin{remark}\label{asp}
Let us make precise what we mean in Assumption~\ref{as}:
$\Dd\subset \R^d$ is non-empty, open, bounded, and there exist
$\e_0\in (0,1)$ and $\eta>0$ such that for each $x \in \pDd$, there exists $A \in \cI_x$ and 
$\psi_x:B_{d-1}(0,\e_0)\to \R_+$ of class $C^3$ such that $\psi_x(0)=0$, $\nabla \psi_x(0)=0$, 
Hess $\psi_x(v)\geq \eta I_{d-1}$ and $|D^2\psi_x(v)|+|D^3\psi_x(v)|\leq \eta^{-1}$ for all $v \in B_{d-1}(0,\e_0)$ and 
\begin{equation}\label{casp}
\Dd \cap B_{d}(x,\e_0)=\{h_x(A,u) : u\in B_d(0,\e_0) \;\; \text{and} \;\; u_1>\psi_x(u_2,\dots,u_d)\}.
\end{equation}
\end{remark}

\begin{remark}\label{imp4}
Under Assumption~\ref{as}, there is $r>0$ such that for all 
$x \in \pDd$, ${B}_d(x+r\bn_x,r)\subset \Dd$. For $A\in \cI_x$ and $y \in B_d(r\be_1,r)$, 
$h_x(A,y) \in {B}_d(x+r\bn_x,r)$
because $|h_x(A,y)-x-r\bn_x|=|Ay-r\bn_x|=|Ay-rA\be_1|=|y-r\be_1|$.
Thus $\cl_x(A,e)\geq \ell_r(e):=\inf\{t>0 : e(t) \notin B_d(r\be_1,e)\}$ for all $e\in \cE$.
As we will see in Lemma~\ref{imp}, $\ell_r(e)>0$ for $\nn_*$-a.e. $e\in \cE$.
\end{remark}

Let us now define the reflected stable processes starting from a point on the boundary $\pDd$.

\begin{definition}\label{dfr1}
Fix
$\beta \in \{*\}\cup(0,\alpha/2)$, $x\in \pDd$ and suppose Assumption~\ref{as}.
We say that $(R_t)_{t\geq 0}$ is an $(\alpha,\beta)$-stable process reflected in $\cDd$ issued from $x$ if
there exists a filtration $(\cG_u)_{u\geq 0}$, a $(\cG_u)_{u\geq 0}$-Poisson measure 
$\Pi_\beta=\sum_{u\in \JJ}\delta_{(u,e_u)}$ on $\R_+\times\cE$ with intensity measure $\dr u \nn_{\beta}(\dr e)$,
a càdlàg $(\cG_u)_{u\geq 0}$-adapted $\pDd$-valued process $(b_u)_{u\geq 0}$ and a $(\cG_u)_{u\geq 0}$-predictable 
process $(a_u)_{u \geq 0}$ such that a.s., for all $u \geq 0$, $a_u \in \cI_{b_{u-}}$ and
\begin{equation}\label{sdeb}
b_u = x + \int_0^u\int_{\mathcal{E}}\Big(g_{b_\vm}(a_{v}, e) - b_\vm\Big)\Pi_\beta (\dr v, \dr e)
\end{equation}
such that, introducing the càdlàg increasing $(\cG_u)_{u\geq 0}$-adapted $\R_+$-valued process 
\begin{equation}\label{sdet}
\tau_u = \int_0^u\int_{\mathcal{E}}\cl_{b_\vm}(a_v, e)\Pi_\beta(\dr v, \dr e)
\end{equation}
and its generalized inverse $L_t=\inf\{u\geq 0 : \tau_u>t\}$ for all $t\geq 0$, we have, for all $t\geq 0$,
\begin{equation}\label{defR}
R_t = 
\begin{cases}
h_{b_{L_t -}}(a_{L_t},e_{L_t}(t-\tau_{L_t -})) & \text{if }  \tau_{L_t} > t, \\
g_{b_{L_t-}}(a_{L_t},e_{L_t}) & \text{if }   L_t \in \JJ \text{ and } \tau_{L_t} = t, \\
b_{L_t} &\text{if } L_t \notin  \JJ.
\end{cases}
\end{equation}
\end{definition}

As already mentioned, in dimension $2$, we can find
a family $(A_y)_{y \in \pDd}$ such that $A_y \in \cI_y$ and $y\mapsto A_y$ is Lipschitz-continuous.
In such a case, the choice $a_u=A_{b_{u-}}$ is the simplest one. In the general case, we will choose $(a_u)_{u\geq 0}$
in a quite complicated way, see Proposition~\ref{newb}-(b). We thus prefer to allow for any predictable process 
$(a_u)_{u\ge 0}$ such that $a_u\in \cI_{b_{u-}}$, and 
to show that the law of the resulting process does not depend on this choice.

\vip

The Poisson integrals in~\eqref{sdeb} and~\eqref{sdet} make sense, by~\eqref{se1} and since
$|g_x(A,e)-x|\leq M(e)$ and $\cl_x(A,e) \leq \ell(e)$. Moreover, we will see in Lemma~\ref{tti} 
that $(\tau_u)_{u\geq 0}$ is automatically a.s. strictly increasing 
and that $\lim_{u\to \infty} \tau_u=\infty$ a.s., which classically implies the following.

\begin{remark}\label{rkr2}
The map $t\mapsto L_t$ is continuous from $\R_+$ into $\R_+$.
For all $t\geq 0$, $t\in [\tau_{L_t-},\tau_{L_t}]$.
For all $t\geq 0$, all $u\geq 0$, $L_t=u$ if and only if $t\in [\tau_{u-},\tau_u]$. For
all $u\geq 0$, $L_{\tau_u}=L_{\tau_{u-}}=u$.
\end{remark}

Although we build things in the reverse
way, the main idea is that $(L_t)_{t\geq 0}$ is the local time of $(R_t)_{t\geq 0}$ 
at $\pDd$, that
$(\tau_u)_{u\geq 0}$ is its generalized inverse and that $b_u=R_{\tau_u}$ is the {\it boundary process} 
describing the successive positions, in a suitable time-scale, of $R$ when hitting $\pDd$.

\vip

If $L_t \notin \JJ$, then $R_t\in \pDd$ and $R_t=b_{L_t}=b_{L_{t}-}$, since
$L_t$ is a continuity point of $b$.
If $L_t \in \JJ$, there are two possibilities. If first $\tau_{L_t}>t$, 
which means that 
$t$ is not at the right extremity of an excursion, $R_t$ is built by using the 
excursion $e_{L_t}$, mapped to $\HH_{b_{L_t-}}$, by setting $R_t=h_{b_{L_t-}}(a_{L_t},e_{L_t}(t-\tau_{L_t-}))$.
If next $\tau_{L_t}=t$, which means that 
$t$ is precisely the end of an excursion, we set $R_t= g_{b_{L_t-}}(a_{L_t},e_{L_t})$, and
we actually also have $R_t=b_{L_t}$, because 
$b_{L_t}=b_{L_{t}-}+(g_{b_{L_t}-}(a_{L_t},e_{L_t})-b_{L_{t}-}) =g_{b_{L_t-}}(a_{L_t},e_{L_t})$ by~\eqref{sdeb}. 

\vip
As seen in Remark~\ref{imp4}, for all $x\in\pDd$, all $A\in \cI_x$, for $\nn_*$-a.e.
$e\in \cE$, $\cl_x(A,e)>0$. Thus when $\beta=*$, 
for all $t\geq 0$ such that $L_t\in \JJ$, we have $\tau_{L_t}>\tau_{L_t-}$
(because $\Delta \tau_{L_t}=\cl_{b_{L_t-}}(a_{L_t},e_{L_t})$) and thus if $t=\tau_{L_t-}$, we have 
$R_t=h_{b_{L_t-}}(a_{L_t},e_{L_t}(0))=b_{L_t-}$ (because $e_{L_t}(0)=0$ since $e_{L_t} \in \cE_0$).

\vip

For $e \in \cEd$, we naturally have $\cl_x(A,e)>0$ if and only if $h_x(A,e(0)) \in \Dd$. When 
$\beta \in (0,\alpha/2)$, $\nn_\beta$ is carried by $\cEd$. Thus when $\beta \in (0,\alpha/2)$,
if $L_t \in \JJ$ and $t=\tau_{L_t-}$ and $\tau_{L_t}>\tau_{L_t-}$, which means that 
$\cl_{b_{L_t-}}(a_{L_t},e_{L_t})>0$, {\it i.e.} that $h_{b_{L_t-}}(a_{L_t},e_{L_t}(0)) \in \Dd$, 
we have $R_t=h_{b_{L_t-}}(a_{L_t},e_{L_t}(0))\neq b_{L_t-}$. If now $L_t \in \JJ$ and $t=\tau_{L_t-}$ and 
$\tau_{L_t}=\tau_{L_t-}$, which means that $\cl_{b_{L_t-}}(a_{L_t},e_{L_t}(0))=0$, {\it i.e.} that 
$h_{b_{L_t-}}(a_{L_t},e_{L_t}(0))\notin \Dd$, then $R_t=g_{b_{L_t-}}(a_{L_t},e_{L_t})=
\Lambda(b_{L_t-},h_{b_{L_t-}}(a_{L_t},e_{L_t}(0)))\neq b_{L_t-}$.

\vip

Let us summarize all this for future reference.

\begin{remark}\label{rkr}
For all $t\geq 0$, we have $t\in [\tau_{L_t-},\tau_{L_t}]$ and
\vip
(a) $R_t=b_{L_t}=b_{L_{t}-}$ if $L_t \notin \JJ$,
\vip
(b) $R_t=h_{b_{L_t-}}(a_{L_t},e_{L_t}(t-\tau_{L_t-}))\in \Dd$ if $L_t \in \JJ$ and $t \in (\tau_{L_t-},\tau_{L_t})$,
\vip
(c) $R_t=b_{L_t-}$ if $L_t\in \JJ$ and $t=\tau_{L_t-}$ when $\beta=*$,
\vip
(d) $R_t=h_{b_{L_t-}}(a_{L_t},e_{L_t}(0))\neq b_{L_t-}$  
if $L_t\in \JJ$ and $t=\tau_{L_t-}< \tau_{L_t}$ when $\beta \in (0,\alpha/2)$,
\vip
(e)  $R_t=\Lambda(b_{L_t-},h_{b_{L_t-}}(a_{L_t},e_{L_t}(0)))\neq b_{L_t-}$
if $L_t\in \JJ$ and $t=\tau_{L_t-}= \tau_{L_t} $ when $\beta \in (0,\alpha/2)$,
\vip
(f) $R_t=g_{b_{L_t-}}(a_{L_t},e_{L_t})=b_{L_t}$ if $L_t\in \JJ$ and $t=\tau_{L_t}$.
\end{remark}

Our first result concerns the existence and uniqueness in law of such a process.
Let $(X^*_t)_{t\geq 0}$ be the canonical process on the set of càdlàg $\cDd$-valued functions
$\Omega^*=\DD(\R_+,\cDd)$, defined by $X_t^*(w)=w(t)$.
We endow $\Omega^*$ with its canonical $\sigma$-field $\cF^*$ and its canonical filtration $(\cF^*_t)_{t\geq 0}$.

\begin{theorem}\label{mr1}
Fix $\beta \in \{*\}\cup(0,\alpha/2)$ and suppose Assumption~\ref{as}.
\vip
(a) For all $x \in \pDd$, there exists an $(\alpha,\beta)$-stable process $(R_t)_{t\geq 0}$
reflected in $\cDd$ issued from $x$. It is a.s. càdlàg and $\cDd$-valued.
\vip
(b) For all $x\in \pDd$, all the $(\alpha,\beta)$-stable processes reflected in $\cDd$ issued from $x$ have
the same law, which we denote by $\QQ_x$. It is a probability measure on the canonical space $(\Omega^*,\cF^*)$.
\vip
(c) If $t_n \in \R_+$ and $x_n \in \pDd$ satisfy $\lim_{n} t_n=t\geq0$ and $\lim_n x_n=x \in \pDd$, then
$$
\text{for all } \varphi \in C_b(\cDd), \qquad \lim_n \QQ_{x_n}[\varphi(X^*_{t_n})]=\QQ_x[\varphi(X_t^*)].
$$

(d) For all $B \in \cB(\DD(\R_+,\cDd))$, the map $x\mapsto \QQ_x(B)$ is measurable from $\pDd$ into $[0,1]$.
\end{theorem}

Finally, we extend Definition~\ref{dfr1} to the case where $x\in \Dd$. 

\begin{definition}\label{dfr2}
Fix $\beta \in \{*\}\cup(0,\alpha/2)$, $x\in \Dd$ and suppose Assumption~\ref{as}.
Consider an isotropic $\alpha$-stable process $(Z_t)_{t\geq 0}$ issued from $x$ and set
$$
\tell(Z)=\inf\{t>0 : Z_t \notin \Dd\}\quad \hbox{and} \quad Y= \Lambda(Z_{\tell(Z)-}, Z_{\tell(Z)}) \in \pDd.
$$
Conditionally on $(Z_{t\land\tell(Z)})_{t\geq 0}$, pick some $\QQ_{Y}$-distributed
process $(S_t)_{t\geq 0}$ and set
\begin{equation}
R_t =
\begin{cases}
Z_t & \text{if }t < \tell(Z), \\
S_{t-\tell(Z)} &  \text{if }t \geq \tell(Z).
\end{cases}
\end{equation}
We say that $(R_t)_{t\geq 0}$ is an $(\alpha,\beta)$-stable process reflected in $\cDd$ issued from $x$.
We denote by $\QQ_x=\cL((R_t)_{t\geq0})$ the resulting law. By Theorem~\ref{mr1}-(a),
$\QQ_x$ is carried by $\DD(\R_+,\cDd)$.
\end{definition}

By Theorem~\ref{mr1}-(d), this definition makes sense.
Here is the first main result of this paper.

\begin{theorem}\label{mr2}
Fix $\beta \in \{*\}\cup(0,\alpha/2)$ and suppose Assumption~\ref{as}.
The quintuple
$$(\Omega^*,\cF^*,(\cF_t^*)_{t\geq0},(\QQ_x)_{x\in \cDd}, (X^*_t)_{t\geq 0})$$ 
defines a Feller Markov process. 
Moreover, for all $x\in \cDd$, it holds that
$\QQ_x[\int_0^\infty \indiq_{\{X_t^* \in \pDd\}} \dr t]=0$.
\end{theorem}

We do not know how to clearly state uniqueness of the process built in Theorem~\ref{mr2}, 
except by saying that the boundary process is unique, and that there is only one natural way to concatenate some 
rotated/translated excursions once this boundary process is built. 
As we will see in Theorem~\ref{mr} below,
we have enough structure (and ``uniqueness'') to ensure that the (position of the) scattering process 
converges to this 
process, without extracting subsequences. That being said, when $\beta = *$, it might possible to show 
that the Markov process from Theorem~\ref{mr2} is the only $\closure{\Dd}$-valued strong Markov process 
satisfying \textit{(i)} when started from inside the domain, the stopped process 
(when hitting $\pDd$) has the same law as $(R_{t\wedge \tell(Z)})_{t\geq0}$ from Definition~\ref{dfr2}; 
\textit{(ii)} it comes out continuously from the boundary and \textit{(iii)} 
it spends zero Lebesgue time on $\pDd$.

\vip

We will also check that when $\beta = *$, the reflected process exits the boundary continuously whereas 
when $\beta \in(0,\alpha /2)$, it comes out with a jump.

\begin{proposition}\label{exit}
Fix $\beta \in \{*\}\cup(0,\alpha/2)$ and grant Assumption~\ref{as}. 
Consider an $(\alpha,\beta)$-stable process  $(R_t)_{t\geq 0}$ reflected in $\cDd$.
Introduce $\cZ=\closure{\{t\geq 0 : R_t \in \pDd\}}$ and write its complementary set as a countable union
of disjoint intervals: $\cZ^c=\cup_{n\in\N}(g_n,d_n)$.
\vip
(i) When $\beta =*$, $\pDd$ is a continuous exit set: a.s., for all $n\in\N$, $R_{g_n}=R_{g_n-}$.
\vip
(ii) When $\beta \in (0,\alpha/2)$, $\pDd$ is a discontinuous exit set: a.s., for all $n \in \N$,
$R_{g_n} \neq R_{g_n-}$.
\end{proposition}

The reflected process inherits, in some sense, 
the scaling property of the $\alpha$-stable process.

\begin{proposition}\label{scal}
Fix $\beta \in \{*\}\cup(0,\alpha/2)$ and grant Assumption~\ref{as}.
For $(R_t)_{t\geq 0}$ an $(\alpha,\beta)$-stable process reflected in $\cDd$ issued 
from $x\in\cDd$ and for $\lambda>0$,  $(\lambda^{1/\alpha} R_{t/\lambda})_{t\geq 0}$ is an $(\alpha,\beta)$-stable 
process reflected in $\lambda^{1/\alpha}\cDd=\{\lambda^{1/\alpha}y : y \in \cDd\}$ issued 
from $\lambda^{1/\alpha}x$.
\end{proposition}

\subsection{Infinitesimal generators and associated P.D.E.s}\label{ssig}

We introduce some fractional Laplacian operator in the domain $\Dd$.

\begin{definition}\label{opbd}
Grant Assumption~\ref{as}.
For $\varphi \in C(\cDd)\cap C^2(\Dd)$, let
$$
\cL\varphi(x)= \int_{\R^d} [\varphi(\Lambda(x,x+z))-\varphi(x) - z \cdot \nabla \varphi(x) \indiq_{\{|z|<1\}}]
\frac{\dr z}{|z|^{d+\alpha}},
\qquad x\in\Dd,
$$
which is well-defined and continuous on $\Dd$, see Remark~\ref{lcont}.
We say that $\varphi \in D_\alpha$ if it holds that $\sup_{x \in \Dd} |\cL \varphi(x)|<\infty$.
\end{definition}
We will show the following in Appendix~\ref{tf}.

\begin{remark}\label{dnonvide}
(a) If $\alpha \in (0,1)$, then for all $\e>0$, $C^2(\Dd)\cap C^{\alpha+\e}(\cDd) \subset D_\alpha$.
\vip
(b) If $\alpha \in [1,2)$, then for all $\e>0$, 
$\{\varphi \in C^2(\Dd)\cap C^{\alpha+\e}(\cDd) : \forall \; x \in \pDd,\;
\nabla\varphi(x)\cdot \bn_x=0\} \subset D_\alpha$.
\end{remark}

We next introduce some test functions satisfying some suitable boundary conditions.

\begin{definition}\label{test}
Grant Assumption~\ref{as}.
\vip
(a) Fix $\beta \in (0,\alpha/2)$. We say that $\varphi \in H_\beta$ if $\varphi \in C(\cDd)$
and if 
\begin{gather*}
\lim_{\e\to 0} \sup_{x \in \pDd} |\cH_{\beta,\e}\varphi(x)|=0,
\; \text{where}\;\;
\cH_{\beta,\e}\varphi(x)=
\int_{\HH \setminus B_d(0,\e)}\!\! [\varphi(\Lambda(x,h_x(A,z)))- \varphi(x)] \frac{\dr z}{|z|^{d+\beta}},
\end{gather*}
the value of $\cH_{\beta,\e}\varphi(x)$ not depending on the choice of $A\in \cI_x$.
\vip

(b) Fix $r>0$ such that $B_d(x+r\bn_x,r)\subset \Dd$ for all $x \in \pDd$ as in Remark~\ref{imp4}
and define $G_{r,\e}=B_2(r\be_1,r)\cap B_2(0,\e)$.
Let $\Sp_*=\{\rho \in \R^d :  |\rho|=1, \rho\cdot \be_1=0\}$ be endowed with its
(normalized) uniform measure $\varsigma$.
We say that $\varphi \in H_*$ if $\varphi \in C(\cDd)$ and if
\begin{gather*}
\lim_{\e\to 0}  \!\sup_{x\in \pDd, h \in G_{r,\e}}\!\!\!\!|\cH_{*}\varphi(x,h)|\!=\!0,\;
\text{where}\;\;
\cH_{*}\varphi(x,h)\!=\! \frac{1}{|h|^{\alpha/2}}\int_{\Sp_*}\!\! [\varphi(x+A(h_1\be_1+h_2 \rho))-\varphi(x)] 
\varsigma(\dr \rho),
\end{gather*}
the value of $\cH_{*}\varphi(x,h)$ not depending on the choice of $A \in \cI_x$.
Note that for all $x \in \pDd$, all $A\in \cI_x$, all $h \in B_2(r\be_1,r)$, all $\rho \in \Sp_*$,
we have $x+A(h_1\be_1+h_2 \rho) \in B_d(x+r\bn_x,r)\subset \Dd$.
\end{definition}

Let us mention that in (b), when $d=2$,
$\Sp_*=\{-\be_2,\be_2\}$ and $\varsigma=\frac12(\delta_{-\be_2}+\delta_{\be_2})$.

\vip
Concerning (a), we may also write, using the substitution $y= h_x(A,z)$,
$$
\cH_{\beta,\e}\varphi(x) = \int_{\HH_x} [\varphi(\Lambda(x,y))- \varphi(x)]\indiq_{\{|y-x|\geq \e\}} 
\frac{\dr y}{|y-x|^{d+\beta}},
$$
where $\HH_x$ the half-space tangent to $\pDd$ at $x$ containing $\Dd$. Thus 
$\lim_{\e \to 0} \cH_{\beta,\e}\varphi(x)$ is a kind of fractional derivative of $\varphi$ of order
$\beta$, and the condition $\lim_{\e \to 0} \cH_{\beta,\e}\varphi(x)=0$ for all $x \in \pDd$
may be interpreted as a fractional Neumann condition. As the process exits the boundary by a jump 
when $\beta \in(0,\alpha /2)$, the boundary condition is naturally {\it non-local}.
\vip
Concerning (b), for $\Sp_{x}=\{\rho \in \R^d : |\rho|=1, \rho\cdot \bn_x=0\}$ and
for $\varsigma_{x}$ the uniform measure on $\Sp_x$,
$$
\cH_{*}\varphi(x,h)= \frac{1}{|h|^{\alpha/2}}\int_{\Sp_{x}} [\varphi(x+h_1\bn_x+ h_2 \rho)-\varphi(x)] 
\varsigma_x(\dr \rho),
$$
Under some regularity conditions on $\varphi$ in the directions of $\bn_x^\perp$, we have
$$
\lim_{h\in B_2(r\be_1,r),|h|\to 0} \cH_{*}\varphi(x,h) = 
\lim_{u\to 0+} \frac {\varphi(x+u\bn_x)-\varphi(x)}{u^{\alpha/2}}. 
$$
Hence the condition 
$\lim_{h\in B_2(r\be_1,r),|h|\to 0} \cH_{*}\varphi(x,h)=0$ for all $x\in \pDd$ is a 
{\it local} fractional Neumann condition, which is natural since the process 
leaves the boundary continuously when $\beta=*$.

\vip

We can now give some information about the infinitesimal generator of our process.
We recall the notion of {\it bounded pointwise} convergence, classical in the
framework of generators, see e.g. Ethier and Kurtz~\cite[Appendix 3]{ek}: $\psi:\R_+\times \Dd\to \R$ 
is said to converge bounded pointwise to 
$\pi:\Dd\to \R$
as $t\to 0$ if $\sup_{t\geq 0, x \in \Dd} |\psi(t,x)|<\infty$ and if for all $x \in \Dd$,
$\lim_{t\to 0} \psi(t,x)=\pi(x)$.

\begin{theorem}\label{ig}
Fix $\beta \in \{*\}\cup(0,\alpha/2)$, suppose Assumption~\ref{as}
and consider the family $(\QQ_x)_{x\in \cDd}$ introduced in Definitions~\ref{dfr1} and~\ref{dfr2}.
For any $\varphi \in D_\alpha \cap H_\beta$,
$$
\frac{\QQ_x[\varphi(X_t^*)-\varphi(x)]}t\quad {\longrightarrow} \quad
\cL\varphi(x)\quad \text{bounded pointwise as $t \to 0$.}
$$
\end{theorem}

We can write down (some weak formulations of) the corresponding P.D.E.s.

\begin{proposition}\label{pde}
Fix $\beta \in \{*\}\cup(0,\alpha/2)$, suppose Assumption~\ref{as}
and consider the family $(\QQ_x)_{x\in \cDd}$ introduced in Definitions~\ref{dfr1} and~\ref{dfr2}.
For $t\geq 0$ and $x \in \cDd$, consider the probability measure $f(t,x,\dr y)=\QQ_x(X_t^* \in \dr y)$
on $\cDd$. For all $x \in \cDd$, it holds that $f(t,x,\pDd)=0$ for a.e. $t>0$ and,
for any $\varphi \in D_\alpha \cap H_\beta$, any $t> 0$,
\begin{equation}\label{wpde}
\int_{\cDd} \varphi(y) f(t,x,\dr y) = \varphi(x)
+ \int_0^t \int_{\cDd}\cL \varphi(y) f(s,x,\dr y) \dr s.
\end{equation}
Moreover, for any $t\geq 0$, the map $x \mapsto f(t,x,\dr y)$ is weakly continuous.
\end{proposition}

We will check the following remark in Appendix~\ref{tf}. We restrict ourselves to the case where
$\Dd$ is a Euclidean ball, the general case being much more intricate.

\begin{remark}\label{pdi}
Assume that $\Dd=B_d(0,1)$.
\vip
(a)  Let $\varphi \in C^2(\Dd)$ such that $\varphi(x)=[d(x,\Dd^c)]^{\alpha/2}$ 
as soon as $d(x,\Dd^c)\leq1/2$.
Then $\varphi \in D_\alpha$.
\vip
(b) Let $\beta\in (0,\alpha/2)$. For any pair of nonnegative, non identically zero, radially symmetric
$\varphi_1,\varphi_2 \in C^{\alpha/2}(\cDd)$ such that
$\varphi_1|_{\pDd}=\varphi_2|_{\pDd}=0$, there is 
$a>0$ such that $\varphi_1-a\varphi_2 \in H_\beta$.
\vip
(c) For all $\beta,\beta' \in \{*\}\cup\{0,\alpha/2\}$ with $\beta\neq \beta'$,
$(D_\alpha \cap H_\beta)\setminus H_{\beta'}\neq \emptyset$.
\end{remark}

\subsection{The scattering process}

The scattering process is kinetic model describing the motion of a particle, of which the velocity is 
reset at (high) constant
rate, according to some given isotropic distribution $\rF$.
We endow this equation with a diffusive
boundary condition, meaning that when the particle reaches the boundary, it is restarted with a new velocity 
distributed according to some other given isotropic distribution $\rG$ (restricted to the set of 
admissible directions).

\begin{assumption}\label{assump_equi}
The two probability densities $\mathrm{F}$ and $\mathrm{G}$ on $\R^d$ are radially symmetric, and there exist
$\alpha \in  (0,2)$, $\kappa_\rF>0$ and $C_\rF>0$ such that 
\[
\rF(v)\leq \frac{C_\rF}{(1+|v|)^{d+\alpha}} \quad \text{for all $v \in \R^d$ and} \quad
\mathrm{F}(v) \sim \frac{\kappa_{\mathrm{F}}}{|v|^{d+\alpha}}  \quad \text{as} \quad |v|\to \infty.
\]
\end{assumption}

We introduce, for each $\e \in (0,1]$, the kinetic equation with nonnegative unknown $f^\e_t(x,v)$:
\begin{equation}\label{rescale_scattering_eq}
\left\lbrace
\begin{aligned}
\e&^{(\alpha-1)/\alpha}\partial_t f_t^\e(x,v) + v\cdot\nabla_x f_t^\e(x,v) 
= \e^{-1/\alpha} \cA f_t^\e(x,v), 
\quad &\:\: (t,x,v)\in(0,\infty) \times \Dd \times \mathbb{R}^d, \\
(v & \cdot \bn_x) f_t^\e(x,v) = 2\mathrm{G}(v)\int_{\{w\cdot \bn_x < 0\}}|w\cdot \bn_x| f_t^\e(x,w) \dr w, 
& \:\: t>0, x \in \partial\Dd, v\cdot \bn_x > 0, \\
f_0^\e &(x,v) = f_{in}(x,v) & \:\: (x,v) \in \Dd \times \mathbb{R}^d.
\end{aligned}
\right.
\end{equation}
The initial condition is a given probability density $f_{in}$ on $\Dd\times \R^d$. 
For $x\in \pDd$, $\bn_x$ denotes, as previously, the inward unit normal vector. 
Finally, the scattering operator $\cA$
acts only on the velocity variable $v$ and is defined, for $f:\R^d\to \R$ and $v \in \R^d$, by
\[
  \cA f(v) = \mathrm{F}(v)\int_{\mathbb{R}^d}f(w) \dr w - f(v).
\]
Observe that $2\int_{v\cdot \bn_x>0} \mathrm{G}(v)\dr v = 1$ for any $x \in \partial \Dd$, 
by rotational invariance of $\mathrm{G}$, and the reader familiar with such equations will deduce 
that~\eqref{rescale_scattering_eq} {\it a priori} preserves mass and positivity, so that 
$f^\e_t$ should be a probability density on $\Dd\times \R^d$ for each $t\geq 0$. The solution we will build
indeed enjoys this property.
The scaling factors $\e^{(\alpha-1)/\alpha}$ and $\e^{-1/\alpha}$ in~\eqref{rescale_scattering_eq}
are the same as in Cesbron, Mellet and Puel~\cite[Equation 5]{cesbron2020fractional}. 
Actually, in their notation, $s=\alpha/2$,
their $\e$ corresponds to our $\e^{1/\alpha}$, and their $\bn_x$ corresponds to our $-\bn_x$.
Note that these scaling factors do depend on the tail parameter $\alpha$
of the equilibrium $\mathrm{F}$.
We refer to~\cite{cesbron2020fractional,cesbron2021fractional} for some motivations and explanations 
about the chosen scales, we will also comment on these scales in a few lines.

\vip

We will study~\eqref{rescale_scattering_eq} through the underlying stochastic process, whose time marginals 
will be given by
$(f_t^\e)_{t\geq 0}$.
Consider
\begin{gather}
\bE=(\Dd\times\R^d)\cup\{(x,v) : x \in \partial D, v \in \R^d, v\cdot \bn_x >0\}, \label{E}\\
\bE_-=(\Dd\times\R^d)\cup\{(x,v) : x \in \partial D, v \in \R^d, v\cdot \bn_x <0\}. \label{Emoins}
\end{gather}
For $(x,v)\in \bE$ and $s>0$, we introduce $\lambda(x,v,s)\in (0,s]$ defined by
\begin{equation}\label{lambda}
\lambda(x,v,s)=\begin{cases} s & \text{if $x+vs \in \Dd$}, \\
\inf\{u>0 : x+vu\notin \Dd \}& \text{otherwise}.\end{cases}
\end{equation}
Observe at once that for $(x,v)\in \bE$ and $s>0$, recalling~\eqref{Lambda},
\begin{equation}\label{Lamlam}
x+\lambda(x,v,s)v=\Lambda(x,x+vs).
\end{equation}
Let us also introduce the probability density $\rG_+(v)= 2 \rG(v) \indiq_{\{v\cdot \be_1 > 0\}}$ and,
for $x\in\partial\mathcal{D}$, the probability density $\rG_x(v) = 
2 \rG(v) \indiq_{\{v\cdot \bn_x > 0\}}$. It should be clear that if $W$ 
is $\rG_+$-distributed, then $A W$ is $\rG_x$-distributed for any $A \in \cI_x$.

\begin{definition}\label{dfsp}
Grant Assumptions~\ref{as} and~\ref{assump_equi} and let $(A_y)_{y \in \pDd}$ be a measurable family
such that $A_y  \in \cI_y$
for each $y \in \pDd$.
Fix $\e\in (0,1]$, $(x,v) \in \bE$ and consider an i.i.d. $\mathrm{Exp}(\e^{-1}$)-distributed sequence 
$(E_n^\e)_{n\geq 1}$, an i.i.d. $\rF$-distributed sequence $(U_n)_{n\geq 1}$ and an i.i.d. $\rG_+$-distributed
sequence $(W_n)_{n\geq 1}$, all these objects being independent.
We introduce the $\e$-\emph{scattering process} $(\bX_t^\e,\bV_t^\e)_{t\geq 0}$ starting 
from $(x,v)$, valued in $\bE$, defined by induction as follows.
\vip
We set $(\bX_0^\e,\bV_0^\e)=(x,v)$ and $T_1^\e=\lambda(\bX_0^\e,\e^{(1-\alpha)/\alpha}\bV_0^\e,E^\e_1)$ and
$$
\text{for all $t\in [0,T_{1}^{\e})$,}\quad
\bV^\e_t=\bV^\e_{0} \quad \text{and} \quad \bX^\e_{t}=\bX^\e_{0}+\e^{(1-\alpha)/\alpha}\bV^\e_{0}t.
$$

Assuming that $T_n^\e$ and $(\bX^\e_t,\bV^\e_t)_{t\in [0,T_n^\e)}$ have been built for some $n\geq 1$, we set
\begin{gather*}
\bX^\e_{T_n^\e}=\bX^\e_{T_n^\e-} \quad \text{and} \quad  
\bV^\e_{T_n^\e}=U_n \indiq_{\{\bX^\e_{T_n^\e} \in \Dd\}} +A_{\bX^\e_{T_n^\e}}W_n \indiq_{\{\bX^\e_{T_n^\e} \in \pDd\}},\\
T_{n+1}^\e = T_n^\e + \lambda(\bX^\e_{T_n^\e},\e^{(1-\alpha)/\alpha}\bV^\e_{T_n^\e},E^\e_{n+1}), \\
\text{for all $t\in [T_n^\e,T_{n+1}^{\e})$,} \quad \bV^\e_t=\bV^\e_{T_n^\e} \quad \text{and} \quad 
\bX^\e_{t}=\bX^\e_{T_n^\e}+\e^{(1-\alpha)/\alpha}\bV^\e_{T_n^\e}(t-T_n^\e).
\end{gather*}
Observe that the process $(\bX_{t-}^\e, \bV_{t-}^\e)_{t\geq0}$ is valued in $\bE_-$
and that the law of $(\bX_{t}^\e, \bV_{t}^\e)_{t\geq0}$ does not depend on the
choice of the family $(A_y)_{y \in \pDd}$ (since for $y\in \pDd$, for $W \sim \rG_+$
and for $A,B \in \cI_y$, $AW$ and $BW$ have the same law $\rG_y$).
\end{definition}

By memorylessness of $(E_n^\e)_{n\geq1}$, we can summarize the dynamics of 
$(\bX_t^\e,\bV_t^\e)_{t\geq 0}$ as follows:

\vip

\noindent $\bullet$ the position process $\bX^\e$ moves according to its velocity $\e^{(1-\alpha)/\alpha}\bV^\e$;
\vip
\noindent $\bullet$  the velocity process $\bV^\e$ is refreshed at rate
$\e^{-1}$ and its new value is chosen according to $\rF$;
\vip
\noindent $\bullet$  when $\bX^\e$ reaches $\partial \mathcal{D}$ 
(at some $y \in \pDd$), $\bV^\e$ it is restarted according to $\rG_{y}$.

\vip

We will check the following easy observation.

\begin{remark}\label{scwd}
The sequence $(T_n^\e)_{n\geq1}$ introduced in Definition~\ref{dfsp} a.s. strictly increases to infinity.
For all $T>0$, $\E[\bM^\e_T]<\infty$, where $\bM^\e_T=\sum_{n\geq 1} \indiq_{\{T_n^\e\leq T\}}$.
\end{remark}

We will verify the following, see Appendix~\ref{deredp} for a precise notion of weak solutions
to~\eqref{rescale_scattering_eq}.

\begin{remark}\label{edp}
Consider the $\e$-scattering process $(\bX_t^\e, \bV_t^\e)_{t\geq0}$ issued from $(x_0,v_0)\in \bE$.
For $t\geq0$, let $f^\e_t(\dr x,\dr v)=\PP(\bX_t^\e \in \dr x, \bV_t^\e \in \dr v)$. 
Then $(f^\e_t)_{t\geq 0}$
is a weak solution to~\eqref{rescale_scattering_eq} with $f_0^\e=\delta_{(x_0,v_0)}$.
\end{remark}

Let us mention at once that when $\Dd=\R^d$, we have $T_{n+1}^\e = T_n^\e + E^\e_{n+1}$ 
and $\bV^\e_{T_n^\e}=U_n$ for all $n\geq 1$, and one may check that setting
$N^\e_t=\sum_{n\geq 1} \indiq_{\{T_n^\e\leq t\}}$,
$$
\bX^\e_t=x+\e^{(1-\alpha)/\alpha}E_1^\e \bV_0+ \e^{(1-\alpha)/\alpha}\sum_{n=1}^{N^\e_t - 1} 
E_{n+1}^\e U_n+(t-T_{N_t^\e}^\e)U_{N_t^\e},
$$
at least if $t>T_1^\e$. But $(N^\e_t)_{t\geq 0}$ is a Poisson process with rate $\e^{-1}$, so that 
$N^\e_t\simeq \e^{-1}t$. Moreover, $E_n^\e$ is of order $\e$, so that $\e^{(1-\alpha)/\alpha}E_n^\e\simeq \e^{1/\alpha}$.
Admitting that we can neglect the last term in the above expression,
we end with $\bX_t^\e\simeq x+\e^{1/\alpha}\sum_{n=1}^{\e^{-1}t}U_{n}$,
which classically converges in law to a radially symmetric $\alpha$-stable process $Z_t$ 
(issued from $x$)  as $\e\to 0$ under Assumption~\ref{assump_equi}.
We hope that this explains why the scalings
in Definition~\ref{dfsp} are relevant.

\vip

Concerning the boundary velocity distribution $\rG$, we will assume one of the following.

\begin{assumption}\label{assump:moments}
(a) It holds that $\int_{\R^d} |v|^{\alpha/2}\mathrm{G}(v)\dr v<\infty$.
\vip
(b) There exists $\beta \in(0, \alpha/2)$ and  some constants $\kappa_{\mathrm{G}}>0$ and $C_\rG>0$ such that 
$$
\rG(v)\leq \frac{C_\rG}{(1+|v|)^{d+\beta}} \quad \text{for all $v \in \R^d$ and} \quad
\rG(v) \sim \frac{\kappa_{\mathrm{G}}}{|v|^{d+\beta}}  \quad \text{as} \quad |v|\to \infty.
$$
\end{assumption}

Here is the second main result of this paper. See Appendix~\ref{sko} about the $\MS$-topology.

\begin{theorem}\label{mr}
Grant Assumption~\ref{as} and Assumption~\ref{assump_equi} with some $\alpha \in (0,2)$ and with 
$\kappa_\rF=1/\Gamma(\alpha+1)$. Grant either Assumption~\ref{assump:moments}-(a)
(in which case, set $\beta=*$) or (b) (in which case $\beta \in (0,\alpha/2)$).
Consider the family $(\QQ_x)_{x \in \cDd}$ as in Theorem~\ref{mr2} with these values of $\alpha$ and $\beta$ 
and, for each $\e\in (0,1]$, 
consider the $\e$-scattering process $(\bX_t^\e, \bV_t^\e)_{t\geq0}$ issued from $(x,v)\in \bE$. Then
$$
(\bX_t^\e)_{t\geq0} \quad \text{converges in law to} \quad \QQ_x \quad
\text{as $\e\to 0$}
$$
in $\DD(\R_+,\cDd)$ endowed with the $\MS$-topology.
\end{theorem}

This result holds true for any value of $\kappa_\rF>0$, modifying the 
definition of $(\QQ_x)_{x \in \cDd}$ as follows: use  
$\kappa_\rF\Gamma(\alpha+1)\dr s |z|^{-\alpha-d} \dr z$ for the 
intensity of the Poisson measure $N$ appearing in~\eqref{eqs}.

\vip

Let us emphasize that the limit law depends on $\rG$ only through $\beta\in \{*\}\cup(0,\alpha/2)$.
For example, when $\beta \in (0,\alpha/2)$, it does not depend on $\kappa_\rG>0$.

\vip

Finally, let us present a P.D.E. version of this result.

\begin{corollary}\label{pdev}
Adopt the same assumptions and notations as in Theorem~\ref{mr}. For each $t\geq 0$,
set $f^\e_t(\dr y,\dr v)=\PP(\bX^\e_t \in \dr y,\bV^\e_t \in \dr v)$, 
$\rho^\e_t(\dr y)= \int_{v \in \R^d}f^\e_t(\dr y,\dr v)$ and $f_t(\dr y)=\QQ_x(X_t^* \in \dr y)$.
We know from Remark~\ref{edp} that $(f^\e_t)_{t\geq 0}$ is a weak solution to~\eqref{rescale_scattering_eq},
while Proposition~\ref{pde} tells us that $(f_t)_{t\geq 0}$ solves~\eqref{wpde}.
For a.e. $t\geq 0$, $\rho^\e_t$ tends to $f_t$ as $\e\to 0$ for
the weak convergence of probability measures on $\R^d$.
\end{corollary}

\begin{proof}
Theorem~\ref{mr} implies that $\rho^\e_t$ (which is the law of $\bX^\e_t$) weakly goes to $f_t$
for any $t\geq 0$ such that $\QQ_x(\Delta X^*_t \neq 0)=0$. Indeed, 
$X_t^*:\DD(\R_+,\cDd)\to \cDd$ is continuous for the $\MS$-topology 
at any $w\in \DD(\R_+,\cDd)$ such that $\Delta X^*_t(w)=0$,
so that $\Pi_t:\cP(\DD(\R_+,\cDd))\to \cP(\cDd)$ defined by $\Pi_t(q)=q(X^*_t \in \cdot)$ is continuous
at any $q\in \cP(\DD(\R_+,\cDd))$ such that $q(\Delta X^*_t\neq 0)=0$.
For any $q\in \cP(\DD(\R_+,\cDd))$, the set $\{t\geq 0 : q(\Delta X^*_t\neq 0)>0\}$
is classically Lebesgue-null.
\end{proof}

\section{Preliminaries: properties of the excursion measures}\label{sec:excursion}

In this section, we establish a few properties of the excursion measures $\nn_\beta$, for $\beta \in \{*\}\cup 
(0,\alpha/2)$ and of the stable process. First, we make precise the paragraph around~\eqref{et}.

\begin{lemma}\label{oncemore}
For  $(Z_t)_{t\geq0}$ an ISP$_{\alpha,0}$ and $Z^1_t=Z_t\cdot \be_1$, the process
$(Y_t)_{t\geq0}=(Z_t^1 - \inf_{s\in[0,t]}Z_s^1)_{t\geq0}$ is Markov, possesses a local time 
$(\xi_t)_{t\geq0}$ at $0$, and this local time is uniquely defined if we impose the condition 
$\E[\int_0^\infty e^{-t} \dr  \xi_t ]=1$. Its right-continuous inverse 
$(\gamma_u=\inf\{s\geq 0 : \xi_s  > u\})_{u \geq0}$ 
is a $(1/2)$-stable $\R_+$-valued subordinator. 
We introduce $\JJ = \{u\geq0, \: \Delta\gamma_u > 0\}$ and, for $u \in \JJ$,
$e_u = (Z_{(\gamma_{u-} + s)\wedge\gamma_u} - Z_{\gamma_{u-}})_{s\geq0}$.
Almost surely, $e_u \in \cEc$ and $\ell(e_u)=\Delta \gamma_u$ for all $u \in \JJ$.
\end{lemma}

\begin{proof} It is classical, see Bertoin~\cite[Chapter VI, pages 156-157]{bertoin1996levy},
that $(Y_t)_{t\geq0}$ is Markov and possesses a local time 
$(\xi_t)_{t\geq0}$ at $0$, \textit{i.e.} a continuous nondecreasing 
additive functional whose support equals 
$\cZ_Y=\closure{\{t\geq 0 : Y_t=0\}}$. It is unique up to a multiplicative constant and we impose the condition 
$\E[\int_0^\infty e^{-t} \dr \xi_t ]=1$. Since $(Z_t^1)_{t\geq0}$ is a 
symmetric stable process, the process
$(\gamma_u)_{u\geq 0}$ is a $(1/2)$-stable subordinator, see Bertoin~\cite[Section VIII, page 218]{bertoin1996levy}.

\vip

It is also classical, see for instance Blumenthal and Getoor~\cite[equation 2.4 page 58]{MR165569}, that
$\R\setminus\cZ_Y = \cup_{u \in \JJ} (\gamma_{u-},\gamma_u)$
and that for any $u\in\JJ$, $Y_{\gamma_u} = 0$. This implies that for $u\in\JJ$,
$Z^1_{\gamma_u}=\inf_{s\in[0,\gamma_{u}]}Z_s^1$ and thus 
$Z_{\gamma_{u-}}^1\geq \inf_{s\in[0,\gamma_{u-}]}Z_s^1 \geq \inf_{s\in[0,\gamma_{u}]}Z_s^1=Z^1_{\gamma_u}$.
\vip
Let us now show that for all $u \in \JJ$, all $t\in (\gamma_{u-},\gamma_u)$, we have $Z^1_t>Z^1_{\gamma_{u-}}$.
Since $(\gamma_{u-},\gamma_u)\subset \R\setminus\cZ_Y$, we have 
$\inf_{s\in [0,\gamma_u)}Z_s^1=\inf_{s\in[0,\gamma_{u-}]}Z_s^1$
and for all $t\in (\gamma_{u-},\gamma_u)$, $Z^1_t>\inf_{s\in[0,t]}Z_s^1=\inf_{s\in[0,\gamma_{u-}]}Z_s^1$.
But $Z^1_{\gamma_{u-}}=\inf_{s\in[0,\gamma_{u-}]}Z_s^1$, because $Y_{\gamma_{u-}}=0$:
by e.g.~\cite[Lemma A.5]{persistence_bbt} or~\cite[Section 2]{MR4146890}, this 
follows from the fact that for the process $(Z_t^1)_{t\geq0}$, 
the point $0$ is regular for both $(0,\infty)$ and $(-\infty, 0)$, meaning that for 
$T_{+} = \inf\{t\geq0, \: Z_t^1 > 0\}$ and $T_{-} = \inf\{t\geq0, \: Z_t^1 < 0\}$, 
we have $\mathbb{P}_0(T_{+} = T_{-} = 0) = 1$.
\vip
We now fix $u\in \JJ$ and show that $e_u \in \cEc$ and $\ell(e_u)=\Delta\gamma_u$. 
We have $e_u(0)=0$ by definition. We have seen that $Z_t^1 > Z_{\gamma_{u-}}^1$ for any 
$t\in(\gamma_{u-}, \gamma_u)$, implying that $e_u(s)=Z_{\gamma_{u-}+s}-Z_{\gamma_{u-}}\in\HH$
for all $s\in (0,\Delta\gamma_u)$. 
Finally, $e_u(\Delta \gamma_u)\notin \HH$, since $e_{u}^1(\Delta\gamma_u)=Z^1_{\gamma_u}-Z^1_{\gamma_{u-}}\leq 0$.
\end{proof}

Let us now show that the excursion measures possess a scaling property.

\begin{lemma}\label{scaling}
For $\lambda>0$, let $\Phi_\lambda : \cE \to \cE$ be defined by $\Phi_\lambda(e)(t)=\lambda^{1/\alpha}e(t/\lambda)$.
It holds that
$$
\Phi_\lambda \# \nn_* = \lambda^{1/2} \nn_* \qquad \hbox{and} \qquad \Phi_\lambda \# \nn_\beta = \lambda^{\beta/\alpha} 
\nn_\beta \quad \text{for all $\beta>0$}.
$$
\end{lemma}

\begin{proof}
We start with the case where $\beta=*$. We recall that $\nn_*$ was defined in the paragraph around~\eqref{et},
starting from an ISP$_{\alpha,0}$ $(Z_t)_{t\geq 0}$. By the scaling property of the stable 
process, $(Z^\lambda_t=\lambda^{1/\alpha}Z_{t/\lambda})_{t\geq 0}$ has the same law as $(Z_t)_{t\geq0}$.
For any $a>0$,  $(a\xi_{t/\lambda})_{t\geq 0}$  is a local time of the 
first coordinate of $(Z^\lambda_t)_{t\geq 0}$  reflected on its infimum, as in Lemma~\ref{oncemore}.
But $(\xi^\lambda_t=\lambda^{1/2}\xi_{t/\lambda})_{t\geq 0}$
is the only choice such that, recalling that $\gamma_u=\inf\{t\geq 0 : \xi_t>u\}$ and observing
that $\gamma^\lambda_u:=\lambda \gamma_{u/\lambda^{1/2}}=\inf\{t\geq 0 : \lambda^{1/2} \xi_{t/\lambda} >u\}$,
$$
\E\Big[\int_0^\infty e^{-t}\dr \xi^\lambda_t\Big] = 
\E\Big[\int_0^\infty e^{-\gamma^\lambda_u} \dr u  \Big]
= \E\Big[\int_0^\infty e^{-\gamma_u} \dr u  \Big]=\E\Big[\int_0^\infty e^{-t}\dr  \xi_{t} \Big]=1.
$$
We used the substitution $ \xi^\lambda_{t} = u$ for the first equality, that 
$(\gamma_u)_{u\geq 0}$ is $(1/2)$-stable for the second one,
and the substitution $u= \xi_t$ for the third one. Hence the triples
$((Z_t)_{t\geq 0}, ( \xi_t)_{t\geq0},(\gamma_u)_{u\geq 0})$ and 
$((Z_t^\lambda)_{t\geq 0}, ( \xi_t^\lambda )_{t\geq0},(\gamma_u^\lambda)_{u\geq 0})$
have the same law. Setting  now $\JJ^\lambda = \{u\geq0, \: \Delta\gamma_u^\lambda > 0\}
= \{\lambda^{1/2}u : u \in \JJ\}$ 
and, for $u \in \JJ^\lambda$,
$$
e_u^\lambda = \Big(Z^\lambda_{(\gamma^\lambda_{u-} + s)\wedge\gamma_u^\lambda} - Z^\lambda_{\gamma^\lambda_{u-}}\Big)_{s\geq0},
$$
we conclude that $\Pi_*^\lambda = \sum_{u \in \JJ^\lambda} \delta_{(u,e^\lambda_u)}$ has the same law as 
$\Pi_*= \sum_{u \in \JJ} \delta_{(u,e_u)}$: it is a Poisson measure on
$\mathbb{R}_+\times\cE$ with intensity $\dr u \nn_*(\dr e)$. 

\vip

On the other hand, for $u=\lambda^{1/2}v \in \JJ^\lambda$, we have
$e_u^\lambda=\Phi_\lambda(e_v)$. Thus $\Pi_*^\lambda =  \sum_{v \in \JJ} \delta_{(\lambda^{1/2}v, \Phi_\lambda(e_v))}$, 
of which the intensity is $\lambda^{-1/2}\dr u ( \Phi_\lambda\# \nn_*) (\dr e)$. 
Hence  $\lambda^{-1/2} \Phi_\lambda\# \nn_*=\nn_*$ as desired.

\vip

We carry on with the case where $\beta>0$, which is simpler. For a measurable $\varphi : \cE \to \R_+$ we simply
write, recalling~\eqref{nnb} and using the scaling property of the stable process,
\begin{align*}
\int_\cE \varphi(\Phi_\lambda(e)) \nn_\beta(\dr e) = \int_\HH |x|^{-d-\beta}\E_x[\varphi(\Phi_\lambda(Z_{\cdot \land \ell(Z)}))]
\dr x = \int_\HH |x|^{-d-\beta}\E_{\lambda^{1/\alpha}x}[\varphi(Z_{\cdot \land \ell(Z)})]\dr x.
\end{align*}
Using the change of variables $y=\lambda^{1/\alpha}x$, we conclude that 
$$
\int_\cE \varphi(\Phi_\lambda(e)) \nn_\beta(\dr e) 
=\lambda^{\beta/\alpha} \int_\HH |y|^{-d-\beta}\E_{y}[\varphi(Z_{\cdot \land \ell(Z)})]\dr y
=\lambda^{\beta/\alpha}\int_\cE \varphi(e) \nn_\beta(\dr e)
$$
as desired.
\end{proof}

We then give some tail estimates of the excursion measure $\nn_*$.

\begin{lemma}\label{qdist} 
Recall that 
$\ell(e) = \inf\{t> 0, \: e(t) \notin {\HH}\}$ and $M(e)=\sup_{t\in[0,\ell(e)]}|e(t)|$
and set $M_1(e)=\sup_{t\in[0,\ell(e))}e_1(t)$ for all
$e \in \cE$, where $e_1(t)=e(t)\cdot \be_1$. There are some constants $c_*,d_*,e_* \in (0,\infty)$ 
such that for all $t>0$, all $m>0$,
\begin{equation}\label{ssc}
\nn_*(\ell>t) = c_* t^{-1/2}, \quad \nn_*(M>m) = d_* m^{-\alpha/2}
\quad \hbox{and} \quad \nn_*(M_1>m) = e_* m^{-\alpha/2}.
\end{equation}
Moreover, \eqref{se1} holds true when $\beta=*$.
\end{lemma}

\begin{proof}
Using the notations of Lemma~\ref{scaling}, we have $\ell(\Phi_\lambda(e))=\lambda \ell(e)$,
$M(\Phi_\lambda(e))=\lambda^{1/\alpha}M(e)$ and $M_1(\Phi_\lambda(e))=\lambda^{1/\alpha}M_1(e)$. 
By Lemma~\ref{scaling}, for any $\lambda>0$,
\begin{gather*}
\nn_*(\ell>t)=\frac{\nn_*(\lambda \ell>t)}{\lambda^{1/2}},\quad
\nn_*(M>m)=\frac{\nn_*(\lambda^{1/\alpha}M>m)}{\lambda^{1/2}},\quad
\nn_*(M_1>m)=\frac{\nn_*(\lambda^{1/\alpha}M_1>m)}{\lambda^{1/2}}.
\end{gather*}
Choosing $\lambda=t$ and $\lambda=m^{\alpha}$, we find~\eqref{ssc},
with $c_*=\nn_*(\ell>1)$, $d_*=\nn_*(M>1)$ and $e_*=\nn_*(M_1>1)$.
First, $c_*>0$, because otherwise, we would have $\ell(e)=0$ for $\nn_*$-a.e. $e \in \cE$.
Similarly, $e_*>0$ (whence $d_*>0$), because otherwise, we would have  $M_1(e)=0$ for $\nn_*$-a.e. $e \in \cE$.
Next, if $c_*=\infty$, the process $\gamma_u=\int_0^u\int_\cE \ell(e)\Pi_* (\dr v,\dr e)$ 
(this formula follows from
the construction of $\Pi_*$,
see the paragraph around~\eqref{et}) explodes immediately, which is not possible.
Moreover, $e_*<d_*$ and $d_*$ is finite: otherwise, we would have $\nn_*(M>A)=\infty$ for all $A>0$;
hence a.s., for any $A\in \mathbb{N}$, there would be infinitely many 
$u \in \JJ\cap[0,1]$ such that $M(e_u)>A$,
making explode $Z$ during $[0,\gamma_1]$.
Finally,
$$
\int_\cE [\ell(e)\land 1+M(e)\land 1]\nn_*(\dr e)= 
\int_0^1 \nn_*(\ell>t) \dr t + \int_0^1 \nn_*(M>m) \dr m,
$$
which is finite by~\eqref{ssc}: \eqref{se1} holds true when $\beta=*$.
\end{proof}

We will need the following property concerning the entrance law of $\nn_*$. 

\begin{lemma}\label{encore1}
For $t> 0$, let $k_t=\nn_*(e(t) \in \cdot,\ell(e)>t)$. There
are some constants $c_0,c_1\in (0,\infty)$ such that 
for all $\varphi \in C_b(\closure{\HH})$,
$$
\lim_{t\to 0} \int_{\HH} \varphi(a) a_1^{\alpha/2} k_t(\dr a) = c_0 \varphi(0)
\quad \text{and}\quad \lim_{t\to 0} \int_{\HH} \varphi(a) |a|^{\alpha/2} k_t(\dr a) = c_1 \varphi(0).
$$
\end{lemma}

\begin{proof} We divide the proof in two steps.
\vip
\textit{Step 1.} Here we show that there exists a constant $C \in (0,\infty)$ such that for any $x>0$,
\begin{equation}\label{eqent}
 \nn_*\big(|e(1)|> x, \ell(e) > 1\big) \leq C  (x\lor 3)^{-\alpha} \log(x\lor 3).
\end{equation}

First recall the definition of $\Pi_*$, see the paragraph around~\eqref{et} and Lemma~\ref{oncemore}: 
consider an ISP$_{\alpha,0}$ $(Z_t)_{t\geq0}$, the right-continuous inverse
$(\gamma_u)_{u\geq0}$ of its local time at $\partial \HH$  
and set $e_u = (Z_{(\gamma_{u-} + s)\wedge \gamma_u} - Z_{\gamma_{u-}})_{s\geq0}$ 
for $u\in \JJ = \{u \geq0, \:  \Delta\gamma_u > 0\}$. Then $\Pi_* = \sum_{u\in\JJ}\delta_{(u, e_u)}$ 
is a Poisson measure on $\mathbb{R}_+ \times \cE$ with intensity $\dr u\nn_*(\dr e)$. Let us set 
$\sigma = \inf\{u \in \JJ, \: \ell(e_u) > 1\}$. Then, see for instance 
Revuz and Yor~\cite[Chapter XII, Lemma 1.13]{revuz2013continuous}, we have
\begin{equation}\label{aa}
\nn_*\big(|e(1)|> x, \ell(e) > 1\big) = \nn_*(\ell(e) > 1)
\mathbb{P}\big(|Z_{\gamma_{\sigma -} + 1} - Z_{\gamma_{\sigma -}}| > x\big).
\end{equation}
Let us now introduce
$\tilde{\gamma}_u = \int_0^u \int_{\mathcal{E}}\ell(e) \indiq_{\{\ell(e) \leq 1\}}\Pi_*(\dr v, \dr e)$, 
which is a subordinator whose Lévy measure is 
$\nu(\dr r) := \nn_*(\ell(e) \in \dr r, \ell(e) \leq 1) =\frac{c_*}{2} r^{-3/2}\indiq_{\{r \leq 1\}}\dr r$ 
by Lemma~\ref{qdist}.  Therefore, $\mathbb{E}[e^{\theta \tilde{\gamma}_u}] = e^{c_\theta u}$ for any $\theta>0$, 
any $u\geq0$, where $c_\theta = \frac{c_*}2\int_0^1(e^{\theta r} - 1)r^{-3/2} \dr r$.
\vip
Now, recalling the definition of $\sigma$ and that 
$\gamma_u = \int_0^u \int_{\mathcal{E}}\ell(e) \Pi_*(\dr v, \dr e)$, 
we observe that $\gamma_{\sigma-} = \tilde{\gamma}_\sigma$. Moreover, $\sigma$ is independent of $(\tilde{\gamma}_u)_{u\geq0}$ because the two Poisson measures 
$\indiq_{\{\ell(e) > 1\}}\Pi_*(\dr v, \dr e)$ 
and $\indiq_{\{\ell(e) \leq 1\}}\Pi_*(\dr v, \dr e)$ are independent, and $\sigma$ is a functional of the first 
one whereas $(\tilde{\gamma}_u)_{u\geq0}$ is a functional of the second one. Finally, $\sigma$ is 
exponentially distributed with parameter $\lambda:=\nn_*(\ell(e) > 1)$. All in all,
$$
\E[e^{\theta \gamma_{\sigma-}}]=\E[e^{\theta \tilde \gamma_{\sigma}}]=\lambda \int_0^\infty \E[e^{\theta \tilde \gamma_{u}}]e^{-\lambda u}
\dr u= \lambda \int_0^\infty e^{c_\theta u - \lambda u}\dr u.
$$
Since $\lim_0 c_\theta=0$, there is $\theta_* \in (0,\alpha)$ such that 
$\mathbb{E}[e^{\theta_* \gamma_{\sigma -}}] < \infty$. We write, for $x>1$ and $t\geq 1$,
\[
 \mathbb{P}\big(|Z_{\gamma_{\sigma -} + 1} - Z_{\gamma_{\sigma -}}| > x\big) \leq \mathbb{P}(\gamma_{\sigma-} > t) 
+ \mathbb{P}( Z_{2t}^* > x / 2).
\]
where $Z_t^* = \sup_{s\in[0,t]}|Z_s|$. By scale invariance of the stable process, we have 
$\mathbb{P}(Z_{t}^* > x) = \mathbb{P}(Z_1^* > x / t^{1/\alpha})$. Moreover, there is
$M > 0$ such that $\mathbb{P}(Z_1^* > z) \leq M z^{-\alpha}$, see for instance 
Bertoin~\cite[Chapter VIII, Proposition 4]{bertoin1996levy} in dimension one (which is enough). We conclude that
\[
 \mathbb{P}\big(|Z_{\gamma_{\sigma -} + 1} - Z_{\gamma_{\sigma -}}| > x\big) 
\leq e^{-\theta_* t} \mathbb{E}[e^{\theta_* \tau_{\sigma-}}] + 2^{\alpha+1}Mtx^{-\alpha},
\]
If $x> 3> e^{\theta_*/\alpha}$, we choose $t = \alpha \log(x) / \theta_*\geq 1$ and find, for some constant $C>0$,
$$
\mathbb{P}\big(|Z_{\gamma_{\sigma -} + 1} - Z_{\gamma_{\sigma -}}| > x\big) \leq C x^{-\alpha}(1+ \log x)
\leq 2C x^{-\alpha}\log x.
$$
Recalling~\eqref{aa}, this shows~\eqref{eqent} when $x>3$. Since $\nn_*(\ell>1)<\infty$,
the case $x\leq 3$ is obvious.

\vip
\textit{Step 2.} 
By Lemma~\ref{scaling}, for any $\lambda>0$, since $\ell(\Phi_\lambda(e))=\lambda\ell(e)$,
\begin{align*}
\int_{\HH} \varphi(a) a_1^{\alpha/2} k_t(\dr a)=&\int_\cE \varphi(e(t))(e_1(t))^{\alpha/2}\indiq_{\{\ell(e)>t\}}
\nn_*(\dr e)\\
=&\frac1{\lambda^{1/2}}\int_\cE \varphi(\lambda^{1/\alpha}e(t/\lambda))(\lambda^{1/\alpha}e_1(t/\lambda))^{\alpha/2}
\indiq_{\{\ell(e)>t/\lambda\}}
\nn_*(\dr e).
\end{align*}
Choosing $\lambda=t$, we find
$$
\int_{\HH} \varphi(a) a_1^{\alpha/2} k_t(\dr a)=\int_\cE \varphi(t^{1/\alpha}e(1)) 
[e_1(1)]^{\alpha/2} \indiq_{\{\ell(e)>1\}} \nn_*(\dr e)
=\int_\HH \varphi(t^{1/\alpha}a) 
a_1^{\alpha/2} k_{1}(\dr a)\to c_0\varphi(0)
$$
as $t\to 0$, where $c_0=\int_\HH a_1^{\alpha/2} k_{1}(\dr a)= \int_{\cE} [e_1(1)]^{\alpha/2}\indiq_{\{\ell(e)>1\}} 
\nn_*(\dr e)$. We have $c_0>0$, since $\ell(e)>1$ implies $e_1(1)>0$ and since $\nn_*(\ell>1)>0$ 
by Lemma~\ref{qdist}. We next write
\[
 c_0 = \int_0^\infty \nn_*(e_1(1) > x^{2/\alpha}, \ell(e) > 1) \dr x,
\]
which is finite by~\eqref{eqent}. The very same arguments 
show that  $\lim_{t\to0}\int_{\HH} \varphi(a) |a|^{\alpha/2} k_t(\dr a) = 
c_1\varphi(0)$ for any $\varphi \in C_b(\closure{\HH})$, where 
$c_1 := \int_{\HH}|a|^{\alpha / 2}k_1(\dr a) \in (0,\infty)$.
\end{proof}

The following result describes the strong Markov property of the excursion measures.

\begin{lemma}\label{mark}
Fix $\beta \in \{*\}\cup(0,\alpha/2)$.
Recall that $(Z_t)_{t\geq 0}$ is, under $\PP_x$, an ISP$_{\alpha,x}$.
Endow $\cE$ with its canonical filtration $\cG_t=\sigma(X_s, s\in [0,t])$, where $X_s(e)=e(s)$ 
is the canonical process. Consider $\rho : \cE \to \R_+\cup\{\infty\}$ a $(\cG_t)_{t\geq 0}$-stopping time, that is,
for all $t\geq 0$, $\{e \in \cE : \rho(e)\leq t\} \in \cG_t$. If $\beta=*$, assume moreover that
$\rho(e)>0$ for $\nn_*$-a.e. $e\in \cE$.
For any pair $\psi_1,\psi_2$ of
measurable functions from $\cE$ into $\R_+$,
\begin{align*}
\int_\cE \psi_1[(e(s\land \rho(e)))_{s\geq 0}]&\psi_2[(e((\rho(e)+u)\land\ell(e)))_{u\geq 0}]
\indiq_{\{\ell(e)>\rho(e)\}} 
\nn_\beta(\dr e)\\
=&\int_\cE \psi_1[(e(s\land \rho(e)))_{s\geq 0}]\E_{e(\rho(e))}[\psi_2[(Z_{u\land \ell(Z)})_{u\geq 0}]]
\indiq_{\{\ell(e)>\rho(e)\}} \nn_\beta(\dr e).
\end{align*}
\end{lemma}

\begin{proof}
Recalling the definition~\eqref{nnb} of $\nn_\beta$, this formula is directly inherited from the strong 
Markov property of the stable process $(Z_t)_{t\geq 0}$ when $\beta \in (0,\alpha/2)$. 
When $\beta=*$, the proof is exactly the same as when $d=1$, 
see e.g. Blumenthal~\cite[Theorem 3.28 page 102]{blu1992}.
\end{proof}

We next give some estimates concerning the stable process.

\begin{lemma}\label{pr5}
Recall that $(Z_t)_{t\geq 0}$ is, under $\PP_x$, an ISP$_{\alpha,x}$.
There are some constants $c,C\in (0,\infty)$ such that
for any $x \in \HH$, denoting by $x_1=x\cdot \be_1$ and 
recalling that $\ell(Z)=\inf\{t>0 : Z_t \notin {\HH}\}$,
$$
\PP_x(\ell(Z)>1) \sim c \; x_1^{\alpha/2} \quad \hbox{ as $|x|\to 0$}\quad \hbox{ and } \quad
\E_x\Big[\sup_{t\in [0,\ell(Z)]} |Z_t-Z_0| \land 1 \Big] \leq C (x_1^{\alpha/2}\land 1).
$$
\end{lemma}

\begin{proof}
The estimate concerning $\PP_x(\ell(Z)>1)$ can be found in 
Bertoin~\cite[Chapter VIII, Proposition 2 page 219]{bertoin1996levy}, 
because by symmetry,
$\PP_x(\ell(Z)>1)=\PP_0(\sup_{t\in [0,1]}Z_t^1 <x_1)$.
For the second estimate, we fix $x\in \HH$ and use Lemmas~\ref{qdist} and~\ref{mark}.
We introduce the stopping time $\rho(e)=\inf\{t\geq0 : e_1(t) > x_1\}$, 
where $e_1(t)=e(t)\cdot\be_1$.
Observe that $\rho$ is $\nn_*$-a.e. positive (because $\nn_*$ is carried by $\cE_0$) and
that for any $m>0$,
$$
\{\rho(e)<\ell(e)\}\cap\Big\{\sup_{t\in [\rho(e),\ell(e)]} |e(t)-e(\rho(e))|>m\Big\} \subset \{M(e)>m/2\}..
$$
Indeed, if $|e(\rho(e))|>m/2$, then clearly $M(e)>m/2$, while if $|e(\rho(e))|\leq m/2$ then 
$M(e) \geq \sup_{t\in [\rho(e),\ell(e)]} |e(t)|\geq \sup_{t\in [\rho(e),\ell(e)]} |e(t)-e(\rho(e))|-m/2 > m/2$.
Consequently,
\begin{align*}
\nn_*(M>m/2)\geq& \int_\cE \indiq_{\{\ell(e)>\rho(e)\}} \indiq_{\{\sup_{t\in [\rho(e),\ell(e)]} |e(t)-e(\rho(e))|>m \}}\nn_*(\dr e)\\
=&\int_\cE \indiq_{\{\ell(e)>\rho(e)\}}\PP_{e(\rho(e))}\Big(\sup_{t\in [0,\ell(Z)]}|Z_t-Z_0|>m\Big)\nn_*(\dr e)
\end{align*}
by Lemma~\ref{mark}. But for all $e\in \cE$ such that $\ell(e)>\rho(e)$, we have $e_1(\rho(e))>x_1$, whence clearly
$$
\PP_{e(\rho(e))}\Big(\sup_{t\in [0,\ell(Z)]}|Z_t-Z_0|>m\Big)\geq \PP_{x}\Big(\sup_{t\in [0,\ell(Z)]}|Z_t-Z_0|>m\Big).
$$
Next,
$\int_\cE \indiq_{\{\ell(e)>\rho(e)\}}\nn_*(\dr e) = \nn_*(M_1>x_1)$, where
$M_1(e)=\sup_{t\in [0,\ell(e))}e_1(t)$. All this gives
$$
\PP_{x}\Big(\sup_{t\in [0,\ell(Z)]}|Z_t-Z_0|>m\Big)\leq \frac{\nn_*(M>m/2)}{\nn_*(M_1>x_1)}
=\frac{2^{\alpha/2}d_* x_1^{\alpha/2}}{e_* m^{\alpha/2}}
$$
by Lemma~\ref{qdist}. We conclude that
$$
\E_x\Big[ \sup_{t\in [0,\ell(Z)]}|Z_t-Z_0|\land 1 \Big]=\int_0^1 \PP_{x}\Big(\sup_{t\in [0,\ell(Z)]}|Z_t-Z_0|>m\Big)\dr m
\leq C x_1^{\alpha/2},
$$
which completes the proof since the left hand side is of course also bounded by $1$.
\end{proof}

We can now compute some tail distributions under the excursion measure $\nn_\beta$.

\begin{lemma}\label{qdist2}  
Fix $\beta \in (0,\alpha/2)$. Recall that for $e\in \cE$, 
$\ell(e) = \inf\{t> 0, \: e(t) \notin {\HH}\}$ and $M(e)=\sup_{t\in[0,\ell(e)]}|e(t)|$. 
There  are $c_\beta,d_\beta \in (0,\infty)$ such that for all $t>0$, all $m>0$,
\begin{equation}\label{scc}
\nn_\beta(\ell>t) = c_\beta t^{-\beta/\alpha}\quad \hbox{and} \quad \nn_*(M>m) = d_\beta m^{-\beta}.
\end{equation}
Moreover, \eqref{se1} holds true when $\beta \in (0,\alpha/2)$. Finally, $\nn_\beta(\ell>1)=\infty$ when
$\beta\geq \alpha/2$.
\end{lemma}

\begin{proof}
Proceeding as in the proof of Lemma~\ref{qdist}, we find
$$
\nn_\beta(\ell>t)=\lambda^{-\beta/\alpha}\nn_\beta(\ell>t/\lambda) \quad \hbox{and} \quad 
\nn_\beta(M>m)=\lambda^{-\beta/\alpha}\nn_\beta(M>m/\lambda^{1/\alpha}).
$$
Choosing $\lambda=t$ and $\lambda=m^{\alpha}$, we get~\eqref{scc}
with $c_\beta=\nn_\beta(\ell>1)$ and $d_\beta=\nn_\beta(M>1)$. Recalling~\eqref{nnb},
$c_\beta = \int_\HH |x|^{-d-\beta} \PP_x(\ell(Z)>1)\dr x$
is positive and finite, since $\PP_x(\ell(Z)>1)\leq C(1\land x_1^{\alpha/2})$ by Lemma~\ref{pr5} and since
$\beta \in (0,\alpha/2)$.
Similarly, $d_\beta= \int_\HH |x|^{-d-\beta} \PP_x(M(Z)>1)\dr x$ is positive and finite, because
$$
\PP_x(M(Z)>1)\leq \E_x\Big[\sup_{t\in [0,\ell(Z)]}|Z_t| \land 1\Big]\leq 
\E_x\Big[\Big(|x|+ \sup_{t\in [0,\ell(Z)]}|Z_t-Z_0|\Big) \land 1\Big],
$$
which is controlled thanks to Lemma~\ref{pr5} by $|x|\land 1+C(x_1^{\alpha/2}\land 1) \leq C (|x|^{\alpha/2} \land 1)$.
\vip
Moreover, \eqref{se1} holds true since
$$
\int_\cE [\ell(e)\land 1+M(e)\land 1]\nn_*(\dr e)= 
\int_0^1 \nn_*(\ell>t) \dr t + \int_0^1 \nn_*(M>m) \dr m < \infty
$$
by~\eqref{scc}. Finally, if $\beta\geq \alpha/2$,
$$
\nn_\beta(\ell>1)=\int_\HH |x|^{-d-\beta} \PP_x(\ell(Z)>1)\dr x
$$
is infinite, since $\PP_x(\ell(Z)>1)\sim c x_1^{\alpha/2}$ as $|x|\to 0$.
\end{proof}

We now show that
$\nn_*$-a.e. $e \in \cE$ does not instantaneously leave a ball tangent to $\partial \HH$ at $0$.

\begin{lemma}\label{imp}
Fix $r>0$ and set $\ell_r(e)=\inf\{t>0 : e(t)\notin {B}_d(r\be_1,r)\}$. Then $\nn_*(\ell_r=0)=0$.
\end{lemma}

\begin{proof}
We have $\HH\setminus B_d(r\be_1,r)=\{x \in \HH : 2rx_1 \leq |x|^2\}\subset \{x \in \HH : rx_1<|x|^2\}=:D_r$,
so that $\ell_r(e)\geq \rho_r(e)$, where $\rho_r(e)=\inf\{t>0 : e(t)\in D_r\}$. It thus suffices
that $\nn_*(\rho_r=0)=0$. Since $\ell>0$ on $\cE$, it suffices that
$\nn_*(\rho_r=0,\ell>\delta)=0$ for all $\delta>0$ and, by 
scaling  (see Lemma~\ref{scaling}), we may assume that $\delta=2$.
We would like to apply the Markov property (Lemma~\ref{mark}) at time $\rho_r$
but this would require that $\rho_r>0$, which is precisely what we want to check.

\vip

For $n\geq1$, we set $D_{r,n}=\{x \in \HH : rx_1<|x|^2-1/n\}$ and $\rho_{r,n}(e)=\inf\{t>0 : e(t) \in D_{r,n}\}$.
It suffices to check that
\begin{equation}\label{bpfqa}
\lim_{\eta\to 0}\liminf_{n} \nn_*(\rho_{r,n}<\eta, \ell>2)=0.
\end{equation}
Indeed, for each $\eta>0$, $\{\rho_r<\eta\} \subset 
\liminf_{n} \{\rho_{r,n}<\eta\}$, because $\rho_r(e)<\eta$ implies that there is $t\in [0,\eta)$
such that $e(t) \in D_r$, so that there is $n_0\geq 1$ such that $e(t) \in \cap_{n\geq n_0} D_{r,n}$,
whence $\rho_{r,n}(e)<\eta$ for all $n\geq n_0$.
Consequently,
$\nn_*(\rho_{r}<\eta, \ell>2)\leq \liminf_{n} \nn_*(\rho_{r,n}<\eta, \ell>2)$. Thus
$\nn_*(\rho_{r}=0, \ell>2)\leq \lim_{\eta\to 0}\liminf_{n} \nn_*(\rho_{r,n}<\eta, \ell>2)$.

\vip
Recall that $\nn_*$ is carried by  $\cE_0=\{e \in \cE, e(0)=0\}$.
For each $n\geq 1$, each $e\in \cE_0$, there is 
$\e>0$ such that for all $t\in (0,\e]$, $e_1(t)>0$ and $|e(t)|^2<1/n$, 
implying that $\rho_{r,n}(e)\geq \e>0$. Thus we may apply Lemma~\ref{mark} to write, if $\eta \in (0,1]$,
\begin{align*}
\nn_*(\rho_{r,n}<\eta, \ell>2)=&\int_{\cE} \indiq_{\{\rho_{r,n}(e)<\ell(e),\rho_{r,n}(e)<\eta\}}
\PP_{e(\rho_{r,n}(e))}(\ell(Z)>2-\rho_{r,n}(e)) \nn_*(\dr e)\\
\leq&\int_{\cE} \indiq_{\{\rho_{r,n}(e)<\eta\land \ell(e)\}}\PP_{e(\rho_{r,n}(e))}(\ell(Z)>1)\nn_*(\dr e).
\end{align*}
Recalling Lemma~\ref{pr5} and using that $e(\rho_{r,n}(e)) \in D_{r,n}\subset D_r$,
we conclude that 
$$
\PP_{e(\rho_{r,n}(e))}(\ell(Z)>1) \leq C [e_1(\rho_{r,n}(e))  \land 1 ]^{\alpha/2}
\leq C [|e(\rho_{r,n}(e))|^2\land 1]^{\alpha/2}=C [|e(\rho_{r,n}(e))|^\alpha\land 1], 
$$
the constant $C$ being allowed to vary and to depend on $r$. Thus if $\rho_{r,n}(e)<\eta\land \ell(e)$,
$$
\PP_{e(\rho_{r,n}(e))}(\ell(Z)>1) \leq C \Big(\sup_{t\in [0,\eta \land \ell(e)]} |e(t)|^\alpha \land 1\Big).
$$
Consequently,
$$
\nn_*(\rho_{r,n}<\eta, \ell>2)
\leq  C \int_{\cE} \Big(\sup_{t\in [0,\eta \land \ell(e)]} |e(t)|^\alpha \land 1\Big) \nn_*(\dr e).
$$
This last quantity does not depend on $n\geq 1$ and tends to $0$ as $\eta\to 0$ by dominated convergence, since 
$\sup_{t\in [0,\eta \land \ell(e)]} |e(t)|^\alpha \to 0$ for all $e\in \cE_0$ and since
$$
\int_{\cE} \Big(\sup_{t\in [0,\ell(e)]} |e(t)|^\alpha \land 1\Big) \nn_*(\dr e)=
\int_{\cE} ((M(e))^\alpha\land 1) \nn_*(\dr e)=\int_0^1 \nn_*(M>m^{1/\alpha})\dr m,
$$
which is finite since $\nn_*(M>m^{1/\alpha})=d_* m ^{-1/2}$ 
by Lemma~\ref{qdist}. We have proved~\eqref{bpfqa}.
\end{proof}

Finally, the following result is almost immediate.

\begin{lemma}\label{ai}
Suppose Assumption~\ref{as}, fix $x\in\pDd$ and recall that $Z$ is, under $\PP_x$, an $ISP_{\alpha,x}$.
Then $\PP_x(\tell(Z)=0)=1$, where $\tell(Z)=\inf\{t>0 : Z_t \notin \Dd\}$.
\end{lemma}

\begin{proof} The process $Y_t=(Z_t - x)\cdot \bn_x$ is a one-dimensional symmetric $\alpha$-stable process 
issued from $0$,
and by convexity of $\Dd$, $\tell(Z)\leq \inf\{t>0 : Y_t\leq 0\}=:\rho(Y)$.
Using Bertoin~\cite[Theorem~5 p~222]{bertoin1996levy}, $\liminf_{t\to 0} t^{-1/\alpha}Y_t=-\infty$,
which implies that $\rho(Y)=0$.
\end{proof}

\section{A crucial estimate}\label{sec:crucial}

The goal of this section is to establish the following Lipschitz estimate, which is necessary to show that
the S.D.E.~\eqref{sdeb} defining the boundary process $(b_u)_{u\geq 0}$ is well-posed, 
with a continuous dependence in the initial condition.

\begin{proposition}\label{ttoopp}
Fix $\beta \in \{*\}\cup (0,\alpha/2)$ and suppose Assumption~\ref{as}.
There is a constant $C>0$ such that for all $x,x' \in \pDd$, all
$A \in \cI_x$, $A'\in\cI_{x'}$,
$$
\Delta(x,x',A,A'):= \int_\cE \Big||g_x(A,e)-g_{x'}(A',e)|- |x-x'|\Big| 
\nn_\beta (\dr e) \leq C (|x-x'|+||A-A'||).
$$
\end{proposition}
The proof of this proposition is long and tedious, and relies on fine estimates on the excursion 
measure $\nn_\beta$, as well as on a geometric inequality which will be proved in Appendix~\ref{ageo}. However, 
this long section (as well as the whole Appendix~\ref{ageo}) can be omitted when $\Dd$ is a Euclidean ball, 
in which case this inequality is quite straightforward to obtain.
\begin{proof}[Proof of Proposition~\ref{ttoopp} when $\Dd=B_d(0,1)$]
Let us observe that 
for $x \in \partial B_d(0,1)$ and $A \in \cI_x$, we have 
$\{y \in \R^d : h_x(A,y) \in {B}_d(0,1)\}={B}_d(\be_1,1)$. Indeed, since $|x|=1$ and
$\bn_x=-x$, we have $|x+A y|^2< 1$ if and only if $|Ay|^2+2x\cdot Ay <0$ if and only if
$|y|^2-2 \be_1\cdot y < 0$ (because $|Ay|=|y|$ and $x\cdot Ay=-A\be_1 \cdot Ay=-\be_1\cdot y$)
if and only if $|y-\be_1|^2<1$. Thus
$$
\cl_x(A,e)=\inf\{t>0 : h_x(A, e(t)) \notin {B}_d(0,1)\}
=\inf\{t>0 : e(t) \notin {B}_d(\be_1,1)\}=:\tell(e),
$$
and $g_x(A,e)=(h_x(A,e(\tell(e)-)),h_x(A,e(\tell(e)))]\cap \partial B_d(0,1)=
x+A \tilde g(e)$, where 
$$
\tilde g(e)= (e(\tell(e)-),e(\tell(e))]\cap \partial B_d(\be_1,1).
$$
As a consequence, $g_x(A,e)-g_{x'}(A',e)=x-x' + (A-A')\tilde g(e)$, so that 
$$
\Big||g_x(A,e)-g_{x'}(A',e)|-|x-x'|\Big| \leq |(A-A')\tilde g(e)|\leq ||A-A'|||\tilde g(e)|.
$$
Since $|\tilde g(e)| \leq M(e) \land 2$, we find 
$\Delta(x,x',A,A')\leq ||A-A'||\int_{\cE} (M(e)\land 2)\nn_\beta(\dr e)$ and
the conclusion follows from~\eqref{se1}.
\end{proof}
We will frequently use the following lemma.

\begin{lemma}\label{paradurf}
Grant Assumption~\ref{as}.
For $x\in \pDd$ and $A\in\cI_x$, let $\Dd_{x,A}=\{y \in \R^d : h_x(A,y)\in\Dd\}$.
\vip
(i) For all $x \in \pDd$, all $A\in \cI_x$, all $y \in \cDd_{x,A}$, $|y| \leq \mathrm{diam}(\Dd)$.

\vip

(ii) There is $C>0$ such that for all $x,x' \in \pDd$, all $A \in \cI_x$, all $A'\in \cI_{x'}$,
$$
{\rm Vol}_d(\Dd_{x',A'}\setminus\Dd_{x,A})\leq C \rho_{x,x',A,A'}
$$
where $\rho_{x,x',A,A'}=|x-x'|+||A-A'||$, and for all $z \in \cDd_{x',A'}\setminus \Dd_{x,A}$,
\begin{gather*}
d(z,\pDd_{x,A})=d(h_{x}(A,z),\pDd) \leq C \rho_{x,x',A,A'},\\
d(z,\pDd_{x',A'})=d(h_{x'}(A',z),\pDd) \leq C \rho_{x,x',A,A'}.
\end{gather*}
\end{lemma}

\begin{proof} 
For (i), take $y \in \cDd_{x,A}$. Then $|y|=|Ay|=|h_x(A,y)-x| \leq \mathrm{diam}(\Dd)$,
since $h_x(A,y) \in \cDd$.
\vip
For (ii), take $z \in \cDd_{x',A'}\setminus \Dd_{x,A}$, {\it i.e.} $h_{x'}(A',z)\in \cDd$ but 
$h_{x}(A,z)\notin \Dd$. Thus
$$
d(h_{x}(A,z),\pDd)\lor d(h_{x'}(A',z),\pDd) \leq |h_{x'}(A',z)-h_{x}(A,z)|
\leq |x-x'|+||A-A'||\,|z|\leq C  \rho_{x,x',A,A'},
$$
where $C=1\lor \mathrm{diam}(\Dd)$, since $|z|\leq \mathrm{diam}(\Dd)$ by (i).
Since $\pDd$ is Lebesgue-null, we deduce that
$$
\mathrm{Vol}_d(\Dd_{x',A'}\setminus \Dd_{x,A})\leq \int_{\R^d} \indiq_{\{0<d(h_x(A,z),\Dd)\leq C \rho_{x,x',A,A'}\}}\dr z
=\int_{\R^d} \indiq_{\{0<d(u,\Dd)\leq C \rho_{x,x',A,A'}\}}\dr u,
$$
since $h_x(A,\cdot)$ is an isometry. Substituting $u=\Phi^{-1}(v)$ with $\Phi$ defined in 
Lemma~\ref{parafac}, using that the Jacobian is bounded and that 
$\kappa^{-1} d(v,{B}_d(0,1))\leq d(\Phi^{-1}(v),\Dd)\leq \kappa d(v,{B}_d(0,1))$, we get
$$
\mathrm{Vol}_d(\Dd_{x',A'}\setminus \Dd_{x,A})\leq C\int_{\R^d} \indiq_{\{0<d(v,{B}_d(0,1))\leq \kappa C \rho_{x,x',A,A'}\}}
\dr v
= C \int_{\Sp_{d-1}} \dr \sigma \int_1^{1+\kappa C  \rho_{x,x',A,A'}}r^{d-1}\dr r,
$$
which is smaller than $C\rho_{x,x',A,A'}$ as desired since $\rho_{x,x',A,A'}$ is uniformly bounded 
(recall that $x,x'$ lie at the boundary of a bounded domain and that $A,A'$ are isometries).
\end{proof}

\subsection{Joint law of the undershoot and overshoot}

For $x \in \pDd$, $A \in \cI_x$ and $y,z \in \R^d$ such that $h_x(A,y)\in\cDd$, we set
\begin{equation}\label{bg}
\bg_x(A,y,z) = \Lambda\Big(h_x(A,y),h_x(A,z)\Big).
\end{equation}
When $h_x(A,z) \notin \Dd$, $\bg_x(A,y,z) =(h_x(A,y),h_x(A,z)]\cap\pDd$.
To motivate this section, let us start with the following observation,
that immediately follows from~\eqref{g1} and~\eqref{bg}.

\begin{remark}\label{ggbar}
For any $x\in \pDd$, any $A \in \cI_x$, any $e\in \cE$, recalling that $h_x(A,y)=x+Ay$, 
that $\cl_x(A,e)=\inf\{t>0 : h_x(A,e(t))\notin \Dd\}$
and introducing $\cu(x,A,e)=e(\cl_x(A,e)-)$ and $\co(x,A,e)=e(\cl_x(A,e))$, it holds that
$$
g_x(A,e)=\bg_x(A,\cu(x,A,e),\co(x,A,e)).
$$
\end{remark}

Observe that $h_x(A,\cu(x,A,e))$ (resp. $h_x(A,\co(x,A,e))$) is the position of $h_x(A,e)$ 
just before (resp. just after) exiting $\Dd$. As $g_x(A,e)$ is a deterministic function of 
$(\cu(x,A,e),\co(x,A,e))$, a first step towards the proof of Proposition~\ref{ttoopp} is to obtain 
estimates on the law of $(\cu(x,A,e),\co(x,A,e))$ under $\nn_\beta$. For $y\in \R^d$, we introduce $\delta(y)=d(y,\pDd)$.

\begin{proposition}\label{underover}
Fix $\beta \in \{*\}\cup (0,\alpha/2)$ and grant Assumption~\ref{as}.
There is a constant $C$ such that for all $x \in \pDd$, all $A\in \cI_x$,
\begin{align*}
q_x(\dr y,\dr z&):=\nn_\beta(\cu(x,A,e) \in \dr y, \co(x,A,e)\in \dr z) \\
\leq& \frac{C[\delta(h_x(A,y))]^{\alpha/2}}{|z-y|^{d+\alpha}|y|^{d}}\indiq_{\{h_x(A,y)\in \Dd,h_x(A,z)\notin \Dd\}} 
\dr y \dr z
\!+\!\indiq_{\{\beta \neq *\}} \frac{C}{|z|^{d+\beta}}\indiq_{\{z\in \HH,h_x(A, z)\notin \Dd\}} \delta_0(\dr y) \dr z,
\end{align*}
where $\dr y$ and $\dr z$ stand for the Lebesgue measure on $\R^d$ and $\delta_0$ for the Dirac mass at $0$.
\end{proposition}

We believe that these bounds are sharp when $\beta=*$ 
as they are derived from sharp estimates on the Green function 
of isotropic stable processes in $C^{1,1}$ domains. Moreover, we also believe that this inequality is an equality 
when $\beta=*$ and when $\Dd$ is a ball. 
The following result is likely to be well-known, although we found no precise reference.

\begin{lemma}\label{uos} 
Grant Assumption~\ref{as}. 
Recall that $(Z_t)_{t\geq 0}$ is, under $\PP_x$, an ISP$_{\alpha,x}$. There is a constant $C>0$
such that for all $x \in \Dd$, setting $\tell(Z)=\inf\{t> 0 : Z_t \notin \Dd\}$,
$$
\PP_x(Z_{\tell(Z)-}\in \dr y,Z_{\tell(Z)}\in \dr z ) \leq C \indiq_{\{y \in \Dd, z \notin \Dd\}}
K(x,y,z) \dr y \dr z,
$$
where
$$
K(x,y,z)=|z-y|^{-d-\alpha}|x-y|^{\alpha-d} \min \Big(\Big[\frac{\delta(x)\delta(y)}{|x-y|^2}\Big]^{\alpha/2},1 \Big).
$$
\end{lemma}

Observe that $\cu(x,A,e)=y$ and $\co(x,A,e)=z$ means that the undershoot (resp. overshoot) of $h_x(A,e)$ equals
$h_x(A,y)$ (resp. $h_x(A,z)$). Observe also that since $h_x(A,\cdot)$ is an isometry, for $(Z_t)_{t\geq 0}$ 
an ISP$_{\alpha,a}$, 
the process $(h_x(A,Z_t))_{t\geq 0}$ is an ISP$_{\alpha,h_x(A,a)}$.
Admitting~\eqref{newformula} and Lemma~\ref{uos}, we thus find that when $\beta=*$, informally,
\begin{align*}
q_x(\dr y,\dr z)=& a_*\lim_{r\to 0} r^{-\alpha/2}\PP_{r\be_1}\Big(\cu(x,A,Z)\in \dr y, \co(x,A,Z) \in \dr z\Big)\\
\leq & C \indiq_{\{h_x(A,y) \in \Dd,h_x(A,z)\notin \Dd\}}\lim_{r\to 0} r^{-\alpha/2} K(h_x(A,r\be_1),h_x(A,y),h_x(A,z))\\
=& C  \indiq_{\{h_x(A,y) \in \Dd,h_x(A,z)\notin \Dd\}} |z-y|^{-d-\alpha}|y|^{-d} [\delta(h_x(A,y))]^{\alpha/2}.
\end{align*}
However, even if we assume~\eqref{newformula}, such an argument is not so easy to make rigorous, 
in particular because the undershoot/overshoot is
not a continuous function of the path.

\begin{proof}[Proof of Lemma~\ref{uos}] By Chen~\cite[Theorem 2.4]{chen1999} 
(with much less assumptions on the domain)
we know that there is a constant $C>0$ such that for all $x \in \Dd$, 
$$
G(x,\dr y):=\E_x\Big[\int_0^{\tell(Z)} \indiq_{\{Z_t \in \dr y\}}\dr t\Big]
\leq C |x-y|^{\alpha-d} \min\Big(\Big[\frac{\delta(x)\delta(y)}{|x-y|^2}\Big]^{\alpha/2},1\Big) \indiq_{\{y \in \Dd\}}
\dr y.
$$
We next apply a classical method, see e.g. Bertoin~\cite[Section III, page 76]{bertoin1996levy}.
For $x \in \Dd$  and $f:(\R^d)^2 \to \R_+$ bounded, we write, using the Poisson measure $N$ defining
the stable process $Z$, see~\eqref{eqs}, as well as the compensation formula,
\begin{align*}
\E_x[f(Z_{\tell (Z)-},Z_{\tell(Z)})]=&\E_x\Big[\int_0^\infty \int_{\R^d}
f(Z_\sm,Z_\sm+z) \indiq_{\{\tell(Z)\geq s\}}\indiq_{\{Z_\sm +z \notin \Dd\}} N(\dr s,\dr z)  \Big]\\
=&\E_x\Big[\int_0^\infty \int_{\R^d}
f(Z_s,Z_s+z) \indiq_{\{\tell(Z)\geq s\}}\indiq_{\{Z_s +z \notin \Dd\}} \frac{\dr z}{|z|^{d+\alpha}} \dr s \Big]\\
=&\E_x\Big[\int_0^{\tell(Z)} \int_{\R^d} f(Z_s,Z_s+z) 
\indiq_{\{Z_s +z \notin \Dd\}} \frac{\dr z}{|z|^{d+\alpha}} \dr s\Big]\\
=& \int_{\R^d} \int_{\R^d} f(y,y+z) \indiq_{\{y+z \notin \Dd\}} \frac{\dr z}{|z|^{d+\alpha}} G(x,\dr y).
\end{align*}
Thus 
\begin{align*}
\E_x[f(Z_{\tell (Z)-},Z_{\tell(Z)})]
\leq & C \int_{\R^d} \! \int_{\R^d} f(y,y+z)
|x-y|^{\alpha-d} \min\Big(\Big[\frac{\delta(x)\delta(y)}{|x-y|^2}\Big]^{\alpha/2}\!\!\!,1\Big) 
\indiq_{\{y+z \notin \Dd,y \in \Dd\}}
\frac{\dr y \dr z}{|z|^{d+\alpha}}\\
=& C  \int_{\R^d} \! \int_{\R^d} f(y,z)
|x-y|^{\alpha-d} \min\Big(\Big[\frac{\delta(x)\delta(y)}{|x-y|^2}\Big]^{\alpha/2}\!\!\!,1\Big) 
\indiq_{\{z \notin \Dd,y \in \Dd\}}
\frac{\dr y\dr z}{|z-y|^{d+\alpha}},
\end{align*}
whence the result.
\end{proof}

\begin{proof}[Proof of Proposition~\ref{underover}] 
We fix $x\in\pDd$, $A\in\cI_x$ and use the short notation ${h_x(y)=h_x(A,y)}$,
$\cl_x(e)=\cl_x(A,e)$, $\cu(x,e)=\cu(x,A,e)$ and $\co(x,e)=\co(x,A,e)$.
We set $\Dd_x=\{y \in \R^d : h_x(y) \in \Dd\}$.

\vip

{\it Case $\beta=*$.} Our goal is to prove that for any measurable $F:(\R^d)^2 \to \R_+$,
\begin{gather}
q_x(F):=\int_{(\R^d)^2}F(y,z)q_x(\dr y,\dr z) \leq C \int_{\Dd_x\times\Dd_x^c}F(y,z)G(x,y,z) \dr y \dr z,
\label{tbp1}\\
\hbox{where}\qquad G(x,y,z)=|y-z|^{-d-\alpha} |y|^{-d}[\delta(h_x(y))]^{\alpha/2}.\notag
\end{gather}
By density, it is sufficient to prove~\eqref{tbp1} when $F\in C_c(\Dd_x\times\cDd_x^c,\R_+)$, 
provided we can show that
\begin{equation}\label{tbp2}
q_x((\Dd_x^c\times\R^d) \cup (\R^d \times \cDd_x))=0.
\end{equation}

Fix any measurable $F:(\R^d)^2\to \R_+$.
By Remark~\ref{imp4}, $\cl_x(e)>0$ for $\nn_*$-a.e. $e$, so that
$$
q_x(F) = \lim_{t \to 0} q_{x,t}(F),\quad \hbox{where} \quad
q_{x,t}(F)= \int_\cE F(\cu(x,e),\co(x,e))\indiq_{\{\cl_x(e)>t\}}\nn_*(\dr e).
$$
We now apply the Markov property at time $t>0$, see Lemma~\ref{mark}, to find
$$
q_{x,t}(F)= \int_\cE \E_{e(t)}[F(\cu(x,Z),\co(x,Z))]\indiq_{\{\cl_x(e)>t\}}\nn_*(\dr e),
$$
Since $Z$ is, under $\PP_a$, an ISP$_{\alpha,a}$ and since $h_x$ is an isometry,
$\tilde Z=h_x(Z)$ is, under $\PP_a$, an ISP$_{\alpha,h_x(a)}$. Moreover,
$\tell (\tilde Z):=\inf\{t> 0 : \tilde Z_t \notin \Dd\}=\cl_x(Z)$,
$\cu(x,Z)=h_x^{-1}(\tilde Z_{\tell(\tilde Z)-})$ and $\co(x,Z)=h_x^{-1}(\tilde Z_{\tell(\tilde Z)})$,
whence 
$$
\E_a[F(\cu(x,Z),\co(x,Z))]\!=\!\E_a[F(h_x^{-1}(\tilde Z_{\tell(\tilde Z)-}),h_x^{-1}(\tilde Z_{\tell(\tilde Z)})))
]\!=\!\E_{h_x(a)}[F(h_x^{-1}(Z_{\tell(Z)-}),h_x^{-1}(Z_{\tell(Z)}))]
$$
and thus, by Lemma~\ref{uos},
\begin{equation}\label{ttuu}
\E_a[F(\cu(x,Z),\co(x,Z))]\leq C \int_{(\R^d)^2}F(h_x^{-1}(y),h_x^{-1}(z))\indiq_{\{y \in \Dd, z \notin \Dd\}}K(h_x(a),y,z)
\dr y \dr z.
\end{equation}
Consequently,
\begin{align*}
q_{x,t}(F)
\leq & C \int_\cE \int_{(\R^d)^2}F(h_x^{-1}(y),h_x^{-1}(z))\indiq_{\{y \in \Dd, z \notin \Dd\}}K(h_x(e(t)),y,z)
\dr y \dr z \indiq_{\{\cl_x(e)>t\}}\nn_*(\dr e).
\end{align*}
Performing the substitution $y\mapsto h_x(y), z\mapsto h_x(z)$, we find
\begin{align*}
q_{x,t}(F)\leq & C \int_\cE \int_{(\R^d)^2}F(y,z)\indiq_{\{h_x(y) \in \Dd, h_x(z) \notin \Dd\}}K(h_x(e(t)),h_x(y),h_x(z))
\dr y \dr z \indiq_{\{\cl_x(e)>t\}}\nn_*(\dr e)\\
\leq & C \int_\cE \int_{(\R^d)^2}F(y,z)\indiq_{\{y \in \Dd_x, z \notin \Dd_x\}}K(h_x(e(t)),h_x(y),h_x(z))
\dr y \dr z \indiq_{\{\ell(e)>t,e(t)\in \Dd_x\}}\nn_*(\dr e),
\end{align*}
since $\cl_x(e)>t$ implies that $\ell(e)>t$ and $h_x(e(t))\in \Dd$.
Introducing the measure  $k_t(\dr a)=n_*(e(t)\in \dr a,\ell(e)>t)$, which is carried by $\HH$,
we get
$$
q_{x,t}(F) \leq C \int_{\HH} \int_{(\R^d)^2}F(y,z)\indiq_{\{y \in \Dd_x, z \notin \Dd_x, a \in \Dd_x\}}
K(h_x(a),h_x(y),h_x(z))\dr y \dr z k_t(\dr a).
$$
For all $a\in \Dd_x$, it holds that $\delta(h_x(a))=d(a,\pDd_x)\leq a_1$, where $a_1:=a\cdot\be_1>0$:
since $\Dd_x \subset \HH$, $a-a_1\be_1 \in \partial \HH$ cannot belong to $\Dd_x$,
so that $d(a,\pDd_x)\leq |a-(a-a_1\be_1)|=a_1$.
Consequently, recalling the expression of $K$ and that $h_x$ is an isometry,
\begin{gather*}
K(h_x(a),h_x(y),h_x(z))\leq a_1^{\alpha/2} H(a,x,y,z),\\
\hbox{where} \quad H(a,x,y,z)=a_1^{-\alpha/2}|z-y|^{-d-\alpha}|a-y|^{\alpha-d}
\Big(\Big[\frac{a_1\delta(h_x(y))}
{|a-y|^2}\Big]^{\alpha/2}\land 1\Big).
\end{gather*} 
Moreover, $a\in \Dd_x$ implies that $|a|\leq$ diam$(\Dd)=:D$, see Lemma~\ref{paradurf}-(i).
All in all,
\begin{align*}
q_{x,t}(F)\leq & C \int_{\HH\cap \closure{B}_d(0,D)} f(a) a_1^{\alpha/2} k_t(\dr a),\quad \hbox{where}\quad f(a)=
\int_{\Dd_x\times\Dd_x^c} F(y,z)H(a,x,y,z) \dr y \dr z.
\end{align*}
Choosing $F=\indiq_{(\Dd_x^c\times\R^d) \cup (\R^d \times \cDd_x)}$, we have $f\equiv0$
(since $\pDd_x$ is Lebesgue-null), whence
$$
q_x((\Dd_x^c\times\R^d) \cup (\R^d \times \cDd_x))=\lim_{t\to 0} q_{x,t}(F)=0,
$$
and we have checked~\eqref{tbp2}. Consider next $F\in C_c(\Dd_x\times\cDd_x^c,\R_+)$ and assume
we can show that 
\begin{gather}
\hbox{$f$ is bounded on $\HH\cap \closure{B}_d(0,D)$} \quad \text{and} \quad
\lim_{a \to 0} f(a)= \int_{\Dd_x\times\Dd_x^c} F(y,z)G(x,y,z) \dr y \dr z. \label{tbp3}
\end{gather}
Then by Lemma~\ref{encore1}, which tells us that $a_1^{\alpha/2} k_t(\dr a) \to c_0\delta_0(\dr a)$
weakly as $t\to 0$,
$$
q_x(F)=\lim_{t\to 0} q_{x,t}(F) \leq C \lim_{t\to 0} \int_{\HH \cap \bar B_d(0,D)} f(a) a_1^{\alpha/2}k_t(\dr a)=
c_0C \int_{\Dd_x\times\Dd_x^c} F(y,z)G(x,y,z)\dr y \dr z,
$$
which shows~\eqref{tbp1}.

\vip
It remains to prove~\eqref{tbp3}. Since $F\in C_c(\Dd_x\times\cDd_x^c,\R_+)$, there exists
$\e>0$ such that $F(y,z)=0$ as soon as $|y-z|\leq \e$ or $|z|\geq 1/\e$, whence
$$
F(y,z)H(a,x,y,z)\leq ||F||_\infty \indiq_{\{|z|< 1/\e\}} 
a_1^{-\alpha/2} \e^{-d-\alpha} |a-y|^{\alpha-d} \Big(\frac{a_1^{\alpha/2}[\delta(h_x(y))]^{\alpha/2}}{|a-y|^\alpha}\land 1\Big).
$$
But $\delta(h_x(y))\leq |h_x(a)-h_x(y)|+\delta(h_x(a))\leq |a-y|+a_1$, whence, for some $C$
depending on $F$,
\begin{align}
F(y,z)H(a,x,y,z)\leq& C \indiq_{\{|z|< 1/\e\}} \Big(|a-y|^{\alpha/2-d} + \frac{a_1^{\alpha/2}}{|a-y|^d}\land 
\frac1{a_1^{\alpha/2}|a-y|^{d-\alpha}}\Big) \notag\\
\leq&  C \indiq_{\{|z|< 1/\e\}} |y-a|^{\alpha/2-d}, \label{pasrelou}
\end{align}
since $u\land v \leq \sqrt{uv}$.
Consequently, for all $a \in \HH\cap \closure{B}_d(0,D)$, recalling the definition of $f$ and
that $\Dd_x \subset B_d(0,D)$
by Lemma~\ref{paradurf}-(i), we find
\begin{align*}
f(a)\leq& C \int_{{B}_d(0,D)} |y-a|^{\alpha/2-d} \dr y \int_{B_d(0,1/\e)}\dr z
= C \int_{{B}_d(-a,D)} |u|^{\alpha/2-d} \dr u \leq C  \int_{{B}_d(0,2D)} |u|^{\alpha/2-d} \dr u,
\end{align*}
so that $f$ is bounded on $\HH \cap \closure{B}_d(0,D)$. Next, we clearly have
$\lim_{a\to 0} H(a,x,y,z)=G(x,y,z)$ for each fixed $y,z \in \Dd_x\times\cDd_x^c$.
Hence to establish the limit in~\eqref{tbp3}, we only have to check the uniform integrability
property $\lim_{M\to \infty} \sup_{a \in \HH\cap \closure{B}_d(0,D)} R_M(a)=0$, where
$$
R_M(a)=\int_{\Dd_x\times\Dd_x^c} F(y,z)H(a,x,y,z)\indiq_{\{F(y,z)H(a,x,y,z)\geq M\}} \dr y \dr z.
$$
By~\eqref{pasrelou} and since $\cDd_x \subset \closure{B}_d(0,D)$, for all $a \in \HH\cap \closure{B}_d(0,D)$, 
substituting $u=y-a$
$$
R_M(a)\leq C \int_{{B}_d(0,D)} |y-a|^{\alpha/2-d} \indiq_{\{|y-a|^{\alpha/2-d}\geq M/C\}}\dr y
\leq C \int_{{B}_d(0,2D)} |u|^{\alpha/2-d} \indiq_{\{|u|^{\alpha/2-d}\geq M/C\}}\dr u,
$$
which does not depend on $a$ and tends to $0$ as $M\to \infty$.
\vip

{\it Case $\beta\in (0,\alpha/2)$.}  We write $q_x=q_x^1+q_x^2$, where
\begin{gather*}
q_x^1(\dr y,\dr z) = \nn_\beta(h_x(e(0)) \in \Dd,\cu(x,e) \in \dr y, \co(x,e)\in \dr z),\\
q_x^2(\dr y,\dr z) = \nn_\beta(h_x(e(0)) \notin \Dd,\cu(x,e) \in \dr y, \co(x,e)\in \dr z).
\end{gather*}
We clearly have 
$\cl_x(e)=0$ when $h_x(e(0)) \notin \Dd$ and, recalling the definitions of $\cu$ and $\co$ and 
the convention that $e(0-)=0$,
$$
q_x^2(\dr y,\dr z)=\nn_\beta(h_x(e(0)) \notin \Dd, 0 \in \dr y, e(0)\in \dr z)= \delta_0(\dr y)
\nn_\beta(h_x(e(0)) \notin \Dd, e(0)\in \dr z).
$$
By definition of $\nn_\beta$, see~\eqref{nnb}, and since $\PP_a(Z(0)\in \dr z)=\delta_a(\dr z)$, this gives 
$$
q_x^2(\dr y,\dr z)=  \delta_0(\dr y) 
\int_{a \in \HH} |a|^{-d-\beta}\indiq_{\{h_x(z) \notin \Dd\}}  \delta_a(\dr z) \dr a
=  \delta_0(\dr y) |z|^{-d-\beta}\indiq_{\{z\in\HH,h_x(z) \notin \Dd\}} \dr z.
$$
By definition of $\nn_\beta$ and by~\eqref{ttuu},
\begin{align*}
q^1_x(\dr y,\dr z)=&\int_{a \in \HH} |a|^{-d-\beta} 
\PP_a(h_x(Z(0))\in \Dd, \cu(x,Z) \in \dr y, \co(x,Z)\in \dr z) \dr a\\
=& \int_{a \in \HH} |a|^{-d-\beta} \indiq_{\{h_x(a) \in \Dd\}}
\PP_a(\cu(x,Z) \in \dr y, \co(x,Z)\in \dr z) \dr a\\
\leq &  C \int_{a \in \HH} |a|^{-d-\beta} \indiq_{\{h_x(a) \in \Dd\}} \indiq_{\{h_x(y) \in \Dd, h_x(z) \notin \Dd\}}
K(h_x(a),h_x(y),h_x(z))\dr a \dr y \dr z\\
=& C  \indiq_{\{h_x(y) \in \Dd, h_x(z) \notin \Dd\}} |z-y|^{-d-\alpha} m(x,y) \dr y \dr z,
\end{align*}
where
$$
m(x,y)=\int_{a \in \HH} |a|^{-d-\beta} \indiq_{\{h_x(a) \in \Dd\}}|a-y|^{\alpha-d} 
\Big(\Big[\frac{\delta(h_x(a))\delta(h_x(y))}{|a-y|^2}\Big]^{\alpha/2}\land 1 \Big)\dr a.
$$
It only remains to show that for some constant $C>0$, for all $x\in\pDd$, all $y \in \R^d$ with
$h_x(y)\in\Dd$, 
\begin{equation}\label{cqv}
m(x,y)\leq C |y|^{-d} [\delta(h_x(y))]^{\alpha/2}.
\end{equation}
We write $m(x,y)=m_1(x,y)+m_2(x,y)$, where $m_1$ (resp. $m_2$) corresponds to the integral on $a \in \HH\setminus B_d(y,|y|/2)$ (resp. $a\in B_d(y,|y|/2)$).

\vip
First, $a \notin B_d(y,|y|/2)$ implies that
$|y-a|\geq |y|/2$, and $h_x(a) \in \Dd$ implies that $|a|\leq$ diam$(\Dd)=:D$, see Lemma~\ref{paradurf}-(i).
Moreover,  $\delta(h_x(a))\leq |h_x(a)-x|=|a|$, since $x \in \pDd$. Thus
\begin{align*}
m_1(x,y)\leq& \int_{B_d(0,D)\cap B_d(y,|y|/2)^c}  |a|^{-d-\beta} |a-y|^{-d} 
[\delta(h_x(a))\delta(h_x(y))]^{\alpha/2} \dr a\\
\leq& C [\delta(h_x(y))]^{\alpha/2}|y|^{-d} \int_{B_d(0,D)} |a|^{-d-\beta+\alpha/2} \dr a.
\end{align*}
Since $\beta<\alpha/2$, we conclude that $m_1(x,y) \leq C [\delta(h_x(y))]^{\alpha/2}|y|^{-d}$.

\vip

Next, $a \in B_d(y,|y|/2)$ implies that $|a|>|y|/2$. Using that $\delta(h_x(a))\leq \delta(h_x(y))+|y-a|$,
\begin{align*}
m_2(x,y)\leq& C |y|^{-d-\beta} \int_{B_d(y,|y|/2)} |a-y|^{\alpha-d}\Big(\Big[
\frac{\delta(h_x(y))(\delta(h_x(y))+|a-y|)}{|a-y|^2}\Big]^{\alpha/2}\land 1 \Big)  \dr a.
\end{align*}
But
\begin{align*}
\frac{\delta(h_x(y))(\delta(h_x(y))+|a-y|)}{|a-y|^2}\land 1
\leq&\frac{(\delta(h_x(y)))^2}{|a-y|^2}\land 1 +
\frac{\delta(h_x(y))}{|a-y|}\land 1
\leq  2 \frac{\delta(h_x(y))}{|a-y|},
\end{align*}
so that 
$$
m_2(x,y)\leq C |y|^{-d-\beta} [\delta(h_x(y))]^{\alpha/2} \int_{B_d(y,|y|/2)} |a-y|^{\alpha/2-d} \dr a
=C |y|^{\alpha/2-d-\beta} [\delta(h_x(y))]^{\alpha/2}.
$$
This last quantity is bounded by $C |y|^{-d} [\delta(h_x(y))]^{\alpha/2}$, because 
$\alpha/2>\beta$ and because $|y|\leq $ diam$(\Dd)$, see Lemma~\ref{paradurf}-(i).
\end{proof}

\subsection{A geometric inequality}

Recall that $\delta(y)=d(y,\pDd)$ and that $\bg_x(A,y,z)=\Lambda(h_x(A,y),h_x(A,z))$.

\begin{proposition}\label{tyvmmm}
Grant Assumption~\ref{as}.
There is a constant $C\in (0,\infty)$ such that for all ${x,x' \in \pDd}$, all $A \in \cI_x$, $A'\in \cI_{x'}$,
setting $\rho_{x,x',A,A'}=|x-x'|+||A-A'||$,
\vip
(i) for all $y,z \in \R^d$
such that $h_x(A,y)\in\Dd$, $h_{x'}(A',y)\in\Dd$, $h_x(A,z)\notin\Dd$, $h_{x'}(A',z)\notin\Dd$,
$$
\Big||\bg_x(A,y,z)- \bg_{x'}(A',y,z)| - |x-x'|\Big|
\leq C \rho_{x,x',A,A'}\Big(|z|\land 1 + \frac{|y|(|y-z|\land 1)}{\delta(h_x(A,y))\land
\delta(h_{x'}(A',y)) } \Big),
$$

(ii) for all $y,z \in \R^d$ 
such that $h_x(A,y)\in\Dd$, $h_{x'}(A',y)\in\Dd$, $h_x(A,z)\notin\Dd$, $h_{x'}(A',z)\in\Dd$,
$$
\Big||\bg_x(A,y,z)- h_{x'}(A',z)| - |x-x'|\Big|
\leq C \rho_{x,x',A,A'} \Big(|z|\land 1 + \frac{|y|(|y-z|\land 1)}{\delta(h_x(A,y))} \Big),
$$

(iii) for all $z \in \R^d$ such that $h_x(A,z)\notin\Dd$, $h_{x'}(A',z)\notin \Dd$,
$$
\Big||\bg_x(A,0,z)- \bg_{x'}(A',0,z)| - |x-x'|\Big|
\leq C \rho_{x,x',A,A'},
$$

(iv) for all $z \in \R^d$,
$$
\Big||\bg_x(A,0,z)- \bg_{x'}(A',0,z)| - |x-x'|\Big| \leq C.
$$
\end{proposition}

The proof of this proposition is handled in Appendix~\ref{ageo}, to which we refer for some illustrations.
The shape of these upperbounds is of course motivated by the estimates we have on the joint law of the 
undershoot/overshoot, see Proposition~\ref{underover}. However, these bounds look quite sharp.
Let us also mention Remark~\ref{strange}, where we show that the strong convexity assumption of $\Dd$
cannot be removed, although it is likely that some more tricky upperbounds should hold without such an assumption.

\subsection{Bounding a few integrals}

We will need the following bounds.

\begin{proposition}\label{int}
Fix $\beta \in (0,1)$ and grant Assumption~\ref{as}.
There is $C\in (0,\infty)$ such that for all $x,x' \in \pDd$, all $A\in \cI_x$, $A' \in \cI_{x'}$,
setting $\rho_{x,x',A,A'}=|x-x'|+||A-A'||$, 
\begin{gather}
\int_{(\R^d)^2}\Big(|z|\land 1 + \frac{|y|(|y-z|\land 1)}{\delta(h_x(A,y))} \Big)
\frac{[\delta(h_x(A,y))]^{\alpha/2}}{|z-y|^{d+\alpha}|y|^{d}}\indiq_{\{h_x(A,y) \in \Dd, h_x(A,z)\notin \Dd\}}
\dr y \dr z \leq C, \label{jjab1} \\
\int_{(\R^d)^2}\!\!\!  \frac{[\delta(h_{x'}(A',z))]^{\alpha/2}[\delta(h_x(A,y))]^{\alpha/2}}{|z-y|^{d+\alpha}|y|^d}
\indiq_{\{h_x(A,y) \in \Dd,h_x(A,z)\notin \Dd,h_{x'}(A',z) \in \Dd\}}\dr y \dr z
\leq C\rho_{x,x',A,A'}, \label{jjab3}\\
\int_{\HH} \indiq_{\{h_x(A,z)\notin \Dd\}} \frac{\dr z}{|z|^{d+\beta}} \leq C, \label{jjab5}\\
\int_\HH \indiq_{\{h_x(A,z)\notin \Dd,h_{x'}(A',z) \in \Dd\}} \frac1{|z|^{d+\beta}}\dr z\leq C\rho_{x,x',A,A'}. \label{jjab2}
\end{gather}
\end{proposition}

We will use the following inequality.

\begin{lemma}\label{egrelou}
There is a constant $c_{d,\alpha}>0$ such that
for $x,z \in \R^d$
$$
\hbox{if $|x|=1<|z|$,} \qquad \int_{{B}_d(0,1)} \frac{(1-|a|^2)^{\alpha/2}}
{|a-z|^{d+\alpha}|a-x|^{d}} \dr a
\leq \frac{c_{d,\alpha}}{|z-x|^{d}(|z|^2-1)^{\alpha/2}}.
$$
\end{lemma}

\begin{proof} This inequality may actually be an equality and a more straightforward proof could likely be 
provided.
Let $Z$ be some ISP$_{\alpha,0}$ and set $\tau(Z)=\inf\{t>0, |Z_t|>1\}$.
By Blumenthal, Getoor and Ray~\cite[Theorem A]{blumenthal1961distribution}, 
there is $c_{1,\alpha,d}$ such that for all $x\in \R^d$ 
with $|x|<1$,
$$
\PP_x(Z_{\tau(Z)}\in \dr z)= \frac{c_{1,\alpha,d} (1-|x|^2)^{\alpha/2}}{|z-x|^d(|z|^2-1)^{\alpha/2}} \indiq_{\{|z|>1\}}\dr z.
$$
By Kyprianou, Rivero and Satitkanitkul~\cite[Corollary 1.4-(ii)]{KypRivSat}, for all $x\in \R^d$ with $|x|<1$,
$$
\PP_x(Z_{\tau(Z)-}\in \dr a, Z_{\tau(Z)}\in \dr z)= \frac{c_{2,\alpha,d} \,
\varphi(\zeta(x,a))}{|x-a|^{d-\alpha}|z-a|^{d+\alpha}} 
\indiq_{\{|a|<1< |z|\}}\dr a \dr z,
$$
where $\varphi(u)=\int_0^u (v+1)^{-d/2}v^{\alpha/2-1}\dr v$ and 
$\zeta(x,a)=\frac{(1-|x|^2)(1-|a|^2)}{|x-a|^{2}}$.
Hence if $|x|<1<|z|$,
\begin{equation}\label{deqq}
\int_{{B}_d(0,1)} \frac{\varphi(\zeta(x,a))}{ (1-|x|^2)^{\alpha/2}|x-a|^{d-\alpha}|z-a|^{d+\alpha}} \dr a
= \frac{c_{1,\alpha,d}}{c_{2,\alpha,d}|z-x|^d(|z|^2-1)^{\alpha/2}}.
\end{equation}
Fix now $x,z \in \R^d$ with $|x|=1<|z|$, and consider $x_n=\frac{n-1}n x$.
Since $\varphi(u)\sim_0 (2/\alpha) u^{\alpha/2}$,
$$
\text{for all $a\in B_d(0,1)$,}\quad \lim_n\frac{\varphi(\zeta(x_n,a))}
{ (1-|x_n|^2)^{\alpha/2}|x_n-a|^{d-\alpha}|z-a|^{d+\alpha}}= 
\frac{2 (1-|a|^2)^{\alpha/2}}{\alpha |x-a|^d|z-a|^{d+\alpha}}.
$$
The conclusion follows from~\eqref{deqq} applied to $x_n$ and from the Fatou lemma.
\end{proof}

We can now give the

\begin{proof}[Proof of Proposition~\ref{int}] We divide the proof in several steps.
\vip

{\it Step 1.} There is $C>0$ such that for all $x \in \pDd$, $A\in\cI_x$,
$z \in \R^d$ with $h_x(A,z) \notin \Dd$,
$$
f_x(z):= \int_{\R^d}  \frac{[\delta(h_x(A,y))]^{\alpha/2}}{|z-y|^{d+\alpha}|y|^{d}}\indiq_{\{h_x(A,y)\in \Dd\}} \dr y
\leq \frac{C}{|z|^d [\delta(h_x(A,z))]^{\alpha/2}}.
$$
Indeed, we substitute $u=h_x(A,y)$, $v=h_x(A,z)$ and write
$$
f_x(z)=\int_{\Dd}  \frac{[\delta(u)]^{\alpha/2}}{|v-u|^{d+\alpha}|u-x|^{d}} \dr u.
$$
We now use the diffeomorphism $\Phi:\R^d \to \R^d$ such that $\Phi(\Dd)=B_d(0,1)$ 
introduced in Lemma~\ref{parafac} and substitute
$a=\Phi(u)$. The Jacobian is bounded, and we get
$$
f_x(z) \leq C \int_{{B}_d(0,1)} \frac{[\delta(\Phi^{-1}(a))]^{\alpha/2}}
{|v-\Phi^{-1}(a)|^{d+\alpha}|\Phi^{-1}(a)-x|^{d}} \dr a.
$$
But, see Lemma~\ref{parafac}, we have, for $a \in B_d(0,1)$,
\begin{gather*}
\delta(\Phi^{-1}(a))\leq \kappa d(a,\partial B_d(0,1))=\kappa (1-|a|)
\leq \kappa (1-|a|^2),\\
|v-\Phi^{-1}(a)|\geq \kappa^{-1} |\Phi(v)-a| \qquad \hbox{and}\qquad
|\Phi^{-1}(a)-x|\geq \kappa^{-1} |a-\Phi(x)|. 
\end{gather*}
Hence
$$
f_x(z) \leq C \int_{{B}_d(0,1)} \frac{(1-|a|^2)^{\alpha/2}}
{|a-\Phi(v)|^{d+\alpha}|a-\Phi(x)|^{d}} \dr a
\leq \frac{C}{|\Phi(v)-\Phi(x)|^{d}(|\Phi(v)|^2-1)^{\alpha/2}}
$$
because $|\Phi(v)|>1=|\Phi(x)|$, see Lemma~\ref{egrelou}. 
To conclude the step, it suffices to notice that
\begin{gather*}
|\Phi(v)|^2-1\geq |\Phi(v)|-1=d(\Phi(v),\partial B_d(0,1)) \geq \kappa^{-1} d(v,\pDd)=
\kappa^{-1} \delta(h_x(A,z)),\\
|\Phi(v)-\Phi(x)|\geq \kappa^{-1} |v-x|=\kappa^{-1}|h_x(A,z)-x|= \kappa^{-1}|z|.
\end{gather*}

{\it Step 2.} We now bound the first term (with $|z|\land 1$) in~\eqref{jjab1}. Calling it $I_1(x)$,
we see that
$$
I_1(x)=\int_{\R^d} (|z|\land 1) f_x(z)\indiq_{\{h_x(A,z)\notin \Dd\}} \dr z
\leq C \int_{\R^d} \frac{|z|\land 1}{|z|^d [\delta(h_x(A,z))]^{\alpha/2}} \indiq_{\{h_x(A,z)\notin \Dd\}} \dr z.
$$
As in Step 1, we substitute $v=h_x(A,z)$ (whence $|z|=|v-x|$) and then $b=\Phi(v)$ to write
\begin{align*}
I_1(x) \leq C  \int_{\Dd^c} \frac{|v-x|\land 1}{|v-x|^d [\delta(v)]^{\alpha/2}} \dr v
\leq C \int_{{B}_d(0,1)^c}  \frac{|\Phi^{-1}(b)-x|\land 1}{|\Phi^{-1}(b)-x|^d [\delta(\Phi^{-1}(b))]^{\alpha/2}} \dr b.
\end{align*}
By Lemma~\ref{parafac}, we have $\kappa^{-1}|b-\Phi(x)|\leq |\Phi^{-1}(b)-x|\leq \kappa|b-\Phi(x)|$, as well as
$\delta(\Phi^{-1}(b))\geq \kappa^{-1} d(b,\partial B_d(0,1))=\kappa(|b|-1)$. Thus
$$
I_1(x)\leq C \int_{{B}_d(0,1)^c}  \frac{|b-\Phi(x)|\land 1}{|b-\Phi(x)|^d (|b|-1)^{\alpha/2}} \dr b
=C\int_{\Sp_{d-1}} \int_1^\infty r^{d-1}\frac{|r\sigma-\Phi(x)|\land 1}{|r\sigma-\Phi(x)|^d (r-1)^{\alpha/2}}  
\dr r \dr \sigma.
$$
By rotational invariance and since $|\Phi(x)|=1$, we have $I_1(x)\leq C J_1$, where
$$
J_1=\int_{\Sp_{d-1}} \int_1^\infty r^{d-1}\frac{|r\sigma-\be_1|\land 1}{|r\sigma-\be_1|^d (r-1)^{\alpha/2}}  
\dr r \dr \sigma,
$$
and it only remains to check that $J_1$ is finite. We write $J_1\leq J_{11}+J_{12}$, where
\begin{gather*}
J_{11}=\int_{\Sp_{d-1}} \int_2^\infty r^{d-1}\frac{|r\sigma-\be_1|\land 1}{|r\sigma-\be_1|^d (r-1)^{\alpha/2}}  
\dr r \dr \sigma \leq \int_{\Sp_{d-1}} \int_2^\infty \frac{r^{d-1}}{|r\sigma-\be_1|^d (r-1)^{\alpha/2}}  
\dr r \dr \sigma, \\
J_{12}=  \int_{\Sp_{d-1}} \int_1^2 r^{d-1}\frac{|r\sigma-\be_1|\land 1}{|r\sigma-\be_1|^d (r-1)^{\alpha/2}}  
\dr r \dr \sigma \leq  C\int_{\Sp_{d-1}} \int_1^2 \frac{1}{|r\sigma-\be_1|^{d-1} (r-1)^{\alpha/2}}  
\dr r \dr \sigma.
\end{gather*}
For $r\geq 2$, we have $|r\sigma -\be_1|\geq r-1 \geq \frac r 2$, whence
$$
J_{11}\leq C \int_{\Sp_{d-1}} \int_2^\infty \frac{1}{r^{1+\alpha/2}}
\dr r \dr \sigma <\infty.
$$
For $r\in [1,2]$, we have $|r\sigma -\be_1|\geq r-1$ and $|r\sigma -\be_1|\geq |\sigma-\be_1|$
(because $|r\sigma -\be_1|^2=1+r^2-2r\sigma_1\geq 2r(1-\sigma_1)\geq 2(1-\sigma_1)=|\sigma-\be_1|^2$), so that
$|r\sigma -\be_1|\geq \frac{r-1+|\sigma-\be_1|}2$. Hence
$$
J_{12}\leq C \int_1^2 \frac{\dr r}{(r-1)^{\alpha/2}} \int_{\Sp_{d-1}} \frac{\dr \sigma}{(r-1+|\sigma-\be_1|)^{d-1}}
\leq C \int_1^2 \frac{\dr r}{(r-1)^{\alpha/2}} \Big(1+\log \frac{1}{r-1}\Big)<\infty.
$$

{\it Step 3.} We now show that for all $\theta\in (0,1)$, there is $C_\theta>0$ such that
for all $x \in \pDd$,
$$
K_\theta(x)= \int_{\Dd} \frac{\dr u}{|u-x|^{d-1}[\delta(u)]^{\theta}} \leq C_\theta.
$$
We substitute $a=\Phi(u)$ and find, similarly as in Steps 1 and 2, 
$$
K_\theta(x)\leq  C \int_{{B}_d(0,1)} \frac{\dr a}{|\Phi^{-1}(a)-x|^{d-1}[\delta(\Phi^{-1}(a))]^{\theta}}
\leq C \int_{{B}_d(0,1)} \frac{\dr a}{|a-\Phi(x)|^{d-1}(1-|a|)^{\theta}}.
$$
By rotational invariance and since $|\Phi(x)|=1$, $K_\theta(x)\leq CL_\theta$, where
$$
L_\theta = C \int_{{B}_d(0,1)} \frac{\dr a}{|a-\be_1|^{d-1}(1-|a|)^{\theta}}
=C\int_0^1 \int_{\Sp_{d-1}} \frac{r^{d-1} \dr \sigma \dr r}{|r\sigma - \be_1|^{d-1}(1-r)^{\theta}}.
$$
But for $r\in (0,1)$, 
$|r\sigma-\be_1|\geq 1-r$ and $|r\sigma-\be_1|\geq r |\sigma-\be_1|$ (because $|r\sigma-\be_1|^2=1+r^2-2r\sigma_1
\geq 2r  (1-\sigma_1)=r |\sigma-\be_1|^2\geq r^2 |\sigma-\be_1|^2$), so that 
$|r\sigma-\be_1|\geq \frac{1-r + r|\sigma-\be_1|}2$ and thus
$$
\frac r{|r\sigma-\be_1|}\leq \frac{2r}{1-r + r|\sigma-\be_1|}
=\frac{2}{\frac{1-r}r + |\sigma-\be_1|}
\leq \frac{2}{1-r + |\sigma-\be_1|}.
$$
Consequently,
$$
L_\theta\leq C\int_0^1 \frac{\dr r}{(1-r)^\theta} \int_{\Sp_{d-1}} \frac{\dr \sigma}{(1-r+|\sigma - \be_1|)^{d-1}}
\leq C \int_0^1 \frac{\dr r}{(1-r)^\theta}\Big(1+\log\frac 1{1-r} \Big)<\infty.
$$

{\it Step 4.} We next bound the second term in~\eqref{jjab1}. Naming it $I_2(x)$, we have
$$
I_2(x)= \int_{(\R^d)^2} \frac{|y-z|\land 1}{|y|^{d-1}|y-z|^{d+\alpha} [\delta(h_x(A,y))]^{1-\alpha/2}}
\indiq_{\{h_x(A,y) \in \Dd, h_x(A,z)\notin \Dd\}} \dr y \dr z.
$$
We substitute $u=h_x(A,y)$, $v=h_x(A,z)$ to find
$$
I_2(x)=\int_{\Dd} \int_{\Dd^c} \frac{|u-v|\land 1}{|u-x|^{d-1}|u-v|^{d+\alpha} [\delta(u)]^{1-\alpha/2}}  \dr v \dr u.
$$
But for $u \in \Dd$, we have $\Dd^c \subset B_d(u,\delta(u))^c$ (because 
$\delta(u)=d(u,\Dd^c)$). Moreover,
$$
\int_{B_d(u,\delta(u))^c}\frac{|u-v|\land 1}{|u-v|^{d+\alpha}} \dr v
=\int_{|w|>\delta(u)} \frac{|w|\land 1}{|w|^{d+\alpha}}\dr w =: \rho_\alpha(u).
$$
We thus have
$$
I_2(x) \leq \int_{\Dd} \frac{\rho_\alpha(u)}{|u-x|^{d-1} [\delta(u)]^{1-\alpha/2}} \dr u.
$$
Recall that $\delta(u)$ is bounded (for $u\in \Dd$) by the diameter of $\Dd$.
One easily checks that $\rho_\alpha(u) \leq C$ if $\alpha<1$, that $\rho_\alpha(u) 
\leq C(1+\log (1+\frac 1 {\delta(u)}))\leq (\delta(u))^{-1/4}$ if $\alpha=1$, and that $\rho_\alpha(u) 
\leq C(\delta(u))^{1-\alpha}$ if $\alpha \in (1,2)$, so that in any case,
$$ 
\frac{\rho_\alpha(u)}{|u-x|^{d-1} [\delta(u)]^{1-\alpha/2}} \leq \frac{C}{|u-x|^{d-1} [\delta(u)]^{\theta}}
\quad \text{whence} \quad I_2(x)\leq CK_\theta(x),
$$
with $\theta=1-\alpha/2$ if $\alpha \in (0,1)$, $\theta=3/4$ if $\alpha=1$ and $\theta=\alpha/2$ 
if $\alpha \in (1,2)$.
We always have $\theta \in (0,1)$, whence $\sup_{x\in \pDd} I_2(x)<\infty$ by Step~3.
Recalling Step 2, we have checked~\eqref{jjab1}.

\vip

{\it Step 5.} Calling $I_3(x,x')$ the left hand side of~\eqref{jjab3} and recalling Step 1, we have
\begin{align*}
I_3(x,x')\leq& \int_{\R^d} [\delta(h_{x'}(A',z))]^{\alpha/2} f_x(z)\indiq_{\{h_x(A,z)\notin \Dd, h_{x'}(A',z) \in \Dd\}} \dr z\\
\leq&  C \int_{\R^d} \frac{[\delta(h_{x'}(A',z))]^{\alpha/2}}{|z|^d [\delta(h_x(A,z))]^{\alpha/2}} 
\indiq_{\{h_x(A,z)\notin \Dd, h_{x'}(A',z) \in \Dd\}} \dr z.
\end{align*}
We now use the notation of Lemma~\ref{paradur} and observe that $h_x(A,z)\notin \Dd$ and $h_{x'}(A',z) \in \Dd$
if and only if $z \in \Dd_{x',A'}\setminus \Dd_{x,A}$, while $\delta(h_x(A,z))=d(z,\pDd_{x,A})$. Hence
$$
I_3(x,x')\leq C \int_{\R^d} \frac{[d(z,\pDd_{x',A'})]^{\alpha/2}}{|z|^d [d(z,\pDd_{x,A})]^{\alpha/2}} 
\indiq_{\{z \in \Dd_{x',A'}\setminus \Dd_{x,A}\}} \dr z=I_{31}(x,x')+I_{32}(x,x'),
$$
$I_{31}(x,x')$ (resp. $I_{32}(x,x')$) being the integral on $B_d(0,\e_1)$ (resp. $B_d(0,\e_1)^c$).
Using the functions $\psi_{x,A},\psi_{x',A'}:B_{d-1}(0,\e_1)\to\R$ introduced in Lemma~\ref{paradur},
we may write
\begin{align*}
I_{31}(x,x')=&C \int_{B_d(0,\e_1)} \frac{[d(z,\pDd_{x',A'})]^{\alpha/2}}{|z|^d [d(z,\pDd_{x,A})]^{\alpha/2}} 
\indiq_{\{\psi_{x',A'}(z_2,\dots,z_d)< z_1 \leq  \psi_{x,A}(z_2,\dots,z_d)\}} \dr z\\
\leq &  C \int_{B_{d-1}(0,\e_1)} \indiq_{\{\psi_{x',A'}(v)<\psi_{x,A}(v)\}}\frac{\dr v}{|v|^d} 
\int_{\psi_{x',A'}(v)}^{ \psi_{x,A}(v)} \frac{[d(z,\pDd_{x',A'})]^{\alpha/2}}{[d(z,\pDd_{x,A})]^{\alpha/2}} \dr z_1,
\end{align*}
where we have set $z=(z_1,v)$ and we used that $|z| \geq |v|$. But, see Lemma~\ref{paradur}, we have
$d(z,\pDd_{x',A'}) \leq  z_1-\psi_{x',A'}(v)$ and 
$d(z,\pDd_{x,A})\geq \frac12 (\psi_{x,A}(v)-z_1)$, so that
$$
I_{31}(x,x')
\leq  C \int_{B_{d-1}(0,\e_1)} \indiq_{\{\psi_{x',A'}(v)<\psi_{x,A}(v)\}} \frac{\dr v}{|v|^d} 
\int_{\psi_{x',A'}(v)}^{ \psi_{x,A}(v)} 
\frac{[z_1-\psi_{x',A'}(v)]^{\alpha/2}}{[\psi_{x,A}(v)-z_1]^{\alpha/2}} \dr z_1.
$$
Setting $c_\alpha=\int_0^1 (\frac{u}{1-u})^{\alpha/2}\dr u$, we have 
$\int_a^b (\frac{z_1-a}{b-z_1})^{\alpha/2}\dr z_1=c_\alpha (b-a)$ for all $0\leq a \leq b$: use
the change of variables $u=\frac{z_1-a}{b-a}$. Consequently, see Lemma~\ref{paradur},
\begin{align*}
I_{31}(x,x') \leq & C \int_{B_{d-1}(0,\e_1)} |\psi_{x',A'}(v)-\psi_{x,A}(v)|\frac{\dr v}{|v|^d}
\leq C \rho_{x,x',A,A'} \int_{B_{d-1}(0,\e_1)} \frac{|v|^2\dr v}{|v|^d}\leq C\rho_{x,x',A,A'}.
\end{align*}

Next,
\begin{align*}
I_{32}(x,x') \leq & C\int_{\R^d} \frac{[d(z,\pDd_{x',A'})]^{\alpha/2}}{[d(z,\pDd_{x,A})]^{\alpha/2}}
\indiq_{\{z \in \Dd_{x',A'}\setminus \Dd_{x,A}\}} \dr z..
\end{align*}
For $z \in \Dd_{x',A'}\setminus \Dd_{x,A}$, we have
$d(z,\Dd_{x,A})\lor d(z,\pDd_{x',A'}) \leq C \rho_{x,x',A,A'}$, see Lemma~\ref{paradurf}. Hence
\begin{align*}
I_{32}(x,x') \leq& C\rho_{x,x',A,A'}^{\alpha/2} 
\int_{\R^d} \frac{\indiq_{\{d(z,\Dd_{x,A})\in (0,C\rho_{x,x',A,A'}]\}}}{[d(z,\pDd_{x,A})]^{\alpha/2}}\dr z\\
=&C\rho_{x,x',A,A'}^{\alpha/2} 
\int_{\R^d} \frac{\indiq_{\{d(h_x(A,z),\Dd)\in (0,C\rho_{x,x',A,A'}]\}}}{[d(h_x(A,z),\pDd)]^{\alpha/2}}\dr z.
\end{align*}
We substitute $u=h_x(A,z)$ and then $a=\Phi(u)$. The Jacobian of $\Phi$ being bounded, see Lemma~\ref{parafac}, 
\begin{align*}
I_{32}(x,x') \leq& C\rho_{x,x',A,A'}^{\alpha/2} 
\int_{\R^d} \frac{\indiq_{\{d(u,\Dd)\in (0,C\rho_{x,x',A,A'}]\}}}{[d(u,\pDd)]^{\alpha/2}}\dr z\\
\leq& C\rho_{x,x',A,A'}^{\alpha/2} \int_{\R^d} \frac{\indiq_{\{d(\Phi^{-1}(a),\Dd)\in (0,C\rho_{x,x',A,A'}]\}}}
{[d(\Phi^{-1}(a),\pDd)]^{\alpha/2}}\dr a.
\end{align*}
But, see Lemma~\ref{parafac}, $\kappa^{-1}d(a,{B}_d(0,1))\leq 
d(\Phi^{-1}(a),\Dd)\leq \kappa d(a,{B}_d(0,1))$, whence
\begin{align*}
I_{32}(x,x') \leq& C\rho_{x,x',A,A'}^{\alpha/2} 
\int_{\R^d} \frac{\indiq_{\{d(a,{B}_d(0,1))\in (0,C\rho_{x,x',A,A'}]\}}}{[d(a,{B}_d(0,1))]^{\alpha/2}}\dr a\\
=&C\rho_{x,x',A,A'}^{\alpha/2} \int_1^{1+C\rho_{x,x',A,A'}} \frac{r^{d-1}\dr r}{(r-1)^{\alpha/2}}  \int_{\Sp_{d-1}} \dr \sigma\\
\leq& C\rho_{x,x',A,A'}^{\alpha/2} \int_1^{1+C\rho_{x,x',A,A'}} \frac{\dr r}{(r-1)^{\alpha/2}},
\end{align*}
because $\rho_{x,x',A,A'}$ is uniformly bounded. Finally, we conclude that
$$
I_{32}(x,x') \leq C\rho_{x,x',A,A'}^{\alpha/2} \rho_{x,x',A,A'}^{1-\alpha/2}=C\rho_{x,x',A,A'}.
$$
We have proved that $I_3(x,x')\leq C \rho_{x,x',A,A'}$, {\it i.e.}~\eqref{jjab3}.
\vip

{\it Step 6.} Let $I_4(x)$ be the left hand side of~\eqref{jjab5} and recall that 
$\Dd_{x,A}=\{y \in \R^d : h_x(A,y)\in \Dd\}$. 
Using the function $\psi_{x,A}:B_{d-1}(0,\e_1)\to \R_+$ introduced in Lemma~\ref{paradur},
$$
I_4(x)=\int_\HH \indiq_{\{z \notin \Dd_{x,A}\}} \frac{\dr z}{|z|^{d+\beta}} \leq 
\int_{\HH \cap B_d(0,\e_1)^c} \frac{\dr z}{|z|^{d+\beta}} + \int_{\HH \cap B_d(0,\e_1)} 
\indiq_{\{z_1 < \psi_{x,A}(z_2,\dots,z_d)\}}\frac{\dr z}{|z|^{d+\beta}}.
$$
Hence, writing $z=(z_1,v)$ and using that $|z|\geq|v|$,
$$
I_4(x) \leq C + \int_{B_{d-1}(0,\e_1)} \frac{\dr v}{|v|^{d+\beta}} \int_0^{\psi_{x,A}(v)}\dr z_1
\leq C +  \int_{B_{d-1}(0,\e_1)} \frac{\psi_{x,A}(v)\dr v}{|v|^{d+\beta}}.
$$
By Lemma~\ref{paradur}, there is $C>0$, which does not depend on $x\in \pDd$, 
such that $|\psi_{x,A}(v)|\leq C |v|^2$:
$$
I_4(x) \leq C + C \int_{B_{d-1}(0,\e_1)} \frac{\dr v}{|v|^{d+\beta-2}}.
$$
The conclusion follows, since $d+\beta-2<d-1$.
\vip

{\it Step 7.} We call $I_5(x,x')$ the left hand side of~\eqref{jjab2}. Using the functions
$\psi_{x,A},\psi_{x',A'}:B_{d-1}(0,\e_1)\to \R_+$ introduced in Lemma~\ref{paradur},
$$
I_5(x,x')= \int_\HH \indiq_{\{z \in \Dd_{x',A'}\setminus \Dd_{x,A}\}}  \frac{\dr z}{|z|^{d+\beta}} \leq I_{51}(x,x')+
I_{52}(x,x'),
$$
where
\begin{align*}
I_{51}(x,x')=& \int_{\HH \cap B_d(0,\e_1)^c} \indiq_{\{z \in \Dd_{x',A'}\setminus \Dd_{x,A}\}} \frac{\dr z}{|z|^{d+\beta}} \leq
C {\rm Vol}_d(\Dd_{x',A'}\setminus \Dd_{x,A}),\\
I_{52}(x,x')=& \int_{\HH \cap B_d(0,\e_1)} \!\!\frac{\indiq_{\{ \psi_{x',A'}(z_2,\dots,z_d)\leq z_1 < \psi_{x,A}(z_2,\dots,z_d)\}}
\dr z}{|z|^{d+\beta}}
\leq C \int_{B_{d-1}(0,\e_1)}\!\! \frac{|\psi_{x,A}(v)-\psi_{x',A'}(v)|\dr v}{|v|^{d+\beta}}.
\end{align*}
We have written $z=(z_1,v)$, used that $|z|\geq |v|$ and integrated in $z_1\in [\psi_{x',A'}(v),\psi_{x,A}(v))$.
Using now Lemmas~\ref{paradurf} and~\ref{paradur}, we find
$$
I_5(x,x')\leq C \rho_{x,x',A,A'} + C\rho_{x,x',A,A'}\int_{B_{d-1}(0,\e_1)} 
\frac{\dr v}{|v|^{d+\beta-2}} \leq  C\rho_{x,x',A,A'}.
$$
\vskip-1.1cm
\end{proof}

\subsection{Conclusion}

We can now give the

\begin{proof}[Proof of Proposition~\ref{ttoopp}.]
We fix
$\alpha \in (0,2)$, $\beta \in \{*\}\cup(0,\alpha/2)$, as well as 
$x,x' \in \pDd$ and $A \in \cI_x$, $A' \in \cI_{x'}$.
We use the shortened notation
$h_x(y)=h_x(A,y)$, $g_x(e)=g_x(A,e)$, 
$\cu(x,e)=\cu(x,A, e)$, $\co(x,e)=\co(x,A,e)$, $\bg_x(y,z)=\bg_x(A,y,z)$ and $\rho_{x,x',A,A'}=|x-x'|+||A-A'||$. 

\vip

{\it Step 1.} By Remark~\ref{ggbar}, it holds that
\begin{align*}
\Delta(x,x',A,A')
=&\int_\cE \Big||\bg_x(\cu(x,e),\co(x,e))-\bg_{x'}(\cu(x',e),\co(x',e))| -|x-x'|\Big| \nn_\beta(\dr e).
\end{align*}
We recall that $\cl_x(e)=\inf\{t>0 : h_x(e(t))\notin \Dd\}$,
$\cu(x,e)=e(\cl_x(e)-)$ and $\co(x,e)=e(\cl_x(e))$. We distinguish 
four possibilities:

\vip

\noindent $\bullet$ $\cl_x(e)=\cl_{x'}(e)$ and $\delta(h_x(\cu(x,e))\leq \delta(h_{x'}(\cu(x',e))$, 
in which case $\cu(x,e)=\cu(x',e)$, $\co(x,e)=\co(x',e)$,
and $h_x(\cu(x,e)) \in \Dd$, $h_{x'}(\cu(x,e)) \in \Dd$, $h_x(\co(x,e))\notin \Dd$ and 
$h_{x'}(\co(x,e)) \notin \Dd$; 

\vip

\noindent $\bullet$  $\cl_x(e)=\cl_{x'}(e)$ and $\delta(h_{x'}(\cu(x',e))\leq \delta(h_x(\cu(x,e))$,
in which case $\cu(x,e)=\cu(x',e)$, $\co(x,e)=\co(x',e)$,
and $h_x(\cu(x',e)) \in \Dd$, $h_{x'}(\cu(x',e)) \in \Dd$, $h_x(\co(x',e))\notin \Dd$ and 
$h_{x'}(\co(x',e)) \notin \Dd$;

\vip

\noindent $\bullet$ $\cl_x(e)<\cl_{x'}(e)$, in which case
$h_x(\cu(x,e))\!\in \!\Dd$, $h_{x'}(\cu(x,e)) \!\in\! \Dd$, $h_x(\co(x,e))\!\notin \!\Dd$,
$h_{x'}(\co(x,e)) \! \in \! \Dd$;

\vip

\noindent $\bullet$ $\cl_{x'}(e)<\cl_x(e)$, in which case
$h_{x'}(\cu(x',e))\!\in\! \Dd$, $h_{x}(\cu(x',e)) \!\in \!\Dd$, $h_{x'}(\co(x',e))\!\notin \!\Dd$,
$h_{x}(\co(x',e)) \!\in \!\Dd$.

\vip

Thus $\Delta(x,x')\leq \Delta_1(x,x',A,A')+\Delta_1(x',x,A',A)+\Delta_2(x,x',A,A')+\Delta_2(x',x,A',A)$, where
\begin{align*}
\Delta_1(x,x'\!,A,A')\!=\!&\int_\cE\! \Big||\bg_x(\cu(x,e),\co(x,e))\!-\!\bg_{x'}(\cu(x,e),\co(x,e))| 
\!-\!|x-x'|\Big| \indiq_{\{\delta(h_x(\cu(x,e))\leq \delta( h_{x'}(\cu(x,e))\}} 
\\
&\hskip3cm
\indiq_{\{h_x(\cu(x,e)) \in \Dd, h_{x'}(\cu(x,e)) \in \Dd, h_x(\co(x,e))\notin \Dd,h_{x'}(\co(x,e)) \notin \Dd\}}
\nn_\beta(\dr e)\\
=&\int_{(\R^d)^2} \Big||\bg_x(y,z)\!-\!\bg_{x'}(y,z)| \!-\!|x-x'|\Big| \indiq_{\{\delta(h_x(y))\leq \delta(h_{x'}(y))\}}\\
&\hskip3cm 
\indiq_{\{h_x(y) \in \Dd, h_{x'}(y) \in \Dd, h_x(z)\notin \Dd,h_{x'}(z) \notin \Dd\}}
q_x(\dr y, \dr z),
\end{align*}
with $q_x$ defined in Proposition~\ref{underover}, and
\begin{align*}
\Delta_2(x,x',A,A')=&\int_\cE \Big||\bg_x(\cu(x,e),\co(x,e))-\bg_{x'}(\cu(x',e),\co(x',e))| -|x-x'|\Big| \\
&\hskip1.5cm \indiq_{\{\cl_{x'}(e)>\cl_x(e),h_x(\cu(x,e)) \in \Dd, h_{x'}(\cu(x,e)) \in \Dd, 
h_x(\co(x,e))\notin \Dd,h_{x'}(\co(x,e)) \in \Dd\}} \nn_\beta(\dr e).
\end{align*}
We write $\Delta_{2}(x,x',A,A')\leq \Delta_{21}(x,x',A,A')+\Delta_{22}(x,x',A,A')$, where
\begin{align*}
\Delta_{21}(x,x',A,A')=&\int_\cE \Big||\bg_x(\cu(x,e),\co(x,e))-h_{x'}(\co(x,e))| -|x-x'|\Big| \\
&\hskip3cm \indiq_{\{h_x(\cu(x,e)) \in \Dd, h_{x'}(\cu(x,e)) \in \Dd, h_x(\co(x,e))\notin \Dd,h_{x'}(\co(x,e)) \in \Dd\}}
\nn_\beta(\dr e)\\
=& \int_{(\R^d)^2} \Big||\bg_x(y,z)-h_{x'}(z)| -|x-x'|\Big| \\
&\hskip3cm \indiq_{\{h_x(y) \in \Dd, h_{x'}(y) \in \Dd, h_x(z)\notin \Dd,h_{x'}(z) \in \Dd\}}
q_x(\dr y, \dr z),
\end{align*}
and
\begin{align*}
\Delta_{22}(x,x',A,A')=&\int_\cE \Big|h_{x'}(\co(x,e))-\bg_{x'}(\cu(x',e),\co(x',e))\Big| \\
&\hskip1.5cm \indiq_{\{\cl_{x'}(e)>\cl_x(e),h_x(\cu(x,e)) \in \Dd, h_{x'}(\cu(x,e)) \in \Dd, h_x(\co(x,e))\notin \Dd,h_{x'}(\co(x,e)) \in \Dd\}}
\nn_\beta(\dr e).
\end{align*}
By the Markov property (see Lemma~\ref{mark}) 
with the stopping time $\cl_x(e)$ (which is indeed positive
when $\beta=*$, see Remark~\ref{imp}) and since
$\cl_x(e)<\cl_{x'}(e)$ implies $\cl_x(e)<\ell (e)$,
\begin{align*}
\Delta_{22}(x,x',A,A')=&\int_\cE \E_{\co(x,e)}\Big[\Big|h_{x'}(Z_0)-\bg_{x'}(Z_{\cl_{x'}(Z)-},Z_{\cl_{x'}(Z)})\Big|\Big] \\
&\hskip1.5cm \indiq_{\{\cl_{x'}(e)>\cl_x(e),h_x(\cu(x,e)) \in \Dd, h_{x'}(\cu(x,e)) \in \Dd, h_x(\co(x,e))\notin \Dd,h_{x'}(\co(x,e)) \in \Dd\}}
\nn_\beta(\dr e)\\
\leq &\int_{(\R^d)^2} \E_{z}\Big[\Big|h_{x'}(Z_0)-\bg_{x'}(Z_{\cl_{x'}(Z)-},Z_{\cl_{x'}(Z)})\Big|\Big] \\
&\hskip1.5cm \indiq_{\{h_x(y) \in \Dd, h_{x'}(y) \in \Dd, h_x(z)\notin \Dd,h_{x'}(z) \in \Dd\}} q_x(\dr y, \dr z).
\end{align*}

\vip

{\it Step 2.} By Proposition~\ref{underover}, we see that 
\begin{align*}
\Delta_1(x,x',A,A')\leq & C \int_{(\R^d)^2} \Big||\bg_x(y,z)-\bg_{x'}(y,z)| -|x-x'|\Big|
\indiq_{\{\delta(h_x(y))\leq \delta(h_{x'}(y))\}}\\
&\hskip3cm \indiq_{\{h_x(y) \in \Dd, h_{x'}(y) \in \Dd, h_x(z)\notin \Dd,h_{x'}(z) \notin \Dd\}}
\frac{[\delta(h_x(y))]^{\alpha/2}}{|z-y|^{d+\alpha}|y|^{d}}\dr y \dr z\\
&+ C\indiq_{\{\beta\neq *\}} \int_{\HH} \Big||\bg_x(0,z)-\bg_{x'}(0,z)| -|x-x'|\Big| 
\indiq_{\{h_x(z)\notin \Dd,h_{x'}(z) \notin \Dd\}} \frac1{|z|^{d+\beta}}\dr z.
\end{align*}
Using next Proposition~\ref{tyvmmm}-(i)-(iii), we find
\begin{align*}
\Delta_1(x,x',A,A')\leq & C\rho_{x,x',A,A'} \int_{(\R^d)^2} \Big(|z|\land 1 + \frac{|y|(|y-z|\land 1)}{\delta(h_x(y))} \Big)\\
&\hskip3cm \indiq_{\{h_x(y) \in \Dd, h_{x'}(y) \in \Dd, h_x(z)\notin \Dd,h_{x'}(z) \notin \Dd\}}
\frac{[\delta(h_x(y))]^{\alpha/2}}{|z-y|^{d+\alpha}|y|^{d}}\dr y \dr z\\
&+ C\indiq_{\{\beta\neq *\}} \rho_{x,x',A,A'} \int_{\HH}
\indiq_{\{h_x(z)\notin \Dd,h_{x'}(z) \notin \Dd\}} \frac1{|z|^{d+\beta}}\dr z.
\end{align*}
Thus $\Delta_1(x,x',A,A')\leq  C\rho_{x,x',A,A'}$
by~\eqref{jjab1} and~\eqref{jjab5}.

\vip

{\it Step 3.}  By Proposition~\ref{underover}, we get
\begin{align*}
\Delta_{21}(x,x',A,A')\leq & C \int_{(\R^d)^2}\! \Big||\bg_x(y,z)-h_{x'}(z)| -|x-x'|\Big| \\
&\hskip3cm \indiq_{\{h_x(y) \in \Dd, h_{x'}(y) \in \Dd, h_x(z)\notin \Dd,h_{x'}(z) \in \Dd\}}
\frac{[\delta(h_x(y))]^{\alpha/2}}{|z-y|^{d+\alpha}|y|^{d}}\dr y \dr z\\
&+C\indiq_{\{\beta\neq *\}} \int_\HH\Big||\bg_x(0,z)-h_{x'}(z)| -|x-x'|\Big|
\indiq_{\{h_x(z)\notin \Dd,h_{x'}(z) \in \Dd\}} \frac1{|z|^{d+\beta}}\dr z.
\end{align*}
We next use Proposition~\ref{tyvmmm}-(ii)-(iv) to obtain
\begin{align*}
\Delta_{21}(x,x',A,A')\leq & C\rho_{x,x',A,A'} \int_{(\R^d)^2}\! 
\Big(|z|\land 1 + \frac{|y|(|y-z|\land 1)}{\delta(h_x(y))} \Big) \\
&\hskip3cm \indiq_{\{h_x(y) \in \Dd, h_{x'}(y) \in \Dd, h_x(z)\notin \Dd,h_{x'}(z) \in \Dd\}}
\frac{[\delta(h_x(y))]^{\alpha/2}}{|z-y|^{d+\alpha}|y|^{d}}\dr y \dr z\\
&+C\indiq_{\{\beta\neq *\}} \int_\HH \indiq_{\{h_x(z)\notin \Dd,h_{x'}(z) \in \Dd\}} \frac1{|z|^{d+\beta}}\dr z.
\end{align*}
We conclude from~\eqref{jjab1} and~\eqref{jjab2} that
$\Delta_{21}(x,x',A,A')\leq  C\rho_{x,x',A,A'}$.

\vip

{\it Step 4.} To treat $\Delta_{22}$, let us first prove that there is a constant $C$ 
(not depending on $x' \in \pDd$) such that
for any $z \in \R^d$ such that $h_{x'}(z)\in \Dd$,
\begin{equation}\label{ttc}
f(z):=\E_{z}\Big[\Big|h_{x'}(Z_0)-\bg_{x'}(Z_{\cl_{x'}(Z)-},Z_{\cl_{x'}(Z)})\Big|\Big]
\leq C [\delta(h_{x'}(z))]^{\alpha/2}.
\end{equation}
Since $h_{x'}$ is an affine isometry, $h_{x'}(Z)$ has the same law as $Z$ (with the suitable
initial condition), so that we can write, recalling~\eqref{bg},
$$
f(z)=\E_{h_{x'}(z)}\Big[\Big|Z_0-\Lambda(Z_{\tell(Z)-},Z_{\tell(Z)})\Big|\Big]
\leq \E_{h_{x'}(z)}\Big[\sup_{t \in [0,\tell(Z)]}|Z_t-Z_0| \land \mathrm{diam}(\Dd)\Big],
$$
where $\tell (Z) = \inf\{t>0 : Z_t \notin \Dd\}$. Consider now $u\in \pDd$ such that 
$\delta(h_{x'}(z))=|h_{x'}(z)-u|$, as well as the half-space $\HH_u$ tangent to $\pDd$ at $u$
and containing $\Dd$ (recall that $\Dd$ is convex).
It holds that $\hat \ell_u(Z)=\inf\{t>0 : Z_t \notin {\HH}_u\} \geq \tell(Z)$, whence
$$
f(z)\leq \E_{h_{x'}(z)}\Big[\sup_{t \in [0,\hat\ell_u(Z)]}|Z_t-Z_0|\land \mathrm{diam}(\Dd)\Big]
= \E_{\delta(h_{x'}(z))\be_1} \Big[\sup_{t \in [0,\ell(Z)]}|Z_t-Z_0|\land \mathrm{diam}(\Dd)\Big]
$$
by isometry-invariance of the stable process (recall that $\ell(Z)=\inf\{t>0 : Z_t \notin {\HH}\}$)
and since $h_{x'}(z)-u$ is orthogonal to $\partial \HH_u$.
We then deduce~\eqref{ttc} from Lemma~\ref{pr5}. With the very same arguments,
we have the following (to be used in other proofs): for any $y \in \Dd$,
with $u\in \pDd$ such that $\delta(y)=|y-u|$ and $\HH_u$ as above, 
\begin{equation}\label{agarder}
\PP_{y}(\tell(Z)> 1)\leq \PP_y(\hat \ell_u(Z)>1)= \PP_{\delta(y)\be_1}(\ell(Z)>1) \sim c (\delta(y))^{\alpha/2}
\text{ as $\delta(y)\to 0$.}
\end{equation}
We thus have
\begin{align*}
\Delta_{22}(x,x',A,A')\leq & C\int_{(\R^d)^2}  [\delta(h_{x'}(z))]^{\alpha/2}
\indiq_{\{h_x(y) \in \Dd, h_{x'}(y) \in \Dd, h_x(z)\notin \Dd,h_{x'}(z) \in \Dd\}} q_x(\dr y, \dr z)\\
\leq & C\int_{(\R^d)^2} [\delta(h_{x'}(z))]^{\alpha/2} \frac{[\delta(h_x(y))]^{\alpha/2}}{|z-y|^{d+\alpha}|y|^d}
\indiq_{\{h_x(y) \in \Dd, h_{x'}(y) \in \Dd, h_x(z)\notin \Dd,h_{x'}(z) \in \Dd\}}\dr y \dr z\\
&+C \indiq_{\{\beta \neq *\}}\int_{\HH}  [\delta(h_{x'}(z))]^{\alpha/2} 
\indiq_{\{h_x(z)\notin \Dd,h_{x'}(z) \in \Dd\}} \frac{\dr z}{|z|^{d+\beta}}
\end{align*}
by Proposition~\ref{underover}. By~\eqref{jjab3} and~\eqref{jjab2}, we end with $\Delta_{22}(x,x',A,A')
\leq C\rho_{x,x',A,A'}$.

\vip

{\it Conclusion.} Gathering Steps 1 to 4, we find $\Delta(x,x',A,A')\leq C\rho_{x,x',A,A'}$ as desired.
\end{proof}

\section{The reflected process starting from the boundary}\label{sec:start_bounda}

The goal of this section is to prove Theorem~\ref{mr1}. We start with three lemmas.

\begin{lemma}\label{tti}
Fix $\beta \in \{*\}\cup(0,\alpha/2)$, $x\in \pDd$ and suppose Assumption~\ref{as}.
Consider a filtration $(\cG_u)_{u\geq 0}$, a $(\cG_u)_{u\geq 0}$-Poisson measure 
$\Pi_\beta=\sum_{u\in \JJ}\delta_{(u,e_u)}$ on $\R_+\times\cE$ with intensity measure $\dr u \nn_{\beta}(\dr e)$,
a càdlàg $(\cG_u)_{u\geq 0}$-adapted $\pDd$-valued process $(b_u)_{u\geq 0}$, a $(\cG_u)_{u\geq 0}$-predictable 
process $(a_u)_{u \geq 0}$ such that a.s., for all $u \geq 0$, $a_u \in \cI_{b_{u-}}$ and define $(\tau_u)_{u\geq 0}$
by~\eqref{sdet}. Then a.s., $u \mapsto \tau_u$ is strictly increasing on $\R_+$ and $\lim_{u\to\infty} \tau_u=\infty$.
\end{lemma}

\begin{proof}
By Remark~\ref{imp4}, we have
$\cl_x(A,e) \geq \ell_r(e)$ for all $e \in \cE$, all $x \in \pDd$, all $A \in \cI_x$.
If $\beta=*$, we have $\ell_r(e)>0$ for $\nn_*$-a.e. $e \in \cE$ by Lemma~\ref{imp}, so that
$\nn_*(\ell_r>0)=\infty$ and the result follows.
If next $\beta \in (0,\alpha/2)$, we  have $\ell_r(e)>0$ for all $e \in \cE_r:=\{e \in \cE :
e(0) \in  B_d(r\be_1,e)\}$. Since $\nn_\beta (\cE_r)=\int_{\HH\cap B_d(r\be_1,e)} |x|^{-d-\beta} \dr x= \infty$,
recall~\eqref{nnb},
the result follows.
\end{proof}

The second one deals with Poisson measures.

\begin{lemma}\label{ttaann}
Fix $\beta \in \{*\}\cup(0,\alpha/2)$. Consider a filtration $(\cG_u)_{u\geq 0}$, 
a $(\cG_u)_{u\geq 0}$-Poisson measure 
$\Pi_\beta=\sum_{u\in \JJ}\delta_{(u,e_u)}$ on $\R_+\times\cE$ with intensity measure $\dr u \nn_{\beta}(\dr e)$.
For any $(\cG_u)_{u\geq 0}$-predictable process $\Theta_u$ such that a.s., for all $u\geq 0$,
$\Theta_u$ is a linear isometry of $\R^d$ sending $\be_1$ to $\be_1$, the random measure
$\Pi'_\beta=\sum_{u\in \JJ}\delta_{(u,\Theta_u e_u)}$ is again a $(\cG_u)_{u\geq 0}$-Poisson measure 
on $\R_+\times\cE$ with intensity measure $\dr u \nn_{\beta}(\dr e)$.
\end{lemma}

\begin{proof}
By rotational invariance of the stable process, the excursion measure $\nn_\beta$ is invariant by
any linear isometry of $\R^d$ sending $\be_1$ to $\be_1$. Consequently, 
the compensator (see Jacod and Shiryaev~\cite[Theorem 1.8 page 66]{jacod2013limit})
of the integer-valued random
measure (see~\cite[Definition 1.13 page 68]{jacod2013limit}) $\Pi_\beta'$ is again  $\dr u \nn_{\beta}(\dr e)$. 
The conclusion follows from~\cite[Theorem 4.8 page 104]{jacod2013limit}.
\end{proof}

The third one states some continuity results.

\begin{lemma}\label{relou}
Fix $\beta \in \{*\}\cup(0,\alpha/2)$ and suppose Assumption~\ref{as}.

\vip

(i) For any $x\in \pDd$, any $A \in \cI_x$, any $t>0$, we have  $\nn_\beta(\{e \in \cE : \cl_x(A,e)=t\})=0$.

\vip

(ii) Let $x_n\in \pDd$, $A_n \in \cI_{x_n}$ such that $\lim_n x_n= x \in \pDd$ and 
$\lim_n A_n = A \in \cI_x$. For $\nn_\beta$-a.e. $e \in \cE$, 
there is $n_e \geq 1$ such that for all $n\geq n_e$, $\cl_{x_n}(A_n,e)=\cl_x(A,e)$.
\end{lemma}

\begin{proof}
For (i), we fix $t>0$ and observe that for all $x\in \pDd$, all 
$z\in \Dd_{x,A}=\{y \in \R^d : h_x(y)\in \Dd \}$,
$\PP_z(\ell_x(A,Z)=t)=0$, where $Z$ is an ISP$_{\alpha,z}$ under $\PP_z$:
$\cl_x(A,Z)$ has a density under $\PP_z$, as shown (with more generality and with some informative 
formula) by Bogdan, Jastrz\c{e}bski, Kassmann, Kijaczko and Pop\l{}awski~\cite[Theorem 1.3]{bjkkp}.
\vip

Since $t>0$, $\cl_x(A,e)=t$ implies
that $e(0)\in \Dd_{x,A}$ (because $e(0)\in \Dd_{x,A}^c$ implies that $\cl_x(A,e)=0$).
Recalling~\eqref{nnb}, we conclude that if $\beta \in (0,\alpha/2)$,
$$
\nn_\beta(\cl_x(A,\cdot)=t)=\int_{\Dd_{x,A}} |z|^{-d-\beta} \PP_z(\cl_x(A,Z)=t) \dr z =0.
$$

If now $\beta=*$, we use the Markov property of $\nn_*$ at time $t/2$,
see Lemma~\ref{mark}, to write
$$
\nn_*(\cl_x(A,\cdot)=t)=\int_{\cE} \indiq_{\{\ell(e)>t/2\}} \PP_{e(t/2)}(\cl_x(A,Z)=t/2) \nn_*(\dr e)=0.
$$

We next check (ii).
We fix $x_n \in \pDd$, $A_n\in \cI_{x_n}$ such that $x_n\to x \in \pDd$ and $A_n \to A \in \cI_x$.
We set $\Dd_{x_n,A_n}=\{y \in \HH : h_{x_n}(A_n,y)\in\Dd\}$ and $\Dd_{x,A}=\{y \in \HH : h_{x}(A,y)\in\Dd\}$.
We recall that $\cl_{x_n}(A_n,e)=\inf\{t>0 : e(t)\notin \Dd_{x_n,A_n}\}$ and 
$\cl_{x}(A,e)=\inf\{t>0 : e(t)\notin \Dd_{x,A}\}$.
\vip

(a) If $z \in \Dd_{x,A}$, then $\PP_z(\inf_{t\in[0,\cl_x(A,Z))}d(Z_t,\Dd_{x,A}^c)>0, 
Z_{\cl_x(A,Z)} \in \cDd_{x,A}^c)=1$. 

\vip

By invariance of the stable process, it suffices that, setting  
$\tell(Z)=\inf\{t> 0 : Z_t \notin \Dd\}$,
\begin{equation}\label{eer}
\text{for all $z\in \Dd$,}\qquad p(z):=\PP_z\Big(\inf_{t\in[0,\tell(Z))}\delta(Z_t)>0, 
Z_{\tell(Z)} \in \cDd^c \Big)=1,
\end{equation}
where $\delta(y)=d(y,\pDd)$.
Fix $z \in \Dd$. By Lemma~\ref{uos}, the laws of $Z_{\tell(Z)-}$ and $Z_{\tell(Z)}$ have densities under 
$\PP_z$, so that $\PP_z$-a.s., $Z_{\tell(Z)} \in \cDd^c$ and 
$Z_{\tell(Z)-} \in \Dd$, whence there is $\e>0$ such that
$\inf_{t\in[\tell(Z)-\e,\tell(Z))}\delta(Z_t)>0$.
Hence we only have to check that for all $\e>0$, 
$$
q_\e(z):=\PP_z\Big(\tell(Z)>\e,\inf_{t\in[0,\tell(Z)-\e]}\delta(Z_t)=0\Big)=0.
$$
For $k\geq 1$, set $\rho_k(Z)=\inf\{t\geq 0 :\delta(Z_t)\leq 1/k \}$: we need 
that $\lim_k q_{k,\e}(z)=0$, where
$$
q_{k,\e}(z)=\PP_z\Big(\tell(Z)>\e,\inf_{t\in[0,\tell(Z)-\e]}\delta(Z_t)\leq 1/k\Big)
\leq \PP_z(\rho_k(Z)\leq \tell(Z)-\e). 
$$
On the event $\{\rho_k(Z)\leq \tell(Z)-\e\}$, we have $Z_{\rho_k(Z)} \in \Dd$, so that, applying the 
strong Markov property at time $\rho_k(Z)$, we get
$$
q_{k,\e}(z)\leq \E_z\Big[\indiq_{\{Z_{\rho_k(Z)} \in \Dd\}}\PP_{Z_{\rho_k(Z)}}(\tell(Z)\geq \e) \Big].
$$
By~\eqref{agarder} and a scaling argument, there is $c_\e>0$ such that
$\PP_{y}(\tell(Z)\geq \e)\leq c_\e [\delta(y)]^{\alpha/2}$ for $y \in \Dd$.  Since
$\delta(Z_{\rho_k(Z)})\leq 1/k$ on  $\{Z_{\rho_k(Z)} \in \Dd\}$, we conclude that
$q_{k,\e}(z)\leq c_\e k^{-\alpha/2} \to 0$ as desired.

\vip

(b) We conclude when $\beta=*$.  Let us first mention that as $n\to \infty$,
\begin{align}\label{ssz}
\sup_{y\in \R^d}|d(y,\Dd_{x_n,A_n}^c)-d(y,\Dd_{x,A}^c)|
=\sup_{y\in \R^d}|d(h_x(A,y),\Dd^c)-d(h_{x_n}(A_n,y),\Dd^c)|
\to 0,
\end{align}
since $|d(h_x(A,y),\Dd^c)-d(h_{x_n}(A_n,y),\Dd^c)|\leq 
|h_x(A,y)-h_{x_n}(A_n,y)|\leq |x-x_n|+||A-A_n|||y|$
and since $|d(h_x(A,y),\Dd^c)-d(h_{x_n}(A_n,y),\Dd^c)|=0$ if $|y|\geq \text{diam}(\Dd)$.

\vip
By Remark~\ref{imp4},
$\cl_x(A,e)\geq \ell_r(e)$, $\cl_{x_n}(A_n,e)\geq \ell_r(e)$ and by Lemma~\ref{imp},  
$\ell_r>0$ $\nn_*$-a.e.
By (a) and the Markov property applied at time $\ell_r(e) >0$, see Lemma~\ref{mark},
we see that for $\nn_*$-a.e. $e \in \cE$, 
\vip
\noindent$\bullet$ either $e(\ell_r(e)) \notin \cDd_{x,A}$,
\vip
\noindent $\bullet$ or 
$e(\ell_r(e))\in \Dd_{x,A}$ and $\inf_{t \in [\ell_r(e),\cl_x(A,e))}d(e(t),\Dd_{x,A}^c)>0$ 
and $e(\cl_x(A,e))\in \cDd_{x,A}^c$. 
\vip
We used that for $\nn_*$-a.e. $e \in \cE$, $e(\ell_r(e))\notin \pDd_{x,A}$,
because the law of $e(\ell_r(e))$ under $\nn_*$ has a density (as in 
Proposition~\ref{underover}). 

\vip

If first $e(\ell_r(e)) \notin \cDd_{x,A}$, then $\cl_x(A,e)=\ell_r(e)$ 
and $\cl_{x_n}(A_n,e)=\ell_r(e)$ for all $n$ large enough, since 
$h_{x_n}(A_n,e(\ell_r(e)))\to h_{x}(A,e(\ell_r(e))) \in\R^d\setminus \cDd$, which is open, 
so that $e(\ell_r(e)) \notin \cDd_{x_n,A_n}$ for $n$ large enough.

\vip

If next $e(\ell_r(e))\in \Dd_{x,A}$ and $\inf_{t \in [\ell_r(e),\cl_x(A,e))}d(e(t),\Dd_{x,A}^c)>0$ 
and 
$e(\cl_x(A,e))\in \cDd_{x,A}^c$, then for all $n$ large enough, 
$e(\ell_r(e))\in \Dd_{x_n,A_n}$ (because $h_{x_n}(A_n,e(\ell_r(e)))\to h_{x}(A,e(\ell_r(e))) \in \Dd$ which
is open),
$\inf_{t \in [\ell_r(e),\cl_{x}(A,e))}d(e(t),\Dd_{x_n,A_n}^c)>0$
by~\eqref{ssz}
and $e(\cl_x(A,e))\in \cDd_{x_n,A_n}^c$ (because $h_{x_n}(A_n,e(\cl_x(A,e)))
\to h_{x}(A,e(\cl_x(A,e))) \in \cDd^c$ which is open).
Thus $\cl_{x_n}(A_n,e)=\cl_x(A,e)$ for all $n$ large enough.

\vip

(c) The case where $\beta \in (0,\alpha/2)$ is easier. By (a) and~\eqref{nnb}, 
we see that for $\nn_\beta$-a.e. $e \in \cE$,
\vip
\noindent $\bullet$ 
either $e(0) \in \R^d\setminus \cDd_{x,A}$, 
\vip
\noindent $\bullet$ or $e(0) \in \Dd_{x,A}$ and 
$\inf_{t \in [0,\cl_x(A,e))}d(e(t), \Dd_{x,A}^c)>0$ and 
$e(\cl_x(A,e))\in \cDd_{x,A}^c$. 
\vip
We conclude as in (b).
\end{proof}

We now prove some well-posedness result for the boundary process.

\begin{proposition}\label{newb} Fix 
$\beta \in \{*\}\cup(0,\alpha/2)$ and suppose Assumption~\ref{as}. We consider a filtration
$(\cG_u)_{u\geq 0}$ and a $(\cG_u)_{u\geq 0}$-Poisson measure $\Pi_\beta=\sum_{u \in \JJ}\delta_{(u,e_u)}$
on $\R_+\times\cE$ with intensity $\dr u \nn_\beta(\dr e)$.

\vip

(a) If $d = 2$, consider the Lipschitz family $(A_y)_{y\in \pDd}$ introduced in Lemma~\ref{locglob}-(i).
For any $x \in \pDd$,
there is strong existence and uniqueness
of a càdlàg $(\cG_u)_{u\geq 0}$-adapted $\pDd$-valued process $(b_u^{x})_{u\geq 0}$ solving
\begin{equation}\label{sdeb21}
b_u^{x}=x+\int_0^u \int_\cE [g_{b^{x}_\vm}(A_{b^{x}_\vm},e)-b^{x}_\vm] \Pi_\beta(\dr v,\dr e).
\end{equation}

(b) If $d\geq 3$, we denote by 
$\varsigma$ the uniform measure on $\pDd$, that is, the (normalized) Hausdorff measure
of dimension $d-1$ (in $\R^d$) restricted to $\pDd$. We
consider for each $z \in \pDd$ the family $(A^z_y)_{y\in \pDd}$ introduced in 
Lemma~\ref{locglob}-(ii), such that $y\mapsto A^z_y$ is locally Lipschitz on $\pDd\setminus\{z\}$.
For any $x \in \pDd$, for $\varsigma$-a.e. $z \in \pDd$, 
there is strong existence and uniqueness
of a càdlàg $(\cG_u)_{u\geq 0}$-adapted $\pDd$-valued process $(b_u^{x,z})_{u\geq 0}$ solving
\begin{equation}\label{sdeb2}
b_u^{x,z}=x+\int_0^u \int_\cE [g_{b^{x,z}_\vm}(A^z_{ b^{x,z}_\vm},e)- b^{x,z}_\vm ] \Pi_\beta(\dr v,\dr e),
\end{equation}
with moreover $\rho_{x,z}:=\lim_{k\to \infty} \rho_{x,z}^k=\infty$ a.s., where 
$\rho_{x,z}^k:=\inf\{u\geq 0 : |b_u^{x,z}-z|\leq 1/k\}$. 
\end{proposition}

Point (b) might actually hold true for all $z \in \pDd\setminus \{x\}$, but this seems difficult to prove.

\begin{proof}
Point (a) classically follows from the Lipschitz estimate
$$
\int_\cE \Big||g_y(A_y,e)-g_{y'}(A_{y'},e)|-|y-y'| \Big| \nn_\beta(\dr e) \leq C |y-y'|,
$$
which is a consequence of Proposition~\ref{ttoopp} and Lemma~\ref{locglob}-(i). The resulting 
process is indeed $\pDd$-valued since $g_y(A_y,e) \in \pDd$ for all $y \in \pDd$, all $e\in \cE$.
We now turn to (b) and fix $x \in \pDd$.

\vip
{\it Step 1.}
We fix $z \in \pDd$. By Proposition~\ref{ttoopp} and Lemma~\ref{locglob}-(ii),
there exists, for any $k\geq 1$, a constant $C_k$ such that 
for any $y,y' \in \pDd \setminus B_d(z,1/k)$,
\begin{equation}\label{tbru}
\int_\cE \Big||g_y(A_y^z,e)-g_{y'}(A_{y'}^z,e)|-|y-y'| \Big| \nn_\beta(\dr e) \leq C_k |y-y'|.
\end{equation}
Moreover, $g_y(A_y^z,e) \in \pDd$ for all $y \in \pDd$, all $e\in \cE$.
We classically deduce the strong existence and uniqueness of a càdlàg $(\cG_u)_{u\geq 0}$-adapted 
$\pDd$-valued process
$b^{x,z}$ solving~\eqref{sdeb2} on  $[0,\rho_{x,z}^k]$.  Letting $k\to \infty$, we conclude the strong 
existence and uniqueness of a càdlàg $(\cG_u)_{u\geq 0}$-adapted 
$\pDd$-valued process
$b^{x,z}$ solving~\eqref{sdeb2} on  $I_{x,z}=\cup_{k\geq 1} [0,\rho_{x,z}^k]$. 

\vip

{\it Step 2.} For $z \in \pDd$, we introduce $\cR_{x,z}:=\overline{\{b^{x,z}_u, u \in I_{x,z}\}}$.
Then $\{\rho_{x,z}<\infty\}\subset \{z\in\cR_{x,z}\}$.
\vip
Indeed, recall that $\rho_{x,z}=\lim_k \rho_{x,z}^k$. On the event $\{\rho_{x,z}< \infty\}$, we have
\vip
\noindent $\bullet$ either $I_{x,z}=[0,\rho_{x,z})$ if the sequence $k\mapsto \rho_{x,z}^k$ 
is not eventually constant,
in which case $b^{x,z}_{\rho_{x,z}-}=z$ (because $|b^{x,z}_{\rho^k_{x,z}}-z|\leq 1/k$),
\vip
\noindent $\bullet$ or $I_{x,z}=[0,\rho_{x,z}]$ if the sequence $k\mapsto \rho_{x,z}^k$ is 
eventually constant, {\it i.e.} if $b^{x,z}_{\rho_{x,z}^k}=z$ for all $k$ large enough, 
in which case $b^{x,z}_{\rho_{x,z}}=z$.

\vip

{\it Step 3.} For all $z\in\pDd$, for $\varsigma$-a.e. $z'\in\pDd$, $\PP(z' \in \cR_{x,z})=0$.
\vip
Using~\eqref{se1} and that $|g_y(A,e)-y|\leq M(e)\land\text{diam }\Dd$ for all $y \in \pDd$, $A \in \cI_y$, 
$e\in\cE$, 
we see that
$$
\text{for all $T>0$,}\qquad 
\E\Big[\sum_{u\in \JJ\cap I_{x,z} \cap[0,T]} |\Delta b^{x,z}_u|\Big]\leq T\int_\cE (M(e)\land \text{diam }\Dd) <\infty.
$$ 
This implies, see L\'epingle~\cite[Theorem 1]{lepingle}, that $(b^{x,z}_u)_{u\in I_{x,z}}$ a.s. 
has finite variation 
on any finite time interval. By McKean~\cite[Page 570]{McKean}, this implies that
the Hausdorff dimension of $\cR_{x,z}$
is a.s. smaller than $1$ (McKean gives the result for $\{b^{x,z}_u, u \in I_{x,z}\}$, 
but taking the closure
only adds a countable number of points, namely $\{b^{x,z}_{u-}, u \in \JJ\cap I_{x,z}\}$). Since 
$\varsigma$ is the Hausdorff measure
of dimension $d-1\geq 2$ restricted to $\pDd$
we conclude that a.s., $\varsigma(\cR_{x,z})=0$, whence
$\int_{\pDd} \PP(z' \in \cR_{x,z}) \varsigma (\dr z')= \E[\varsigma(\cR_{x,z})]=0$.
\vip

{\it Step 4.} For $z,z' \in \pDd$, 
set $\rho^k_{x,z,z'}=\inf\{t \geq 0 : |b^{x,z}_t-z|\leq 1/k$ or $|b^{x,z}_t-z'|\leq 1/k\}$ and 
$I_{x,z,z'}=\cup_{k\geq 1}[0,\rho^k_{x,z,z'}]$.
The processes
$(b^{x,z}_u)_{u \in I_{x,z,z'}}$ and $(b^{x,z'}_u)_{u \in I_{x,z',z}}$
have the same law (when extended to the time-interval $\R_+$ by sending them to a 
cemetery point after their lifetime).
\vip
Indeed, for all $u \in [0,\rho^k_{x,z',z}]$, we have
\begin{gather*}
b_u^{x,z'}\!=\!x+\!\int_0^u \int_\cE [g_{b^{x,z'}_\vm}(A^{z'}_{b^{x,z'}_\vm},e)- b^{x,z'}_\vm ] \Pi_\beta(\dr v,\dr e)
\!=\!x+\!\int_0^u \int_\cE [g_{b^{x,z'}_\vm}(A^{z}_{b^{x,z'}_\vm},e)- b^{x,z'}_\vm ] \Pi_\beta'(\dr v,\dr e),\\
\text{where}\quad  \Pi_\beta' =\sum_{u \in \JJ}\delta_{(u, e_u')}, \quad \text{with}\quad 
e_u'=\Theta_ue_u\quad  \text{and} \quad \Theta_u= (A^{z}_{b^{x,z'}_{u-}})^{-1}A^{z'}_{b^{x,z'}_{u-}}.
\end{gather*}
But $(\Theta_u)_{u\geq 0}$ is $(\cG_u)_{u\geq 0}$-predictable and 
a.s., for any $u\geq 0$, $\Theta_u$ is a linear isometry sending 
$\be_1$ to $\be_1$. By Lemma~\ref{ttaann}, $\Pi_\beta'$ is a $(\cG_u)_{u\geq 0}$-Poisson measure
with intensity $\dr u \nn_\beta(\dr e)$.
Consequently, $b^{x,z'}$ solves~\eqref{sdeb2} during $[0,\rho^k_{x,z',z}]$, {\it i.e.} until it enters 
$\bar B_d(z,1/k)\cup\bar B_d(z',1/k)$, with the Poisson measure
$\Pi'_\beta$. Since $b^{x,z}$ also solves~\eqref{sdeb2} during $[0,\rho^k_{x,z,z'}]$, {\it i.e.} until it enters 
$\bar B_d(z,1/k)\cup\bar B_d(z',1/k)$, with the Poisson measure
$\Pi_\beta$ and since this S.D.E. is well-posed, see Step 1, we conclude that
$(b^{x,z}_u)_{u \in  [0,\rho^k_{x,z,z'}]}$ and $(b^{x,z'}_u)_{u \in [0,\rho^k_{x,z',z}]}$ 
share the same law. The conclusion follows.

\vip

{\it Step 5.} For $\varsigma\otimes \varsigma$-a.e. $(z,z') \in \pDd^2$,
$\cR_{x,z}$ and $\cR_{x,z'}$ have the same law.
\vip

By Step~4, we know that for all $z,z' \in \pDd$, $\closure{\{b^{x,z}_u : u \in I_{x,z,z'}\}}$ 
and $\closure{\{b^{x,z'}_u : u \in I_{x,z',z}\}}$ have the same law.
It thus suffices to verify that for  $\varsigma\otimes\varsigma$-a.e. $(z,z') \in \pDd^2$,
we both have $I_{x,z,z'}=I_{x,z}$ a.s. ({\it i.e.} $z' \notin \cR_{x,z}$ a.s.) and  $I_{x,z',z}= I_{x,z'}$ a.s. 
({\it i.e.} $z \notin \cR_{x,z'}$ a.s.). This follows from Step 3.

\vip

{\it Step 6.} By Step 2, it suffices to check that for $\varsigma$-a.e. $z\in  \pDd$,
$\PP(z\in\cR_{x,z})=0$.
But for $\varsigma\otimes\varsigma$-a.e. $(z,z') \in (\pDd)^2$, we have
$\PP(z\in\cR_{x,z})=\PP(z\in \cR_{x,z'})=0$ by Steps 5 and 3.
\end{proof}

We are now ready to give the

\begin{proof}[Proof of Theorem~\ref{mr1}] 
We only treat the case $d\geq 3$. The case $d=2$ is easier, as we can use Proposition~\ref{newb}-(a) 
instead of Proposition~\ref{newb}-(b).
We fix
$\beta \in \{*\}\cup(0,\alpha/2)$ and suppose Assumption~\ref{as}. We also consider a filtration
$(\cG_u)_{u\geq 0}$ and a $(\cG_u)_{u\geq 0}$-Poisson measure $\Pi_\beta=\sum_{u \in \JJ}\delta_{(u,e_u)}$
on $\R_+\times\cE$ with intensity measure $\dr u \nn_\beta(\dr e)$.

\vip

{\it Step 1.} We prove (a). The existence, for any $x \in \pDd$, of an 
$(\alpha,\beta)$-stable process reflected in $\cDd$ issued from $x$, see Definition~\ref{dfr1}, 
immediately follows from Proposition~\ref{newb}: fix one value of $z \in \pDd$ such that 
$(b_u:=b^{x,z}_u)_{u\geq 0}$ exists.
Then $(b_u)_{u\geq 0}$ solves~\eqref{sdeb} with the choice $a_u=A^z_{b^{x,z}_{u-}}$
and it suffices to define $(\tau_u)_{u\geq 0}$ by~\eqref{sdet}, to introduce
$(L_t=\inf\{u\geq 0 : \tau_u > t\})_{t\geq 0}$ and to define $(R_t)_{t\geq 0}$ by~\eqref{defR}.

\vip

The process $(R_t)_{t\geq 0}$ is $\cDd$-valued. Indeed, recall from Remark~\ref{rkr2} that we always have
$t\in [\tau_{L_t-},\tau_{L_t}]$.
If first $L_t \in \JJ$ and $\tau_{L_t}>t$, then $R_t=h_{b_{L_t-}}(a_{L_t},e_{L_t}(t-\tau_{L_t-}))$, which belongs to $\cDd$ 
because $t-\tau_{L_t-}\in [0,\Delta \tau_{L_t})=[0,\cl_{b_{L_t-}}(a_{L_t},e))$ and by definition of $\cl$;
if next $L_t \in \JJ$ and $\tau_{L_t}=t$, then $R_t=b_{L_t}\in \pDd$, see Remark~\ref{rkr}-(f);
if finally  $L_t\notin \JJ$, then $R_t=b_{L_t} \in \pDd$, see Remark~\ref{rkr}-(a).

\vip

It remains to show that $R=(R_t)_{t\geq 0}$ is a.s. càdlàg on $\R_+$.
We will use that for all $t\geq 0$,
\begin{equation}\label{ttu}
|R_t-b_{L_{t-}}|=\left.
\begin{cases}
|h_{b_{L_t -}}(a_{L_t},e_{L_t}(t-\tau_{L_t -}))-b_{L_t-}| & \text{if }  \tau_{L_t} > t, \\
|g_{b_{L_t-}}(a_{L_t},e_{L_t})-b_{L_t-}| & \text{if }   L_t \in \JJ,  \tau_{L_t} = t, \\
|b_{L_t} - b_{L_t-}| &\text{if } L_t \notin  \JJ.
\end{cases} \right\}
\leq \indiq_{\{L_t \in \JJ\}}M(e_{L_t}),
\end{equation}
where $M(e)=\sup_{t\in [0,\ell(e)]}|e(t)|$, 
because $b_{L_t}=b_{L_t-}$ if $L_{t}\notin \JJ$, see Remark~\ref{rkr}-(a). 
\vip
Let us now check that 
\begin{equation}\label{tuu}
\text{a.s., for all } u\geq 0, \lim_{v\to u,v\neq u}
\indiq_{\{v \in \JJ\}}M(e_v)=0.
\end{equation}
The process $\hat \tau_u=\int_0^u \int_\cE \ell(e) \Pi_\beta(\dr v,\dr e)$ being càdlàg,
\begin{equation}\label{pttu}
\text{a.s., for all } u\geq 0, \lim_{v\to u,v\neq u} \Delta \hat\tau_v=0,\quad \text{i.e. } \lim_{v\to u,v\neq u}
\indiq_{\{v \in \JJ\}}\ell(e_v)=0.
\end{equation}
For $T>0$ and $\delta>0$, set $G_{T,\delta}=\sup_{v \in \JJ\cap[0,T]} M(e_v)\indiq_{\{\ell(e_v)\leq \delta\}}$. 
We have
$$
\E[G_{T,\delta}\land 1]\leq \E\Big[\sum_{v \in \JJ\cap[0,T]} (M(e_v)\land 1)\indiq_{\{\ell(e_v)\leq \delta\}}\Big]
=T \int_\cE (M(e)\land 1)\indiq_{\{\ell(e)\leq \delta\}}\nn_\beta(\dr e) \to 0 \text{ as } \delta\to 0
$$
by~\eqref{se1} and since $\ell>0$ $\nn_\beta$-a.e.
Since $G_{T,\delta}$ is monotone in $\delta$, we conclude that $\lim_{\delta\to 0}  G_{T,\delta}=0$ a.s. for all $T>0$
which, together with~\eqref{pttu}, implies~\eqref{tuu}.

\vip

We claim that
\begin{gather*}
L_{\tau_u+h}>u \quad \text{for all $u\in \JJ$, all $h>0$},\\
L_{\tau_{u-}-h}<u \quad \text{for all $u\in \JJ$, all $h\in (0,\tau_{u-})$},\\
L_{t+h}\neq L_t \quad  \text{for all $t\geq 0$ such that $L_t\notin \JJ$ and all $h\in [-t,\infty)\setminus\{0\}$.}
\end{gather*}
The first claim follows from the fact that $L_{\tau_u+h}=\inf\{v>0 : \tau_v > \tau_{u}+h\}$ and that 
$(\tau_u)_{u\geq 0}$ is càd. Indeed, if $L_{\tau_u+h} \leq u$, then for any $v > u$, $\tau_v > \tau_u + h$, 
which would contradict the right-continuity of $\tau$.
The second one uses that $L_{\tau_{u-}-h}=\inf\{v>0 : \tau_v > \tau_{u-}-h\}$ and that $(\tau_u)_{u\geq 0}$ is làg.
The last claim uses that for $t\geq 0$ such that $L_t \notin \JJ$, if e.g. $h>0$,
$L_{t+h}=\inf\{v>0 : \tau_v > t+h\}>\inf\{v>0 : \tau_v > t\}$, because $\tau$ is continuous at $L_t$ and
$\tau_{L_t}=t$.
\vip
We now have all the tools to check that $R$ is càdlàg.
\vip
\noindent $\bullet$ Since $e_u$ is càdlàg for each $u \in \JJ$ and
since $R_t=h_{b_{u-}}(a_u,e_u(t-\tau_{u-}))$ for each $t \in [\tau_{u-},\tau_u)$, $R$ is clearly càdlàg on 
$\cup_{u\in \JJ} (\tau_{u-},\tau_u)$, càd at each $t\in \{\tau_{u-} : u \in \JJ, \Delta \tau_u>0\}$ and làg at each 
$t\in \{\tau_{u} : u \in \JJ, \Delta \tau_u>0\}$.
\vip
\noindent $\bullet$ Let us check that $R$ is càd at $t=\tau_u$ for each $u \in \JJ$. 
For $h>0$, by Remark~\ref{rkr}-(f) with $t=\tau_u$,
$$
|R_{\tau_u+h}-R_{\tau_u}|=|R_{\tau_u+h}-b_u|\leq |R_{\tau_u+h}-b_{L_{\tau_u+h}-}| + |b_{L_{\tau_u+h}-}-b_u|
=:\Delta_{u,h}^1+\Delta_{u,h}^2.
$$
As seen above, $L_{\tau_u+h}>u$ for all $h>0$.
Hence $\lim_{h\searrow 0} \Delta_{u,h}^2=0$ because $b$ is càd, $L$ is continuous and $L_{\tau_u}=u$.
Next, $\Delta_{u,h}^1\leq \indiq_{\{L_{\tau_u+h} \in \JJ\}}M(e_{L_{\tau_u+h}})$ by~\eqref{ttu}.
Thus $\lim_{h\searrow 0} \Delta_{u,h}^1=0$ by~\eqref{tuu}, because $L_{\tau_u+h}\neq u$ for all $h>0$
and because $L$ is continuous and $L_{\tau_u}=u$.
\vip
\noindent $\bullet$ We now verify that $R$ is làg at $t=\tau_{u-}$ for each $u \in \JJ$, and more 
precisely that $\lim_{h \searrow 0}R_{\tau_{u-}-h}=b_{u-}$.
We write, for $h\in (0,\tau_{u-})$,
$$
|R_{\tau_{u-}-h}-b_{u-}|\leq |R_{\tau_{u-}-h}-b_{(L_{\tau_{u-}-h})-}| + |b_{(L_{\tau_{u-}-h})-}-b_{u-}|
=:\Delta_{u,h}^3+\Delta_{u,h}^4.
$$
First, $\lim_{h\searrow 0} \Delta_{u,h}^4=0$ because $b$ is làg, $L$ is continuous, $L_{\tau_{u-}-h}< u$
and $L_{\tau_{u-}}=u$. Next, $\Delta_{u,h}^3\leq \indiq_{\{L_{\tau_{u-}-h} \in \JJ\}}M(e_{L_{\tau_{u-}-h}})$ by~\eqref{ttu}.
Since $L_{\tau_{u-}-h}<u$ for all $h\in [0,\tau_{u-})$ and since $L$ is continuous,
$\lim_{h\searrow 0} \Delta_{u,h}^3=0$ by~\eqref{tuu}.
\vip
\noindent $\bullet$ Let us show that $R$ is continuous at each $t$ such that $L_t\notin \JJ$.
By Remark~\ref{rkr}-(a), for $h \in (-t,\infty)$,
$$
|R_{t+h}-R_t|=|R_{t+h}-b_{L_t}|\leq |R_{t+h}-b_{L_{t+h}-}| + |b_{L_{t+h}-}-b_{L_t}|
=:\Delta_{t,h}^5+\Delta_{t,h}^6.
$$
Since $L$ is continuous and $b$ is continuous at $L_t \notin \JJ$, we have
$\lim_{h\to 0} \Delta_{t,h}^6=0$. Next, 
$\Delta_{t,h}^5\leq \indiq_{\{L_{t+h} \in \JJ\}} M(e_{L_{t+h}})$ by~\eqref{ttu}.
But $\lim_{h\to 0}\indiq_{\{L_{t+h} \in \JJ\}} M(e_{L_{t+h}})=0$ by~\eqref{tuu}, since $L_{t+h}\neq L_t$ for all $h \neq 0$
as recalled above.

\vip

{\it Step 2.} Here we show (b), {\it i.e.} that for any $x\in \pDd$, any 
$(\alpha,\beta)$-stable process $(R'_t)_{t\geq 0}$ reflected in $\cDd$ issued from $x$, associated to some
$(\cG_u')_{u\geq 0}$-Poisson measure $\Pi_\beta'$ and some processes $(b'_u)_{u\geq 0}$, $(a'_u)_{u\geq 0}$, 
$(\tau'_u)_{u\geq 0}$ and $(L_t')_{t\geq 0}$,
has the same law as the process of Step 1 (built with the same value of $z\in \pDd$
as in Step 1, with  
some $(\cG_u)_{u\geq 0}$-Poisson measure 
$\Pi_\beta$ and some processes $(b_u)_{u\geq 0}$, $(a_u:=A^z_{b_{u-}})_{u\geq 0}$, 
$(\tau_u)_{u\geq 0}$ and $(L_t)_{t\geq 0}$).

\vip

Write $\Pi_\beta'=\sum_{u \in \JJ'}\delta_{(u,e_u')}$ and, setting $e_u''= \Theta_u e_u'$
with $\Theta_u=(A^z_{b'_{u-}})^{-1}a_u'$
for all $u \in \JJ'$,
\begin{align*}
b_u'=& x + \int_0^u\int_{\mathcal{E}}\Big(g_{b'_\vm}(a'_v,e) - b_\vm'\Big)\Pi_\beta' (\dr v, \dr e)\\
=& x + \int_0^u\int_{\mathcal{E}}\Big(g_{b'_{\vm}}(A^z_{b'_\vm},e) - b'_\vm\Big)\Pi_\beta''(\dr v, \dr e),\;\; 
\hbox{where}\;\; \Pi_\beta''=\sum_{u \in \JJ'}\delta_{(u,e_u'')}.
\end{align*}
This follows from the fact that $g_{b'_\vm}(a'_v,e)=g_{b'_{\vm}}(A^z_{b'_\vm}, \Theta_u e)$.
But $(\Theta_u)_{u\geq 0}$ is $(\cG_u)_{u\geq 0}$- predictable and a.s., 
for any $u\geq 0$, $\Theta_u$ is a linear isometry sending 
$\be_1$ to $\be_1$. By Lemma~\ref{ttaann}, $\Pi_\beta''$ is a $(\cG_u)_{u\geq 0}$-Poisson measure
with intensity $\dr u \nn_\beta(\dr e)$.
By strong uniqueness, see Proposition~\ref{newb}-(b),
$((b_u')_{u\geq 0},\Pi''_\beta)$ has the same law as $((b_u)_{u\geq 0},\Pi_\beta)$. 
Next, we have
$$
\tau_u'=\int_0^u\int_{\mathcal{E}}\cl_{b'_\vm}(a'_v,e)\Pi_\beta' (\dr v, \dr e)
=\int_0^u\int_{\mathcal{E}}\cl_{b'_\vm}(A^z_{b'_\vm},e)\Pi_\beta''(\dr v, \dr e)
$$
and $L'_t=\inf\{u\geq 0 : \tau'_u>t\}$,
so that $((b_u')_{u\geq 0},(\tau'_u)_{u\geq 0},(L'_t)_{t\geq0},\Pi''_\beta)$ has the same law as 
$((b_u)_{u\geq 0},(\tau_u)_{u\geq 0},(L_t)_{t\geq 0},\Pi_\beta)$. Finally,
\begin{equation*}
R_t' = 
\begin{cases}
h_{b'_{L_t' -}}(a'_{L'_t},e'_{L'_t}(t-\tau'_{L'_t -}))= h_{b'_{L_t' -}}(A^z_{ b'_{L'_t-}},e''_{L'_t}(t-\tau'_{L'_t -}))
& \text{if }  \tau'_{L'_t} > t, \\
g_{b'_{L'_t-}}(a'_{L'_t},e'_{L'_t})=g_{b'_{L'_t-}}(A^z_{b'_{L'_t-}},e''_{L'_t}) & 
\text{if }   L'_t \in \JJ' \text{ and } \tau'_{L'_t} = t, \\
b'_{L_t'} &\text{if } L'_t \notin  \JJ',
\end{cases}
\end{equation*}
so that $((b_u')_{u\geq 0},(\tau'_u)_{u\geq 0},(L'_t)_{t\geq0},(R'_t)_{t\geq 0},\Pi''_\beta)$ and 
$((b_u)_{u\geq 0},(\tau_u)_{u\geq 0},(L_t)_{t\geq 0},(R_t)_{t\geq 0},\Pi_\beta)$ have the same law. In particular,
$(R'_t)_{t\geq 0}$ has the same law as $(R_t)_{t\geq 0}$ as desired. 

\vip

{\it Step 3.} Here we prove point (c). We 
call $\QQ_x$ the law of the $(\alpha,\beta)$-stable process $(R_t)_{t\geq 0}$ reflected in $\cDd$ issued from $x$,
consider some sequences $t_n\geq 0$ and $x_n \in \pDd$ such that $\lim_n t_n=t\geq 0$ and  $\lim_n x_n=x \in \pDd$,
some bounded continuous function $\varphi : \cDd \to \R$, and verify that
$\lim_n \QQ_{x_n}[\varphi(X^*_{t_n})]=\QQ_x[\varphi(X^*_t)]$, where we recall that 
$(X_t^*)_{t\geq0}$ is the canonical process.
\vip

{\it Step 3.1.} We fix a Poisson measure $\Pi_\beta$ and
$z \in \pDd$ such that Proposition~\ref{newb}-(b) applies both to $x$ and to every $x_n$ for all $n\geq 1$.
We consider the solutions $(b_u:=b^{x,z}_u)_{u\geq 0}$ and $(b^n_u:=b^{x_n,z}_u)_{u\geq 0}$ to~\eqref{sdeb2},
and we then build $((\tau_u)_{u,\geq 0}, (L_t)_{t\geq 0}, (R_t)_{t\geq 0})$ 
(resp. $((\tau_u^n)_{u,\geq 0}, (L_t^n)_{t\geq 0}, (R_t^n)_{t\geq 0})$) as in Definition~\ref{dfr1}, using
$(a_u:=A^z_{b_{u-}})_{u \geq 0}$ (resp.$(a_u^n:=A^z_{b^n_{u-}})_{u \geq 0}$) and the Poisson measure $\Pi_\beta$.
We will check that $\lim_n\E[|R^n_{t_n}-R_t|]=0$ and this will end the proof, since
$(R^n_s)_{s\geq 0} \sim \QQ_{x_n}$ and $(R_s)_{s\geq 0} \sim \QQ_x$.

\vip
{\it Step 3.2.} We verify that for any $T>0$, $\sup_{[0,T)}(|b^n_u-b_u|+|a^n_u-a_u|)$ 
tends to $0$ in probability.

\vip
For $k\geq 1$, we introduce $\rho^k = \inf\{u \geq 0 : |b_u-z|\leq 1/k\}$ and 
$\rho^k_n = \inf\{u \geq 0 : |b_u^n-z |\leq 1/k\}$. Using~\eqref{tbru} and the Gronwall Lemma, one easily gets
\begin{equation}\label{gk}
\E\Big[\sup_{[0,T\land \rho^k \land \rho^k_n)}|b^n_u-b_u|\Big]\leq |x_n-x| e^{C_kT}.
\end{equation}
Since $y\mapsto A^z_y$ is Lipschitz continuous on $\pDd\setminus B_d(z,1/k)$, this gives,
for some other constant $C'_k$,
\begin{equation}\label{gk2}
\E\Big[\sup_{[0,T\land \rho^k \land \rho^k_n)}|a^n_u-a_u|\Big]
=\E\Big[\sup_{[0,T\land \rho^k \land \rho^k_n)}|A^z_{b^n_{u-}}-A^z_{b_{u-}}|\Big]
\leq C'_k|x_n-x| e^{C_kT}.
\end{equation}
For $\e>0$ and $k\geq1$, we write 
$$
p_n(\e):=\PP\Big(\sup_{[0,T)}(|b^n_u-b_u|+|a^n_u-a_n|)>\e\Big)\leq \PP(\rho^{k} \leq T)+
p_{n,k}^1 +p_{n,k}^2(\e),
$$ 
where
\begin{gather*}
p_{n,k}^1=  \PP(\rho^{k} >T,\rho^{2k}_n\leq T) \quad \text{and} \quad
p_{n,k}^2(\e) = \PP\Big(\sup_{[0,T\land \rho^{k} \land \rho^{2k}_n)}(|b^n_u-b_u|+|a^n_u-a_u|)>\e\Big).
\end{gather*}
First, $\lim_n p_{n,k}^2(\e)=0$ for each $k\geq 1$ by~\eqref{gk}-\eqref{gk2}. Next, we observe that
$\rho^{k} >T,\rho^{2k}_n\leq T$ implies that $\sup_{[0,T \land \rho^k\land \rho^{2k}_n)}|b^n_u-b_u|>1/(2k)$, whence
$p_{n,k}^1\leq p_{n,k}^2(1/2k)$ and $\lim_n p_{n,k}^1=0$ for each $k\geq 1$ as well.
We thus find $\limsup_n p_n(\e) \leq  \PP(\rho^{k} \leq T)$. Since 
$\lim_k \rho^k = \infty$ a.s. by Proposition~\ref{newb}-(b), the conclusion follows.
\vip

{\it Step 3.3.} We next check that for any $T>0$, $\sup_{[0,T]}|\tau^n_u-\tau_u|$ tends to $0$ in probability.
Recalling~\eqref{sdet} and using the subadditivity of the function $r \mapsto r\land 1$ on $[0,\infty)$, we write
$$
\E\Big[\sup_{[0,T]}|\tau^n_u-\tau_u|\land 1\Big] \leq 
\int_0^T \int_\cE \E[|\cl_{b^n_\vm}(a^n_v,e)-\cl_{b_\vm}(a_v,e)|\land 1] \nn_\beta(\dr e) \dr v.
$$
Recall that $\cl_y(A,e)\leq \ell(e)$ for all $y\in \pDd$, all $A\in \cI_y$, and that 
$\int_\cE (\ell(e)\land 1 ) \nn_\beta(\dr e)<\infty$, see~\eqref{se1}. By dominated convergence,
it thus suffices to verify that for all $v \in [0,T]$, for $\nn_\beta$-a.e. $e \in \cE$,
$\lim_n \cl_{b^n_\vm}(a^n_v,e)=\cl_{b_\vm}(a_v,e)$ in probability. This follows from Lemma~\ref{relou}-(ii)
and Step 3.2.

\vip

{\it Step 3.4.} By Lemma~\ref{tti}, the generalized inverse $L_t=\inf\{u\geq 0 : \tau_u>t\}$ is a.s. finite 
for all $t\geq 0$ and continuous on $\R_+$. We thus classically deduce from Step 3.3 that for all $T>0$, $\sup_{[0,T]} |L^n_t-L_t| \to 0$ in probability.

\vip

{\it Step 3.5.} Here we prove that for all $t\geq 0$, $\PP(L_t\in \JJ, \tau_{L_t-}=t)=
\PP(L_t\in \JJ, \tau_{L_t}=t)=0$ and for all $n\geq 0$, all $t\geq0$,
$\PP(L_t^n\in \JJ, \tau^n_{L_t^n-}=t)=\PP(L_t^n\in \JJ, \tau^n_{L^n_t}=t)=0$.
\vip
We only study $L_t,\tau_{L_t-},\tau_{L_t}$, the case of  $L_t^n,\tau^n_{L^n_t-},\tau^n_{L^n_t}$ 
being treated in the very same way. Since $L_0=0\notin \JJ$ a.s., it suffices to treat the case where $t>0$. By Remark~\ref{rkr2}, we always have $\tau_{L_t-}\leq t \leq \tau_{L_t}$ and 
$L_t=u \in \JJ$ if and only if $u\in\JJ$ and $t\in [\tau_{u-},\tau_u]$, so that
\begin{align*}
\PP(L_t\in \JJ, \tau_{L_t-}=t)=&\E\Big[\sum_{u \in \JJ} \indiq_{\{\tau_{u-}=t\leq \tau_{u}\}}\Big]
=\E\Big[\sum_{u \in \JJ} \indiq_{\{\tau_{u-}=t\leq\tau_{u-}+\cl_{b_{u-}}(a_u,e_u)\}}\Big].
\end{align*}
By the Poisson compensation formula,
\begin{align*}
\PP(L_t\in \JJ, \tau_{L_t-}=t)=&\E\Big[\int_0^\infty \int_\cE \indiq_{\{\tau_{u}= t\leq \tau_{u}+\cl_{b_{u}}(a_u,e)\}} 
\nn_\beta(\dr e) \dr u\Big]
\leq \E\Big[\int_0^\infty \int_\cE \indiq_{\{\tau_{u}= t\}} 
\nn_\beta(\dr e) \dr u\Big].
\end{align*}
This last quantity equals $0$, since there is a.s. at most one $u \in [0,\infty)$ such that $\tau_u=t$ 
by Lemma~\ref{tti}. Similarly (and using that $\PP(L_t\in \JJ, \tau_{L_t-}=t)=0$),
\begin{align*}
\PP(L_t\in \JJ, \tau_{L_t}=t)=\PP(L_t\in \JJ, \tau_{L_t-}<t=\tau_{L_t})=
\E\Big[\int_0^\infty \int_\cE \indiq_{\{\tau_{u}<t=\tau_{u}+\cl_{b_{u}}(a_u,e)\}} 
\nn_\beta(\dr e) \dr u\Big],
\end{align*}
which equals $0$ by by Lemma~\ref{relou}-(i).

\vip

{\it Step 3.6.} Here we prove that $I_n=\E[|R^n_{t_n}-R_t|\indiq_{\{L_t \notin \JJ\}}]\to 0$.
Everywhere in this step we may use the dominated convergence theorem, since $R$ and $R^n$ take values in $\cDd$
which is bounded.
\vip

We have $R_t=b_{L_t}$ on $\{L_t\notin \JJ\}$, so that $I_n\leq I_{n,A}^1+I_{n,A}^2+I_{n,A}^3$, where,
setting $D=\text{diam}(\Dd)$,
\begin{gather*}
I_{n,A}^1=2 D (\PP(L_t> A)+\PP(L^n_{t_n}> A)), \quad 
I_{n,A}^2=\E[|b^n_{L^n_{t_n}-} - b_{L_t}|\indiq_{\{L_t \notin \JJ, L_t \leq A, L^n_{t_n}\leq A\}}],\\
I_{n,A}^3=\E[|R^n_{t_n}-b^n_{L^n_{t_n}-}| \indiq_{\{L_t \notin \JJ, L_t \leq A, L^n_{t_n}\leq A\}}].
\end{gather*}
By Step 3.4 and since $L$ is continuous, $L^n_{t_n}\to L_t$. Hence 
$\limsup_n I^1_{n,A} \leq 4 D\PP(L_t> A)$,
so that $\lim_{A\to \infty} \limsup_n I^1_{n,A} =0$. Next, 
$$
I^2_{n,A} \leq \E\Big[\sup_{[0,A]}|b^n_u-b_u|\Big]+\E\Big[|b_{L^n_{t_n}-}-b_{L_t}|\indiq_{\{L_t \notin \JJ\}}\Big]. 
$$
The first term
tends to $0$ as $n\to \infty$ by Step 3.2, as well as the second one, since  $L^n_{t_n}\to L_t$
and since $L_t \notin \JJ$ implies that $b$ is continuous at $L_t$.
Thus $\lim_n I^2_{n,A} =0$ for all $A>0$.
We now recall that if $L^n_{t_n} \notin \JJ$, then $R^n_{t_n}=b^n_{L^n_{t_n}}=b^n_{L^n_{t_n}-} $ (see Remark~\ref{rkr}-(a)) 
and write
$$
I_{n,A}^3= \E[|R^n_{t_n}-b^n_{L^n_{t_n}-}|\indiq_{\{L^n_{t_n} \in \JJ,L_t \notin \JJ, L_t \leq A, L^n_{t_n}\leq A\}}]\leq
2DI_{n,A,\delta}^{31} +I_{n,A,\delta}^{32}+ I_{n,A}^{33}, 
$$
where
\begin{gather*}
I_{n,A,\delta}^{31}=\PP(L_t \leq A, L^n_{t_n}\leq A, L_t \notin \JJ, \tau^n_{L^n_{t_n}}- \tau^n_{L^n_{t_n}-}>\delta),\\
I_{n,A,\delta}^{32}= \mathbb{E}\Big[|R^n_{t_n}-b^n_{L^n_{t_n-}}| \indiq_{\{L_t \leq A, L^n_{t_n}\leq A, L_t \notin \JJ,
0 < \tau^n_{L^n_t}-\tau^n_{L^n_{t_n}-} < \delta\}}\Big],\\
I_{n,A}^{33}=\E\Big[|R^n_{t_n}-b^n_{L^n_{t_n}-}| \indiq_{\{L_t \leq A, L^n_{t_n}\leq A, L_t \notin \JJ,L^n_{t_n} \in \JJ,
\tau^n_{L^n_t}=\tau^n_{L^n_{t_n}-}\}}\Big].
\end{gather*}
Observe that $I_{n,A}^{33}=0$ when $\beta=*$, since we have $\cl_y(B,e)>0$ for all $y \in \pDd$, all
$B \in \cI_y$ and $\nn_*$-a.e. $e\in \cE$ (see Remark~\ref{imp4}) and thus $\tau^n_u>\tau^n_{u-}$ for all 
$u \in \JJ$, but this is not the case when $\beta \neq *$ since $\cl_y(B,e)=0$ as soon as $h_y(B,e(0))\notin \Dd$.

\vip
We have $\lim_n I_{n,A,\delta}^{31}=0$ for each $A>0$, each $\delta>0$, because $\tau^n$ converges uniformly to
$\tau$ on $[0,A]$ by Step 3.3, because $L^n_{t_n}\to L_t$, and because $\tau$ is 
continuous at $L_t$ when $L_t \notin \JJ$. Next, 
we recall that when $L^n_{t_n}=u \in \JJ$, we have
$|R^n_{t_n}-b^n_{L^n_{t_n}-}|\leq M(e_u)\land D$, see~\eqref{ttu}. Since $L^n_{t_n}=u \in \JJ$ if and only if 
$t_n \in [\tau^n_{u-},\tau^n_u]=[\tau^n_{u-},\tau^n_{u-}+\cl_{b^n_{u-}}(a^n_u,e)]$,
\begin{align*}
I_{n,A,\delta}^{32}\leq& \mathbb{E}\Big[|R^n_{t_n}-b^n_{L^n_{t_n-}}| \indiq_{\{ L^n_{t_n}\leq A,
0 < \tau^n_{L^n_t}-\tau^n_{L^n_{t_n}-} < \delta\}}\Big]\\
\leq &\E \Big[\sum_{u \in \JJ, u \leq A} [M(e_u)\land D]
\indiq_{\{t_n \in [\tau^n_{u-},\tau^n_{u-}+\cl_{b^n_{u-}}(a^n_u,e)], \ell_{b^n_{u-}}(a_u^n,e_u)\in (0,\delta)\}} \Big]\\
=&\E\Big[\int_0^A \int_{\cE} [M(e)\land D] \indiq_{\{t_n\in[\tau^n_u,\tau^n_u+\cl_{b^n_{u}}(a^n_u,e)]\}}
\indiq_{\{\cl_{b^n_{u}}(a_u^n,e_u)\in (0,\delta)\}} \nn_\beta(\dr e) \dr u\Big]\\
\leq &\E\Big[\int_0^A \int_{\cE} [M(e)\land D]\indiq_{\{\cl_{b^n_{u}}(a_u^n,e_u)\in (0,\delta)\}} \nn_\beta(\dr e) \dr u\Big].
\end{align*}
By Step~3.2, 
Lemma~\ref{relou}-(ii) and dominated convergence (recall~\eqref{se1}), we conclude that
$$
\limsup_n I_{n,A,\delta}^{32} \leq \E\Big[\int_0^A \int_{\cE} [M(e)\land D]
\indiq_{\{\cl_{b_{u}}(a_u,e_u)\in (0,\delta)\}} \nn_\beta(\dr e) \dr u\Big],
$$
so that $\lim_{\delta \to 0} \limsup_n I_{n,A,\delta}^{32}=0$ for all $A>0$ by dominated convergence again.
Finally, observe that $L^n_{t_n} \in \JJ$ and $\tau^n_{L^n_{t_n}}=\tau^n_{L^n_{t_n-}}$ implies that $t_n=\tau^n_{L^n_{t_n}}
=\tau^n_{L^n_{t_n-}}$ and thus that $R^n_{t_n}=b^n_{L^n_{t_n}}$, see Remark~\ref{rkr}-(f). Thus
$$
I_{n,A}^{33}\leq \E\Big[|b^n_{L^n_{t_n}}-b^n_{L^n_{t_n-}}| \indiq_{\{L_t \leq A, L^n_{t_n}\leq A,L_t \notin \JJ\}}\Big]
\leq  2\E\Big[\sup_{[0,A]}|b^n_u-b_u|\Big]+\E\Big[|b_{L^n_{t_n}}-b_{L^n_{t_n-}}|\indiq_{\{L_t \notin \JJ\}}\Big].
$$
Using Step 3.2, that $L^n_{t_n}\to L_t$ by Step~3.4
and that $L_t\notin \JJ$ implies that $b$ is continuous at $L_t$,
we conclude that $\lim_n I^{33}_{n,A}=0$ for each $A>0$.
The step is complete.

\vip

{\it Step 3.7.} We finally check that $J_n=\E[|R^n_{t_n}-R_t|\indiq_{\{L_t \in \JJ\}}]$ tends to $0$.
Again, we may use everywhere the dominated convergence theorem, since $R$ and $R^n$ take values in $\cDd$.

\vip
By Step 3.5, we see that $\{L_t \in \JJ\}=\{\tau_{L_t-}<t<\tau_{L_t}\}$ up to some negligible event.
We write $J_n \leq D J_n^1+J_n^2$, where
\begin{gather*}
J_n^1=\PP(L^n_{t_n}\neq L_t,\tau_{L_t-}<t<\tau_{L_t})\quad \text{and} \quad
J_n^2=\E\Big[|R^n_{t_n}-R_t|\indiq_{\{L^n_{t_n}=L_t, \tau_{L_t-}<t<\tau_{L_t}\}}\Big].
\end{gather*}
By Remark~\ref{rkr2}, we know that for all $u \in \JJ$, $\{L^n_{t_n}=u\}=\{t_n\in[\tau^n_{u-},\tau^n_u]\}$. Thus
$\{L^n_{t_n}=L_t\}=\{t_n\in[\tau^n_{L_t-},\tau^n_{L_t}]\}$, and
$$
J_n^1 = \PP\Big(t_n\notin[\tau^n_{L_t-},\tau^n_{L_t}], t \in (\tau_{L_t-},\tau_{L_t})\Big)\to 0,
$$
because $t_n\to t$, $\tau^n_{L_t-}\to \tau_{L_t-}$ and $\tau^n_{L_t}\to \tau_{L_t}$ by Step 3.3.
We next write, recalling~\eqref{defR},
\begin{align*}
J_n^2=&\E\Big[\sum_{u \in \JJ}\Big|h_{b^n_{u-}}(a^n_u,e_u(t_n -\tau^n_{u-}))-h_{b_{u-}}(a_u,e_u(t-\tau_{u-}))\Big|
\indiq_{\{t_n\in[\tau^n_{u-},\tau^n_u]\}}\indiq_{\{t \in (\tau_{u-},\tau_u)\}} \Big]\\
=&\E\Big[\int_0^\infty \int_\cE \Big|h_{b^n_{u}}(a^n_u,e(t_n-\tau^n_{u}))-h_{b_{u}}(a_u,e(t-\tau_{u}))\Big|
\indiq_{\{t_n\in[\tau^n_{u},\tau^n_{u}+\cl_{b^n_u}(a^n_u,e)]\}}\\
&\hskip8cm\indiq_{\{t \in (\tau_{u},\tau_{u}+\cl_{b_u}(a_u,e)]\}} \nn_\beta(\dr e)\dr u\Big]
\end{align*}
by the compensation formula and since $\tau^n_{u}=\tau^n_{u-}+\cl_{b^n_u}(a^n_u,e_u)$ and $\tau_{u}=\tau_{u-}
+\cl_{b_u}(a_u,e_u)$ for all $u \in \JJ$. By dominated convergence and the following arguments, we conclude that
$\lim_n J_n^2=0$:
\vip
\noindent $\bullet$ $\E[\int_0^\infty \int_\cE \indiq_{\{t \in (\tau_{u},\tau_{u}+\cl_{b_u}(a_u,e))\}} 
\nn_\beta(\dr e)\dr u]= \PP(t\in (\tau_{L_t-},\tau_{L_t}))<\infty$ by the compensation formula,
\vip
\noindent $\bullet$  for $\nn_\beta$-a.e. $e \in \cE$, for all $u\geq 0$ such that 
$t \in (\tau_{u},\tau_{u}+\cl_{b_u}(a_u,e))$ and $t_n \in (\tau_{u}^n,\tau_{u}^n+\cl_{b_u^n}(a_u^n,e))$, 
\[
|h_{b^n_{u}}(a^n_u,e(t-\tau^n_{u}))- h_{b_{u}}(a_u,e(t-\tau_{u}))|\leq |b^n_u-b_u|+||a^n_u-a_u|||e(t_n-\tau^n_u)|
+ |e(t_n -\tau^n_u)-e(t-\tau_u)|,
\]
which vanishes as $n\to \infty$, because
$b^n_u\to b_u $ and $a^n_u\to a_u$ by Step 3.2, $\tau^n_u\to \tau_u$ by Step 3.3,
and $e(t-\tau^n_u)\to e(t-\tau_u)$ since 
$t-\tau_u>0$ is not a jump time of $e$ (for  $\nn_\beta$-a.e. $e \in \cE$).

\vip

{\it Step 4.} We finally check (d). 
Since 
$$
\cB(\DD(\R_+,\cDd))= \sigma(\{\{w \in\DD(\R_+,\cDd) : w(t) \in A \} 
: t \geq 0, A \in \cB(\cDd)\}),
$$ 
it suffices, by a monotone class argument,
to verify that for all $0\leq t_1<\dots<t_k$,
for all $A \in \cB(\cDd^k)$, the map $x \mapsto \QQ_x[(X^*_{t_1},\dots,X^*_{t_k}) \in A]$
is measurable from $\pDd$ to $\R$. Since now $\{A \in \cB(\cDd^k) : x \mapsto 
\QQ_x[(X^*_{t_1},\dots,X^*_{t_k}) \in A]$ 
is measurable$\}$ is a $\sigma$-field, we may assume that $A$ is a closed 
subset of $\cDd^k$. In such a case, we may find a sequence $\varphi_\ell \in C_b(\cDd^k)$ decreasing to $\indiq_A$,
and it finally suffices to show that for any $\varphi \in C_b(\cDd^k)$,
$x\mapsto \QQ_x[\varphi(X^*_{t_1}, \dots, X^*_{t_k})]$ is measurable from $\pDd$ to $\R$.
But this map is actually continuous: this follows from Step 3, where we have seen that if $x_n \in \pDd$ satisfies 
$\lim_n x_n=x \in \pDd$, then it is possible to build $R^{x_n} \sim \QQ_{x_n}$ and 
$R^x \sim \QQ_x$ on the same probability space in such a way that for each $t\geq 0$,
$R^{x_n}_t\to R^x_t$ in probability.
\end{proof}

\section{The reflected stable process starting from anywhere}\label{any}

Our goal is now to prove Theorem~\ref{mr2}.
We start with the Markov property,
of which the proof is fastidious. We do not follow Blumenthal's 
approach~\cite[Chapter V, Section 2]{blu1992}, which rely more on resolvents, and approximations 
of the process. Our proof is closer in spirit to the one given by Salisbury~\cite{MR859838}, although 
it is more involved as we deal with excursions of Markov processes in domains. It basically relies on the 
compensation formula and the Markov property of the excursion measure. The rotational 
invariance of the isotropic stable process is also crucial.

\begin{proposition}\label{markoff}
Fix $\beta \in \{*\}\cup (0,\alpha/2)$ and suppose Assumption~\ref{as}.
 The family $(\QQ_x)_{x\in\closure{\Dd}}$ defines a Markov process on the 
canonical filtered probability space of c\`adl\`ag $\closure{\Dd}$-valued processes.
\end{proposition}

\begin{proof}
We denote by $\Omega^* = \DD(\mathbb{R}_+, \closure{\Dd})$ the canonical space endowed with the Skorokhod $\JS$-topology, by $X^*=(X_t^*)_{t\geq0}$ the canonical process and by 
$(\mathcal{F}_t^*)_{t\geq0}$ the canonical filtration. We have to show that for any $x\in\closure{\Dd}$,
any $t\geq 0$, any bounded measurable functions $\psi_1, \psi_2 : \Omega^* \to \mathbb{R}$,
\begin{equation}\label{eq:markov_prop}
\QQ_x\Big[\psi_1((X_{s\wedge t}^*)_{s\geq0})\psi_2((X_{s + t}^*)_{s\geq0})\Big] 
= \QQ_x\Big[\psi_1((X_{s\wedge t}^*)_{s\geq0})\QQ_{X_t^*}\big[\psi_2((X_{s }^*)_{s\geq0})\big]\Big].
\end{equation}

\textit{Step 1.} We first rephrase Definition~\ref{dfr2}. We denote by $\DD_0$ the subset of 
$\DD(\mathbb{R}_+, \mathbb{R}^d)$ of paths $z$ such that $z(0) \in \cDd$ and for $z\in\DD_0$, 
we set $\tell(z) = \inf\{r >0, \: z(r) \notin {\Dd}\}$. Then we define
$\mathcal{C}: \DD_0 \to \partial\Dd$ and $\Psi : \DD_0 \to \Omega^*$ by
\begin{equation}\label{eq:psi_c}
 \mathcal{C}(z) = \Lambda(z(\tell(z)-), z(\tell(z))) \quad \text{and} \quad 
\Psi(z) = \big(z(t) \indiq_{\{t < \tell(z)\}} + \mathcal{C}(z)\indiq_{\{t \geq \tell(z)\}}\big)_{t\geq0}.
\end{equation}
Next, we introduce the stopping time $d_0^* = \inf\{s >0, \: X_s^* \in \partial\Dd\}$. 
Recall that under $\mathbb{P}_x$, the process $Z=(Z_s)_{s\geq0}$ is an ISP$_{\alpha,x}$.
From Definition~\ref{dfr2}, for any $x\in \Dd$, any bounded measurable 
$\varphi : \Omega^* \times \Omega^* \to \mathbb{R}$,
\begin{gather}\label{eq:def_5}
 \QQ_x\Big[\varphi((X_{s\wedge d_0^*}^*)_{s\geq0},(X_{s + d_0^*}^*)_{s\geq0})\Big] 
= \mathbb{E}_x\Big[\tilde{\varphi}(\Psi(Z), \mathcal{C}(Z))\Big],\\
\text{where} \quad \tilde{\varphi}(w, x) = \QQ_x\big[\varphi(w, X^*)\big] \quad 
\text{for all $w\in \Omega^*$, all $x \in \pDd$.}\notag
\end{gather}

\textit{Step 2.} Here we  show that~\eqref{eq:markov_prop} holds when $x\in \partial\Dd$. 
Consider, on some filtered probability space $(\Omega, \mathcal{G}, (\mathcal{G}_u)_{u\geq0}, \mathbb{P})$, 
a $(\mathcal{G}_u)_{u\geq0}$-Poisson measure $\Pi_\beta = \sum_{u\in \mathrm{J}}\delta_{(u, e_u)}$ 
on $\mathbb{R}_+ \times \mathcal{E}$ with intensity $\dr u\nn_\beta(\dr e)$, together with the processes 
$(a_u^x, b_u^x, \tau_u^x)_{u\geq0}$ and $(L^x_t,R^x_t)_{t\geq 0}$ from Definition~\ref{dfr1}. By definition, 
the law of $(R_t^x)_{t\geq0}$ is $\QQ_x$. For $t> 0$, we also set $g_t^x = \tau_{L_t^x -}^x$ and
$d_t^x = \tau_{L_t^x}^x$.
To lighten the notations, we will most of the time drop the superscript $x$, except for $(R_t^x)_{t\geq0}$.

\vip
\textit{Step 2.1.} Let us verify that 
with the convention that $(v,v)=\emptyset$ for all $v\geq 0$, it holds that
\begin{gather}
\text{for} \quad \cZ:=\closure{\{t\geq 0 : R^x_t \in \pDd\}},
\quad \text{we have} \quad \cZ^c=\cup_{u \in \JJ} (\tau_{u-},\tau_u) \notag\\
\quad \text{and} \quad \cZ= \{\tau_u, \: u\geq 0\}\cup \{\tau_{u-}, \: u> 0\}.\label{zdur}
\end{gather}
Pick $t$ in $\mathcal{O} := \cup_{u \in \JJ} (\tau_{u-},\tau_u)$.
There exists $u \in \JJ$ such that $\Delta\tau_u > 0$ and $t \in (\tau_{u-}, \tau_u)$. By 
Remark~\ref{rkr2}, $L_t = u$ and $t\in(\tau_{L_t-}, \tau_{L_t})$, so that
$R_t^x\in \Dd$ by Remark~\ref{rkr}-(b). Consequently, 
$\mathcal{O} \subset \{t\geq 0, R_t^x \in \Dd\}= \{t\geq 0, R_t^x \in \pDd\}^c$, whence
$\mathcal{O}\subset \cZ^c$ since $\mathcal{O}$ is open.
Now, we infer from Bertoin~\cite[page 13]{bertoin_subordinators} 
that $\mathcal{O}^c = \{\tau_u, \: u\geq 0\}\cup \{\tau_{u-}, \: u> 0\}$,
because $(\tau_u)_{u\geq0}$ is a strictly increasing càdlàg path such that $\lim_{u\to\infty}\tau_u = \infty$,
see Lemma~\ref{tti}. For all $u\geq 0$, $R_{\tau_u}=b_u \in \pDd$: by Remark~\ref{rkr} with $t=\tau_u$
(whence $L_t=u$ by Remark~\ref{rkr2}), we either have $u \notin \JJ$ and thus $R_t=b_u$ 
(see  Remark~\ref{rkr}-(a))
or  $u \in \JJ$ and $t=\tau_u$ and thus $R_t=b_u$  (see  Remark~\ref{rkr}-(f)). Consequently, 
$\{\tau_u, \: u\geq 0\}\subset \cZ$. 
Hence for all $u>0$, all $h \in (0,u)$, $\tau_{u-h}\in \cZ$, so that
$\tau_{u-}=\lim_{h\searrow 0}\tau_{u-h} \in \cZ$, since $\cZ$ is closed. All in all
$\mathcal{O}^c= \{\tau_u, \: u\geq 0\}\cup \{\tau_{u-}, \: u> 0\}\subset \mathcal{Z}$. Since $\mathcal{Z} \subset \mathcal{O}^c$ (because we have seen that $\mathcal{O}\subset \cZ^c$),
we conclude that $\cZ=\mathcal{O}^c= \{\tau_u, \: u\geq 0\}\cup \{\tau_{u-}, \: u> 0\}$
and that $\mathcal{Z}^c=\mathcal{O}= \cup_{u \in \JJ} (\tau_{u-},\tau_u)$. 
We have proved~\eqref{zdur}.
\vip
We next check that for all $t>0$,
\[
g_t= \sup\{s < t, \: R_s^x \in \partial\Dd\}\quad \text{and}\quad 
d_t = \inf\{s > t, \: R_s^x \in \partial\Dd\}.
\] 
Indeed, recall that $g_t=\tau_{L_t-}$ and observe that by~\eqref{zdur}
\begin{align*}
\sup\{s < t, \: R_s^x \in \partial\Dd\}=&\sup \closure{\{s < t, \: R_s^x \in \partial\Dd\}}\\
=& \sup ( \cZ \cap [0,t) ) \\
=& \sup ((\{\tau_u : u \geq 0\}\cup\{\tau_{u-} : u>0\})\cap [0,t)  ),
\end{align*}
which equals $\tau_{L_t-}$ since $t \in [\tau_{L_t-},\tau_{L_t}]$ by Remark~\ref{rkr2}.
The other identity is checked similarly.

\vip

Setting now $d_0 = \inf\{s > 0 : R_s \in \pDd\}$, we observe that $d_0=0$ a.s. 
(recall that $x \in \pDd$). This implies that $\QQ_x(d_0^*=0)=1$
when $x \in \pDd$, a fact we will use later. Indeed, for any $t>0$, $d_0\leq d_t=\tau_{L_t}$, which tends to 
$\tau_0=0$ as $t\to 0$ since $\tau$ is right continuous and since $\lim_{t\to 0}L_t=0$ because 
$\tau$ is strictly increasing (so that $L$ is continuous).

\vip

By Remark~\ref{rkr}-(a)-(b)-(c)-(d)-(f) and since $\tau_{L_t}=t$ when $L_t \notin \JJ$ (because 
$t\in [\tau_{L_t-},\tau_{L_t}]$ and $\Delta \tau_{L_t}=0$), we see that for all $t\geq 0$,
\begin{equation}\label{ptd}
R_t^x = h_{b_{L_t}-}(a_{L_t}, e_{L_t}(t-\tau_{L_t -}))\indiq_{\{\tau_{L_t} > t\}} + b_{L_t}\indiq_{\{\tau_{L_t} = t\}}.
\end{equation}
Recalling that $d_t = \tau_{L_t}$, we deduce that $(R_{s\wedge d_t}^x)_{s\geq0}$ is $\mathcal{G}_{L_t}$-measurable.

\vip

\textit{Step 2.2.} We now show that for any $t\geq0$, the law of $(R_{d_t + s}^x)_{s\geq0}$ conditionally on 
$(R_{s\wedge d_t}^x)_{s\geq0}$ is $\QQ_{R_{d_t}^x}$. To this end, we fix $t\geq0$ and we introduce
$$
 \Pi_\beta^t = \sum_{u \in \mathrm{J}_t}\delta_{(u , e^t_u)} \quad \text{where} \quad \mathrm{J}_t 
= \{u \geq0, \: u + L_t \in \mathrm{J}\} \quad \text{and} \quad e^t_u=e_{u+L_t}. 
$$
Since $L_t$ is a $(\mathcal{G}_u)_{u\geq0}$-stopping time, it comes that $\Pi_\beta^t$ is a 
$(\mathcal{G}_{u+L_t})_{u\geq0}$-Poisson measure distributed as $\Pi_\beta$ and is independent of 
$\mathcal{G}_{L_t}$. 
One easily checks that $(b_u^t)_{u\geq0} = (b_{L_t + u})_{u\geq0}$ and 
$(\tau_u^t)_{u\geq0} = (\tau_{L_t + u} - \tau_{L_t})_{u\geq0}$ solve~\eqref{sdeb} and~\eqref{sdet} with the Poisson
measure $\Pi^t_\beta$, with $(a^t_u=a_{L_t+u})_{u\geq 0}$ and with $b^t_0=b_{L_t}=R_{d_t}^x$ (recall that
$d_t=\tau_{L_t}$ and use Remark~\ref{rkr}-(f)).
Next,
$(L_s^t)_{s\geq0} = (L_{d_t+s} - L_{d_t})_{s\geq0}$ is the right-continuous inverse of $(\tau_u^t)_{u\geq0}$:
this is easily checked, recalling that $d_t=\tau_{L_t}$ and using that
$L_{d_t}=L_t$ (because $L$ is constant on $[\tau_{L_t -}, \tau_{L_t}]$, see Remark~\ref{rkr2}).
By~\eqref{ptd},
\begin{equation*}
R_{d_t+s}^x = 
\begin{cases}
h_{b_{(L_{d_t + s})-}}(a_{L_{d_t + s}}, e_{L_{d_t+s}}(d_t + s - \tau_{(L_{d_t +s})-}))& \text{if } \tau_{L_{d_t + s}} > d_t + s,  \\
b_{L_{d_t+s}} &\text{otherwise}. 
\end{cases}
\end{equation*}
Since $L_{d_t} = L_t$, we find $b_{(L_{d_t +s})-} = b_{(L_s^t)-}^t$, $a_{L_{d_t +s}} = a_{L_s^t}^t$
and $\tau_{(L_{d_t+s})-}-d_t=\tau^t_{(L^t_s)-}$, so that
\begin{equation*}
R_{d_t+s}^x = 
\begin{cases}
h_{b_{(L_{s}^t)-}^t}(a_{L_s^t}^t, e_{L_s^t}^t( s - \tau_{L_s^t}^t))& \text{if } \tau_{L_s^t}^t > s,  \\
b_{L_s^t}^t &\text{otherwise}. 
\end{cases}
\end{equation*}
Thus, conditionally on $\mathcal{G}_{L_t}$,  
$(R_{d_t + s}^x)_{s\geq0}$ is an $(\alpha,\beta)$-stable process reflected in $\cDd$, 
driven by the Poisson measure $\Pi^t_\beta$ and starting from $R^x_{d_t}$.
Hence by Theorem~\ref{mr1}-(b), 
conditionally on $\mathcal{G}_{L_t}$, $(R_{d_t + s}^x)_{s\geq0}$ is $\QQ_{R_{d_t}^x}$-distributed.
This completes the step since 
$\sigma((R_{s\wedge d_t}^x)_{s\geq0}) \subset \mathcal{G}_{L_t}$ by Step~2.1.

\vip

\textit{Step 2.3.} We show here that for any $t>0$, any $x\in\partial\Dd$, any bounded measurable functions
$\psi_1,\psi_2 : \Omega^* \to \R$, we have~\eqref{eq:markov_prop}. By a monotone class argument,
it suffices that for any bounded measurable functions 
$\varphi_1, \varphi_2, \varphi_3, \varphi_4 : \Omega^* \to \mathbb{R}$, we have
\begin{align*}
 \bm{\mathrm{I}} :=& \QQ_x\Big[\varphi_1((X_{s\wedge g^*_t}^*)_{s\geq0})\varphi_2((X_{(s +g^*_t) \wedge t}^*)_{s\geq0})
\varphi_3((X_{(t+s) \wedge d_t^*}^*)_{s\geq0})\varphi_4((X_{d_t^* +s}^*)_{s\geq0})\Big] \\ 
=& \QQ_x\Big[\varphi_1((X_{s\wedge g^*_t}^*)_{s\geq0})
\varphi_2((X_{(s +g^*_t) \wedge t}^*)_{s\geq0})\QQ_{X_t^*}\Big[\varphi_3((X_{s \wedge d_0^*}^*)_{s\geq0})
\varphi_4((X_{d_0^* +s}^*)_{s\geq0})\Big]\Big]
\end{align*}
where $g^*_t = \sup\{s < t, \: X_s^* \in \partial\Dd\}$ and $d_t^* = \inf\{s > t, \: X_s^* \in \partial\Dd\}$. 
By definition, we have
$$
 \bm{\mathrm{I}}= \E\Big[\varphi_1((R_{s\wedge g_t}^x)_{s\geq0})
\varphi_2((R_{(s +g_t) \wedge t}^x)_{s\geq0})\varphi_3((R_{(t+s) \wedge d_t}^x)_{s\geq0})\varphi_4((R_{d_t +s}^x)_{s\geq0})\Big].
$$
Conditioning on $(R^x_{s\wedge d_t})_{s\geq0}$ and applying Step 2.2, we get
$$
 \bm{\mathrm{I}} = \E\Big[\varphi_1((R_{s\wedge g_t}^x)_{s\geq0})\varphi_2((R_{(s +g_t) \wedge t}^x)_{s\geq0})
\varphi_3((R_{(t+s) \wedge d_t}^x)_{s\geq0})\QQ_{R_{d_t}^x}\big[\varphi_4((X_s^*)_{s\geq0})\big]\Big].
$$
We split this expectation according to whether $L_t \notin\mathrm{J}$ or $\tau_{L_t} > t$. 
This covers all the cases, since $\mathbb{P}(L_t \in \mathrm{J}, \tau_{L_t} = t) = 0$
as seen in Step 3.5 of the proof of Theorem~\ref{mr1}. 

\vip

First, when $L_t\notin\mathrm{J}$, we have $d_t = \tau_{L_t} = t$ so that
$(R_{(t+s) \wedge d_t}^x)_{s\geq0}=(R_t)_{s\geq 0}$:
\begin{align*}
 \bm{\mathrm{I}}_1 &:= \E\Big[\varphi_1((R_{s\wedge g_t}^x)_{s\geq0})\varphi_2((R_{(s +g_t) \wedge t}^x)_{s\geq0})
\varphi_3((R_{(t+s) \wedge d_t}^x)_{s\geq0})
\QQ_{R_{d_t}^x}\big[\varphi_4((X_s^*)_{s\geq0})\big]\indiq_{\{L_t \notin \mathrm{J}\}}\Big] \\
& = \E\Big[\varphi_1((R_{s\wedge g_t}^x)_{s\geq0})\varphi_2((R_{(s +g_t) \wedge t}^x)_{s\geq0})
\varphi_3((R_{t}^x)_{s\geq0})\QQ_{R_{t}^x}\big[\varphi_4((X_s^*)_{s\geq0})\big]\indiq_{\{L_t \notin \mathrm{J}\}}\Big].
\end{align*}
Since $L_t \notin \mathrm{J}$ implies that $R_t^x = b_{L_t} \in \partial\Dd$ and since
$\QQ_z(d_0^*=0)=1$ for all $z \in \partial\Dd$ (see Step~2.1)
\begin{equation}\label{eq:first_term_markov}
\bm{\mathrm{I}}_1 = \E\Big[\varphi_1((R_{s\wedge g_t}^x)_{s\geq0})\varphi_2((R_{(s +g_t) \wedge t}^x)_{s\geq0})
\QQ_{R_t^x}\Big[\varphi_3((X_{s \wedge d_0^*}^*)_{s\geq0})\varphi_4((X_{d_0^* +s}^*)_{s\geq0})\Big]
\indiq_{\{L_t \notin \mathrm{J}\}}\Big].
\end{equation}

\vip

To treat the case $\tau_{L_t} > t$, we introduce the function 
$\varphi_5 : \Omega^*\times\closure{\Dd} \to \mathbb{R}$ 
defined by $\varphi_5(w,z) = \varphi_3(w)\QQ_z[\varphi_4(X^*)]$, so that
\begin{align*}
\bm{\mathrm{I}}_2 :=& \E\Big[\varphi_1((R_{s\wedge g_t}^x)_{s\geq0})\varphi_2((R_{(s +g_t) \wedge t}^x)_{s\geq0})
\varphi_3((R_{(t+s) \wedge d_t}^x)_{s\geq0})
\QQ_{R_{d_t}^x}\big[\varphi_4((X_s^*)_{s\geq0})\big]\indiq_{\{\tau_{L_t}>t\}}\Big]\\
=& \E\Big[\varphi_1((R_{s\wedge g_t}^x)_{s\geq0})
\varphi_2((R_{(s +g_t) \wedge t}^x)_{s\geq0})\varphi_5((R_{(t+s) \wedge d_t}^x)_{s\geq0}, R_{d_t}^x)\indiq_{\{\tau_{L_t} > t\}}\Big].
\end{align*}
We now aim to use the compensation formula. For a path $w\in\DD(\mathbb{R}_+, \mathbb{R}^d)$ and for $r\geq 0$,
we set $k_rw = (w(s\wedge r))_{s\geq0}$ and $\theta_rw = (w(r+s))_{s\geq0}$. We set
$\mathcal{A} = \{(b,a) : b \in \partial\Dd, a \in \mathcal{I}_b\}$ and introduce
the function $\mathrm{H} : \mathcal{A}\times \mathcal{E} \to \Omega^*$ defined by
$$
 \mathrm{H}(b, a, e) = \Big(h_{b}(a, e(r))\indiq_{\{r < \bar{\ell}_{b}(a, e)\}} 
+ g_b(a, e)\indiq_{\{r \geq \bar{\ell}_{b}(a, e)\}}\Big)_{r \geq0}.
$$
From~\eqref{ptd} and since $g_t=\tau_{L_t-}$ and
$d_t=\tau_{L_t}$, on the event $\{\tau_{L_t} > t\}$, we have
\begin{gather*}
 (R_{(s +g_t) \wedge t}^x)_{s\geq0} = \mathrm{H}(b_{L_t -}, a_{L_t}, k_{t - \tau_{L_t -}}e_{L_t}), 
\quad (R_{(t+s) \wedge d_t}^x)_{s\geq0} = \mathrm{H}(b_{L_t -}, a_{L_t}, \theta_{t - \tau_{L_t -}}e_{L_t}),\\
\text{and} \quad  R_{d_t}^x = g_{b_{L_t-}}(a_{L_t}, \theta_{t- \tau_{L_t-}}e_{L_t}). 
\end{gather*}
We now set 
$\tilde{\varphi}_2(b,a,e) = \varphi_2 (\mathrm{H}(b,a,e))$ and $\tilde{\varphi}_5(b,a,e) =  
\varphi_5(\mathrm{H}(b, a, e), g_b(a, e))$ and we write
\begin{align*}
\bm{\mathrm{I}}_2 =& \E\Big[\varphi_1((R_{s\wedge \tau_{L_t-}}^x)_{s\geq0}) 
\tilde\varphi_2(b_{L_t -}, a_{L_t}, k_{t - \tau_{L_t -}}e_{L_t})
\tilde\varphi_5(b_{L_t -}, a_{L_t}, \theta_{t - \tau_{L_t -}}e_{L_t}) \indiq_{\{\tau_{L_t}>t\}}\Big]\\
=&\E\bigg[\sum_{u\in\mathrm{J}}\varphi_1((R_{s\wedge \tau_{u-}}^x)_{s\geq0})
\tilde{\varphi}_2(b_{u -}, a_{u}, k_{t - \tau_{u -}}e_u)
\tilde{\varphi}_5(b_{u -}, a_{u}, \theta_{t - \tau_{u -}}e_{u})
\indiq_{\{\tau_{u-}\leq t < \tau_{u-} + \bar{\ell}_{b_{u-}}(a_u, e_u)\}}\bigg].
\end{align*}
We used that $\tau_{L_t}>t$ if and only if $L_t \in \JJ$ and $\tau_{L_t-}\leq t < \tau_{L_t}$,
see Remark~\ref{rkr2}, and that in such a case, 
$\tau_{L_t}=\tau_{L_t-}+ \cl_{b_{L_t-}}(a_{L_t},e_{L_t})$ by~\eqref{sdet}.
Finally, we apply the compensation formula, which is licit because $(a_u)_{u\geq0}$ is 
$(\mathcal{G}_u)_{u\geq0}$-predictable and so is $\varphi_1((R_{s\wedge \tau_{u-}}^x)_{s\geq0})$.
We then get
\begin{align*}
 \bm{\mathrm{I}}_2 = \E\bigg[\int_0^\infty \varphi_1((R_{s\wedge \tau_{u}}^x)_{s\geq0})\indiq_{\{\tau_u \leq t\}}
\int_{\mathcal{E}}\tilde{\varphi}_2(&b_{u}, a_{u}, k_{t - \tau_{u }}e)  \\
 &\times\tilde{\varphi}_5(b_{u}, a_{u}, \theta_{t - \tau_{u}}e)\indiq_{\{t < \tau_{u} + \bar{\ell}_{b_{u}}(a_u, e)\}}
\nn_\beta(\dr e) \dr u\bigg].
\end{align*}
Since $u\mapsto \tau_u$ is a.s. strictly increasing by Lemma~\ref{tti}, there is at most one $u$ for which
$\tau_u=t$ and we can replace $\indiq_{\{\tau_u \leq t\}}$ by $\indiq_{\{\tau_u < t\}}$ in the above integral.
Now, for every fixed $u\geq0$,
we apply the Markov property of $\nn_\beta$
at time $t-\tau_u$, see Lemma~\ref{mark}. This is possible because
$\tau_u < t < \tau_{u} + \bar{\ell}_{b_{u}}(a_u, e)$ implies that $\ell(e)\geq \bar{\ell}_{b_{u}}(a_u, e) 
> t- \tau_u>0$.
Finally, we stress that $\indiq_{\{t < \tau_{u} + \bar{\ell}_{b_{u}}(a_u, e)\}}$ is a function of 
($\tau_u,b_u,a_u$ and) $k_{t-\tau_u}e$. It comes that
\begin{align*}
 \bm{\mathrm{I}}_2 = \E\bigg[\int_0^\infty \!\!\!\varphi_1((R_{s\wedge \tau_{u}}^x)_{s\geq0})&
\int_{\mathcal{E}}\!\!\tilde{\varphi}_2(b_{u}, a_{u}, k_{t - \tau_{u }}e) 
\varphi_6(b_u,a_u,e(t-\tau_u))\indiq_{\{\tau_u< t < \tau_{u} + \bar{\ell}_{b_{u}}(a_u, e)\}}\nn_\beta(\dr e) \dr u\bigg],
\end{align*}
where 
$$
\varphi_6(b,a,z)=\mathbb{E}_{z}\Big[\tilde{\varphi}_5(b,a,(Z_{r\wedge\ell(Z)})_{r\geq0})\Big]
$$
and where we recall that under $\mathbb{P}_z$, $(Z_t)_{t\geq0}$ is an ISP$_{\alpha,z}$.
Recalling the definition~\eqref{eq:psi_c} of the functions $\Psi$ and $\mathcal{C}$, it should be clear that
$H(b,a,e)=\Psi(b+ae)$ and $g_b(a,e)=\mathcal{C}(b+ae)$, whence $ \tilde{\varphi}_5(b,a,e) 
= \varphi_5(\Psi(b + a e), \mathcal{C}(b+ a e))$. By translation and rotational invariance of the isotropic 
stable process, we thus have
$$
\varphi_6(b,a,z)=\mathbb{E}_{h_b(a,z)}[\varphi_5(\Psi(Z),\mathcal{C}(Z))],
$$
whence
\begin{align*}
 \bm{\mathrm{I}}_2 = \E\bigg[\int_0^\infty \!\!\!\varphi_1((R_{s\wedge \tau_{u}}^x)_{s\geq0})
\int_{\mathcal{E}}\!\!\tilde{\varphi}_2(b_{u}, a_{u}, k_{t - \tau_{u }}e) 
\mathbb{E}_{h_{b_u}(a_u,e(t-\tau_u))}&\Big[\varphi_5(\Psi(Z),\mathcal{C}(Z))\Big]\\
&\indiq_{\{\tau_u<  t < \tau_{u} + \bar{\ell}_{b_{u}}(a_u, e)\}}\nn_\beta(\dr e) \dr u\bigg].
\end{align*}
We can then use the compensation formula in the reverse way to conclude that
\begin{align*}
\bm{\mathrm{I}}_2 =&\E\Big[\varphi_1((R_{s\wedge g_t}^x)_{s\geq0})\varphi_2((R_{(s +g_t) \wedge t}^x)_{s\geq0})
\mathbb{E}_{h_{b_{{L_t}-}}(a_{L_t},e_{L_t}(t-\tau_{L_t-}))}[\varphi_5(\Psi(Z), \mathcal{C}(Z))]\indiq_{\{\tau_{L_t} > t\}}\Big]\\
=& \E\Big[\varphi_1((R_{s\wedge g_t}^x)_{s\geq0})\varphi_2((R_{(s +g_t) \wedge t}^x)_{s\geq0})
\mathbb{E}_{R_t^x}[\varphi_5(\Psi(Z), \mathcal{C}(Z))]\indiq_{\{\tau_{L_t} > t\}}\Big],
\end{align*}
since $R^x_t=h_{b_{L_t-}}(a_{L_t},e_{L_t}(t-\tau_{L_t-}))$ when $\tau_{L_t} > t$, see~\eqref{ptd}.
Recalling now~\eqref{eq:def_5} and the definition of $\varphi_5(w,x) 
= \varphi_3(w)\QQ_x[\varphi_4(X^*)]$, we see that for any $z \in \Dd$,
$$
 \mathbb{E}_{z}[\varphi_5(\Psi(Z), \mathcal{C}(Z))] = 
\mathbb{E}_{z}[\varphi_3(\Psi(Z)) \QQ_{\mathcal{C}(Z)}[\varphi_4(X^*)]] =
\QQ_z
\Big[\varphi_3((X_{s \wedge d_0^*}^*)_{s\geq0})\varphi_4((X_{d_0^*+s}^*)_{s\geq0})\Big].
$$
Since we work on $\{\tau_{L_t} > t\}$, we have $R_t^x \in \Dd$ so that
$$
 \bm{\mathrm{I}}_2 = \E\Big[\varphi_1((R_{s\wedge g_t}^x)_{s\geq0})\varphi_2((R_{(s +g_t) \wedge t}^x)_{s\geq0})
\QQ_{R_t^x}\Big[\varphi_3((X_{s \wedge d_0^*}^*)_{s\geq0})\varphi_4((X_{d_0^* +s}^*)_{s\geq0})\Big]
\indiq_{\{\tau_{L_t} > t\}}\Big]
$$
Recalling now the expression~\eqref{eq:first_term_markov} of $\bm{\mathrm{I}}_1$ and that 
$\bm{\mathrm{I}} =\bm{\mathrm{I}}_1 + \bm{\mathrm{I}}_2$, we finally get
\begin{align*}
 \bm{\mathrm{I}} & = \E\Big[\varphi_1((R_{s\wedge g_t}^x)_{s\geq0})\varphi_2((R_{(s +g_t) \wedge t}^x)_{s\geq0})
\QQ_{R_t^x}\Big[\varphi_3((X_{s \wedge d_0^*}^*)_{s\geq0})\varphi_4((X_{d_0^* +s}^*)_{s\geq0})\Big]\Big] \\
 & = \QQ_x\Big[\varphi_1((X_{s\wedge g_t^*}^*)_{s\geq0})\varphi_2((X_{(s +g_t^*) \wedge t}^*)_{s\geq0})
\QQ_{X_t^*}\Big[\varphi_3((X_{s \wedge d_0^*}^*)_{s\geq0})\varphi_4((X_{d_0^* +s}^*)_{s\geq0})\Big]\Big],
\end{align*}
which was our goal.

\vip

\textit{Step 3.} We finally show that~\eqref{eq:markov_prop} holds when $x \in \Dd$. 
We first show that for any  measurable bounded functions 
$\varphi_1, \varphi_2, \varphi_3 :\Omega^* \to \mathbb{R}$,
\begin{align}\label{eq:mark_pas_bord1}
 \bm{\mathrm{J}} := &\QQ_x\Big[\varphi_1((X_{s\wedge t}^*)_{s\geq0})\varphi_2((X_{(s+t)\wedge d_0^*}^*)_{s\geq0})
\varphi_3((X_{d_0^*+ s}^*)_{s\geq0}) \indiq_{\{d_0^*>t\}}\Big] \nonumber \\ 
  = & \QQ_x\Big[\varphi_1((X_{s\wedge t}^*)_{s\geq0})\QQ_{X_t^*}\big[\varphi_2((X_{s\wedge d_0^*}^*)_{s\geq0})
\varphi_3((X_{d_0^* + s}^*)_{s\geq0})\big]\indiq_{\{d_0^*>t\}}\Big].
\end{align}
First, we see from~\eqref{eq:def_5} that
\begin{align*}
\bm{\mathrm{J}} 
=&\mathbb{E}_x\Big[\varphi_1((Z_{s\wedge t})_{s\geq0})\varphi_2(\Psi((Z_{(s+t)\wedge \tell(Z)})_{s\geq0}))
\QQ_{\mathcal{C}(Z)}\big[\varphi_3(X^*)\big]\indiq_{\{\tell(Z) > t\}}\Big].
\end{align*}
Conditionally on $(Z_{s\wedge t})_{s\geq0}$ and $\{\tell(Z)>t\}$, 
$(Z_{t+s})_{s\geq0}$ is an ISP$_{\alpha,Z_t}$, whence
$$
\bm{\mathrm{J}} = \mathbb{E}_x\Big[\varphi_1((Z_{s\wedge t})_{s\geq0})
\mathbb{E}_{Z_t}\Big[\varphi_2(\Psi((Z_{s\wedge \tell(Z)})_{s\geq0}))
\QQ_{\mathcal{C}(Z)}\big[\varphi_3(X^*)\big]\Big]\indiq_{\{\tell(Z) > t\}}\Big],
$$
and~\eqref{eq:mark_pas_bord1} follows from~\eqref{eq:def_5}. Using a monotone class argument,
we deduce from~\eqref{eq:mark_pas_bord1} that for any bounded measurable $\psi_1,\psi_2 : \Omega^* \to \R$, 
\begin{align}\label{tg1}
\QQ_x\Big[\psi_1((X_{s\wedge t}^*)_{s\geq0})\psi_2((X_{t+s}^*)_{s\geq0})\indiq_{\{d_0^*> t\}}\Big]
=& \QQ_x\Big[\psi_1((X_{s\wedge t}^*)_{s\geq0})\QQ_{X_{t}^*}[\psi_2(X^*)]\indiq_{\{d_0^*> t\}}\Big].
\end{align}
We next prove that for any  measurable bounded functions 
$\varphi_1, \varphi_2, \varphi_3 :\Omega^* \to \mathbb{R}$,
\begin{align}\label{eq:mark_pas_bord2}
\bm{\mathrm{K}} := &\QQ_x\Big[\varphi_1((X_{s\wedge d_0^*}^*)_{s\geq0})\varphi_2((X_{(s+d_0^*)\wedge t}^*)_{s\geq0})
\varphi_3((X_{t+ s}^*)_{s\geq0}) \indiq_{\{d_0^*\leq t\}}\Big] \nonumber \\ 
  = & \QQ_x\Big[\varphi_1((X_{s\wedge d_0^*}^*)_{s\geq0})\varphi_2((X_{(s+ d_0^*)\wedge t}^*)_{s\geq0})
\QQ_{X_t^*}[\varphi_3(X^*)]\indiq_{\{d_0^*\leq t\}}\Big].
\end{align}
We first write $\bm{\mathrm{K}}=\QQ_x[\varphi_1((X_{s\wedge d_0^*}^*)_{s\geq 0})\varphi_4(t-d_0^*,(X_{s+d_0^*}^*)_{s\geq0}) 
\indiq_{\{d_0^*\leq t\}}]$, where
$$
\text{for all $r\geq0$, all $w\in \Omega^*$},
\quad \varphi_4(r,w)=\varphi_2(w((s\land r)_{r\geq 0})\varphi_3(w((s+r)_{s\geq 0}).
$$
Since $d_0^*$ is a function of $(X_{s\wedge d_0^*}^*)_{s\geq 0}$, we find, using~\eqref{eq:def_5}, that
\begin{align*}
\bm{\mathrm{K}}=\E_x\Big[\varphi_1(\Psi((Z_{s\wedge \tell(Z)})_{s\geq0})))\varphi_5(t-\tell(Z),
\mathcal{C}(Z))\indiq_{\{\tell(Z) \leq t \}}   \Big],
\end{align*}
where for $r\geq 0$ and $z \in \pDd$,
$$
\varphi_5(r,z)=\QQ_z\Big[\varphi_2((X_{s\land r}^*)_{s\geq0})
\varphi_3((X_{r+ s}^*)_{s\geq0})\Big]= \QQ_z\Big[\varphi_2((X_{s\land r}^*)_{s\geq0}) \QQ_{X_r^*}[
\varphi_3(X^*)]\Big]
$$
by Step 2.
Finally, we use~\eqref{eq:def_5} again to conclude the proof of~\eqref{eq:mark_pas_bord2}.
By a monotone class argument,
we deduce from~\eqref{eq:mark_pas_bord2} that for any bounded measurable $\psi_1,\psi_2 : \Omega^* \to \R$, 
\begin{align*}
\QQ_x\Big[\psi_1((X_{s\wedge t}^*)_{s\geq0})\psi_2((X_{t+s}^*)_{s\geq0})\indiq_{\{d_0^*\leq t\}}\Big]
=& \QQ_x\Big[\psi_1((X_{s\wedge t}^*)_{s\geq0})\QQ_{X_{t}^*}[\psi_2(X^*)]\indiq_{\{d_0^*\leq t\}}\Big].
\end{align*}
Together with~\eqref{tg1}, this shows that~\eqref{eq:markov_prop} holds true.
The proof is complete.
\end{proof}

We can now give the

\begin{proof}[Proof of Theorem~\ref{mr2}]
By Proposition~\ref{markoff}, the family $(\QQ_x)_{x\in \cDd}$ defines a Markov process
on the canonical filtered probability space of càdlàg $\cDd$-valued processes. 

\vip

For any $x\in \cDd$ it holds that $\QQ_x[\int_0^\infty \indiq_{\{X^*_t\in\pDd\}}\dr t]=0$. 
Indeed, recalling Definition~\ref{dfr2}, it suffices to treat the case where $x \in \pDd$. 
With the notation of the proof of Proposition~\ref{markoff}, it suffices that
$\int_0^\infty \indiq_{\{R_t^x\in\pDd\}}\dr t=0$ a.s. Since $\lim_{u\to \infty} \tau_u=\infty$ a.s. 
by Lemma~\ref{tti}, it is enough that $\int_0^{\tau_u} \indiq_{\{R_t^x\in\pDd\}}\dr t=0$ a.s. for all $u>0$.
Recalling~\eqref{zdur}, we thus need that $\int_0^{\tau_u} \indiq_{\{t \in \mathcal{Z}\}}\dr t=0$ for all $u>0$,
{\it i.e.} that $\int_0^{\tau_u} \indiq_{\{t \in \mathcal{Z}^c\}}\dr t=\tau_u$. This follows from the facts that
$\mathcal{Z}^c\cap [0,\tau_u)=\cup_{v \in \JJ,v\leq u}(\tau_{v-},\tau_v)$, see~\eqref{zdur}, and 
that $\tau_u=\sum_{v\in \JJ,v\leq u} \Delta \tau_v$.

\vip

We finally check  
that this Markov process is Feller, {\it i.e.} that for any $t>0$, any $\varphi \in C_b(\cDd)$,
we have $\QQ_{x_n}[\varphi(X^*_t)]\to \QQ_x[\varphi(X^*_t)]$ if $x_n \to x$,
where $X^*$ is the canonical process on $\Omega^*=\DD(\R_+,\cDd)$.
\vip

Fix $t> 0$, $\varphi \in C_b(\cDd)$ and $x_n \in \cDd$ such that $x_n\to x \in \cDd$.
Let us introduce $\psi(r,z)=\QQ_z[\varphi(X^*_{r})]$ for $r\geq 0$ and $z\in \pDd$.
We know from Theorem~\ref{mr1}-(c) that the map $(r,z)\mapsto \psi(r,z)$ is continuous on $\R_+\times \pDd$.
By Definition~\ref{dfr2}, recalling that $Z$ is, under $\PP_0$, an ISP$_{\alpha,0}$,
\begin{gather*}
\QQ_x[\varphi(X_t^*)]= \E_0[\varphi(x+Z_t)\indiq_{\{\sigma >t\}}]+\E_0[\psi(t-\sigma,\Lambda(x+Z_{\sigma-},x+Z_\sigma))
\indiq_{\{\sigma \leq t\}}], \\
\QQ_{x_n}[\varphi(X_t^*)]= \E_0[\varphi(x_n+Z_t)\indiq_{\{\sigma_n >t\}}]+
\E_0[\psi(t-\sigma_n,\Lambda(x_n+Z_{\sigma_n-},x_n+Z_{\sigma_n}))\indiq_{\{\sigma_n \leq t\}}],
\end{gather*}
where $\sigma=\tell(x+Z)$ ({\it i.e.} $\sigma=\inf\{t>0 : x+Z_t \notin \Dd\}$) 
and $\sigma_n=\tell(x_n+Z)$. Note here that  $\mathbb{P}_0(\sigma = 0) = 1$ when $x \in \pDd$ (see Lemma~\ref{ai}),
in which case the above identity is trivial.
\vip

{\it Case 1: $x \in \Dd$.} By~\eqref{eer},
we a.s. have 
$$
\inf_{t\in [0,\sigma)}d(x+Z_t,\Dd^c)>0 \quad\text{and}\quad x+Z_\sigma \in \cDd^c.
$$
Thus a.s., for all $n$ large enough, $\inf_{t\in [0,\sigma)}d(x_n+Z_t, \Dd^c)>0$ and 
$x_n+Z_\sigma \in \cDd^c$, implying that $\sigma_n=\sigma$ and thus that
$\Lambda(x_n+Z_{\sigma_n-},x_n+Z_{\sigma_n})\to \Lambda(x+Z_{\sigma-},x+Z_\sigma)$ by Lemma~\ref{Lambdacon}.
It follows by dominated
convergence that  $\QQ_{x_n}[\varphi(X^*_t)]\to \QQ_x[\varphi(X^*_t)]$.

\vip

{\it Case 2: $x \in \pDd$.} Then $\mathbb{P}_0(\sigma=0) = 1$, so that $\QQ_x[\varphi(X^*_t)]=\psi(t,x)$.
By~\eqref{agarder} and a scaling argument, for any $\e>0$, there is $c_\e>0$ such that
$\PP(\sigma_n>\e)\leq c_\e d(x_n,\pDd)^{\alpha/2}$, so that $\sigma_n\to 0$ in probability.
This implies that in probability, $\indiq_{\{\sigma_n>t\}}\to 0$,  $\indiq_{\{\sigma_n\leq t\}}\to 1$,
and $\Lambda(x_n+Z_{\sigma_n-},x_n+Z_{\sigma_n})\to x$ (because 
$\lim_n Z_{\sigma_n-}=\lim_n Z_{\sigma_n}=0$ in probability by right-continuity of $Z$ at time $0$). 
It follows by dominated convergence that  $\QQ_{x_n}[\varphi(X^*_t)]\to \psi(t,x)$.
\end{proof}

Let us check Proposition~\ref{exit} about the behavior of the process at the boundary.

\begin{proof}[Proof of Proposition~\ref{exit}]
Recall Definitions~\ref{dfr1} and~\ref{dfr2}.
By the strong Markov property, we may assume that $R_0=x\in \pDd$. 
Also recall that $\cZ=\closure{\{t \geq0, \: R_t \in \pDd\}}$ and that $\cZ^c = \cup_{u\in\JJ}(\tau_{u-}, \tau_u)$,
see~\eqref{zdur}. 
It remains to check that for any $u\in\JJ$, such that $\Delta\tau_u > 0$, we have 
$R_{\tau_{u-}} = R_{(\tau_{u-})-}$ if $\beta = *$ and $R_{\tau_{u-}} \neq R_{(\tau_{u-})-}$ if $\beta\in(0,\alpha/2)$.

\vip
Let thus $u\in \JJ$ such that $\Delta \tau_u>0$. By Remark~\ref{rkr2}, we have $u = L_{\tau_{u}} = L_{\tau_{u-}}$. 
From the definition~\eqref{defR} of $R$, with $t=\tau_{u-}$, which satisfies $\tau_{L_t}=\tau_u>\tau_{u-}=t$,
we get
\[
R_{\tau_{u-}} = h_{b_{u-}}(a_u,e_u(t-\tau_{u-}))=b_{u -} + a_{u}e_{u}(0)=R_{(\tau_{u-})-} + a_{u}e_{u}(0).
\]
We finally used that $R_{(\tau_{u-})-} = b_{u-}$, as seen in the proof of Theorem~\ref{mr1} (see Step 1).
This completes the proof since $e_u(0)=0$ when $\beta=*$, while $e_u(0) \in \HH$ when $\beta\in(0,\alpha/2)$.
\end{proof}

Finally, we deal with the scaling property of our processes.

\begin{proof}[Proof of Proposition~\ref{scal}] We divide the proof in two parts.

\vip
\textit{Step 1.} Let us first assume that $R_0 = x \in \pDd$. Consider the Poisson measure 
$\Pi_\beta$ and the processes $(a_u, b_u, \tau_u)_{u\geq0}$ from Definition~\ref{dfr1}. We set $\gamma = 1/2$ 
when $\beta = *$ and $\gamma = \beta / \alpha$ when $\beta \in (0,\alpha /2)$. For $\lambda >0$, we introduce 
the map $\Phi_\lambda : \cE \to \cE$ defined by $\Phi_\lambda(e)(t)=\lambda^{1/\alpha}e(t/\lambda)$.  
We set $\Dd^\lambda = \lambda^{1/\alpha} \Dd$ and observe that for $y \in  \pDd$, the normal vector 
at $\lambda^{1/\alpha}y \in \pDd^\lambda$ is $\bn_y$, so that, with obvious notations, 
$\cI_y=\cI^\lambda_{\lambda^{1/\alpha}y}$.
For any $y \in \pDd^\lambda$, any $A\in\cI_y^\lambda$ and any $e\in\cE$, 
we will denote by $\bar{\ell}_y^\lambda(A, e)$ and $g_y^\lambda(A, e)$ the corresponding quantities associated 
with the domain $\Dd^\lambda$. One can check from~\eqref{g1} and the definition of $\bar{\ell}$ that 
for any $y\in\pDd$, any $A\in\cI_y$ and any $e\in\cE$,
\begin{equation}\label{eq:scaling_ell_g}
\lambda\bar{\ell}_y(A, e) = \bar{\ell}_{\lambda^{1/\alpha}y}^\lambda(A, \Phi_\lambda(e)) 
\quad \text{and} \quad\lambda^{1/\alpha}g_y(A, e) = g_{\lambda^{1/\alpha}y}^\lambda(A, \Phi_\lambda(e)).
\end{equation}
We now introduce the Poisson measure $\Pi_\beta^\lambda$ defined by
$$
\Pi_\beta^\lambda = \sum_{s\in \mathrm{J}_\lambda}\delta_{(s, e_s^\lambda)} \quad \text{where} \quad \mathrm{J}_\lambda 
= \{s\geq0, \: s / \lambda^\gamma \in \mathrm{J}\} \quad \text{and} \quad e_s^\lambda 
= \Phi_\lambda(e_{s/\lambda^\gamma}).
$$
Using the scaling property of $\nn_\beta$, see Lemma~\ref{scaling}, one can verify that $\Pi_\beta^\lambda$ 
is a Poisson measure distributed 
as $\Pi_\beta$.
Recall now that $(b_u)_{u\geq0}$ is a solution to~\eqref{sdeb}, so that, setting 
$(b_u^\lambda)_{u\geq0} = (\lambda^{1/\alpha}b_{u / \lambda^\gamma})_{u\geq0}$ and 
$(a_u^\lambda)_{u\geq0} = (a_{u/\lambda^\gamma})_{u\ge0}$, we have for any $u\geq0$,
\begin{align*}
b_u^\lambda & = \lambda^{1/\alpha}x + \lambda^{1/\alpha}\int_0^{u / \lambda^{\gamma}}
\int_{\cE}\Big(g_{b_\vm}(a_{v}, e) - b_\vm\Big)\Pi_\beta (\dr v, \dr e) \\
& = \lambda^{1/\alpha}x +\int_0^{u}\int_{\cE}\Big(g_{b_\vm^\lambda}^\lambda(a_v^\lambda, e) - b_\vm^\lambda\Big)
\Pi_\beta^\lambda (\dr v, \dr e)
\end{align*}
where in the second equality we used~\eqref{eq:scaling_ell_g} and the definition of $\Pi_\beta^\lambda$. 
Thus $(b_u^\lambda)_{u\geq0}$ is a solution to~\eqref{sdeb} valued in $\pDd^\lambda$ and 
started at $\lambda^{1/\alpha}x$. Similarly,
$(\tau_u^\lambda)_{u\geq0} = (\lambda\tau_{u / \lambda^\gamma})_{u\geq0}$ satisfies
$$
\tau_u^\lambda = \int_0^u\int_{\cE}\bar{\ell}_{b_\vm^\lambda}^\lambda(a_v^{\lambda}, e)\Pi_\beta^\lambda(\dr v, \dr e).
$$
Then, if $(L_t)_{t\geq0}$ denotes the right-continuous inverse of $(\tau_u)_{u\geq0}$, one can see 
that $(L_t^\lambda)_{t\geq0} := (\lambda^{\gamma}L_{t/\lambda})_{t\geq0}$ is the right-continuous inverse 
of $(\tau_u^\lambda)_{u\geq0}$. Finally, by~\eqref{ptd} (which follows from
Remark~\ref{rkr}), we have
\begin{align*}
\lambda^{1/\alpha}R_{t/\lambda} = &
\lambda^{1/\alpha}h_{b_{L_{t/\lambda} -}}(a_{L_{t/\lambda}},e_{L_{t/\lambda}}(t/\lambda-\tau_{L_{t/\lambda} -})) 
\indiq_{\{\tau_{L_{t/\lambda}} > t/\lambda\}}+\lambda^{1/\alpha}b_{L_{t/\lambda}} \indiq_{\{\tau_{L_{t/\lambda}} = t/\lambda\}}\\
=&h_{b_{L_t^\lambda -}^\lambda}(a_{L_t^\lambda}^\lambda,e_{L_t^\lambda}^\lambda(t-\tau_{L_t^\lambda -}^\lambda)) 
\indiq_{\{  \tau_{L_t^\lambda}^\lambda > t\}}+b_{L_t^\lambda}^\lambda\indiq_{\{  \tau_{L_t^\lambda}^\lambda =t\}}.
\end{align*}
As a consequence, the process $(\lambda^{1/\alpha} R_{t/\lambda})_{t\geq 0}$ an $(\alpha,\beta)$-stable 
process reflected in $\lambda^{1/\alpha}\cDd=\{\lambda^{1/\alpha}y : y \in \cDd\}$ issued 
from $\lambda^{1/\alpha}x\in\pDd^\lambda$.

\vip
\textit{Step 2.} Consider now an $(\alpha,\beta)$-stable process $(R_t)_{t\geq 0}$ reflected in $\cDd$
issued from $x \in \Dd$, built as in Definition~\ref{dfr2}: for some ISP$_{\alpha,x}$ $(Z_t)_{t\geq 0}$, 
for $\tell(Z)=\inf\{t>0 : Z_t \notin \Dd\}$, set $R_{t}=Z_t$ for $t\in [0,\tell(Z))$, set
$Y=\Lambda(Z_{\tell(Z)-},Z_{\tell(Z)})$, pick some $\QQ_{Y}$-distributed process $(S_t)_{t\geq 0}$,
and set $R_t=S_{t-\tell(Z)}$ for $t\geq \tell(Z)$.
\vip
We introduce the function $\Lambda_\lambda$ as in~\eqref{Lambda} associated with $\Dd^\lambda$ and,
for $z \in \pDd^\lambda$, we call $\QQ^\lambda_z$ the law of the $(\alpha,\beta)$-stable process 
reflected in $\cDd^\lambda$.
\vip

We introduce 
$(Z^\lambda_t=\lambda^{1/\alpha}Z_{t/\lambda})_{t\geq 0}$, which is an ISP$_{\alpha,\lambda^{1/\alpha}x}$,
and observe that we have $\tell_\lambda(Z^\lambda)=\inf\{t>0 : Z_t^\lambda \notin \Dd^\lambda\}=\lambda \tell(Z)$,
that $\lambda^{1/\alpha}Y=\Lambda_\lambda(Z^\lambda_{\tell_\lambda(Z^\lambda)-},Z^\lambda_{\tell_\lambda(Z^\lambda) })$, that 
$(S^\lambda_t=\lambda^{1/\alpha}S_{t/\lambda})_{t\geq 0}$ is $\QQ^\lambda_{\lambda^{1/\alpha}Y}$-distributed by Step 1 
and that $\lambda^{1/\alpha}S_{(t-\sigma)/\lambda}=S^\lambda_{t-\sigma_\lambda}$. 
All this shows that $(R^\lambda_t=\lambda^{1/\alpha}R_{t/\lambda})_{t\geq 0}$
is indeed an $(\alpha,\beta)$-stable process reflected in $\cDd^\lambda$ and
issued from $\lambda^{1/\alpha}x \in \Dd$.
\end{proof}

\section{Infinitesimal generator and P.D.E.s}\label{sec:pde}

The goal of this section is to prove Theorem~\ref{ig} and Proposition~\ref{pde}.
We recall that $\QQ_x$ was introduced in Definitions~\ref{dfr1} and~\ref{dfr2} 
and that $(X^*_t)_{t\geq 0}$ is the canonical process
on $\Omega^*=\DD(\R_+,\cDd)$. Recall also that the sets $D_\alpha$ and $H_\beta$ were introduced in 
Definitions~\ref{opbd} and~\ref{test}.
The main difficulty of the section is to establish the following result.

\begin{proposition}\label{lg2}
Fix $\beta \in \{*\}\cup (0,\alpha/2)$ and suppose Assumption~\ref{as}.
Let $\varphi \in D_\alpha \cap H_\beta$.
For all $x\in \cDd$, all $t\geq 0$, we have
\begin{equation}\label{anres}
\QQ_x[\varphi(X^*_{t})]=\varphi(x)+\QQ_x\Big[\int_0^{t} \cL\varphi(X^*_s)\dr s\Big].
\end{equation}
\end{proposition}

By Theorem~\ref{mr2}, $\QQ_x[\int_0^\infty \indiq_{\{X_s^*\in \pDd\}}\dr s]=0$,
so that everything makes sense in the above identity, 
even if $\cL \varphi$ is not defined on $\pDd$. Let us admit Proposition~\ref{lg2} 
for a moment and handle the
proofs of the results announced in Subsection~\ref{ssig}.

\begin{proof}[Proof of Theorem~\ref{ig}]
For any $\varphi \in D_\alpha \cap H_\beta$, Proposition~\ref{lg2}  
tells us that for all $x \in \Dd$, all $t\geq 0$,
$$
\psi(t,x):=\frac{\QQ_x[\varphi(X_t^*)-\varphi(x)]}t=\QQ_x\Big[\frac1t\int_0^t \cL\varphi(X_s^*)\dr s\Big].
$$
Since $\cL \varphi \in C(\Dd)\cap L^\infty(\Dd)$ by definition of $D_\alpha$ and since
$X^*$ is right continuous, we immediately conclude that $\psi(t,x)$ converges bounded pointwise 
on $\Dd$ to $\cL\varphi(x)$
as $t\to 0$.
\end{proof}

\begin{proof}[Proof of Proposition~\ref{pde}] Recall that $f(t,x,\dr y)=\QQ_x(X_t^* \in \dr y)$ for all
$x\in \cDd$, all $t\geq 0$. By Theorem~\ref{mr2}, we know that for all $x \in \cDd$,
$\int_0^\infty f(t,x,\pDd) \dr t =0$ and that for all $t\geq 0$, the map $x\mapsto f(t,x,\dr y)$
is weakly continuous.
For $\varphi \in D_\alpha \cap H_\beta$, Proposition~\ref{lg2}  
tells us that for all $x \in \cDd$, all $t\geq 0$,
$\QQ_x[\varphi(X_t^*)]=\varphi(x)+\int_0^t  \QQ_x[\cL\varphi(X_s^*)]\dr s$,
which precisely gives us~\eqref{wpde}.
\end{proof}

To prove Proposition~\ref{lg2}, we first study what happens until the process reaches the boundary.

\begin{lemma}\label{lg1}
Fix $\beta \in \{*\}\cup(0,\alpha/2)$ and grant Assumption~\ref{as}.
Let $\varphi \in D_\alpha$. For all $x \in \cDd$, all $t\geq 0$, 
setting $d_0^* =\inf\{t>0 : X_t^* \in \pDd\}$,
\[
\QQ_x\big[\varphi(X^*_{t\land d_0^*})\big]= \varphi(x) 
+ \QQ_x\bigg[\int_0^{t\land d_0^*} \cL\varphi(X^*_s)\dr s\bigg].
\]
\end{lemma}

\begin{proof} 
The case $x\in\pDd$ is obvious, since then $\QQ_x(d_0^*=0)=1$, see Step~2.1 of the proof of 
Proposition~\ref{markoff}.
Fix $x \in \Dd$ and recall Definition~\ref{dfr2}:
$$
\QQ_x\big[\varphi(X^*_{t\land d_0^*})\big] = \mathbb{E}_x\big[\varphi(R_{t\land \tell(Z)})\big]
\quad \text{and} \quad  
\QQ_x\Big[\int_0^{t\land d_0^*} \!\! \cL\varphi(X^*_s)\dr s\Big]=\E_x\Big[\int_0^{t\land \tell(Z)}\!\!
\cL\varphi(R_s)\dr s\Big],
$$ 
where $\mathbb{P}_x$ is the law of an ISP$_{\alpha,x}$ $(Z_t)_{t\geq0}$, where 
$\tell(Z) = \inf \{t> 0 : Z_t \notin\Dd\}$ and where $R_t = Z_t$ for $t\in [0,\tell(Z))$ 
and $R_{\tell(Z)} = \Lambda(Z_{\tell(Z)-}, Z_{\tell(Z)})$. Recall also that $Z$ is defined by~\eqref{eqs},
through a Poisson measure $N$ on $\R_+\times\R^d\setminus\{0\}$ with intensity measure 
$\dr s |z|^{-\alpha-d} \dr z$. We claim that for all $t\in [0,\tell(Z)]$,
$$
R_{t} = x + \int_0^{t}\int_{\{|z|\leq 1\}}\big[\bar{\Lambda}(R_\sm, z) - R_\sm\big]\tilde{N}(\dr s, \dr z) 
+ \int_0^{t}\int_{\{|z|> 1\}}\big[\bar{\Lambda}(R_\sm, z) - R_\sm\big]N(\dr s, \dr z),
$$
where $\bar\Lambda(r,z)=\Lambda(r,r+z)$. Indeed, let $U_t$ be the RHS of this expression.
We have $\bar{\Lambda}(R_\sm, z)=\bar{\Lambda}(Z_\sm, z)=Z_\sm+z=R_{\sm}+z$ for $N$-a.e. 
$(s,z) \in [0,\tell(Z))\times \R^d$ (recall that $\bar\Lambda(x,z)=x+z$ if $x+z \in \Dd$).
Thus for $t\in [0,\tell(Z))$,
$$
U_t=x + \int_0^t \int_{\{|z|\leq 1\}} z\tilde{N}(\dr s, \dr z) 
+ \int_0^{t}\int_{\{|z|> 1\}} z N(\dr s, \dr z)=Z_t=R_t.
$$
Next,
$$
U_{\tell(Z)}=U_{\tell(Z)-}+\bar{\Lambda}(R_{\tell(Z)-}, \Delta Z_{\tell(Z)}) - R_{\tell(Z)-}=
\bar{\Lambda}(Z_{\tell(Z)-}, \Delta Z_{\tell(Z)})=\Lambda(Z_{\tell(Z)-},Z_{\tell(Z)})=R_{\tell(Z)}.
$$
We now introduce $\tell_\e(Z)=\inf\{t\geq 0 : d(Z_t,\Dd^c)\leq \e\}\leq \tell(Z)$. We have
\begin{align*}
R_{t\land\tell_\e(Z)} =& x + \int_0^{t\land\tell_\e(Z)}\int_{\{|z|\leq 1\}}\big[\bar{\Lambda}(R_\sm, z) - R_\sm\big]
\tilde{N}(\dr s, \dr z) \\
&+ \int_0^{t\land\tell_\e(Z)}\int_{\{|z|> 1\}}\big[\bar{\Lambda}(R_\sm, z) - R_\sm\big]N(\dr s, \dr z).
\end{align*}
We can now apply the Itô formula with $\varphi\in C^2(\Dd)\cap C(\cDd)$ and we get
\begin{align*}
\varphi(R_{t\land\tell_\e(Z)}) = & \varphi(x) + \int_0^{t\land\tell_\e(Z)}\int_{\{|z|\leq 1\}}
\big[\varphi(\bar{\Lambda}(R_\sm, z)) - \varphi(R_\sm)\big]
\tilde{N}(\dr s, \dr z) \\
& \hskip-25pt+ \int_0^{t\land\tell_\e(Z)} \int_{\{|z|\leq 1\}}\big[\varphi(\bar{\Lambda}(R_\sm, z)) - \varphi(R_\sm) 
- \nabla \varphi (R_\sm) \cdot (\bar{\Lambda}(R_\sm, z) - R_\sm)\big]\frac{\dr z}{|z|^{d + \alpha}} \dr s \\
& \hskip-25pt+ \int_0^{t\land\tell_\e(Z)}\int_{\{|z|> 1\}}\big[\varphi(\bar{\Lambda}(R_\sm, z)) - \varphi(R_\sm)\big]
N(\dr s, \dr z) \\
=  & \varphi(x) + M_{t}^{\e} + \int_0^{t\land\tell_\e(Z)}\!\!\!\int_{\mathbb{R}^d} \!\!
\big[\varphi(\bar{\Lambda}(R_\sm, z))-\varphi(R_\sm) - \nabla \varphi(R_\sm) \cdot z \indiq_{\{|z|\leq 1\}}\big]
\frac{\dr z}{|z|^{d+\alpha}}\dr s
\end{align*}
where
\[
 M_t^{\e} = \int_0^{t\land\tell_\e(Z)}\int_{\mathbb{R}^d}\big[\varphi(\bar \Lambda(R_\sm, z)) - \varphi(R_\sm)\big]
\tilde{N}(\dr s, \dr z).
\]
Since $\varphi \in C(\cDd)\cap C^2(\Dd)$ and since $d(r,\Dd^c)\geq\e$ and $|z|<\e/2$ imply that
$\bar \Lambda(r,z)=r+z$,
$$
\sup_ {r : d(r,\Dd^c)\geq \e}|\varphi(\bar\Lambda(r,z))-\varphi(r)|\leq 
2||\varphi||_\infty \indiq_{\{|z|\geq \e/2\}} + \Big(\sup_{r : d(r,\Dd^c)\geq \e/2} |\nabla \varphi(r)| \Big)
|z| \indiq_{\{|z|<\e/2\}},
$$
which belongs to $L^2(\R^d,|z|^{-\alpha-d}\dr z)$. Thus $(M_t^\e)_{t\geq0}$ is a true martingale, and we 
conclude, recalling Definition~\ref{opbd}, that 
\begin{align*}
\E_x\big[\varphi(R_{t\land \tell_\e(Z)})\big]=&\varphi(x)+\E_x\Big[\int_0^{t\land \tell_\e(Z)} \cL\varphi(R_s)  
\dr s \Big].
\end{align*}
It then suffices to let $\e\to 0$, using that $\lim_{\e\to 0}\PP_x(\tell_\e(Z)=\tell(Z))=1$ by~\eqref{eer} 
(since $x \in \Dd$), that $\varphi \in C(\cDd)$ and that $\cL\varphi$ is bounded.
\end{proof}

We next show that our reflected process can be approximated by a similar process with a {\it finite} excursion
measure. This is a little fastidious, but it seems unavoidable 
to prove Proposition~\ref{lg2}  without heavy restrictions on the
test functions. Such restrictions would make the conclusion of 
Remark~\ref{pdi} incorrect, which would be an important issue.

\begin{definition}\label{dfd}
Grant Assumption~\ref{as}.
Let $m >0$ and let $j$ be a finite measure on $\HH$  left invariant by any isometry of $\HH$ sending 
$\be_1$ to $\be_1$. Consider the finite measure $\nn_j$ on $\cE$ defined by
\begin{equation}\label{nnj}
\nn_j(B) = \int_{x \in \HH} \PP_x\big((Z_{t\land \ell(Z)})_{t\geq 0} \in B \big)j(\dr x)\quad
\text{for all Borel subset $B$ of $\cE$.}
\end{equation}
For $x \in \pDd$, we say that $(R_t)_{t\geq 0}$ is an $(\alpha,m, j)$-stable process reflected in $\cDd$ issued 
from $x$ if there exists a filtration $(\cG_u)_{u\geq 0}$, a $(\cG_u)_{u\geq 0}$-Poisson measure 
$\Pi_j=\sum_{u\in \JJ}\delta_{(u,e_u)}$ on $\R_+\times\cE$ with intensity measure $\dr u \nn_{j}(\dr e)$,
a càdlàg $(\cG_u)_{u\geq 0}$-adapted $\pDd$-valued process $(b_u)_{u\geq 0}$ and a $(\cG_u)_{u\geq 0}$-predictable 
process $(a_u)_{u \geq 0}$ such that a.s., for all $u \geq 0$, $a_u \in \cI_{b_{u-}}$ and 
\begin{equation}\label{sdebd}
b_u = x + \int_0^u\int_{\mathcal{E}}\Big(g_{b_\vm}(a_{v}, e) - b_\vm\Big)\Pi_j (\dr v, \dr e)
\end{equation}
and such that, introducing the càdlàg increasing $\R_+$-valued $(\cG_u)_{u\geq 0}$-adapted process 
\begin{equation}\label{sdetd}
\tau_u = mu + \int_0^u\int_{\mathcal{E}}\cl_{b_\vm}(a_v, e)\Pi_j(\dr v, \dr e)
\end{equation}
and its generalized inverse $L_t=\inf\{u\geq 0 : \tau_u>t\}$,
$R_t$ is given by~\eqref{defR} for all $t\geq 0$.
We then call $\QQ^{m,j}_x$ the law of $(R_t)_{t\geq 0}$.
\end{definition}

Recall that for any $z \in \pDd$, the measurable family $(A^z_y)_{y \in \pDd}$ (or $(A_y)_{y \in \pDd}$ if $d=2$) 
such that $A_y^z \in \cI_y$ was introduced in Lemma~\ref{locglob}.
For any $z \in \pDd$ fixed,
the strong existence and uniqueness of an $(\alpha,m,j)$-stable process reflected in $\cDd$ 
such that $a_u=A^z_{b_{u-}}$ 
(or $a_u=A_{b_{u-}}$ if $d=2$) is obvious, since the Poisson measure $\Pi_j$ is finite.
The uniqueness in law (with any choice for $(a_u)_{u\geq 0}$) of the $(\alpha,m,j)$-stable process 
reflected in $\cDd$ can be checked exactly as in Step 2 of the proof of Theorem~\ref{mr1},
making use of a result similar to Lemma~\ref{ttaann} for $\Pi_j$,
which is licit since $j$ is invariant by any isometry of $\HH$ sending 
$\be_1$ to $\be_1$. 

\vip

We are now ready to specify the approximating processes. Recall that $r>0$ and $\ell_r$ were 
defined in Remark~\ref{imp4}. 
For $n \geq 1$, we introduce (recall the definition~\eqref{nnb} of $\nn_\beta$)
\begin{gather*}
j_\beta^n(\dr x) = \nn_\beta(e(0) \in \dr x, |e(0)| > 1/n) = \indiq_{\{|x| > 1/n\}}\frac{\dr x}{|x|^{1+\beta}}
\quad \text{if $ \beta \in (0,\alpha/2)$}, \\
j_*^n(\dr x) = \nn_*(e(1/n) \in \dr x, \: \ell_r(e) > 1/n).
\end{gather*}
These two measures on $\HH$ are finite and invariant by any isometry sending $\be_1$ to $\be_1$. 

\begin{lemma}\label{dtc}
Let $\beta \in \{*\}\cup(0,\alpha /2)$ and suppose Assumption~\ref{as}.  For $n\geq 1$, set $m_n = 1/n$. 
For any $x\in\pDd$, in the sense of finite-dimensional distributions,
\[
 \lim_n \QQ_x^{m_n, j_\beta^n} = \QQ_x.
\]
\end{lemma}

\begin{proof} As in the proof of Theorem~\ref{mr1}, we only treat the case $d\geq 3$, the case $d=2$ being easier,
as we can use Lemma~\ref{locglob}-(a) instead of Lemma~\ref{locglob}-(b).
We fix  $\beta \in\{*\}\cup (0,\alpha/2)$ and 
$x \in \pDd$ and consider $z \in \pDd$ such that Proposition~\ref{newb}-(b) applies. 
\vip
Consider, on some probability space with a filtration $(\mathcal{G}_u)_{u\geq0}$,
a $(\mathcal{G}_u)_{u\geq0}$-Poisson measure $\Pi_\beta = \sum_{u\in \mathrm{J}}\delta_{(u, e_u)}$ 
on $\mathbb{R}_+ \times \mathcal{E}$ with intensity $\dr u\nn_\beta(\dr e)$, as well as the boundary process
$(b_u)_{u\geq 0}$ built in Proposition~\ref{newb}-(b). We set $(a_u)_{u\geq 0}=(A^z_{b_{u-}})_{u\geq 0}$, 
define
$(\tau_u)_{u\geq 0}$ by~\eqref{sdet}, introduce $(L_t=\inf\{u\geq 0 : \tau_u>t\})_{t\geq 0}$ and define 
$(R_t)_{t\geq 0}$ by~\eqref{defR}. Then, as in Step 1 of the proof of Theorem~\ref{mr1}, 
$(R_t)_{t\geq 0}$ is $\QQ_x$-distributed.

\vip
If $\beta \in (0, \alpha/2)$, 
we set $\Pi_\beta^n = \sum_{u\in \mathrm{J}}\delta_{(u, e_u)} \indiq_{\{|e_u(0)| > 1/n\}}$ and we see that $\Pi_\beta^n$ is a 
Poisson  measure with intensity $\dr u \nn_{j_\beta^n}(\dr e)$. 

\vip

If $\beta = *$, then for any $u\in\JJ$ such that 
$\ell_r(e_u) > 1/n$, we define $e_u^n = (e_u(1/n + s))_{s\geq 0} =: \theta_{1/n}(e_u)$ and we set 
$\Pi_\beta^n = \sum_{u\in \mathrm{J}}\delta_{(u, e_u^n)} \indiq_{\{\ell_r(e_u) > 1/n\}}$. Let us observe at once that
\begin{equation}\label{tbk}
\text{for all $b \in \pDd$, all $A\in \cI_b$,}
\quad\cl_b(A,e^n_u)=\cl_b(A,e_u)-\frac1n \quad \text{and} \quad g_b(A,e^n_u)=g_b(A,e_u),
\end{equation}
because $\cl_b(a,e)\geq \ell_r(e)$, see Remark~\ref{imp4}.
For any measurable 
$F:\mathbb{R}_+ \times \cE \to \mathbb{R}_+$,
\begin{align*}
\mathbb{E}\bigg[\int_0^\infty\int_{\cE} F(u, e)\Pi_\beta^n(\dr u, \dr e)\bigg]
=& \int_0^\infty\int_{\cE} F(u, \theta_{1/n}(e))\indiq_{\{\ell_r(e) > 1/n\}}\nn_*(\dr e) \dr u \\
=& \int_0^\infty\int_{\cE} \E_{e(1/n)}[F(u,(Z_{t\land \tell(Z)})_{t\geq 0})] \indiq_{\{\ell_r(e) > 1/n\}}\nn_*(\dr e) \dr u,
\end{align*}
where we used the Markov property at time $1/n$ of the measure $\nn_*$, see Lemma~\ref{mark}.
This is possible since $\ell(e) \geq \ell_r(e)$. Recalling~\eqref{nnj} and the definition of $j^n_*$,
this gives
\begin{align*}
\mathbb{E}\bigg[\int_0^\infty\int_{\cE} F(u, e)\Pi_\beta^n(\dr u, \dr e)\bigg]
=& \int_0^\infty\int_{\cE}F(u, e)\nn_{j_*^n}(\dr e) \dr u,
\end{align*}
which  shows that $\Pi_\beta^n$ has intensity $\dr u \nn_{j_*^n}(\dr e)$.
\vip
We then introduce, for each $n\geq 1$, the $(\alpha,m_n,j^n_\beta)$-process $(R^n_t)_{t\geq 0}$
built as in Definition~\ref{dfd} with the Poisson measure $\Pi_\beta^n$ (and with the choice
$a^n_u=A^z_{b^n_{u-}}$,
with the value of $z$ introduced above to build $(R_t)_{t\geq 0}$), 
and denote by $(b^n_u)_{u\geq 0}$, $(a^n_u)_{u\geq 0}$, $(\tau^n_u)_{u\geq 0}$,
$(L^n_t)_{t\geq 0}$ the processes involved in its construction.
\vip
We now show, following  line by line Steps 3.2-3.7 of the proof of Theorem~\ref{mr1}
(with $t_n=t$ and $x_n=x$), that 
for each $t\geq 0$, $R^n_t$ converges in probability to $R_t$. This will complete
the proof.
\vip
{\it Step 3.2.} For any $T>0$, $\sup_{[0,T)}(|b^n_u-b_u|+|a^n_u-a_u|)\to 0$ in probability.
\vip
We introduce $\rho^k=\inf\{u\geq 0 : |b_u-z|\leq 1/k\}$ and $\rho^k_n=\inf\{u\geq 0 : |b_u^n-z|\leq 1/k\}$
and show that $\lim_n\E[\sup_{[0,T\land \rho^k\land \rho^k_n)}|b^n_u-b_u|]=0$ 
for each $k\geq1$, the rest of the proof is similar.

\vip

If $\beta \in (0,\alpha/2)$, we have
\begin{gather*}
b_u^n=x+\int_0^u \int_\cE [g_{b_\vm^n}(A^z_{b_\vm^n},e)-b_\vm^n]\indiq_{\{|e(0)|>1/n\}} \Pi_\beta(\dr v,\dr e),\\
b_u=x+\int_0^u \int_\cE [g_{b_\vm}(A^z_{b_\vm},e)-b_\vm]\Pi_\beta(\dr v,\dr e).
\end{gather*}
Using~\eqref{tbru} and that $|g_b(A,e)-b|\leq M(e)\land D$ (where $D=$diam$(\Dd)$) 
for all $b\in\pDd$, all $A\in \cI_b$, all $e\in \cE$, 
we get, for all $u\in [0,T]$,
$$
\E\Big[\sup_{[0,u\land \rho^k\land \rho^k_n)}|b^n_u-b_u|\Big]
\leq C_k \int_0^u \E\Big[|b^n_{v\land \rho^k\land \rho^k_n}-b_{v\land \rho^k\land \rho^k_n}|\Big] \dr v
+ T \int_\cE (M(e)\land D) \indiq_{\{|e(0)|\leq 1/n\}}\nn_\beta(\dr e).
$$
Thanks to the Gronwall lemma, we conclude that
$$
\E\Big[\sup_{[0,T\land \rho^k\land \rho^k_n)}|b^n_u-b_u|\Big] \leq T e^{C_kT}\int_\cE (M(e)\land D) \indiq_{\{|e(0)|\leq 1/n\}}
\nn_\beta(\dr e),
$$
which tends to $0$ as $n\to \infty$ by~\eqref{se1} and since $|e(0)|>0$ for $\nn_\beta$-a.e. $e\in\cE$.
\vip
If $\beta=*$, we have, thanks to~\eqref{tbk},
\begin{equation*}
b_u^n=x+\int_0^u \int_\cE [g_{b_\vm^n}(A^z_{b_\vm^n},e)-b_\vm^n]\indiq_{\{\ell_r(e) > 1/n\}} \Pi_\beta(\dr v,\dr e).
\end{equation*}
We get as previously that
$$
\E\Big[\sup_{[0,T\land \rho^k\land \rho^k_n)}|b^n_u-b_u|\Big] \leq T e^{c_KT}\int_\cE (M(e)\land D) \indiq_{\{\ell_r(e)\leq 1/n\}}
\nn_\beta(\dr e),
$$
which tends to $0$ as $n\to \infty$ by~\eqref{se1} and since $\ell_r>0$ $\nn_*$-a.e. by Lemma~\ref{imp}.
\vip

{\it Step 3.3.} For any $T>0$, $\sup_{[0,T]}|\tau^n_u-\tau_u|\to 0$ in probability.
\vip
If $\beta \in (0,\alpha/2)$, we have
\begin{gather*}
\tau_u^n=\frac un+\int_0^u \int_\cE \cl_{b_\vm^n}(A^z_{b_\vm^n},e)\indiq_{\{|e(0)|>1/n\}} \Pi_\beta(\dr v,\dr e)
\quad \text{and}
\quad \tau_u=\int_0^u \int_\cE \cl_{b_\vm}(A^z_{b_\vm},e) \Pi_\beta(\dr v,\dr e).
\end{gather*}
Using the subadditivity of $r\mapsto r\land 1$ on $[0,\infty)$ and that $\cl_b(A,e)\leq \ell(e)$, we write
\begin{align*}
\E\Big[\sup_{[0,T]}|\tau^n_u-\tau_u|\land 1\Big]\leq& \frac Tn+\int_0^T \int_{\cE}\E\Big[|\cl_{b_\vm^n}(A^z_{b_\vm^n},e)
-\cl_{b_\vm}(A^z_{b_\vm},e)|\land 1 \Big] \nn_\beta(\dr e)\dr v\\
&+ T \int_\cE (\ell(e)\land 1) \indiq_{\{|e(0)|\leq1/n\}} \nn_\beta(\dr e).
\end{align*}
We conclude by dominated convergence, exactly as in Step 3.3 of the proof of Theorem~\ref{mr1}.
\vip
If now $\beta = *$, we have, thanks to~\eqref{tbk},
\[
\tau_u^n = \frac u n + \int_0^u\int_{\cE}[\ell_{b_{v-}^n}(A^z_{b_\vm^n},e) - 1/n]\indiq_{\{\ell_r(e)\leq 1/n\}}
\Pi_\beta(\dr v,\dr e),
\]
whence as previously
\begin{align*}
\E\Big[\sup_{[0,T]}|\tau^n_u-\tau_u|\land 1\Big]\leq& \frac Tn+\int_0^T \int_{\cE}\E\Big[\Big|
\Big(\cl_{b_\vm^n}(A^z_{b_\vm^n},e)-\frac1n\Big)_+-\cl_{b_\vm}(A^z_{b_\vm},e)\Big|\land 1\Big]\nn_*(\dr e) \dr v\\
&+ T \int_\cE (\ell(e)\land 1) \indiq_{\{\ell_r(e)\leq1/n\}} \nn_*(\dr e).
\end{align*}
We conclude by dominated convergence, using~\eqref{se1} and that $\ell_r>0$ $\nn_*$-a.e. for the last term.

\vip

Steps 3.4 and 3.5 are exactly the same as in the proof of Theorem~\ref{mr1}, while Steps 3.6 and 3.7
are slightly easier (since $t_n=t$ and $x_n=x$).
\end{proof}

Let us now write down some kind of Itô formula for the approximate process.

\begin{lemma}\label{itod}
Grant Assumption~\ref{as}.
Let $m >0$ and let $j$ be a finite measure on $\HH$, invariant by any isometry of $\HH$ sending 
$\be_1$ to $\be_1$. Fix $x\in \pDd$. Let $(R_t)_{t\geq 0}$ be some $(\alpha,m,j)$-stable process
reflected in $\cDd$ issued from $x$, built as in Definition~\ref{dfd} with some Poisson measure $\Pi_j$.
Consider the associated processes $(b_u)_{u \geq 0}$, $(a_u)_{u\geq 0}$,  $(\tau_u)_{u\geq 0}$, $(L_t)_{t\geq 0}$.
\vip
(a) For all $t\geq 0$, setting $\theta= \int_{\cE}(1-\exp(-\ell_r(e))\nn_j(\dr e)$, we have
$\E[L_t]\leq \theta^{-1} e^t.$
\vip
(b) For all $\varphi \in D_\alpha$ (recall Definition~\ref{opbd}), all $t\geq 0$, we have
$$
\E[\varphi(R_t)]=\varphi(x)+\E\Big[\int_0^t \indiq_{\{R_s \notin \pDd\}}\cL\varphi(R_s)\dr s\Big] + \E\Big[\int_0^{L_t} 
\cK_j\varphi(b_{u},a_{u}) \dr u\Big],
$$
where for $b\in \pDd$ and  $A \in \cI_b$,
\begin{equation}\label{ckj}
\cK_j\varphi(b,A)=\int_{\HH} [\varphi(\Lambda(b,h_{b}(A,z)))-\varphi(b)]j(\dr z) .
\end{equation}

(c) For all  $\varphi \in D_\alpha$, all $t\geq 0$, we have
$$
\Big|\E[\varphi(R_t)]-\varphi(x)-\E\Big[\int_0^t \indiq_{\{R_s \notin \pDd\}}\cL\varphi(R_s)\dr s\Big]\Big|
\leq  \theta^{-1}e^t\sup_{b\in \pDd, A \in \cI_b} |\cK_j\varphi(b,A)|.
$$
\end{lemma}

All the expressions in (b) and (c) make sense, thanks to (a), since $j$ is finite, and since 
$\varphi \in D_\alpha$ implies that $\varphi$ and $\cL\varphi$ are bounded on $\Dd$.

\begin{proof} Point (c) immediately follows from (a) and (b).
\vip

We start with (a). For $u>0$, $\PP(L_t > u)\leq \PP(\tau_u \leq t)=\PP(e^{-\tau_u}\geq e^{-t})
\leq e^t \E[e^{-\tau_u}]$. But recalling~\eqref{sdetd} and that $\cl_b(a,e)\geq \ell_r(e)$, see Remark~\ref{imp4}, 
we have
$$
\frac {\dr}{\dr u} \E[e^{-\tau_u}]\leq-\int_{\cE}\E[e^{-\tau_u}(1-\exp(-\cl_{b_u}(a_u,e)))]\nn_j(\dr e)
\leq -\int_{\cE}\E[e^{-\tau_u}](1-\exp(-\ell_r(e)))\nn_j(\dr e).
$$
Hence $\E[e^{-\tau_u}]\leq \exp(-\theta u)$, so that $\PP(L_t>u)\leq e^te^{-\theta u}$ and thus
$\E[L_t] \leq e^t \int_0^\infty e^{-\theta u}\dr u = \theta^{-1}e^t$.
\vip
For (b), we write, recalling~\eqref{sdebd} and that $\Pi_j=\sum_{u\in \JJ}\delta_{(u,e_u)}$ has a finite 
intensity measure,
\begin{align}\label{tbrw}
\varphi(R_t)=&\varphi(x)+ \sum_{u \in \JJ} \indiq_{\{u<L_t\}} [\varphi(b_{u})-\varphi(b_{u-})]
+\varphi(R_t)-\varphi(b_{L_t-})\notag\\
=&\varphi(x)+\sum_{u \in \JJ} \indiq_{\{u<L_t\}} [\varphi(g_{b_{u-}}(a_u,e_u))-\varphi(b_{u-})]
+\varphi(R_t)-\varphi(b_{L_t-}).
\end{align}
But $\cl_{b_{u-}}(a_u,e_u)\leq t-\tau_{u-}$ for all $u\in \JJ$ such that $u<L_t$ (because $u<L_t$ 
implies that $\tau_u\leq t$, {\it i.e.} $\tau_{u-}+\cl_{b_{u-}}(a_u,e_u)\leq t$).
Moreover, by~\eqref{defR} (see also Remark~\ref{rkr}),
$$
R_t=\indiq_{\{L_t\notin \JJ\}}b_{L_t-}+\indiq_{\{L_t \in \JJ, \tau_{L_t>t}\}}h_{b_{L_t-}}(a_{L_t},e_{L_t}(t-\tau_{L_t-}))
+ \indiq_{\{L_t\in \JJ,\tau_{L_t}=t\}}g_{b_{L_t-}}(a_{L_t},e_{L_t}),
$$
so that, since $\tau_{L_t}=\tau_{L_t-}+\cl_{b_{L_t-}}(a_{L_t},e_{L_t})$ when $L_t \in \JJ$,
\begin{align*}
\varphi(R_t)-\varphi(b_{L_{t}-})=&\indiq_{\{L_t\in \JJ\}} \Big(
 \indiq_{\{\cl_{b_{L_t-}}(a_{L_t},e_{L_t})=t-\tau_{L_t-}\}}\varphi(g_{b_{L_t-}}(a_{L_t},e_{L_t}))\\
&\hskip1.5cm+ \indiq_{\{\cl_{b_{L_t-}}(a_{L_t},e_{L_t})<t-\tau_{L_t-}\}}\varphi(h_{b_{L_t-}}(a_{L_t},e_{L_t}(t-\tau_{L_t-})) )
-\varphi(b_{L_t-})\Big).
\end{align*}
All in all, \eqref{tbrw} rewrites as
\begin{align*}
\varphi(R_t)=\varphi(x)+ \sum_{u \in \JJ} \indiq_{\{u\leq L_t\}}
\Big(&\indiq_{\{ \cl_{b_{u-}}(a_u,e_u)\leq t-\tau_{u-}\}}\varphi(g_{b_{u-}}(a_u,e_u)) \\
&\hskip10pt +\indiq_{\{ \cl_{b_{u-}}(a_u,e_u)> t-\tau_{u-}\}}\varphi(h_{b_{u-}}(a_{u},e_{u}(t-\tau_{u-})))
-\varphi(b_{u-})\Big).
\end{align*}
The compensation formula gives us, since $L_t$ is a $(\cG_u)_{u\geq 0}$-stopping time,
$$
\E[\varphi(R_t)]=\varphi(x)+\E\Big[\int_0^{L_t} \Gamma\varphi(t-\tau_u,b_u,a_u) \dr u \Big],
$$
where for $w\geq 0$, $b \in \pDd$ and $A\in \cI_b$,
$$
\Gamma\varphi(w,b,A)=\int_{\cE} [\indiq_{\{\cl_b(A,e)\leq w\}} \varphi(g_b(A,e))+
\indiq_{\{\cl_b(A,e)> w\}} \varphi(h_b(A,e(w)))-\varphi(b)]\nn_j(\dr e).
$$
Recall the definition~\eqref{nnj} of $\nn_j$, from which we get $\nn_j(e(0)\in \dr z)=j(\dr z)$. 
Recall also from~\eqref{ckj} the definition of $\cK_j$.
We now write $\Gamma\varphi(w,b,A)=\Theta\varphi(w,b,A)+\cK_j\varphi(b,A)$, where
$$
\Theta\varphi(w,b,A)\!=\!
\int_{\cE} \![\indiq_{\{\cl_b(A,e)\leq w\}} \varphi(g_b(A,e))+
\indiq_{\{\cl_b(A,e)> w\}} \varphi(h_b(A,e(w)))-\varphi(\Lambda(b,h_b(A,e(0))))]\nn_j(\dr e).
$$
At this point, we have shown that
$$
\E[\varphi(R_t)]=\varphi(x)+\E\Big[\int_0^{L_t}\Theta\varphi(t-\tau_u,b_u,a_u)\dr u\Big] 
+\E\Big[\int_0^{L_t}\cK_j\varphi(b_u,a_u)\dr u \Big],
$$
and it remains to check that $\E[I_t]=\E[J_t]$, where
\begin{equation*}
I_t:=\int_0^{L_t}\Theta\varphi(t-\tau_u,b_u,a_u)\dr u \quad \text{and}  \quad
J_t:=\int_0^t\cL\varphi(R_s)\indiq_{\{R_s\notin\pDd\}}\dr s.
\end{equation*}
Since $(R_t)_{t\geq 0}$ is defined by~\eqref{defR} (see also Remark~\ref{rkr}), 
it holds that $R_s \in \pDd$ if and only if $s\in \cup_{u\in \JJ}(\tau_{u-},\tau_u)$. Since moreover
$\tau_{u-}\leq t$ for all $u\in \JJ \cap[0,L_t]$, we have
\begin{align*}
J_t= \sum_{u \in \JJ} \indiq_{\{u\leq L_t\}} \int_{\tau_{u-}}^{\tau_u\land t} \cL\varphi(R_s) \dr s.
\end{align*}
Again by~\eqref{defR}, we have $R_s=h_{b_{u-}}(a_u,e_u(s-\tau_{u-}))$ for all $u \in \JJ$, all
$s \in (\tau_{u-},\tau_u)$, whence
\begin{align*}
J_t
=&\sum_{u \in \JJ} \indiq_{\{u\leq L_t\}} \int_{\tau_{u-}}^{\tau_u\land t}\cL\varphi(h_{b_{u-}}(a_u,e_u(s-\tau_{u-})) ) \dr s\\
=&\sum_{u \in \JJ} \indiq_{\{u\leq L_t\}} \int_0^{(t-\tau_{u-})\land (\tau_u-\tau_{u-})}
\cL\varphi(h_{b_{u-}}(a_u,e_u(s)) ) \dr s.
\end{align*}
Since finally $\tau_u-\tau_{u-}=\cl_{b_{u-}}(a_u,e_u)$ by~\eqref{sdetd}, the compensation formula
tells us that
\begin{equation}\label{f1}
\E[J_t]=\E\Big[\int_0^{L_t}\int_\cE \Big(\int_0^{(t-\tau_u) \land \cl_{b_{u}}(a_u,e)} 
\cL\varphi(h_{b_{u}}(a_u,e(s)) ) \dr s
\Big) \nn_j(\dr e) \dr u   \Big].
\end{equation}
Next, we show that for all $w\geq 0$, all $b\in \pDd$, all $A \in \cI_b$,
\begin{equation}\label{ultimgoal}
\Theta\varphi(w,b,A)=\int_\cE \Big(\int_0^{w\land \cl_b(A,e)} \cL\varphi(h_b(A,e(s)))\dr s
\Big) \nn_j(\dr e).
\end{equation}
Inserted in the expression of $I_t$ and comparing to~\eqref{f1}, this will show that 
$\E[I_t]=\E[J_t]$ and the proof will be complete.
\vip
Recall the expression~\eqref{nnj} of $\nn_j$: we have, with $Z$ an ISP$_{\alpha,z}$ under $\PP_z$,
\begin{gather*}
\Theta\varphi(w,b,A)=\int_\HH \Delta_z\varphi(w,b,A) j(\dr z), \\
\int_\cE \Big(\int_0^{w\land \cl_b(A,e)} \cL\varphi(h_b(A,e(s)))\dr s
\Big) \nn_j(\dr e)= \int_\HH \Delta'_z\varphi(w,b,A) j(\dr z),
\end{gather*}
where 
\begin{gather*}
\Delta_z\varphi(w,b,A)\!\!=\E_z\!\Big[\!\indiq_{\{\cl_b(A,Z)\leq w\}} \varphi(g_b(A,Z_{\cdot\land\cl_b(A,Z)}))\!+\!
\indiq_{\{\cl_b(A,Z)> w\}} \varphi(h_b(A,Z_w))\!-\!\varphi(\Lambda(b,h_b(A,z)))\!\Big]\\
\Delta_z'\varphi(w,b,A)\!\!= \E_z\Big[ \int_0^{w\land \cl_b(A,Z)} \cL\varphi(h_b(A,Z_s)) \dr s\Big],
\end{gather*}
and it suffices to show that $\Delta_z\varphi(w,b,A)=\Delta_z'\varphi(w,b,A)$ for all $z\in \HH$,
$w\geq0$, $b \in \pDd$, $A \in \cI_b$.

\vip
If $h_b(A,z)\notin  \Dd$, this is obvious, since we $\PP_z$-a.s. have $\cl_b(A,Z)=0$ and 
$g_b(A,Z_{\cdot\land\cl_b(A,Z)})=\Lambda(b,h_b(A,z))$.

\vip

If $h_b(A,z)\in  \Dd$, we have, recalling the definition of $g_b(A,\cdot)$,
using that $h_b(A,Z)$ is an ISP$_{\alpha,h_b(A,z)}$ under $\PP_z$ and setting
$\tell(Z)=\inf\{t>0 : Z_t\notin \Dd\}$,
$$
\Delta_z\varphi(w,b,A)=
\E_{h_b(A,z)}\Big[\indiq_{\{\tell(Z)\leq w\}} \varphi(\Lambda(Z_{\tell(Z)-},Z_{\tell(Z)}))
+ \indiq_{\{\tell(Z)> w\}} \varphi(Z_w) - \varphi( h_b(A,z)) \Big].
$$
Recalling Definition~\ref{dfr2}, we thus find, with $d_0^*=\inf\{s>0 : X^*_s \notin \Dd\}$,
\begin{align*}
\Delta_z\varphi(w,b,A)=&\QQ_{h_b(A,z)}\Big[\varphi(X^*_{w\land d_0^*})-\varphi(h_b(A,z))\Big]
=\QQ_{h_b(A,z)}\Big[\int_0^{w\land d_0^*} \cL\varphi(X^*_s)\dr s \Big]
\end{align*}
by Lemma~\ref{lg1}.
Using one more time Definition~\ref{dfr2} and then the isotropy  of the stable process,
\begin{gather*}
\Delta_z\varphi(w,b,A)\!=\!\E_{h_b(A,z)}\Big[\int_0^{w\land\tell(Z)} \cL\varphi(Z_s)\dr s \Big]
\!=\!\E_{z}\Big[\int_0^{w\land\cl_b(A,Z)} \cL\varphi(h_b(A,Z_s))\dr s \Big]
\!=\!\Delta_z'\varphi(w,b,A)
\end{gather*} 
as desired.
\end{proof}

Finally, we can handle the

\begin{proof}[Proof of Proposition~\ref{lg2}] We consider $\varphi \in D_\alpha\cap H_\beta$ and 
split the proof into 3 steps.
\vip
{\it Step 1.} We first prove~\eqref{anres} when $x \in \pDd$ and $\beta \in (0,\alpha/2)$.
For each $n\geq 1$, we consider $m_n=1/n$ and $j^n_\beta$ as in Lemma~\ref{dtc}. By Lemma~\ref{itod}, we know
that, with  $\theta_n=\int_{\cE} (1-e^{-\ell_r(e)})\nn_{j^n_\beta}(\dr e)$,
\begin{equation}\label{anresd}
\Big|\QQ_x^{m_n,j^n_\beta}[\varphi(X^*_t)]-\varphi(x)-
\int_0^t \QQ_x^{m_n,j^n_\beta}[\indiq_{\{X_s^* \notin \pDd\}}\cL\varphi(X_s^*)]\dr s\Big|
\leq  \theta_n^{-1}e^t\sup_{b\in \pDd, A \in \cI_b} |\cK_{j^n_\beta}\varphi(b,A)|.
\end{equation}
Fix $n_0>1/r$, so that $B_d(r\be_1,r)\cap \{|z|>1/n_0\}\neq \emptyset$. Then for any $n\geq n_0$,
we have $\theta_n\geq \theta_{n_0}>0$. Indeed,
recalling the definition~\eqref{nnj} of $\nn_j$ and that of $j^\beta_n$, we find
$$
\theta_{n_0}=\int_\HH \indiq_{\{|z|>1/n_0\}} \E_z[1-\exp(-\ell_r(Z))] \frac{\dr z}{|z|^{d+\beta}},
$$
which is strictly positive since $\PP_z(\ell_r(Z)>0)=1$ for all $z \in B_d(r\be_1,r)$.
Next, comparing Definition~\ref{test}-(a) and~\eqref{ckj}, we see that 
$\cK_{j^n_\beta}\varphi(b,A)=\cH_{\beta,1/n}\varphi(b)$.
Since $\varphi \in H_\beta$, we conclude that $\lim_n \sup_{b \in \pDd,A\in \cI_b}|\cK_{j^n_\beta}\varphi(b,A)|=0$.
\vip
By Lemma~\ref{dtc} and since $\varphi\in C(\cDd)$, we see that $\lim_n \QQ_x^{m_n,j^n_\beta}[\varphi(X^*_t)]
=\QQ_x[\varphi(X^*_t)]$. Since $\cL \varphi \in C(\Dd)\cap L^\infty(\Dd)$, we also have 
$\lim_n \QQ_x^{m_n,j^n_\beta}[\cL\varphi(X^*_s)\indiq_{\{X^*_s \notin \pDd \}}]
=\QQ_x[\cL\varphi(X^*_s)]$ for all $s\geq 0$ such that $\QQ_x(X^*_s\in\pDd)=0$. Since $\QQ_x(X^*_s\in\pDd)=0$
for a.e. $s\geq 0$, see Theorem~\ref{mr2}, and since $\cL\varphi$ is bounded, 
we conclude that $\lim_n \int_0^t \QQ_x^{m_n,j^n_\beta}[\indiq_{\{X_s^* \notin \pDd\}}
\cL\varphi(X_s^*)]\dr s= \int_0^t \QQ_x[\cL\varphi(X_s^*)]\dr s$.
We thus may let $n\to \infty$ in~\eqref{anresd} and find~\eqref{anres}.

\vip

{\it Step 2.} We next prove~\eqref{anres} when $x \in \pDd$ and $\beta=*$.
As in Step 1, we start from~\eqref{anresd} and, using the very same arguments, we only have to verify that
the RHS of~\eqref{anresd} tends to $0$ as $n\to \infty$, {\it i.e.} that 
$\liminf_n \theta_n>0$ and that $\lim_n \Delta_n=0$, 
where $\Delta_n=\sup_{b \in \pDd,A\in \cI_b}|\cK_{j^n_*}\varphi(b,A)|=0$.

\vip

Let us first recall that $j^n_*(\dr x)=\nn_*(e(1/n)\in \dr x, \ell_r(e)>1/n)$, from which we get that for any 
$n \geq 1$, $\theta_n=\int_\cE (1-e^{-(\ell_r(e)-1/n)})\indiq_{\{\ell_r(e)>1/n\}} \nn_*(\dr e)$.
Using Fatou's lemma, we get
$\liminf_n \theta_n \geq \int_\cE (1-e^{-\ell_r(e)})\nn_*(\dr e)>0$, see Lemma~\ref{imp}.

\vip

Since $j^n_*$ is carried by $B_d(r\be_1,r)$ by definition
of $\ell_r$, since $h_b(A,z)=b+Az \in B_d(b+r\bn_b,r)$ for all $z\in B_d(r\be_1,r)$
and since $B_d(b+r\bn_b,r)\subset \Dd$ by definition of $r$ (see Remark~\ref{imp4}),
we have, recalling~\eqref{ckj},
$$
\cK_{j^n_*}\varphi(b,A)= \int_{B_d(r\be_1,r)}[\varphi(b+Az)-\varphi(b)]j^n_*(\dr z).
$$
Recall that 
$\Sp_*=\{\rho \in \R^d : |\rho|=1, \rho\cdot \be_1=0\}$ and that $\varsigma$ is the uniform measure on $\Sp_*$.
Since $j^n_*$ is invariant by any 
isometry of $\HH$ sending $\be_1$ to $\be_1$, it holds that
$$
j^n_*(B)=\int_{B_2(r\be_1,r)}\int_{\Sp_*} \indiq_{\{h_1 \be_1 + h_2 \rho \in B \}} \zeta(\dr \rho) g^n_* (\dr h)
$$ 
for any $B \in \cB(\HH)$, where $g^n_*$ is the image measure of $j^n_*$ by
$$
B_d(r\be_1,r)\ni z\mapsto h=(z\cdot \be_1,\sqrt{|z|^2-(z\cdot\be_1)^2})\in B_2(r\be_1,r).
$$
As a conclusion, 
\begin{align*}
\cK_{j^n_*}\varphi(b,A)=& \int_{B_2(r\be_1,r)}\int_{\Sp_*}[\varphi(b+A(h_1\be_1+h_2\rho))-\varphi(b)] 
\zeta(\dr\rho) g^n_* (\dr h)\\
=&  \int_{B_2(r\be_1,r)} |h|^{\alpha/2} \cH_{*}\varphi(b,h) g^n_* (\dr h),
\end{align*}
with $\cH_*\varphi$ introduced in Definition~\ref{test}-(b). We set
$\Psi(u)=\sup_{b \in \pDd, h \in B_2(r\be_1,r)\cap B_2(0,u)}|\cH_{*}\varphi(b,h)|$ for $u \in (0,2r)$ 
(observe that $h \in B_2(r\be_1,r)$ implies that $|h|<2r$). Since $\varphi \in H_*$, recall 
Definition~\ref{test}-(b), $\Psi$ is bounded and satisfies $\lim_{u\to 0} \Psi(u) =0$. We now write
$$
\Delta_n \leq  \int_{B_2(r\be_1,r)} |h|^{\alpha/2} \Psi (|h|) g^n_*(\dr h)=\int_{B_d(r\be_1,r)} 
|z|^{\alpha/2} \Psi (|z|) j^n_*(\dr z).
$$
Recalling the definition of $j^n_*$ and that $\ell\geq\ell_r$, we find
\begin{align*}
\Delta_n\leq&\int_{\cE} |e(1/n)|^{\alpha/2} \Psi (|e(1/n)|) \indiq_{\{e(1/n)\in B_d(r\be_1,r)\}} 
\indiq_{\{\ell_r(e)>1/n\}}\nn_*(\dr e).
\end{align*}
Using the notation of Lemma~\ref{encore1}, this gives
\[
\Delta_n\leq \int_\HH  |a|^{\alpha/2}\Psi (|a|)\indiq_{\{a \in B_d(r\be_1,r)\}} k_{1/n}(\dr a).
\]
But Lemma~\ref{encore1} tells us that $|a|^{\alpha/2} k_t(\dr a) \to c_1\delta_0(\dr a)$ weakly as $t\to 0$.
Since $\lim_{u\to 0} \Psi (u)=0$ and since $\Psi$ is bounded, we conclude
that $\lim_n \Delta_n=0$ as desired.

\vip

{\it Step 3.} We finally prove~\eqref{anres} when $x \in \Dd$. With $d_0^*=\inf\{t>0 : X^*_t\in\pDd\}$,
\begin{align*}
\QQ_x[\varphi(X_t^*)]=&\QQ_x[\varphi(X_{t\land d_0^*}^*)]+ \QQ_x[\indiq_{\{t>d_0^*\}}(\varphi(X_t^*)-
\varphi(X_{t\land d_0^*}^*))]\\
=&\varphi(x)+\QQ_x\Big[\int_0^{t\land d_0^*}\cL\varphi(X^*_s)\dr s\Big] + 
\QQ_x\Big[\indiq_{\{t>d_0^*\}} \psi(X^*_{d_0^*},t-d_0^*) \Big],
\end{align*}
where $\psi(z,r)=\QQ_z[\varphi(X^*_r)-\varphi(z)]$. We used Lemma~\ref{lg1} for the first term and the 
strong Markov property for the second one. By Steps 1 and 2, we know that 
$\psi(z,r)=\QQ_z[\int_0^r \cL\varphi(X^*_s)\dr s]$
for all $z \in \pDd$. Since $X^*_{d_0^*}\in \pDd$, we conclude, using the strong Markov property again, that
$$
\QQ_x[\varphi(X_t^*)]=\varphi(x)+\QQ_x\Big[\int_0^{t\land d_0^*}\cL\varphi(X^*_s)\dr s\Big] + 
\QQ_x\Big[\indiq_{\{t>d_0^*\}} \int_{d_0^*}^t\cL\varphi (X^*_s)\dr s \Big].
$$
The conclusion follows.
\end{proof}
 
\section{The scattering process and its Markov version}\label{scatM}

Here we prove Remark~\ref{scwd}, which shows that the scattering process
introduced in Definition~\ref{dfsp} does not explode. We then introduce a Markov variation of the (position
of) the scattering process and explain why it is sufficient to the study scaling limit of this Markov process.

\begin{proof}[Proof of Remark~\ref{scwd}]
Recall the definition of the sequence $(T_n^\e)_{n\geq1}$ from Definition~\ref{dfsp}.
First, the sequence $T_n^\e$ is a.s. strictly increasing because $\lambda(x,v,s)>0$ for all $(x,v) \in \bE$,
all $s>0$. We next show that there exist $a>0$ and $\rho>0$ (depending on $\e$) such that  
for $U\sim \rF(v)\dr v$, $W\sim \rG_+(v)\dr v$ and $E^\e\sim \text{Exp}(\e^{-1})$ independent,
\begin{gather}
\forall x\in \Dd, \quad  p_x(a)=\PP(\lambda(x,U,E^\e)\geq a)\geq \rho,\label{supc1} \\
\forall x\in \pDd,  \forall A\in \cI_x, \quad  q_x(a)=\PP(\lambda(x,AW,E^\e)\geq a)\geq \rho, \label{supc2}
\end{gather}

The domain $\Dd$ being open and $C^2$, there exists $\delta>0$ such that for all $x\in \Dd$, there
exists a ball $C_x\subset \Dd$ with radius $\delta$ such that $x \in \partial C_x$. Calling $y_x$ the center of
$C_x$, we have $C_x=B_d(y_x,\delta)$ and $\delta=|x-y_x|$. One may check that for $v\in \R^d$ and $s>0$,
$x+sv \in C_x$ if and only if $\frac{v}{|v|}\cdot \frac{y_x-x}{|y_x-x|}>\frac{|v|s}{2\delta}$, and in such a case,
$\lambda(x,v,s)=s$. Thus for all $x \in \Dd$,
$$
p_x(a)\geq \PP\Big(E^\e\geq a, \frac{U}{|U|}\cdot \frac{y_x-x}{|y_x-x|}>\frac{|U| E^\e}{2\delta}\Big)
=\PP\Big(E^\e\geq a, \frac{U}{|U|}\cdot \be_1>\frac{|U| E^\e}{2\delta}\Big)
$$
by rotational invariance of $\rF$. Since $\frac{U}{|U|}$ is independent of $|U|$,
$$
p_x(a)\geq \PP(E^\e\in (a,2a))\PP\Big(\frac{U}{|U|}\cdot \be_1>\frac 12\Big)\PP\Big(|U|<\frac\delta{2a}\Big).
$$
This last quantity does not depend on $x$ and is positive if 
$a>0$ is small enough so that ${\PP(|U|<\frac\delta{2a})>0}$, since $\frac{U}{|U|}$ is uniformly distributed
on $\Sp_{d-1}$.

\vip

The domain $\Dd$ being open and $C^2$, there is $\delta>0$ such that for all $x\in \pDd$, 
$B_d(x+\delta\bn_x,\delta)\subset \Dd$. One concludes as previously that $\inf_{x\in\pDd} q_x(a)>0$
if $a>0$ is small enough.

\vip

Recalling Definition~\ref{dfsp}, one easily deduces from~\eqref{supc1}-\eqref{supc2} 
that if setting $\cG_n^\e=\sigma(T_1^\e,\dots,T_n^\e)$, we have $\PP(T_{n+1}^\e-T_n^\e \geq a | \cG_n^\e)\geq \rho$
for all $n\geq 0$, so that $\mathbb{E}[1-e^{-(T_{n+1}^\e - T_n^\e)} | \cG_n^\e] \geq \rho(1-e^{-a})$. Hence, we get
\[
\E[e^{-T_{n+1}^\e}]= \E\Big[e^{-T_n^\e}\E\Big[e^{-(T_{n+1}^\e-T_n^\e)}\Big|\cG_n^\e\Big]\Big]
\leq (1-\rho(1-e^{-a}))\E[e^{-T_{n}^\e}].
\]
Thus $\E[e^{-T_{n}^\e}]\leq \theta^n$, where $\theta=1-\rho(1-e^{-a})<1$ 
and $\PP(T_n^\e \leq T)=\PP(e^{-T_n^\e}\geq e^{-T})
\leq e^T \theta^n$.
Hence $\lim_n T_n^\e=\infty$ a.s. and $\PP(\bM^\e_T \geq n)=\PP(T_n^\e\leq T)\leq e^T\theta^n$,
so that $\E[\bM^\e_T]<\infty$.
\end{proof}

We introduce a few notation.

\begin{notation}\label{fgh}
Grant Assumption~\ref{assump_equi}.
Fix $\e \in (0,1]$ and consider three independent random variables $U\sim \rF(v)\dr v$, $W\sim \rG_+(v)\dr v$ and 
$E^\e\sim \rm{Exp}(\e^{-1})$. We call $\rF_\e$ the density of $\e^{(1-\alpha)/\alpha}E^\e U$, carried by $\R^d$,
$\rG_\e$ the density of $\e^{(1-\alpha)/\alpha}E^\e W$, carried by $\HH$, and 
finally $\rH_\e$ the density of $(\e^{(1-\alpha)/\alpha}E^\e U,\e^{(1-\alpha)/\alpha}E^\e W)$, 
carried by $\R^d\times \HH$.
\end{notation}

The scattering process $(\bX^\e_t,\bV^\e_t)_{t\geq 0}$ is Markov, but the position process $(\bX^\e_t)_{t\geq 0}$ alone
is not. We now introduce some Markov process $(R^\e_t)_{t\geq 0}$ that will be close to $(\bX^\e_t)_{t\geq 0}$.
More precisely, as we will see in the proof of Theorem~\ref{mr}, 
$(\bX^\e_t)_{t\geq 0}$ is the linear interpolation of $(R^\e_{\lambda^\e_t})_{t\geq 0}$, where $\lambda^\e$ is a time-change
close to identity.

\begin{definition}\label{dfre}
Grant Assumptions~\ref{as} and~\ref{assump_equi}.
Fix $\e \in (0,1]$ and
consider a Poisson measure $\rM_\e$ on $\R_+\times\R^d\times \HH$, with intensity 
$\e^{-1}\dr s \rH_\e(u,w)\dr u \dr w$. 
We also consider a measurable family $(A_y)_{y \in \pDd}$ such that $A_y  \in \cI_y$
for each $y \in \pDd$ and recall that $\Lambda$ was defined in~\eqref{Lambda}. 
For $x\in \cDd$, we say that $(R^\e_t)$ is an $\e$-Markov scattering process issued from $x$ if a.s.,
for all $t\geq 0$, 
\begin{equation}\label{eq:def_r_epsilon}
R_t^\e =  x + \int_0^t\int_{\mathbb{R}^d\times\HH} [\Lambda(R_\sm^\e, R_\sm^\e+u\indiq_{\{R_\sm^\e \notin \partial\Dd\}}+ 
A_{R_\sm^\e}w\indiq_{\{R_\sm^\e \in \partial\Dd\}}) - R_\sm^\e]
\rM_\e(\dr s, \dr u,\dr w).
 \end{equation}
\end{definition}

This process is well-defined.

\begin{remark}\label{repswd}
Grant Assumptions~\ref{as} and~\ref{assump_equi} and 
fix $\e\in (0,1]$. The S.D.E.~\eqref{eq:def_r_epsilon} has, for each $x\in \cDd$, a pathwise unique
solution $(R^\e_t)_{t\geq 0}$, which is $\cDd$-valued. As solution to a time-homogeneous
well-posed S.D.E., the resulting process is strongly Markov. The law of $(R^\e_t)_{t\geq 0}$
does not depend on the choice of the family $(A_y)_{y \in \pDd}$.
\end{remark}

\begin{proof}
Note that for any $t > 0$, 
$\rM_\e([0,t]\times \mathbb{R}^d\times\HH)=\e^{-1}t< \infty$. 
Hence~\eqref{eq:def_r_epsilon} may be solved by induction on the jump times of $\rM_\e$,
implying its pathwise well-posedness. The last assertion follows from the fact that for $y \in \pDd$ and 
for $A,B \in \cI_y$, $\rG_\e \#A=\rG_\e \#B$.
\end{proof}

The following convergence result will be the object of Sections~\ref{sconvmark} and~\ref{stech}.

\begin{proposition}\label{thm:conv_markov}
Grant Assumption~\ref{as} and Assumption~\ref{assump_equi} with some $\alpha \in (0,2)$ and with 
$\kappa_\rF=1/\Gamma(\alpha+1)$. Grant either Assumption~\ref{assump:moments}-(a)
(in which case, set $\beta=*$) or Assumption~\ref{assump:moments}-(b)
(in which case $\beta \in (0,\alpha/2)$).
Consider the family
$(\QQ_x)_{x \in \cDd}$ as in Theorem~\ref{mr2}, with these values of $\alpha$ and $\beta$.
Consider, for each $\e \in (0,1]$, the solution $(R_t^\e)$ to~\eqref{eq:def_r_epsilon} starting from
some $x_\e \in \cDd$. If $x_\e \to x \in \cDd$, then
\[
 (R_t^\e)_{t\geq0} \quad \text{converges in law to} \quad \QQ_x \quad \text{as }\e\to0
\]
in $\DD(\R_+,\cDd)$, endowed with the $\JS$-topology.
\end{proposition}

We will also show the following result in Section~\ref{stech}.

\begin{lemma}\label{tmok}
Grant the same assumptions as in Proposition~\ref{thm:conv_markov}.
For $(x,v)\in \bE$ and $\e\in (0,1]$, consider the sequence $(T_n^\e)_{n\geq 1}$ introduced in 
Definition~\ref{dfsp}, set $T_0^\e=0$ and $N^\e_t=\sum_{n\geq 1} \indiq_{\{T_n^\e\leq t\}}$. For any $T>0$,
\begin{gather*}
\sup_{n=0,\dots,N^\e_T} (T^\e_{n+1}-T_n^\e) \to 0 \quad \text{in probability as $\e\to 0$,}\\
\sup_{t \in [0,T]} |\e N^\e_t-t| \to 0 \quad \text{in probability as $\e\to 0$.}
\end{gather*}
\end{lemma}

Admitting Lemma~\ref{tmok} and Proposition~\ref{thm:conv_markov}, we give the

\begin{proof}[Proof of Theorem~\ref{mr}]
We divide the proof in several steps.
\vip

{\it Step 1.} With the notation of Definition~\ref{dfsp} and with $T^\e_0=0$, we introduce the process
$\bar \bX^\e_t=\sum_{n\geq 0}\bX^\e_{T_n^\e}\indiq_{\{t \in [T_n^\e,T^\e_{n+1})\}}$, of which $\bX^\e$ is the linear 
interpolation. We check here that
it is enough to show that $(\bar \bX^\e_t)_{t\geq 0}$ converges in law to $\QQ_x$ for the $\JS$-topology.

\vip

We know from Lemma~\ref{tmok} that $\sup_{n=0,\dots,N^\e_T} |T^\e_{n+1}-T_n^\e| \to 0$
in probability as $\e\to 0$ for any $T>0$. By Lemma~\ref{j1m1}, we conclude that
if $(\bar \bX^\e_t)_{t\geq 0}$ converges in law to $\QQ_x$ in $\DD(\R_+,\cDd)$ for the
$\JS$-topology, then $(\bX^\e_t)_{t\geq 0}$ converges in law to $\QQ_x$
for the $\MS$-topology.

\vip

{\it Step 2.} We next introduce $(Y^\e_t=\bar \bX^\e_{T_1^\e+t})_{t\geq 0}$ and verify that it suffices to show that
$(Y^\e_t)_{t\geq 0}$ converges in law to $\QQ_x$ for the $\JS$-topology.

\vip

We introduce the continuous increasing time-change 
$$
\lambda_\e(t)=\frac{T^\e_2}{T^\e_2-T^\e_1}t \indiq_{\{t\leq T^\e_2-T^\e_1\}} + (T_1^\e+t)\indiq_{\{t> T^\e_2-T^\e_1\}}.
$$
By definition of the $\JS$-topology, see Appendix~\ref{sko}, it suffices to show that
$\lim_{\e\to 0} ||\lambda_\e - I||_\infty=0$ and $\lim_{\e\to 0}||Y^\e -\bar \bX^\e\circ \lambda_\e ||_\infty=0$ 
in probability. But
$||\lambda_\e - I||_\infty=T^\e_1 \leq E^\e_1\sim \text{Exp}(\e^{-1})$, which tends to 0 in probability,
while 
$$
||Y^\e -\bar \bX^\e\circ \lambda_\e ||_\infty= \sup_{t\in [0,T^\e_2-T^\e_1]} |\bar \bX^\e_{T_1^\e+t}-\bar \bX^\e_{\lambda_\e(t)}|
\leq \sup_{t\in [0,T^\e_2-T^\e_1]} |\bar \bX^\e_{T_1^\e+t}-x|+\sup_{t\in [0,T^\e_2-T^\e_1]} |\bar \bX^\e_{\lambda_\e(t)}-x|,
$$
so that
$$
||Y^\e -\bar \bX^\e\circ \lambda_\e ||_\infty\leq 2\sup_{[0,T_2^\e]}|\bar \bX^\e_{t}-x|\leq 
2(| \bX^\e_{T^\e_1}-x|\lor| \bX^\e_{T^\e_2}-x|).
$$
Looking at Definition~\ref{dfsp}, we deduce that
\begin{align*}
||Y^\e -\bar \bX^\e\circ \lambda_\e ||_\infty\leq & 2\e^{(1-\alpha)/\alpha}|v|T_1^\e+2
\e^{(1-\alpha)/\alpha}|\bV^\e_{T_1^\e}|(T_2^\e-T^\e_1)\\
\leq& 2|v|\e^{(1-\alpha)/\alpha}E^\e_1+ 2(|U_1|+|W_1|)\e^{(1-\alpha)/\alpha}E^\e_2,
\end{align*}
which goes to $0$ in probability, since $\e^{(1-\alpha)/\alpha} E^\e_1$ and $\e^{(1-\alpha)/\alpha} E^\e_2$ 
are equal in law to $\e^{1/\alpha}E_1^1$.

\vip

{\it Step 3.} We consider an i.i.d. family $(F_n^\e)_{n\geq 1}$ of $\mathrm{Exp}(\e^{-1})$-distributed random variables 
independent of everything else, set $S^\e_n=F_1^\e+\dots+F_n^\e$ and introduce the Poisson measure
$$
\rM_\e=\sum_{n\geq 1} \delta_{(S^\e_n,U_n^\e,W_n^\e)}, \quad \text{where} \quad 
U_n^\e=\e^{(1-\alpha)/\alpha}E_{n+1}^\e U_{n} \quad \text{and} \quad 
W_n^\e=\e^{(1-\alpha)/\alpha}E_{n+1}^\e W_{n}.
$$
Its intensity is $\e^{-1} \dr s \rH_\e(u,w)\dr u \dr w$, and it is independent of $\bX^\e_{T_1^\e}$
(which is a function of $x,v,E^\e_1$). We then consider the solution $Z^\e_t$ 
to~\eqref{eq:def_r_epsilon} with this Poisson measure and starting at 
$\bX^\e_{T_{1}^\e}$. We claim that $Y^\e_t=Z^\e_{\rho^\e_t}$, where $\rho^\e_t=S^\e_{M^\e_t}$, with 
$M^\e_t=\sum_{n\geq 1} \indiq_{\{T_{n+1}^\e-T_1^\e\leq t\}}$.
\vip

Looking at Definition~\ref{dfsp}, one can check that for all $n\geq 1$,
$$
\bX^\e_{T_{n+1}^\e}= \Lambda(\bX^\e_{T_{n}^\e}, \bX^\e_{T_{n}^\e}+ U_n^\e)\indiq_{\{\bX^\e_{T_{n}^\e}\in \Dd\}}
+\Lambda(\bX^\e_{T_{n}^\e}, \bX^\e_{T_{n}^\e}+ A_{\bX^\e_{T_n^\e}}W_n^\e)\indiq_{\{\bX^\e_{T_{n}^\e}\in \pDd\}}.
$$
Indeed, if first $\bX^\e_{T_{n}^\e}\in \Dd$, then $\bV^\e_{T_{n}^\e}=U_n$ and thus 
$$
\bX^\e_{T_{n+1}^\e}=\bX^\e_{T_{n}^\e}+\e^{(1-\alpha)/\alpha}U_n\lambda(\bX^\e_{T_{n}^\e},\e^{(1-\alpha)/\alpha}U_n,E^{\e}_{n+1})
=\Lambda(\bX^\e_{T_{n}^\e}, \bX^\e_{T_{n}^\e}+ \e^{(1-\alpha)/\alpha}E_{n+1}^\e U_n)
$$
by~\eqref{Lamlam}. If next $\bX^\e_{T_{n}^\e}\in \pDd$, then $\bV^\e_{T_{n}^\e}=A_{\bX^\e_{T_{n}^\e}}W_n$ and thus 
$$
\bX^\e_{T_{n+1}^\e}=\bX^\e_{T_{n}^\e}+\e^{(1-\alpha)/\alpha}A_{\bX^\e_{T_{n}^\e}}W_n
\lambda(\bX^\e_{T_{n}^\e},\e^{(1-\alpha)/\alpha}A_{\bX^\e_{T_{n}^\e}}W_n,E^{\e}_{n+1}),
$$
which equals $\Lambda(\bX^\e_{T_{n}^\e}, \bX^\e_{T_{n}^\e}+
\e^{(1-\alpha)/\alpha}E_{n+1}^\e A_{\bX^\e_{T_{n}^\e}}W_n)$ by~\eqref{Lamlam}.
\vip

Consequently, for all $n\geq 1$, since $Y^\e_{T_{n}^\e-T_{1}^\e}=\bX^\e_{T_{n}^\e}$,
$$
Y^\e_{T_{n+1}^\e-T_{1}^\e}=\Lambda( Y^\e_{T_{n}^\e-T_1^\e}, Y^\e_{T_{n}^\e-T_1^\e}+ U_n^\e)\indiq_{\{Y^\e_{T_{n}^\e-T_1^\e}\in \Dd\}}
+\Lambda(Y^\e_{T_{n}^\e-T_1^\e}, Y^\e_{T_{n}^\e-T_1^\e}+ A_{Y^\e_{T_n^\e-T_1^\e}}W_n^\e)\indiq_{\{Y^\e_{T_{n}^\e-T_1^\e}\in \pDd\}}.
$$

Next, we immediately deduce from~\eqref{eq:def_r_epsilon} and the definition of $\rM_\e$ 
that for all $n\geq 1$,
$$
Z^\e_{S^\e_{n}}=\Lambda(Z^\e_{S^\e_{n-1}}, Z^\e_{S^\e_{n-1}}+ U_{n}^\e)\indiq_{\{Z^\e_{S^\e_{n-1}}\in \Dd\}}
+\Lambda(Z^\e_{S^\e_{n-1}}, Z^\e_{S^\e_{n-1}}+ A_{Z^\e_{S^\e_{n-1}}}W_{n}^\e)\indiq_{\{Z^\e_{S^\e_{n-1}}\in \pDd\}}.
$$
Since $Y^\e_0=\bX^\e_{T_1^\e}=Z^\e_0$, we conclude that for all $n\geq 0$, we have $Y^\e_{T_{n+1}^\e-T_{1}^\e}=Z^\e_{S^\e_{n}}$.

\vip
Thus for all $n\geq 0$, all $t\in [T_{n+1}^\e-T_{1}^\e,T_{n+2}^\e-T_{1}^\e)$, it holds that
$$
Y^\e_t=Y^\e_{T_{n+1}^\e-T_{1}^\e}=Z^\e_{S^\e_{n}}=Z^\e_{\rho^\e_t},
$$ 
because by definition of $M^\e_t$, we have $M^\e_t=n$ for all $t\in [T_{n+1}^\e-T_{1}^\e,T_{n+2}^\e-T_{1}^\e)$.

\vip 

{\it Step 4.} Here we prove that for all $T>0$, $\sup_{[0,T]} |\e M^\e_t-t|\to 0$ in probability.

\vip

Recall that $M^\e_t=\sum_{n\geq 1} \indiq_{\{T_{n+1}^\e-T_1^\e\leq t\}}$ and that $N^\e_t=\sum_{n\geq 1} \indiq_{\{T_{n}^\e\leq t\}}$.
For all $t\geq 0$, we have $M^\e_t=N^\e_{T_1^\e+t}-1$, so that
$$
\sup_{[0,T]}|\e M^\e_t-t| \leq \e + \sup_{[0,T]}|\e N^\e_{t+T^\e_1}-t|
\leq \e +  \sup_{[0,T+T_1^\e]}|\e N^\e_{t}-t| + T_1^\e.
$$
Since $T_1^\e\to 0$ in probability (see Step~2) and
$\sup_{[0,T+1]}|\e N^\e_{t}-t|\to 0$ in probability by Lemma~\ref{tmok}, the conclusion follows.
\vip

{\it Step 5.} By Step~2, it suffices that $(Y^\e_t)_{t\geq 0} \to \QQ_x$ in law for the $\JS$-topology.
Since $Y^\e_t=Z^\e_{\rho^\e_t}$ by Step~3, it is enough, by Lemma~\ref{compj1},
to show that (i) $(Z^\e_t)_{t\geq 0}\to \QQ_x$ in law  for the $\JS$-topology and that (ii)
for all $T>0$, $\sup_{[0,T]}|\rho^\e_t-t|\to 0$ in probability.

\vip

First, (i) follows from Proposition~\ref{thm:conv_markov}, because $(Z^\e)_{t\geq 0}$ is a solution 
to~\eqref{eq:def_r_epsilon}  starting from $\bX^\e_{T_{1}^\e}$ and since $\bX^\e_{T_{1}^\e}
=\bar \bX^\e_{T_{1}^\e}\to x$ in probability (as seen in Step~2).

\vip
Next, we write $\sup_{[0,T]}|\rho^\e_t-t|=A^1_{T,\e}+A^2_{T,\e}$, where
$$
A^1_{T,\e}= \sup_{[0,T]} |S^\e_{M^\e_t}-\e M^\e_t| \quad \text{and} \quad A^2_{T,\e} = \sup_{[0,T]} |\e M^\e_t-t|.
$$
We know from Step~4 that $A^2_{T,\e}\to 0$ in probability. For $\eta>0$
$$
\PP(A^1_{T,\e}>\eta)\leq \PP\Big(M^\e_T>\frac {2T} \e\Big) 
+ \PP\Big(\sup_{k=0,\dots,\lfloor 2T/\e\rfloor}|S^\e_k-\e k|>\eta\Big)
\leq \PP\Big(M^\e_T>\frac {2T} \e\Big)+ \frac{4 \mathrm{Var}(S^\e_{\lfloor 2T/\e\rfloor })}{\eta^2}
$$
by Doob's $L^2$ inequality, since $(S^\e_k-\e k)_{k\geq 0}$ is a martingale. Since 
Var$(S^\e_{\lfloor 2T/\e \rfloor})=\e^2 \lfloor 2T/\e \rfloor\to 0$
and since $\PP(M^\e_T>2T/ \e) \leq \PP(|\e M^\e_T-T|>T) \to 0$ by  Step~4, we have shown (ii).
\end{proof}

\section{Convergence of the Markov scattering process}\label{sconvmark}

The goal of this section is to prove Proposition~\ref{thm:conv_markov}, assuming a few results,
see Proposition~\ref{lemma:estimates}, that will be proved in Section~\ref{stech}.
\vip

\subsection{Excursions outside the boundary}

We consider the solution $(R_t^\e)_{t\geq0}$ to~\eqref{eq:def_r_epsilon}, starting from some $x\in \pDd$. 
We show that this process
shares the same structure as $(R_t)_{t\geq0}$, {\it i.e.} it can be built by concatenating, translating,
rotating and stopping 
excursions outside the half-space, in a way resembling Definition~\ref{dfr1}.

\begin{notation}\label{neps}
Grant Assumptions~\ref{as} and~\ref{assump_equi}, fix $\e\in (0,1]$ and $c_\e>0$ to be chosen later. 
Recall that $\rF_\e$ and $\rG_\e$ were introduced
in Notation~\ref{fgh}. Pick $O_\e\sim \rG_\e(v)\dr v$, independent of a Poisson measure
$\rK_\e$ on $\R_+\times \R^d$ with intensity $\e^{-1}\dr s \rF_\e(u)\dr u$. Set 
$$
Y^\e_t=O_\e +\int_0^t \int_{\R^d} u \rK_\e(\dr s,\dr u)\quad \text{and} 
\quad \ell(Y^\e)=\inf\{t> 0 : Y^\e_t \notin \HH\}.
$$
We denote by 
$\nn^\e=\frac{c_\e}\e \rm{Law}((Y^\e_{t\land \ell(Y^\e)})_{t\geq 0})$,
which is a measure on $\cE$ with mass $\frac{c_\e}\e$.
\end{notation}

Recall that $h$, $g$, $\cl$ and the excursion measures $\nn_*$ and $\nn_\beta$
were introduced in Subsection~\ref{nono}.

\begin{lemma}\label{repsexc}
Grant Assumptions~\ref{as} and~\ref{assump_equi}, fix $\e\in (0,1]$ and $c_\e>0$. Consider $x\in \pDd$, as well as
a measurable family $(A_y)_{y \in \pDd}$ such that $A_y \in \cI_y$ for each $y \in \pDd$,
and a Poisson measure $\Pi^\e=\sum_{u \in \JJ_\e}\delta_{(u,e^\e_u)}$ 
on $\R_+\times \cE$ with intensity $\dr u \nn^\e(\dr e)$. The following equations
have pathwise-unique solutions 
\begin{gather}
b_u^\e = x + \int_0^u\int_{\mathcal{E}}[g_{b_\vm^\e}(A_{b_\sm^\e}, e) - b_\vm^\e]\Pi_\e(\dr v, \dr e),\label{sdebe} \\
\tau_u^\e =c_\e u + \int_0^u\int_{\mathcal{E}}\cl_{b_\vm^\e}(A_{b_\sm^\e},e)\Pi_\e(\dr v, \dr e),\label{sdete}
\end{gather}
and $(b^\e_u)_{u\geq 0}$ is valued in $\pDd$, while $(\tau_u^\e)_{u\geq 0}$ is increasing, valued in $\R_+$ and
$\lim_{u\to \infty} \tau^\e_u=\infty$. 
Let $(L_t^\e=\inf\{u\geq 0 : \tau^\e_u>t\})_{t\geq0}$  and set 
\begin{equation}\label{def:reflected_process_eps}
R_t^\e = 
\begin{cases}
h_{b^\e_{L^\e_t-}}(A_{b^\e_{L^\e_t-}},e^\e_{L_t^\e}(t-\tau^\e_{L_t^\e-})) & \text{if }   \tau^\e_{L_t^\e}>t,\\
g_{b_{L_t^\e -}^\e}(A_{b^\e_{L^\e_t-}},e^\e_{L_t^\e})&  \text{if } L^\e_t \in \JJ_\e \text{ and } \tau^\e_{L_t^\e}=t,\\
b_{L_t^\e}^\e &  \text{if } L^\e_t \notin \JJ_\e.
\end{cases}
\end{equation}
Then $(R^\e_t)_{t\geq 0}$ has the same law as the process built in Remark~\ref{repswd} 
(issued from $x\in \pDd$).
\end{lemma}
Existence and pathwise uniqueness for~\eqref{sdebe}-\eqref{sdete} are straightforward, since 
$\Pi_\e$ is finite of $[0,T]\times \cE$ for all $T>0$.
We clearly have $\lim_{u\to \infty} \tau^\e_u=\infty$ 
(because $\tau^\e_u\geq c_\e u$), so that $(L^\e_t)_{t\geq 0}$ is also well-defined. The process 
$(R^\e_t)_{t\geq 0}$ introduced in~\eqref{def:reflected_process_eps} is thus uniquely defined.

\begin{proof}
We fix $\e>0$ and consider the process $(R^\e_t)_{t\geq 0}$ starting at $x\in\pDd$
built in Remark~\ref{repswd} with some Poisson
measure $\rM_\e=\sum_{s\in \II_\e}\delta_{(s,v_s,w_s)}$ 
on $\R_+\times \R^d\times \HH$ with intensity $\e^{-1} \dr s \rH_\e(v,w)\dr v \dr w$:
\begin{equation}\label{rere}
R^\e_t=x+\int_0^t\int_{\R^d\times \HH} [\Lambda(R_\sm^\e, R_\sm^\e+v\indiq_{\{R_\sm^\e \in \Dd\}}+ 
A_{R_\sm^\e}w\indiq_{\{R_\sm^\e \in \partial\Dd\}}) - R_\sm^\e]\rM_\e(\dr s,\dr v,\dr w).
\end{equation}
Our goal is to 
build a Poisson measure $\Pi_\e$ as in the statement such that $(R^\e_t)_{t\geq 0}$ 
satisfies~\eqref{def:reflected_process_eps}.  Let $(\cF^\e_t)_{t\geq 0}$ be the canonical filtration
of $\rM_\e$.
\vip
{\it Step 1.} We first build the (extended) excursions $(\bar e_n)_{n\geq 1}$ of $(R^\e_t)_{t\geq 0}$ outside $\pDd$: 
we introduce the sequences of 
stopping times $(\sigma_n)_{n \geq 0}$ and $(\rho_n)_{n\geq 1}$ defined by $\sigma_0=0$
and, for any $n\geq 1$,
$$
\rho_{n} = \inf\{t\geq\sigma_{n-1}, \: \rM_\e([\sigma_{n-1}, t] \times \mathbb{R}^d \times \HH) >0\}
\quad \text{and} \quad 
\sigma_{n} = \inf\{t \geq\rho_{n}, \: R^\e_t \in \partial\Dd\}.
$$
We may have $\sigma_n=\rho_n$ if $R^\e$ jumps from $\pDd$ to $\pDd$ at time $\rho_n$
({\it i.e.} if $R^\e_{\sigma_{n-1}}+A_{R^\e_{\sigma_{n-1}}}w_{\rho_n} \notin \Dd$).
We define $(\bar e_n(t))_{t\in [0,\sigma_n-\rho_n]}$ for each $n\geq 1$ by setting, for all 
$t\in [0,\sigma_n-\rho_n]$, 
$$
\bar e_n(t)=R^\e_{\sigma_{n-1}}+A_{R^\e_{\sigma_{n-1}}}w_{\rho_n} + 
\int_{(\rho_n,\rho_n+t]}\int_{\R^d\times \HH} v \rM_\e(\dr s,\dr v,\dr w).
$$

We now check that for all $n\geq 1$,
\vip

\noindent (i) if $\sigma_n=\rho_n$, then $R^\e_{\rho_n}=\Lambda(R^\e_{\sigma_{n-1}},\bar e_n(0))$ and
$\bar e_n(0) \notin \Dd$,
\vip
\noindent (ii) if $\sigma_n>\rho_n$, then $R^\e_{t+\rho_n}=\bar e_n(t)$ for $t\in [0,\sigma_n-\rho_n)$,
while $\bar e_n(\sigma_n-\rho_n)\notin \Dd$
and $R^\e_{\sigma_n}=\Lambda(\bar e_n((\sigma_n-\rho_n)-),\bar e_n(\sigma_n-\rho_n))$.
\vip
For (i), we have $R^\e_{\rho_n}= \Lambda(R^\e_{\rho_n-},R^\e_{\rho_n-}+A_{R^\e_{\rho_n-}}w_{\rho_n})$ by~\eqref{rere} and since
$R^\e_{\rho_n-}=R^\e_{\sigma_{n-1}} \in \pDd$. Thus 
$R^\e_{\rho_n}=\Lambda(R^\e_{\sigma_{n-1}},R^\e_{\sigma_{n-1}}+A_{R^\e_{\sigma_{n-1}}}w_{\rho_n})=\Lambda(R^\e_{\sigma_{n-1}},\bar e_n(0))$.
We have $\bar e_n(0) \notin \Dd$, because else we would have $R^\e_{\sigma_n}=R^\e_{\rho_n}=\bar e_n(0)\in \Dd$.

\vip
For (ii), we have $R^\e_{\rho_n}= R^\e_{\rho_n-}+A_{R^\e_{\rho_n-}}w_{\rho_n}$ by~\eqref{rere} and since 
$R^\e_{\rho_n-}=R^\e_{\sigma_{n-1}} \in \pDd$ and $R^\e_{\rho_n}\in \Dd$ (because $\sigma_n>\rho_n$). Thus
$R^\e_{\rho_n}= R^\e_{\sigma_{n-1}}+A_{R^\e_{\sigma_{n-1}}}w_{\rho_n}= \bar e_n(0)$. Next, by definition of $\sigma_n$, we have 
$R^\e_\sm,R^\e_s \in \Dd$ for all 
$s \in \II_\e\cap(\rho_n,\sigma_n)$, so that 
$$
\Lambda(R_\sm^\e, R_\sm^\e+v_s\indiq_{\{R_\sm^\e \in \Dd\}}+ A_{R_\sm^\e}w_s\indiq_{\{R_\sm^\e \in \partial\Dd\}})-R_\sm^\e
=\Lambda(R_\sm^\e, R_\sm^\e+v_s)-R^\e_\sm=v_s, 
$$
and thus for $t\in (0,\sigma_n-\rho_n)$, recalling~\eqref{rere},
$$
R^\e_{\rho_n+t}=R^\e_{\rho_n}+\int_{(\rho_n,\rho_n+t]}\int_{\R^d\times \HH} v \rM_\e(\dr s,\dr v,\dr w)=\bar e_n(t).
$$
Finally, $R^\e_{\sigma_n}=\Lambda(R^\e_{\sigma_n-},R^\e_{\sigma_n-}+v_{\sigma_n})$ because $R^\e_{\sigma_n-} \in \Dd$.
But $R^\e_{\sigma_n-}=\bar e_n((\sigma_n-\rho_n)-)$ from the above discussion, while 
$R^\e_{\sigma_n-}+v_{\sigma_n}=\bar e_n((\sigma_n-\rho_n)-)+v_{\sigma_n}=\bar e_n(\sigma_n-\rho_n)$ 
by definition of $\bar e_n$. Thus $R^\e_{\sigma_n}=\Lambda(\bar e_n((\sigma_n-\rho_n)-),\bar e_n(\sigma_n-\rho_n))$.
Notice that $\bar e_n(\sigma_n-\rho_n)\notin \Dd$ because else, we would have 
$R_{\sigma_n}=\bar e_n(\sigma_n-\rho_n)\in \Dd$.

\vip

{\it Step 2.} We now rotate/translate/complete the excursions $\bar e_n$ to get some i.i.d. excursions
outside the half-space. We recall that the marginals of $\rH_\e$ are $\rF_\e$ and $\rG_\e$, see Notation~\ref{fgh}.

\vip

We consider an i.i.d. sequence of Poisson measures
$(\mathrm{K}_n)_{n\geq0}$ on $\mathbb{R}_+ \times \mathbb{R}^d$ with intensity $\e^{-1}\dr s\mathrm{F}_\e(\dr z)$, 
independent of everything else. We define $\cG_n=\sigma(\mathrm{K}_n)$ and, for all $n\geq 1$, all $t\geq 0$,
$$
Y^n_t=w_{\rho_n} + \int_{(\rho_n,(\rho_n+t)\land \sigma_n]}\int_{\R^d\times \HH} A_{R^\e_{\sigma_{n-1}}}^{-1} v \rM_\e(\dr s,\dr v,\dr w)
+\indiq_{\{t>\sigma_n-\rho_n\}} \int_0^{t-\sigma_n} v \mathrm{K}_n(\dr s,\dr v).
$$
Since $w_{\rho_n}\sim \rG_\e$ by definition of $\rho_n$ and 
since $\rF_\e$ is rotationally invariant, the process $(Y^n_t)_{t\geq 0}$ has
the same law $\mu_\e$ as $(Y^\e_t)_{t\geq 0}$ introduced in Notation~\ref{neps}.
Moreover, $(Y^n_t)_{t\geq 0}$ 
is $\cF_{\sigma_n}\lor\cG_n$-measurable and independent of $\cF_{\rho_n-}$.

\vip

We introduce $\ell(Y^n)=\inf\{t> 0 : Y^n_t \notin \HH\}$. It holds that $\ell(Y^n)\geq \sigma_n-\rho_n$, 
because for all $t\in [0,\sigma_n-\rho_n)$, $Y^n_t=A_{R^\e_{\sigma_{n-1}}}^{-1}(\bar e_n(t)-R^\e_{\sigma_{n-1}})$
with $\bar e_n(t)=R^\e_{\rho_n+t} \in \Dd$ by Step~1-(ii) and because for all $y\in \pDd$, all $z\in \Dd$,
$A_y^{-1}(z-y) \in \HH$ by convexity of $\Dd$.

\vip

We introduce $e_n(t)= Y^n_{t\land \ell(Y^n)}$. Again, $(e_n(t))_{t\geq 0}$
is $\cF_{\sigma_n}\lor\cG_n$-measurable and independent of $\cF_{\rho_n-}$ (and thus of $\cF_{\sigma_{n-1}}$)  
and has the same law $\nu_\e$
as  $(Y^\e_{t\land \ell(Y^\e)})_{t\geq 0}$ introduced in Notation~\ref{neps}.
\vip
For each $n\geq 0$, by definition of $\sigma_{n-1}$ and of $\rho_n$ and since $\rM_\e$ is Poisson, 
$\rho_n-\sigma_{n-1}$ is Exp$(\e^{-1})$-distributed, $\cF_{\rho_n-}$-measurable (and thus $\cF_{\sigma_n}$-measurable) and
independent of $\cF_{\sigma_{n-1}}$.
\vip
For each $n\geq 2$,  $((e_n(t))_{t\geq 0},\rho_n-\sigma_{n-1})$ is independent of $\cF_{\sigma_{n-1}}$ and
thus of the vector $((e_1(t))_{t\geq 0},\rho_1-\sigma_0, \ldots, (e_{n-1}(t))_{t\geq 0},\rho_{n-1}-\sigma_{n-2})$. 
Furthermore,
$(e_n(t))_{t\geq 0}$ is independent of $\cF_{\rho_n-}$ and thus of $\rho_n-\sigma_{n-1}$.

\vip
All in all, the whole family $((e_n(t))_{t\geq0}, \rho_n - \sigma_{n-1})_{n\geq 1}$ 
is independent and for each $n\geq1$, $(e_n(t))_{t\geq 0}\sim \nu_\e$
and $\rho_n-\sigma_{n-1}\sim \text{Exp}(\e^{-1})$.

\vip

{\it Step 3.} We set $T_0=0$ and, for $n\geq1$,
$T_n = c_\e^{-1}\sum_{k=1}^n(\rho_k - \sigma_{k-1})$. By Step~2 and since e.g. 
$T_1=c_{\e}^{-1}(\rho_1 - \sigma_{0})\sim \text{Exp}(\frac{c_\e}\e)$, the measure
\[
\Pi_\e = \sum_{n\geq1}\delta_{(T_n, e_n)}
\]
is Poisson on $\mathbb{R}_+ \times \mathcal{E}$ with intensity $\dr u  \nn_\e(\dr e)$ 
(because $\nn_\e=\frac{c_\e}\e\nu_\e$, see Notation~\ref{neps}), with
set of jump times $\JJ_\e=\{T_n : n\geq 1\}$.
Note that we may write $\Pi_\e=\sum_{s \in \JJ_\e}\delta_{(s,e^\e_s)}$ as in the statement by setting
$e^\e_{T_n}=e_n$ for each $n\geq 1$.

\vip
For all $n\geq1$,  $\bar e_n(t)=R^\e_{\sigma_{n-1}}+ A_{R^\e_{\sigma_{n-1}}}e_n(t)$ for all $t\in [0,\sigma_n-\rho_n]$:
since $\ell(Y^n)\geq \sigma_n-\rho_n$,
\begin{align*}
R^\e_{\sigma_{n-1}}+ A_{R^\e_{\sigma_{n-1}}}e_n(t)=&R^\e_{\sigma_{n-1}}+ A_{R^\e_{\sigma_{n-1}}}Y^n_t\\
=&R^\e_{\sigma_{n-1}}+ A_{R^\e_{\sigma_{n-1}}}w_{\rho_n} 
+\int_{(\rho_n,\rho_n+t]}\int_{\R^d\times \HH} v \rM_\e(\dr s,\dr v,\dr w)=\bar e_n(t).
\end{align*}

{\it Step 4.} We now check that for all $n\geq 1$, we have 
\vip
\noindent (i) $g_{R^\e_{\sigma_{n-1}}}(A_{R^\e_{\sigma_{n-1}}},e_{n})=R^\e_{\sigma_n}$, 
\vip
\noindent (ii) $\cl_{R^\e_{\sigma_{n-1}}}(A_{R^\e_{\sigma_{n-1}}},e_{n})=\sigma_{n}-\rho_n$,
\vip
\noindent (iii) $R^\e_t=h_{R^\e_{\sigma_{n-1}}}(A_{R^\e_{\sigma_{n-1}}},e_n(t-\rho_n))$ for all $t\in [\rho_n,\sigma_n)$ 
if $\sigma_n>\rho_n$.

\vip

We start with (iii), writing, for $t \in [\rho_n,\sigma_n)$,
\begin{align*}
h_{R^\e_{\sigma_{n-1}}}(A_{R^\e_{\sigma_{n-1}}},e_n(t-\rho_n))=R^\e_{\sigma_{n-1}}+A_{R^\e_{\sigma_{n-1}}}e_n(t-\rho_n)
=\bar e_n(t-\rho_n)=R^\e_t.
\end{align*}
We used the definition of $h$, the link between $e_n$ and $\bar e_n$ (see Step~3) and Step~1-(ii).

\vip

We next check (ii) when $\sigma_n>\rho_n$. In this case, $R^\e_t \in \Dd$ for all $t\in [\rho_n,\sigma_n)$ 
by definition of $\rho_n$ and $\sigma_n$, so that
(iii) implies that $\cl_{R^\e_{\sigma_{n-1}}}(A_{R^\e_{\sigma_{n-1}}},e_{n})\geq \sigma_{n}-\rho_n$, and it remains to verify
that $h_{R^\e_{\sigma_{n-1}}}(A_{R^\e_{\sigma_{n-1}}},e_{n}(\sigma_n-\rho_n)) \notin \Dd$. By definition of $h$ and 
by the link between $e_n$ and $\bar e_n$, we have 
$h_{R^\e_{\sigma_{n-1}}}(A_{R^\e_{\sigma_{n-1}}},e_{n}(\sigma_n-\rho_n))=\bar e_n(\sigma_n-\rho_n) \notin \Dd$ by see Step 1-(ii).

\vip

We now prove (i) when $\sigma_n>\rho_n$: it follows from the definition of $g$ and point (ii) that 
\begin{align*}
g_{R^\e_{\sigma_{n-1}}}(A_{R^\e_{\sigma_{n-1}}},e_{n})=&
\Lambda(h_{R^\e_{\sigma_{n-1}}}(A_{R^\e_{\sigma_{n-1}}},e_{n}((\sigma_n-\rho_n)-),
h_{R^\e_{\sigma_{n-1}}}(A_{R^\e_{\sigma_{n-1}}},e_{n}(\sigma_n-\rho_n)))\\
=& \Lambda(\bar e_n((\sigma_n-\rho_n)-),\bar e_n(\sigma_n-\rho_n)),
\end{align*}
which equals $R^\e_{\sigma_n}$ by Step~1-(ii).

\vip

We finally check (i) and (ii) when $\sigma_n=\rho_n$. We have $\cl_{R^\e_{\sigma_{n-1}}}(A_{R^\e_{\sigma_{n-1}}},e_{n})=0$, because
$h_{R^\e_{\sigma_{n-1}}}(A_{R^\e_{\sigma_{n-1}}},e_{n}(0))=\bar e_n(0) \notin \Dd$, see Step~1-(i).
This also implies that $g_{R^\e_{\sigma_{n-1}}}(A_{R^\e_{\sigma_{n-1}}},e_{n})=\Lambda(R^\e_{\sigma_{n-1}},\bar e_n(0))$,
which equals $R^\e_{\rho_n}$ by Step~1-(i).

\vip

{\it Step 5.} We consider $(b^\e_u)_{u\geq 0}$, $(\tau^\e_u)_{u\geq 0}$, $(L^\e_t)_{t\geq 0}$ as in the statement, defined
with $\Pi_\e$ introduced in Step~3, and we show that $(R^\e_t)_{t\geq 0}$ satisfies~\eqref{def:reflected_process_eps}.

\vip

{\it Step 5.1.} We first check by induction that for all $n\geq 0$, we have (convention: $\sum_1^{0}=0$)
\begin{equation}\label{tdrec}
\text{for all} \; u \in [T_n,T_{n+1}),\quad b^\e_u=R^\e_{\sigma_n} \quad \text{and} \quad
\tau^\e_u=c_\e u+\sum_{k=1}^{n} (\sigma_{k}-\rho_k).
\end{equation}
Recalling~\eqref{sdebe}-\eqref{sdete},
we have $b^\e_u=x=R^\e_0=R^\e_{\sigma_0}$ and $\tau^\e_u=c_\e u$ for all $u \in [0,T_1)$,
which shows~\eqref{tdrec} when $n=0$.
If now~\eqref{tdrec} holds true for some $n\geq 0$, then by~\eqref{sdebe}, for all 
$u \in [T_{n+1},T_{n+2})$,
$$
b^\e_u=b^\e_{T_{n+1}}=g_{b^\e_{T_n}}(A_{b^\e_{T_n}},e_{n+1})=g_{R^\e_{\sigma_n}}(A_{R^\e_{\sigma_n}},e_{n+1})=R^\e_{\sigma_{n+1}}
$$ 
by the induction hypothesis and Step~4. Moreover, by \eqref{sdete}, for all $u \in [T_{n+1},T_{n+2})$,
$$
\tau^\e_u=\tau^\e_{T_{n+1}}+c_\e(u-T_{n+1})=\tau^\e_{T_{n+1}-}+\cl_{b^\e_{T_{n}}}(A_{b^\e_{T_{n}}},e_{n+1})
+c_\e(u-T_{n+1}).
$$
By the induction hypothesis and since $\cl_{b^\e_{T_{n}}}(A_{b^\e_{T_{n}}},e_{n+1})=
\cl_{R^\e_{\sigma_n}}(A_{R^\e_{\sigma_n}},e_{n+1})=\sigma_{n+1}-\rho_{n+1}$ by Step~4, this gives
$$
\tau^\e_u=c_\e T_{n+1}+\sum_{k=1}^{n} (\sigma_{k}-\rho_k) + (\sigma_{n+1}-\rho_{n+1})+c_\e(u-T_{n+1})
=c_\e u + \sum_{k=1}^{n+1} (\sigma_{k}-\rho_k).
$$

{\it Step 5.2.} We now check, using~\eqref{tdrec} and recalling that 
$T_n = c_\e^{-1}\sum_{k=1}^n(\rho_k - \sigma_{k-1})$ and that
$(L_t^\e)_{t\geq0}$ is the right-continuous inverse of $(\tau_u^\e)_{u\geq0}$, that
for all $n\geq 1$,
\begin{equation}\label{ftl}
L_t^\e = T_{n-1}+c_\e^{-1}(t-\sigma_{n-1}) \quad \text{for all} \quad t\in [\sigma_{n-1},\rho_n) \quad \text{and} 
\quad L_t^\e = T_n \quad \text{for all} \quad t \in [\rho_n, \sigma_n).
\end{equation}

\noindent We have $\tau^\e_u=c_\e u$ for $u\in [0,T_1)$, whence $L^\e_t=c_\e^{-1}t=T_0+c_\e^{-1}(t-\sigma_0)$ 
for $t\in [0,c_\e T_1)=[\sigma_0,\rho_1)$.

\vip
\noindent We have  $\tau^\e_{T_1}=c_\e T_1+\sigma_1-\rho_1=\sigma_1$, whence $L^\e_t=T_1$ for $t\in [\rho_1,\sigma_1)$.
\vip
\noindent We have  $\tau^\e_u=c_\e u+\sigma_1-\rho_1$ for $u\in [T_1,T_2)$, whence
$L^\e_t=c_\e^{-1}(t-\sigma_1+\rho_1)=T_1+c_\e^{-1}(t-\sigma_1)$ for 
$t\in [c_\e T_1+\sigma_1-\rho_1,c_\e T_2+\sigma_1-\rho_1)=[\sigma_1,\rho_2)$.
\vip
\noindent We have  $\tau^\e_{T_2}=c_\e T_2+\sigma_1-\rho_1+\sigma_2-\rho_2=\sigma_2$, whence $L^\e_t=T_2$ 
for $t\in [\rho_2,\sigma_2)$.
\vip
\noindent Etc.

\vip
{\it Step 5.3.} We fix $t \geq 0$ and check~\eqref{def:reflected_process_eps}. If $t=0$,
then $L^\e_0=0 \notin\JJ_\e$ and $R^\e_0=x=b^\e_0$, which agrees with the third line 
of~\eqref{def:reflected_process_eps}. If $t>0$, there is $n\geq 1$
such that either $t\in(\sigma_{n-1},\rho_n)$ or $t\in[\rho_{n}, \sigma_{n})$ or $t=\sigma_n$.

\vip
\noindent $\bullet$ If $t\in(\sigma_{n-1},\rho_n)$, then $R^\e_t=R^\e_{\sigma_{n-1}}$ by definition of
$\rho_n$. Moreover, $L_t^\e =T_{n-1}+c_\e^{-1}(t-\sigma_{n-1}) \in (T_{n-1},T_n)$ by~\eqref{ftl}, 
so that $L^\e_t \notin \JJ_\e$ and
$b^\e_{L^\e_t}= R^\e_{\sigma_{n-1}}$ by~\eqref{tdrec}. Thus $R^\e_t=b^\e_{L^\e_t}$, which 
agrees with the third line 
of~\eqref{def:reflected_process_eps}.

\vip

\noindent $\bullet$ If $t \in[\rho_{n}, \sigma_{n})$ (which implies that $\sigma_n>\rho_n$),
then $L^\e_t=T_n \in \JJ_\e$ by~\eqref{ftl}, so that~\eqref{tdrec} tells us that $b^\e_{L^\e_t-}=R^\e_{\sigma_{n-1}}$,
that $\tau_{L^\e_t-}=c_\e T_n+\sum_{k=1}^{n-1} (\sigma_{k}-\rho_k)=\rho_n$ and finally that
$\tau_{L^\e_t}=c_\e T_n+\sum_{k=1}^{n} (\sigma_{k}-\rho_k)=\sigma_n$. 
Thus $\tau^\e_{L^\e_t}>t$, and to agree with the first line of~\eqref{def:reflected_process_eps},
we have to verify that $R^\e_t=h_{R^\e_{\sigma_{n-1}}}(A_{R^\e_{\sigma_{n-1}}},e_n(t-\rho_n))$. This has
been seen in Step 4.
\vip

\noindent $\bullet$ If $t=\sigma_{n}$, then as previously, $L^\e_t=T_n \in \JJ_\e$, $b^\e_{L^\e_t-}=R^\e_{\sigma_{n-1}}$,
$\tau_{L^\e_t-}=\rho_n$ and $\tau_{L^\e_t}=\sigma_n$. Thus $\tau^\e_{L^\e_t}=t$, and to agree with the second line 
of~\eqref{def:reflected_process_eps},
we have to verify that $R^\e_{\sigma_n}=g_{R^\e_{\sigma_{n-1}}}(A_{R^\e_{\sigma_{n-1}}},e_n)$. This has
been seen in Step 4.
\end{proof}

\subsection{Estimates on the excursion measures}

Here we give some crucial estimates on the measure $\nn^\e$ that will proved in Section~\ref{stech}.
Recall that $\ell:\cE\to \R_+$ and $M : \cE\to \R_+$ were defined in Subsection~\ref{nono}.
For $r>0$ and $e\in \cE$, we recall that $\ell_r(e)=\inf\{t>0 : e(t) \notin B_d(r\be_1,r)\}$. 
The constant $\chi_\rG$ is introduced in Definition~\ref{def:chig} when $\beta=*$ and we set
$\chi_\rG=1/(2\kappa_\rG \Gamma(\beta+1))$ when $\beta\in (0,\alpha/2)$, 
with $\kappa_\rG$ defined in Assumption~\ref{assump:moments}-(b).

\begin{proposition}\label{lemma:estimates}
Grant Assumption~\ref{assump_equi} with some $\alpha\in (0,2)$ and with $\kappa_{\rF}=1/\Gamma(\alpha+1)$ 
and either Assumption~\ref{assump:moments}-(a), in which case set $\beta=*$, $c_\e= \chi_{\rG} \, \e^{1/2}$ 
and $\theta_0=1/2$, or Assumption~\ref{assump:moments}-(b) 
with some $\beta\in (0,\alpha/2)$, in which case set $c_\e= \chi_{\rG} \,\e^{1-\beta / \alpha}$ and 
$\theta_0=\beta/\alpha$.
Consider $\nn^\e$ as in Notation~\ref{neps} and $\nn_\beta$ as in Subsection~\ref{nono}.
For $\delta\in (0,1]$, define 
\begin{gather}
\nn_*^\delta (\dr e)=\indiq_{\{\ell(e)>\delta\}}\nn_*(\dr e) \quad \text{and} \quad
\nn^{\e,\delta} (\dr e)=\indiq_{\{\ell(e)>\delta\}}\nn^\e(\dr e) \quad \text{under Assumption~\ref{assump:moments}-(a)},
\label{nd1}\\
\nn_\beta^\delta (\dr e)=\indiq_{\{|e(0)|>\delta\}}\nn_*(\dr e) \quad \text{and} \quad
\nn^{\e,\delta} (\dr e)=\indiq_{\{|e(0)|>\delta\}}\nn^\e(\dr e) \quad \text{under Assumption~\ref{assump:moments}-(b)}.
\label{nd2}
\end{gather}
For any $\delta \in (0,1]$, we have
\begin{equation}\label{aan}
\nn_\beta^\delta(\cE)+\sup_{\e\in (0,1]} \nn^{\e,\delta} (\cE)<\infty.
\end{equation}
For any $\theta \in (0,\theta_0)$, we have
\begin{gather}
\sup_{\e\in(0,1]}\int_{\mathcal{E}}[M(e) \wedge 1 + \ell(e) \wedge [\ell(e)]^\theta]\nn^\e(\dr e) < \infty, \label{j1s}\\
\lim_{\delta\to0} \limsup_{\e\to 0} \int_{\mathcal{E}}\big[M(e) \wedge 1 + \ell(e) \wedge 1\big]
(\nn^\e-\nn^{\e,\delta})(\dr e) = 0. \label{j2s}
\end{gather}
For all $\delta>0$, all $\phi:\cE\to \R$ bounded and continuous for the $\JS$-topology,
\begin{equation}\label{j3s}
\lim_{\e\to 0} \int_\cE \phi(e)  \nn^{\e,\delta} (\dr e) = \int_\cE \phi(e)
\nn_\beta^\delta (\dr e).
\end{equation}
If $\beta=*$, for all $\delta>0$, all $r>0$,
\begin{equation}\label{j4s}
\lim_{\eta\to 0}  \limsup_{\e\to 0} \nn^{\e,\delta}(\ell_r<\eta)=0.
\end{equation}
In any case (when $\beta=*$ or $\beta \in (0,\alpha/2)$), for all $r>0$,
\begin{equation}\label{j5s}
\lim_{\eta\to 0}\liminf_{\e\to 0}\nn^{\e}(\ell_r>\eta)=\infty.
\end{equation} 
When  $\beta \in (0,\alpha/2)$, for all $a>0$,
\begin{equation}\label{j6s}
\lim_{\delta\to 0} \limsup_{\e\to 0}\nn^{\e}(\{e \in \cE : |e(0)|\leq \delta, \ell(e)>a\})=0.
\end{equation}
\end{proposition}

\subsection{Some continuity properties}

We first verify that the exit time does not charge points (except possibly $0$ when $\beta\neq *$) 
under the excursion measure.

\begin{lemma}\label{clac}
Grant the same assumptions and notations as in Proposition~\ref{lemma:estimates} and suppose
Assumption~\ref{as}.
For all $\e\in(0,1]$, for all $t>0$, for all $x \in \pDd$, for all $A\in \cI_x$, we have
$\nn^\e(\{e \in \cE : \cl_x(A,e)=t\})=0$.
\end{lemma}

\begin{proof}
Recall that $\nn^\e=\frac{c_\e}{\e}\text{Law}((Y^\e_{t\land \ell(Y^\e)})_{t\geq 0})$, 
with $(Y^\e_t)_{t\geq 0}$ introduced in 
Notation~\ref{neps}. Let $S_\e$ be the set of jump times of the Poisson measure $\rK_\e$.
Since $(Y^\e_t)_{t\geq 0}$ is piecewise constant with set of jump times $S_\e$, we have
$\cl_x(A,(Y^\e_{t\land \ell(Y^\e)})_{t\geq 0}) \in \{0\}\cup S_\e$. But 
$\PP(t \in S_\e)=0$ for all $t\geq 0$.
\end{proof}

We next check some continuity properties.

\begin{lemma}\label{lemma:cont_g_ell}
Grant the same assumptions and notations as in Proposition~\ref{lemma:estimates} and suppose
Assumption~\ref{as}. Consider a measurable family $(A_y)_{y\in \pDd}$ such that
$A_y \in \cI_y$ for each $y\in \pDd$.
\vip
(i) Fix $\delta>0$, $z\in \pDd$ and a sequence $(z_n)_{n\geq 1}$ such that
$z_n \in \pDd$ and $\lim_n z_n =z$. For any continuous bounded function
$\varphi:\R^d\times \mathbb{R}_+\to \R$, any sequence $(\e_n)_{n\geq 1}$ decreasing to $0$,
$$
\lim_{n}\int_\cE \varphi(g_{z_n}(A_{z_n}, e), \cl_{z_n}(A_{z_n}, e)) \nn^{\e_n,\delta}(\dr e)
=\int_\cE \varphi(g_{z}(A_z, e), \cl_{z}(A_z, e)) \nn_\beta^{\delta}(\dr e).
$$

(ii)  Fix $\delta>0$, $z\in \pDd$ and a sequence $(z_n)_{n\geq 1}$ such that
$z_n \in \pDd$ and $\lim_n z_n =z$. Consider $s\geq 0$, $t\geq 0$ and $(s_n)_{n\geq 1}$, $(t_n)_{n\geq 1}$ 
such that $\lim_n s_n= s$ and $\lim_n t_n=t$. If $s \neq t$, then for 
any sequence $(\e_n)_{n\geq 1}$ decreasing to $0$, any continuous bounded function $\varphi:\R^d\to\R$,
\begin{align*}
\lim_{n}\int_\cE \varphi(h_{z_n}(A_{z_n},e(t_n-s_n))) &\indiq_{\{s_n<t_n<s_n+\cl_{z_n}(A_{z_n},e)\}} \nn^{\e_n,\delta}(\dr e)
\\
=&\int_\cE \varphi(h_{z}(A_{z},e(t-s))) \indiq_{\{s<t<s+\cl_{z}(A_{z},e)\}} \nn_\beta^\delta(\dr e).
\end{align*}

(iii) Consider a process $(Z^\e_t)_{t\geq 0}$ converging in law to an isotropic 
$\alpha$-stable process $(Z_t)_{t\geq 0}$ issued from $z\in \cDd$ for the $\JS$-topology. 
For each $t> 0$, as $\e\to 0$,
$$
\big(Z^\e_t\indiq_{\{t<\tell(Z^\e)\}},\Lambda(Z^\e_{\tell(Z^\e)-},Z^\e_{\tell(Z^\e)}),\tell(Z^\e)\big) 
\;\; \text{goes in law to} \;\;
\big(Z_t\indiq_{\{t<\tell(Z)\}},\Lambda(Z_{\tell(Z)-},Z_{\tell(Z)}),\tell(Z)\big),
$$
where we recall that $\tell(Z)=\inf\{t > 0 : Z_t \notin \Dd\}$.
\end{lemma}

\begin{proof}
We start with (i).
By invariance of $\nn^{\e,\delta}$ and $\nn_\beta^\delta$ by any isometry sending $\be_1$ to $\be_1$,
the integrals do not depend on the choice of $(A_{y})_{y\in \pDd}$. We thus may assume that $y\mapsto A_y$
continuous at $z$, see Lemma~\ref{locglob}, so that $A_{z_n}\to A_z$. We set $m_{n,\delta}=\nn^{\e_n,\delta}(\cE)$ and 
$m_\delta=\nn_\beta^{\delta}(\cE)$. By~\eqref{j3s} with $\phi=1$, we have $\lim_n m_{n,\delta}=m_\delta$.
Still by~\eqref{j3s}, the sequence of probability measures $(m_{n,\delta}^{-1}\nn^{\e_n,\delta})_{n\geq 1}$ on $\cE$ 
converges weakly to $m_\delta^{-1}\nn_\beta^{\delta}$. By Skorokhod's representation theorem, we can find
$\bw_n\sim m_{n,\delta}^{-1}\nn^{\e_n,\delta}$ a.s. converging, for the $\JS$-topology, to some 
$\bw\sim m_\delta^{-1}\nn_\beta^{\delta}$, and we are reduced to show that $\lim_n I_n=I$, where
$$
I_n=\E\Big[\varphi(g_{z_n}(A_{z_n},\bw_n),\cl_{z_n}(A_{z_n},\bw_n))\Big] \quad \text{and} \quad
I=\E\Big[\varphi(g_{z}(A_{z},\bw),\cl_{z}(A_{z},\bw))\Big].
$$
Consider a (random) sequence of time changes $\lambda_n$ such that $||\lambda_n-I||_\infty\to 0$
and such that $\bov_n:=\bw_n\circ \lambda_n$ converges to $\bw$ uniformly on compact intervals.
We have $g_{z_n}(A_{z_n},\bw_n)= g_{z_n}(A_{z_n},\bov_n)$ and $\cl_{z_n}(A_{z_n},\bw_n)=\lambda_n(\cl_{z_n}(A_{z_n},\bov_n))$.
Assume for a moment that $\lim_n q_n=0$, where
$$
q_n:=\PP(\cl_{z_n}(A_{z_n},\bov_n)\neq\cl_z(A_z,\bw)).
$$
Then $\cl_{z_n}(A_{z_n},\bw_n)$ of course goes to $\cl_z(A_z,\bw)$ in probability,
and 
\begin{align*}
&g_{z_n}(A_{z_n},\bw_n)=g_{z_n}(A_{z_n},\bov_n)=
\Lambda(z_n+A_{z_n}\bov_n(\cl_{z_n}(A_{z_n},\bov_n)-),z_n+A_{z_n}\bov_n(\cl_{z_n}(A_{z_n},\bov_n)))\\
&\hskip5cm\to \Lambda(z+A_{z}\bw(\cl_{z}(A_{z},\bw)-),z+A_{z}\bw(\cl_{z}(A_{z},\bw)))=g_z(A_z,\bw)
\end{align*}
in probability by continuity of $\Lambda$, see Lemma~\ref{Lambdacon}. We conclude that $\lim_n I_n=I$ 
as desired by dominated convergence.

\vip

We now check that $\lim_n q_n=0$ when $\beta=*$. Consider $r>0$ as in Remark~\ref{imp4} and recall that
$\cl_{y}(A,e)\geq \ell_r(e)$ for all $y\in \pDd$, all $A\in \cI_y$, all $e\in \cE$. 
For any $\eta>0$, we write 
$q_n\leq q_{n,1}(\eta)+q_2(\eta)+q_{n,3}(\eta)$, where
\begin{gather*}
q_{n,1}(\eta) = \PP(\ell_r(\bov_n)\leq\eta), \qquad q_2(\eta)= \PP(\ell_r(\bw)\leq\eta) \\
q_{n,3}(\eta)=\PP(\ell_{z_n}(A_{z_n},\bov_n)> \eta,\cl_z(A_z,\bw)> \eta, \cl_{z_n}(A_{z_n},\bov_n)
\neq\cl_z(A_z,\bw) ).
\end{gather*}
By Remark~\ref{imp4}, we have $\lim_{\eta\to 0} q_{2}(\eta)=0$. By~\eqref{j4s}, we have
$\lim_{\eta\to 0} \limsup_n\PP(\ell_r(\bw_n)\leq\eta)=0$, which implies that 
$\lim_{\eta\to 0} \limsup_n q_{n,1}(\eta)=0$. It thus 
suffices to prove that
for each $\eta>0$, $\lim_n q_{n,3}(\eta)=0$. By~\eqref{eer} and the Markov property, see Lemma~\ref{mark}, we know
that a.s., on the event $\{\cl_z(A_z,\bw)> \eta\}$,
$$
\inf\{d(h_z(A_z,\bw(t)),\Dd^c), t\in [\eta,\cl_{z}(A_{z},\bw))\}>0
\quad \text{and} \quad \bw(\cl_{z}(A_{z},\bw)) \in \cDd^c.
$$
Since $\bov_n$ a.s. converges locally uniformly to $\bw$, we deduce that 
$h_{z_n}(A_{z_n},\bov_n)$ converges locally uniformly to $h_z(A_z,\bw)$, so that
a.s. on $\{\cl_z(A_z,\bw)> \eta\}$, for all $n$ large enough,
$$
\inf\{d(h_{z_n}(A_z,\bov_n(t)),\Dd^c), t\in [\eta,\cl_{z}(A_{z},\bw))\}>0
\quad \text{and} \quad \bov_n(\cl_{z}(A_{z},\bw)) \in \cDd^c.
$$
Thus a.s. on $\{\cl_z(A_z,\bw)> \eta\}$, we have $\cl_{z_n}(A_{z_n},\bov_n) \in [0,\eta)\cup\{\cl_{z}(A_{z},\bw)\}$ 
for all $n$ large enough.
This implies that $\lim_n q_{n,3}(\eta)=0$.

\vip
We assume that $\beta\in (0,\alpha/2)$ and check that $\lim_n q_n=0$, and more precisely that a.s.,
$\cl_{z_n}(A_{z_n},\bov_n)=\cl_z(A_z,\bw)$ for all $n$ large enough.
We recall from~\eqref{nnb} that
$\bw(0)$ has a density (namely, $\bw(0)\sim m_\delta^{-1}\indiq_{\{x \in \HH,|x|>\delta\}}|x|^{-d-\beta}\dr x$), so that
$h_z(A_z,\bw(0)) \notin \pDd$ a.s.

\vip
On $\{h_z(A_z,\bw(0))\in \cDd^c\}$, we a.s. have 
$h_{z_n}(A_{z_n},\bov_n(0)) \in \cDd^c$ for all $n$ large enough, so that
$\cl_{z_n}(A_{z_n},\bov_n)=\cl_{z}(A_{z},\bw)=0$ for all $n$ large enough.
\vip
On $\{h_z(A_z,\bw(0))\in \Dd\}$, we a.s. have, by~\eqref{eer}, recalling~\eqref{nnb} and that
$\bw \sim m_\delta^{-1}\nn_\beta^\delta$,
$$
\inf\{d(h_z(A_z,\bw(t)),\Dd^c), t\in [0,\cl_{z}(A_{z},\bw))\}>0
\quad \text{and} \quad \bw(\cl_{z}(A_{z},\bw)) \in \cDd^c.
$$
We conclude as in the case $\beta=*$ (but directly with $\eta=0$) that on $\{h_z(A_z,\bw(0))\in \Dd\}$, a.s., 
for all $n$ large enough, $\cl_{z_n}(A_{z_n},\bov_n)=\cl_{z}(A_{z},\bw)$.

\vip

For (ii), we use the same notation as previously: we have to check that
$\lim_n J_n=J$, where
\begin{gather*}
J_n=\E\Big[\varphi(h_{z_n}(A_{z_n},\bw_n(t_n-s_n)))\indiq_{\{s_n<t_n<s_n+\cl_{z_n}(A_{z_n},\bw_n)\}}\Big],\\
J=\E\Big[\varphi(h_{z}(A_{z},\bw(t-s)))\indiq_{\{s<t<s+\cl_{z}(A_{z},\bw)\}}\Big].
\end{gather*}
We already have seen that $h_{z_n}(A_{z_n},\bw_n)$ a.s. converges to $h_{z}(A_{z},\bw)$ 
for the $\JS$-topology and that $\cl_{z_n}(A_{z_n},\bw_n)$
converges to $\cl_{z}(A_{z},\bw)$ in probability. Since $\bw$ a.s. has no jump at time $t-s$ and since $t_n-s_n$
goes to $t-s$, we deduce that $h_{z_n}(A_{z_n},\bw_n(t_n-s_n))$ goes to $h_{z}(A_{z},\bw(t-s))$ in probability.
Since $s\neq t$, we have $\indiq_{\{s_n<t_n\}}\to \indiq_{\{s<t\}}$. Since finally $\cl_{z}(A_{z},\bw)\neq t-s$ a.s.
by Lemma~\ref{clac}, we have $\indiq_{\{t_n<s_n+\cl_{z_n}(A_{z_n},\bw_n)\}}\to\indiq_{\{t<s+\cl_{z}(A_{z},\bw)\}}$ in probability.
The conclusion follows from dominated convergence.
\vip
We check (iii) when $z \in \pDd$. By Lemma~\ref{ai},
$(Z_{t}\indiq_{\{t<\tell(Z)\}},\Lambda(Z_{\tell(Z)-},Z_{\tell(Z)}),\tell(Z))=(0,z,0)$.
By the Skorokhod representation theorem, we may assume that
$(Z^\e_t)_{t\geq 0}$ a.s. goes to $(Z_t)_{t\geq 0}$ for the $\JS$-topology. 
Let us now show that a.s., $\tell(Z^\e)\to 0$. Since $\Dd\subset \{u\in\mathbb{R}^d, (u-z)\cdot\bn_z > 0\}$
by convexity, it suffices to show that $\lim_{\e\to 0} \inf_{s\in[0,t]}(Z^\e_s-z) \cdot \bn_z<0$ a.s. for every $t>0$.
This follows from the facts that  $(Z^\e_t)_{t\geq 0}$ a.s. goes to $(Z_t)_{t\geq 0}$ for the $\JS$-topology
and that for any $t > 0$, $\inf_{s\in[0,t]}(Z_s-z) \cdot \bn_z <0$, see Lemma~\ref{ai}.
This implies that $\indiq_{\{t<\tell(Z^\e)\}}\to 0$ (because $t>0$)
and that both $Z^\e_{\tell(Z^\e)-}$ and $Z^\e_{\tell(Z^\e)}$ a.s. go to $z$, so that 
$\Lambda(Z_{\tell(Z)-},Z_{\tell(Z)})\to \Lambda(z,z)=z$ by Lemma~\ref{Lambdacon}.
\vip

We finally check (iii) when $z \in \Dd$. By the Skorokhod representation theorem, we may assume that
$(Z^\e_t)_{t\geq 0}$ a.s. goes to $(Z_t)_{t\geq 0}$ for the $\JS$-topology.   
Since $Z_0=z \in \Dd$, we know from~\eqref{eer} that
$$
\inf\{d(Z_t, \Dd^c), t\in [0,\tell(Z))\}>0
\quad \text{and} \quad Z_{\tell(Z)} \in \cDd^c,
$$
from which we deduce as in the proof of (i) that a.s.,
$$
\lim_{\e\to 0}\tell(Z^\e)=\tell(Z) \quad \text{and} \quad \lim_{\e\to 0}\Lambda(Z^\e_{\tell(Z^\e)-},Z^\e_{\tell(Z^\e)})
=\Lambda(Z_{\tell(Z)-},Z_{\tell(Z)}).
$$
Since $\PP(\tell(Z)=t)=0$ by Bogdan, Jastrz\c{e}bski, Kassmann, Kijaczko and 
Pop\l{}awski~\cite[Theorem~1.3]{bjkkp}, 
we conclude that $\indiq_{\{t<\tell(Z^\e)\}}$ a.s. goes to $\indiq_{\{t<\tell(Z)\}}$. Finally,  $Z^\e_t$ a.s. goes to $Z_t$ because $Z$ a.s. has no jump at time $t$.
\end{proof}

\subsection{Convergence of the boundary process}

In this subsection, we prove the convergence in law of the process $(b_t^\e, \tau_t^\e)_{t\geq0}$ towards 
the limiting boundary process $(b_t, \tau_t)_{t\geq0}$, through
tightness/martingale problems arguments. 

\begin{proposition}\label{convbord}
Grant the same assumptions and notations as in Proposition~\ref{lemma:estimates} and suppose
Assumption~\ref{as}. Consider a measurable family $(A_y)_{y\in \pDd}$ such that
$A_y \in \cI_y$ for each $y\in \pDd$.
Consider for each $\e \in (0,1]$ some $x_\e\in\partial\Dd$ and the solution $(b^\e_u,\tau^\e_u)_{u\geq 0}$
to~\eqref{sdebe}-\eqref{sdete} with $b^\e_0=x_\e$.
If $\lim_{\e\to 0} x_\e=x \in \pDd$, 
then $(b^\e_u,\tau^\e_u)_{u\geq 0}$ converges in law to $(b_u, \tau_u)_{u\geq0}$ in 
$\mathcal{D}(\mathbb{R}_+, \mathbb{R}^{d} \times \mathbb{R}_+)$ for the $\JS$-topology,
where $(b_u, \tau_u)_{u\geq0}$ is as in Definition~\ref{dfr1} with $b_0=x$.
\end{proposition}

\begin{proof} We divide the proof into five steps.

\vip
 \textit{Step 1.} We first show that the family of processes $((b^\e_u,\tau^\e_u)_{u\geq 0}, \e \in (0,1])$ is tight 
in $\mathcal{D}(\mathbb{R}_+, \mathbb{R}^{d} \times \mathbb{R}_+)$ for the $\JS$-topology
and that any limit point  $(b_u,\tau_u)_{u\geq 0}$ satisfies $\Delta b_u=\Delta\tau_u=0$ a.s. for every 
deterministic $u\geq 0$. To this end,
we apply the Aldous criterion, see for instance 
Jacod and Shiryaev~\cite[Section~VI, Thm.~4.5 and Remark~4.7]{jacod2013limit}.
It suffices that for all $T>0$,
\vip
\noindent (i) $\lim_{A\to \infty} \sup_{\e\in [0,1]}\PP[\sup_{[0,T]} (|b^\e_u|+\tau^\e_u)>A]=0$,
\vip
\noindent (ii) $\lim_{\delta\to 0} \sup_{\e\in (0,1]}\sup_{(S,S')\in \cA^\e_{T,\delta}}\E[(|b^\e_{S'}-b^\e_S|+
\tau^\e_{S'}-\tau^\e_S)\land 1]=0$, where $\cA^\e_{T,\delta}$ is the set of pairs of stopping times $(S,S')$
in the filtration of $(b^\e_u,\tau^\e_u)_{u\geq 0}$
such that $0\leq S \leq S'\leq S+\delta\leq T$.

\vip

We start with (i). Since for any $u\geq 0$, $b_u^\e \in \pDd$ and since $\Dd$ is bounded, 
we immediately deduce that $\sup_{\e \in (0,1]}\E[\sup_{[0,T]} |b^\e_u| ]<\infty$. 
Next, recalling~\eqref{sdete} and that $\cl_y(A,e)\leq \ell(e)$, we find that
$\sup_{[0,T]} \tau^\e_u \leq c_\e T + I_{\e,T} + J_{\e,T}\leq \chi_\rG T+I_{\e,T} + J_{\e,T}$, where 
$$
I_{\e,T}=\int_0^T \int_\cE \ell(e) \indiq_{\{\ell(e) \leq 1\}} \Pi^\e(\dr v, \dr e) \quad
\text{and} \quad J_{\e,T}=\int_0^T \int_\cE \ell(e) \indiq_{\{\ell(e) > 1\}} \Pi^\e(\dr v, \dr e).
$$
Recall that $\theta_0\in (0,1/2]$ was defined in Proposition~\ref{lemma:estimates} and fix
$\theta\in (0,\theta_0)$.
Using that $(a+b)^\theta \leq a^\theta+b^\theta$ for all $a,b\geq 0$, we find
$$
J_{\e,T}^\theta \leq \int_0^T \int_\cE [\ell(e)]^\theta \indiq_{\{\ell(e) > 1\}} \Pi^\e(\dr v, \dr e),\;\;
\text{whence} \;\; I_{\e,T}+J_{\e,T}^\theta\leq \int_0^T \int_\cE (\ell(e) \land[\ell(e)]^\theta )\Pi^\e(\dr v, \dr e).
$$
Thus $\sup_{\e \in(0,1]}\E[ I_{\e,T}+J_{\e,T}^\theta]<\infty$ by~\eqref{j1s}. All this proves (i).

\vip
We carry on with (ii). Using~\eqref{sdebe}, \eqref{sdete}, that $(a+b)\land 1 \leq a\land 1+b\land 1$ 
for all $a,b\geq 0$ and that $|g_y(A,e)-y|\leq M(e)$ and $\cl_y(A,e)\leq \ell(e)$,
we get that for any $(S,S')\in \cA^\e_{T,\delta}$,
\[
(|b_{S'}^\e -b^\e_S |+ \tau_{S'}^\e - \tau^\e_S)\land 1 \leq 
\int_S^{S'}\int_{\mathcal{E}} [M(e)\land 1+\ell(e)\land 1]
\Pi_\e(\dr v, \dr e).
\]
We conclude from~\eqref{j1s} (and since $\ell(e)\land 1 \leq \ell(e)\land[\ell(e)]^\theta)$ 
that 
$$
\sup_{\e\in(0,1]}\E[(|b_{S'}^\e -b^\e_S |+ \tau_{S'}^\e - \tau^\e_S)\land 1] \leq C \delta, 
$$
which implies (ii).

\vip

\textit{Step 2.} We consider a sequence $(b^{\e_k}_u,\tau^{\e_k}_u)_{u\geq 0}$ converging in law to some
$(b_u,\tau_u)_{u\geq 0}$ for the $\JS$-topology, and show that $(b_u,\tau_u)_{u\geq 0}$ 
solves~\eqref{sdeb}-\eqref{sdet} (with e.g. $a_u=A_{b_\um}$). Such a process being unique in law, 
see~Theorem~\ref{mr1} and Step~2 of its proof, this will complete the proof.

\vip

By the theory of martingale problems, it suffices that for any 
$\varphi \in C^1_b( \mathbb{R}^d \times \mathbb{R}_+ )$,
\[
 M_u^\varphi := \varphi(b_u, \tau_u) - \varphi(x,0) - \int_0^u\int_{\mathcal{E}}\Big[\varphi(g_{b_v}(A_{b_{v}}, e), 
\tau_v + \cl_{b_v}(A_{b_{v}}, e)) - \varphi(b_v, \tau_v)\Big]\nn_\beta(\dr e)\dr v
\]
is a martingale in the canonical filtration of $(b_u,\tau_u)_{u\geq 0}$. In other words, we need that
for any $n\geq 1$, any $0 \leq u_1 < \cdots < u_n \leq v < w$, any 
$\varphi, \psi_1, \cdots, \psi_n \in C^1_b( \mathbb{R}^d \times \mathbb{R}_+ )$,
\begin{equation}\label{ggoal}
\E[H((b_u,\tau_u)_{u\geq 0})]=0, 
\end{equation}
where for $(y,r)\in \DD(\R_+,\R^d\times\R_+)$,
\begin{align*}
H(y,r)=&\Big(\prod_{i=1}^n\psi_i(y_{u_i}, r_{u_i})\Big) \bigg(\varphi(y_w,r_w) - \varphi(y_v,r_v)  \\
&\hskip2cm- \int_v^w\int_{\mathcal{E}}\Big[\varphi(g_{y_s}(A_{y_s}, e),r_s+\cl_{y_s}(A_{y_s},e)) - \varphi(y_s,r_s)
\Big]\nn_\beta(\dr e)\dr s\bigg).
\end{align*}
We fix such a functional $H$, the rest of the proof is devoted to establish~\eqref{ggoal}.

\vip

{\it Step 3.} Since $(b^{\e_k}_u,\tau^{\e_k}_u)_{u\geq 0}$ solves~\eqref{sdebe}-\eqref{sdete}, we have 
$\E[K_{\e_k}((b_u^{\e_k},\tau_u^{\e_k})_{u\geq 0})]=0$, where
\begin{align*}
K_\e(y,r)=&\Big(\prod_{i=1}^n\psi_i(y_{u_i}, r_{u_i})\Big) \bigg(\varphi(y_w,r_w) - \varphi(y_v,r_v)  
-\int_v^w  c_\e \partial_r\varphi(y_s,r_s) \dr s\\
&\hskip2cm- \int_v^w\int_{\mathcal{E}}\Big[\varphi(g_{y_s}(A_{b_{y_s}}, e),r_s+\cl_{y_s}(A_{b_{y_s}},e)) - \varphi(y_s,r_s)
\Big]\nn^\e(\dr e)\dr s\bigg).
\end{align*}
We thus may write, for any $k\geq 1$ and any $\delta \in (0,1]$,
\begin{equation}\label{dec}
\E[H((b_u,\tau_u)_{u\geq 0})]=\E[H((b_u,\tau_u)_{u\geq 0})]-\E[K_{\e_k}((b_u^{\e_k},\tau_u^{\e_k})_{u\geq 0})]
= I_1^{\delta}+I_2^{k,\delta}+I_3^{k,\delta}+I_4^{k},
\end{equation}
where, defining $H^\delta$ (resp. $H^\delta_\e$,  $H_\e$) as $H$ replacing $\nn_\beta$ by
$\nn_\beta^\delta$ (resp. $\nn^{\e,\delta}$, $\nn^\e$), see~\eqref{nd1}-\eqref{nd2},
\begin{align*}
I_1^{\delta}=&\E[H((b_u,\tau_u)_{u\geq 0})]-\E[H^\delta((b_u,\tau_u)_{u\geq 0})], \\
I_2^{k,\delta}=&\E[H^\delta((b_u,\tau_u)_{u\geq 0})]-\E[H^\delta_{\e_k}((b_u^{\e_k},\tau_u^{\e_k})_{u\geq 0})],\\
I_3^{k,\delta}=& \E[H^\delta_{\e_k}((b_u^{\e_k},\tau_u^{\e_k})_{u\geq 0})]-\E[H_{\e_k}((b_u^{\e_k},\tau_u^{\e_k})_{u\geq 0})],\\
I_4^{k}=& \E[H_{\e_k}((b_u^{\e_k},\tau_u^{\e_k})_{u\geq 0})]-\E[K_{\e_k}((b_u^{\e_k},\tau_u^{\e_k})_{u\geq 0})].
\end{align*}
We have $\sup_{(y,r)\in \DD(\R_+,\R^d\times\R_+) }|H_\e(y,r)-K_\e(y,r)|\leq C c_\e$ 
for some constant $C>0$, so that 
\begin{equation}\label{rr1}
\lim_{k\to \infty} I_4^{k}=0.
\end{equation}

{\it Step 4.} Here we prove that 
\begin{equation}\label{rr2}
\lim_{\delta\to 0} I_1^\delta= 0 \quad \text{and}\quad \lim_{\delta\to 0} \limsup_{k} |I_3^{k,\delta}|= 0.
\end{equation}
Since $\varphi$ is bounded together with it first derivatives, we see that for all
$(y,r)\in \DD(\R_+,\R^d\times\R_+)$,
$$
|H^\delta_{\e}(y,r)-H_{\e}(y,r)|\leq C \int_v^w \int_\cE\Big(|g_{y_s}(A_{b_{y_s}}, e)-y_s|\land 1 + 
\cl_{y_s}(A_{b_{y_s}},e))\land 1 \Big)   (\nn^{\e}-\nn^{\e,\delta})(\dr e) \dr s.
$$
Recalling that $|g_b(A,e)-b|\leq M(e)$ and $\cl_b(A,e)\leq \ell(e)$, we end with 
$$
|H^\delta_{\e}(y,r)-H_{\e}(y,r)|\leq C (w-v) \int_\cE \Big(M(e)\land 1+ \ell(e)\land 1
\Big)   (\nn^{\e}-\nn^{\e,\delta})(\dr e).
$$
It then directly follows from~\eqref{j2s} that $\lim_{\delta\to 0} \limsup_{k} |I_3^{k,\delta}|= 0$. One
shows similarly that $\lim_{\delta\to 0} I_1^\delta= 0$, using~\eqref{se1}, which implies that
$$
\lim_{\delta\to0}\int_{\mathcal{E}}\big[M(e) \wedge 1 + \ell(e) \wedge 1\big]
(\nn_\beta-\nn_\beta^\delta)(\dr e) = 0.
$$

{\it Step 5.} Recall~\eqref{dec} and that our goal is to prove~\eqref{ggoal}. By~\eqref{rr1}-\eqref{rr2}, 
it suffices to check that for each $\delta \in (0,1]$, $\lim_{k\to \infty} I^{k,\delta}_2=0$.
Fix $\delta\in (0,1]$ and assume for a moment that
\vip
\noindent (a) $\sup_{(y,r) \in  \DD(\R_+,\R^d\times\R_+)}(|H^\delta(y,r)|+|H^{\delta}_{\e_k}(y,r)|)<\infty$, 
\vip
\noindent (b) if $(y_k,r_k)\to (y,r)$ in $\DD(\R_+,\R^d\times\R_+)$ for the the $\JS$-topology and if $(y,r)$ has
no jump at times $u_1,\dots,u_k, v,w $, then 
$\lim_k H^{\delta}_{\e_k}(y_k,r_k)=H^\delta(y,r)$.
\vip
Then, we use the Skorokhod representation theorem to find 
$(\bar b^{\e_k}_u,\bar \tau^{\e_k}_u)_{u\geq 0})$ (with the same law as $(b^{\e_k}_u,\tau^{\e_k}_u)_{u\geq 0})$)
a.s. converging to $(\bar b_u,\bar \tau_u)_{u\geq 0})$ (with the same law as $(b_u,\tau_u)_{u\geq 0})$) and write
$$
|I^{k,\delta}_2|=|\E[H^\delta((\bar b_u,\bar \tau_u)_{u\geq 0})]-
\E[H^{\delta}_{\e_k}((\bar b^{\e_k}_u,\bar \tau^{\e_k}_u)_{u\geq 0}))]|\leq \E[|H^\delta((\bar b_u,\bar \tau_u)_{u\geq 0})
-H^{\delta}_{\e_k}((\bar b^{\e_k}_u,\bar \tau^{\e_k}_u)_{u\geq 0}))|].
$$
This last quantity tends to $0$ as $k\to \infty$ by dominated convergence, thanks to points (a) and (b) and since
$(\bar b_u,\bar \tau_u)_{u\geq 0}$ a.s. has no jump at times $u_1,\dots,u_k, v,w $, see Step~1.

\vip

Point (a) follows from~\eqref{aan} and the fact that the functions involved in the 
definition of $H^\delta$ and $H^\delta_\e$ are bounded.
For (b), the only difficulty is to check that for a.e. $s \in [v,w]$,
\begin{align*}
\lim_k \int_{\mathcal{E}} \varphi(g_{y^k_s}(A_{y^k_s}, e),r_s^k+\cl_{y^k_s}(A_{b_{y^k_s}},e))\nn^{\e_k,\delta}(\dr e)
= \int_{\mathcal{E}} \varphi(g_{y_s}(A_{y_s}, e),r_s+\cl_{y_s}(A_{y_s},e))\nn_\beta^\delta(\dr e).
\end{align*}
This follows from Lemma~\ref{lemma:cont_g_ell}-(i) and the fact that
$\lim_k (y^k_s,r^k_s)=(y_s,r_s)$ for all $s\in [v,w]$ which is not a jump time of $(y,r)$,
which is the case for a.e. $s\in [v,w]$.
\end{proof}

\subsection{Convergence of the Markov scattering process}

We denote by $\QQ^\e_x$ the law of  $(R_t^\e)_{t\geq0}$ (the solution to~\eqref{eq:def_r_epsilon}) 
issued from $x \in \cDd$ and recall that 
$\QQ_x$ was introduced in Definitions~\ref{dfr1} and~\ref{dfr2}. 
We also introduce the corresponding transition semigroups acting on continuous functions 
$\varphi:\closure{\Dd} \to \mathbb{R}$ and defined for any $x\in\cDd$, any $t\geq 0$, as
\[
\cP^\e_t \varphi(x) = \QQ_x^\e[\varphi(X_t^*)],\quad\text{and}\quad  
\cP_t \varphi(x) = \QQ_x[\varphi(X_t^*)],
\]
where $(X_t^*)_{t\geq 0}$ is the canonical process on $\DD(\R_+,\cDd)$.
We first show some tightness result.

\begin{lemma}\label{ret}
Adopt the same assumptions and notations as in Proposition~\ref{thm:conv_markov}.
The family $(\QQ_x^\e,x\in \cDd,\e\in (0,1])$ 
is tight, the set $\DD(\R_+,\cDd)$ being endowed with the $\JS$-topology.
Moreover, 
\begin{equation}\label{ttac}
\lim_{\eta\to 0} \limsup_{\e\to 0} \sup_{x \in \cDd} \QQ_x^\e\Big[\sup_{t\in [0,\eta]} |X^*_t-x|\Big] = 0.
\end{equation}
\end{lemma}

\begin{proof} We divide the proof in several steps.
\vip
{\it Step 1.} Let us first show that~\eqref{ttac} implies the tightness. Since $\cDd$ is compact,
it suffices, by the Aldous criterion, see
Jacod and Shiryaev~\cite[Section~VI, Thm.~4.5]{jacod2013limit}, that for all $T>0$,
\begin{equation}\label{ttax}
\lim_{\eta\to 0} \;\limsup_{\e\to 0} \;\sup_{x \in \cDd} \;\sup_{(S,S')\in \cA_{T,\eta}}\QQ^\e_x[|X^*_{S'}-X^*_S|]=0.
\end{equation}
Here $\cA_{T,\eta}$ is the set of all couples $(S,S')$ of stopping times (in the canonical filtration of $X^*$)
such that $0\leq S \leq S' \leq S+\eta \leq T$. But for $(S,S')\in \cA_{T,\eta}$, for any $x\in \cDd$,
using the strong Markov property at time $S$, we see that
$$
\QQ^\e_x[|X^*_{S'}-X^*_S|]\leq \sup_{y \in \cDd} \QQ^\e_y\Big[\sup_{t\in [0,\eta]}|X^*_t-y|\Big].
$$
Thus~\eqref{ttac} implies~\eqref{ttax} as desired.
\vip

{\it Step 2.} For $x\in\cDd$ and $\eta>0$, we set, recalling that $\tell(X^*)=\inf\{t> 0 : X^*_t\in \pDd\}$,
$$
\rho_\e(x,\eta)=\QQ_x^\e\Big[\sup_{t\in [0,\eta]} |X^*_t-x|\Big]\quad \hbox{and} 
\quad\rho'_\e(x,\eta)=\QQ_x^\e\Big[\sup_{t\in [0,\eta\land \tell(X^*)]} |X^*_t-x|\Big].
$$ 
By the strong Markov property at time $\tell(X^*)$, for all $x \in \cDd$,
$\rho_\e(x,\eta) \leq \rho_\e'(x,\eta)+\sup_{y \in \pDd} \rho_\e(y,\eta)$.

\vip

{\it Step 3.} Here we show that $\lim_{\eta\to 0} \limsup_{\e\to 0}\sup_{x\in \cDd}\rho_\e'(x,\eta) =0$.

\vip

We recall that the solution $(R^\e_t)_{t\geq 0}$ to~\eqref{eq:def_r_epsilon} 
starting at $x\in \cDd$ is $\QQ^\e_x$-distributed, so that 
$\rho_\e'(x,\eta)=\E[\sup_{t\in [0,\eta\land \tell(R^\e)]} |R^\e_t-x|]$. We introduce 
introduce the free process issued from $0$
\begin{equation}\label{freep}
Z^\e_t=\int_0^t \int_{\R^d\times \HH} u \rM_\e(\dr s,\dr u, \dr v),
\end{equation}
with the same Poisson measure $\rM_\e$ as in~\eqref{eq:def_r_epsilon}. One can verify that
$R^\e_t=x+Z^\e_t$ for all $t\in [0,\tell(R^\e))$ and that 
$R^\e_{\tell(R^\e)}=\Lambda(x+Z^\e_{\tell(R^\e)-},x+Z^\e_{\tell(R^\e)-})\in [x+Z^\e_{\tell(R^\e)-},x+Z^\e_{\tell(R^\e)}]$.
As a conclusion,
$$
\sup_{t\in [0,\eta\land \tell(R^\e)]} |R^\e_t-x| \leq \sup_{t\in [0,\eta]} |Z^\e_t|.
$$
Since moreover $\sup_{t\in [0,\eta\land \tell(R^\e)]} |R^\e_t-x|\leq \text{Diam}(\Dd)$,
we only have to check that  
$$
\lim_{\eta\to 0} \sup_{\e\in (0,1]} \rho_\e''(\eta)=0, \quad \text{where} \quad
\rho_\e''(\eta)=\E\Big[ \sup_{t\in [0,\eta]} |Z^\e_t|\land 1\Big].
$$
Using that $\int_{|u|\leq 1} u \rF_\e(u)\dr u=0$ 
by rotational invariance, we may write $Z^\e_t=Z^{\e,1}_t+Z^{\e,2}_t$, where
$$
Z^{\e,1}_t=\int_0^t \int_{\R^d\times \HH} u\indiq_{\{|u|\leq 1\}} \tilde \rM_\e (\dr s,\dr u,\dr v)
\quad\text{ and } \quad
Z^{\e,2}_t=\int_0^t \int_{\R^d\times \HH} u\indiq_{\{|u|> 1\}} \rM_\e (\dr s,\dr u,\dr v). 
$$

Let us first show that $\E[\sup_{[0,\eta]}|Z^{\e,1}_t|^2] \leq C \eta$.
By Doob's inequality, recalling Definition~\ref{dfre} and Notation~\ref{fgh},
$$
\E\Big[\sup_{[0,\eta]}|Z^{\e,1}_t|^2\Big] \leq \frac{4\eta}\e \int_{ \R^d} |u|^2\indiq_{\{|u|\leq 1\}}\rF_\e(u)\dr u
= \frac{4\eta}\e\E\Big[|\e^{(1-\alpha)/\alpha}E^\e U|^2\indiq_{\{|\e^{(1-\alpha)/\alpha}E^\e U|\leq 1\}} \Big],
$$
with $E^\e\sim\text{Exp}(\e^{-1})$ and $U\sim \rF(u)\dr u$. Using Assumption~\ref{assump_equi} and
that $\e^{-1} E^\e \sim\text{Exp}(1)$, we get
$$
\E\Big[\sup_{[0,\eta]}|Z^{\e,1}_t|^2\Big] \leq \frac{4C_{\rF}\eta}{\e}\int_0^\infty e^{-z}\dr z
\int_{\{|u|\leq \e^{-1/\alpha}z^{-1}\}} \frac{\e^{2/\alpha}z^2|u|^2 \dr u}{|u|^{d+\alpha}}.
$$
A simple computation shows that this last quantity equals $C\eta$ for some constant $C>0$.

\vip
Let us next check that $\E[\sup_{[0,\eta]}|Z^{\e,2}_t|\land 1] \leq C\eta$, which will end the step.
We write
$$
\E\Big[\sup_{[0,\eta]}|Z^{\e,2}_t|\land 1\Big] \!\leq\! \PP\Big(\sup_{[0,\eta]}|Z^{\e,2}_t|>0\Big)
\!\leq\! \PP\Big(\rM_\e([0,\eta]\times\{|u|>1\} \times \HH)>0 \Big)\!=\!1-\exp(-\eta a_\e) 
\!\leq\! a_\e \eta,
$$
where 
$$
a_\e= \frac1\e\int_{|u|>1} \rF_\e(u)\dr u = \frac1\e \PP(|\e^{(1-\alpha)/\alpha}E^\e U|>1)\leq 
\frac{C_\rF}\e \int_0^\infty e^{-z}\dr z \int_{\{|u|> \e^{-1/\alpha}z^{-1}\}} \frac{\dr u}{|u|^{d+\alpha}}.
$$
This last quantity is finite and does not depend on $\e$, so that 
$\E[\sup_{[0,\eta]}|Z^{\e,2}_t|\land 1] \leq C\eta$.

\vip

{\it Step 4.} It only remains to verify that  $\lim_{\eta\to 0} \limsup_{\e\to 0}\sup_{x\in \pDd}\rho_\e(x,\eta) =0$.
Here we use that $(R^\e_t)_{t\geq 0}$ defined in Lemma~\ref{repsexc}, issued from $x\in \pDd$,
is $\QQ^\e_x$-distributed.
For any $a\in (0,1]$,
\begin{equation}\label{ccc1}
\rho_\e(x,\eta)=\E\Big[\sup_{t \in [0,\eta)}|R^\e_t-x|\Big] 
\leq D\PP(L^\e_\eta>a) + \E\Big[\indiq_{\{L^\e_\eta\leq a\}}\sup_{t\in [0,\eta]}|R^\e_t-x|\Big],
\end{equation}
where $D=\text{Diam}(\Dd)$.
We have $R^\e_t-x=R^\e_t-b^\e_{L_t^\e-}+b^\e_{L^\e_t-}-x$ and, by~\eqref{def:reflected_process_eps},
$$
|R^\e_t-b^\e_{L^\e_t-}|\leq M(e^\e_{L^\e_t})\indiq_{\{L^\e_t \in \JJ_\e\}},
$$
where we recall that $M(e)=\sup_{t\in [0,\ell(e)]}|e(t)|$. Moreover,
$|R^\e_t-b^\e_{L^\e_t-}|\leq D$, whence 
$$ 
\E\Big[\indiq_{\{L^\e_\eta\leq a\}}\sup_{t\in [0,\eta]}|R^\e_t-x|\Big] \leq \E\Big[\sup_{u \in \JJ_\e\cap [0,a]}(M(e^\e_u)\land D)
+ \sup_{u\in [0,a]} |b^\e_u-x|\Big].
$$
Recalling~\eqref{sdebe} and that $|g_y(A,e)-y|\leq M(e)\land D$ for all $y\in \pDd$, $A\in \cI_y$ and
$e\in \cE$, we find
$$ 
\E\Big[\indiq_{\{L^\e_\eta\leq a\}}\sup_{t\in [0,\eta]}|R^\e_t-x|\Big] \leq 2\E\Big[\int_0^a \int_\cE (M(e)\land D) 
\Pi_\e(\dr v,\dr e)\Big] = 2a  \int_\cE (M(e)\land D) \nn^\e(\dr e).
$$
By~\eqref{j1s}, there is a constant $C>0$ such that for all $a\in (0,1]$,
\begin{equation}\label{ccc2}
\sup_{\e\in (0,1]}\sup_{x\in\pDd}\E\Big[\indiq_{\{L^\e_\eta\leq a\}}\sup_{t\in [0,\eta]}|R^\e_t-x|\Big] \leq C a.
\end{equation}

Define $r>0$ as in Remark~\ref{imp4} and recall that $\cl_y(A,\cdot)\geq \ell_r(\cdot)$. 
Recalling~\eqref{sdete}, we conclude that 
$\tau^\e_u\geq \int_0^u\int_\cE \ell_r(e) \Pi^\e(\dr v,\dr e)=: \tilde\tau_u^\e$ (which does not depend on 
$x\in \pDd$). Hence
$$
\PP(L^\e_\eta>a)=\PP(\tau^\e_a<\eta)\leq \PP(\tilde \tau^\e_a<\eta)
\leq \PP\Big(\Pi_\e([0,a]\times \{\ell_r>\eta\})=0\Big)
=\exp(-a \nn^\e(\ell_r>\eta)).
$$
Using~\eqref{j5s}, we conclude that for any $a>0$,
\begin{equation}\label{ccc3}
\lim_{\eta\to 0} \limsup_{\e\to 0}\sup_{x\in \pDd}\PP(L^\e_\eta>a)\leq \lim_{\eta\to 0} \limsup_{\e\to 0}
\exp(-a \nn^\e(\ell_r>\eta))=0.
\end{equation}
Gathering~\eqref{ccc1}, \eqref{ccc2} and~\eqref{ccc3} completes the step.
\end{proof}

We next show the convergence of the semigroups when starting from the boundary.

\begin{lemma}\label{prop:conv_semi}
Adopt the same assumptions and notations as in Proposition~\ref{thm:conv_markov}.
Let $x \in \pDd$, $t\geq 0$ and, for any $\e\in(0,1]$, let $x_\e \in \pDd$ and $t_\e \geq 0$ 
be such that $x_\e\to x$ and $t_\e\to t$ as $\e\to0$. 
For any continuous function $\varphi:\cDd \to \mathbb{R}$, 
it holds that $\mathcal{P}^\e_{t_\e} \varphi(x_\e) \to \mathcal{P}_t \varphi(x)$ as $\e\to0$.
\end{lemma}

\begin{proof} We classically may assume that $t>0$ and that $\varphi$ is Lipschitz continuous.

\vip

{\it Step 1.} The set $S=\{t\geq 0 : \QQ_x(X^*_t \in \pDd)>0\}$ being Lebesgue-null (because it holds that
$\QQ_x[\int_0^\infty \indiq_{\{X^*_t\in \pDd\}}\dr t] = 0$,
see~Theorem~\ref{mr2}), it suffices to treat the case where $t\notin S$. 

\vip

Indeed, when $t\in S$, let
$\eta_n>0$ be decreasing to $0$ and such that $t+\eta_n \notin S$. We find
$\lim_{\e\to 0}\cP^\e_{t_\e+\eta_n}\varphi(x_\e)=\cP_{t+\eta_n}\varphi(x)$ for each $n$.
But $\lim_n \cP_{t+\eta_n}\varphi(x)=\cP_{t}\varphi(x)$ by right-continuity of $X^*$, and
$\lim_n\lim_{\e\to 0}\cP^\e_{t_\e+\eta_n}\varphi(x_\e)=\lim_{\e\to 0}\cP^\e_{t_\e}\varphi(x_\e)$,
because, using the Markov property at time $t_\e$ and that $\varphi$ is Lipschitz continuous,
$$
\limsup_{\e\to0} |\cP^\e_{t_\e+\eta_n}\varphi(x_\e)-\cP^\e_{t_\e}\varphi(x_\e)|
\leq \limsup_{\e\to 0} \sup_{y \in \cDd} |\cP^\e_{\eta_n}\varphi(y)-\varphi(y)|
\leq  C \limsup_{\e\to 0} \sup_{y \in \cDd} \QQ^\e_y[|X^*_{\eta_n}-y|],
$$
which tends to $0$ as $n\to \infty$ by~\eqref{ttac}.
Thus $\lim_{\e\to 0}\cP^\e_{t_\e}\varphi(x_\e)=\cP_{t}\varphi(x)$ as desired.

\vip

{\it Step 2.} We consider $(R^\e_t)_{t\geq 0}\sim \QQ^\e_{x_\e}$,
together with $\Pi^\e=\sum_{u\in \JJ_\e}\delta_{(u,e^\e_u)}$, $(b^\e_u)_{u\geq 0}$, 
$(\tau^\e_u)_{u\geq 0}$ and $(L^\e_t)_{t\geq 0}$, as in Lemma~\ref{repsexc}. We also consider
$(R_t)_{t\geq 0}\sim\QQ_x$, together with $\Pi=\sum_{u\in \JJ}\delta_{(u,e_u)}$, $(b_u)_{u\geq 0}$, 
$(\tau_u)_{u\geq 0}$ and $(L_t)_{t\geq 0}$, as in Definition~\ref{dfr1}. 
We have to prove that $\lim_{\e\to 0}\E[\varphi(R^\e_{t_\e})]=\E[\varphi(R_{t})]$ if $t\notin S$.
We thus assume that $t\notin S$ and observe that 
$L_t \in \JJ$ a.s., because $R_t=b_{L_t} \in \pDd$ when $L_t \notin \JJ$,
see~\eqref{defR}.

\vip

{\it Step 2.1.} By Proposition~\ref{convbord}, $(b^\e_u,\tau^\e_u)_{u\geq 0}$ 
converges in law to $(b_u,\tau_u)_{u\geq 0}$ for the $\JS$-topology. 
By Lemma~\ref{convL}-(i), this implies that $((b^\e_u,\tau^\e_u)_{u\geq 0},(L^\e_t)_{t\geq 0})$ converges in law
to $((b_u,\tau_u)_{u\geq 0},(L_t)_{t\geq 0})$ in $\DD(\R_+,\cDd\times \R_+)\times\DD(\R_+,\R_+)$,
with $\DD(\R_+,\cDd\times \R_+)$ endowed with $\JS$ and $\DD(\R_+,\R_+)$ endowed with the 
topology of the uniform convergence on compact time intervals.

\vip

{\it Step 2.2.} Exactly as in Step~3.5 of the proof of Proposition~\ref{newb}, making use of Lemma~\ref{clac},
one can check that we have both
$\PP(L^\e_{t_\e} \in \JJ_\e,\tau^\e_{L^\e_{t_\e}-}=t_\e)=\PP(L^\e_{t_\e} \in \JJ_\e,\tau^\e_{L^\e_{t_\e}}=t_\e)=0$
and $\PP(L_{t} \in \JJ,\tau_{L_{t}-}=t)=\PP(L_{t} \in \JJ,\tau_{L_{t}}=t)=0$.

\vip

{\it Step 2.3.} We now show that
$$
\lim_{\delta\to 0}\limsup_{\e\to 0} \PP(L^\e_{t_\e}\in \JJ_\e,\ell(e^\e_{L^\e_{t_\e}})\leq \delta)=
\lim_{\delta\to 0}\PP(L_{t}\in \JJ,\ell(e_{L_{t}})\leq \delta)=0
$$

By Step~2.2 and since $t\notin S$, we a.s. have $L_t \in \JJ$ and $\tau_{L_t-}<t<\tau_{L_t}$.
By Lemma~\ref{convL}-(ii) and Step~2.1, we conclude that
$$
\limsup_{\e\to 0} \PP(\Delta \tau^\e_{L_{t_\e}^\e}\leq \delta) \leq  \PP(\Delta \tau_{L_{t}}\leq \delta)= 
\PP(0<\Delta \tau_{L_{t}}\leq \delta).
$$
Since $\cap_{\delta>0} \{0<\Delta \tau_{L_{t}}\leq \delta\}=\emptyset$ up to a negligible set, we end with
$$
\lim_{\delta\to 0}\limsup_{\e\to 0} \PP(\Delta \tau^\e_{L_{t_\e}^\e}\leq \delta)=0.
$$
The conclusion follows, since $\{L^\e_{t_\e}\in J_\e,\ell(e^\e_{L^\e_{t_\e}})\leq \delta\}\subset 
\{\Delta \tau^\e_{L_{t_\e}^\e}\leq \delta\}$: indeed, $\Delta \tau^\e_{L_{t_\e}^\e}=
\cl_{b^\e_{L_{t_\e}^\e-}}(A_{b^\e_{L_{t_\e}^\e-}},e^\e_{L^\e_{t_\e}})\leq \ell(e^\e_{L^\e_{t_\e}})$.
Similarly, $\PP(L_{t}\in \JJ,\ell(e_{L_{t}})\leq \delta)\leq \PP(0<\Delta \tau_{L_t}\leq \delta)$ tends to $0$
as $\delta\to 0$.

\vip

{\it Step 2.4.} We have $\lim_{A\to \infty} \limsup_{\e\to 0} \PP(L^\e_{t_\e}\geq A)=\lim_{A\to \infty} \PP(L_{t}\geq A)=0$.
\vip
Indeed, we deduce from Step~2.1 that $L^\e_{t_\e}$ goes in law to $L_t$, whence
$\limsup_{\e\to 0} \PP(L^\e_{t_\e}\geq A)\leq \PP(L_{t}\geq A)$, which tends to $0$ as $A\to\infty$
since $L_t$ is a.s. finite.

\vip

{\it Step 2.5.} We now verify that if $\beta=*$, then for any $\delta>0$, 
$$
\lim_{\e\to 0} \E\Big[\varphi(R^\e_{t_\e})\indiq_{\{L^\e_{t_\e}\in \JJ_\e,\ell(e^\e_{L^\e_{t_\e}})>\delta\}}\Big]= 
\E\Big[\varphi(R_{t})\indiq_{\{L_{t}\in \JJ,\ell(e_{L_{t}})>\delta\}}\Big].
$$

Using that $L^\e_{t_\e} \in \JJ_\e$ if and only if there is $u\in \JJ_\e$ such that $\tau^\e_{u-}< t_\e < \tau^\e_u$
(we replaced broad inequalities by strict ones using Step~2.2) and that in such a case
$R^\e_{t_\e}=h_{b^\e_u}(A_{b^\e_u},e^\e_u)$, see~\eqref{def:reflected_process_eps} and
$\tau^\e_u=\tau^\e_{u-}+\cl_{b^\e_{L_{t_\e}^\e-}}(A_{b^\e_{L_{t_\e}^\e-}},e^\e_{L^\e_{t_\e}})$, we write,
for any $A>0$,
\begin{align*} 
\E\Big[\varphi(R^\e_{t_\e})\indiq_{\{L^\e_{t_\e}\in \JJ_\e\cap[0,A],\ell(e^\e_{L^\e_{t_\e}})>\delta\}}\Big]
=& \E\Big[\sum_{u \in \JJ_\e\cap[0,A]} \varphi(h_{b^\e_u}(A_{b^\e_u},e^\e_u))
\indiq_{\{\tau^\e_{u-}<t_\e  < \tau^\e_{u}\}}  
\indiq_{\{\ell(e^\e_u)>\delta\}} \Big]\\
=& \E\Big[\int_0^A \int_\cE \varphi(h_{b^\e_u}(A_{b^\e_u},e))\indiq_{\{\tau^\e_{u}< t_\e< \tau^\e_{u}+ \cl_{b^\e_u}(A_{b^\e_u},e)\}} 
\nn^{\e,\delta}(\dr e) \dr u \Big]
\end{align*}
by the compensation formula. Recall that $\nn^{\e,\delta}$ was introduced in Proposition~\ref{lemma:estimates}. 
Similarly,
\begin{align*}
\E\Big[\varphi(R_t)\indiq_{\{L_{t}\in \JJ\cap[0,A],\ell(e_{L_{t}})>\delta\}}\Big]
=& \E\Big[\int_0^A \int_\cE \varphi(h_{b_u}(A_{b_u},e))\indiq_{\{\tau_{u}< t< \tau_{u}+ \cl_{b_u}(A_{b_u},e)\}} 
\nn^{\delta}_*(\dr e) \dr u \Big].
\end{align*}
By the Skorokhod representation theorem, we may assume that $(\tau^\e_u,b^\e_u)_{u\geq 0}$ a.s. converges to 
$(\tau_u,b_u)_{u\geq 0}$ for the $\JS$-topology. This implies that for a.e. $u\geq 0$,
$\lim_{\e\to 0} \tau^\e_u=\tau_u$ and $\lim_{\e\to 0} b^\e_u=b_u$. Moreover, we have $\tau_u\neq t$ for a.e. $u\geq 0$
(since $(\tau_u)_{u\geq 0}$ is strictly increasing, there is at most one $u$ such that $\tau_u=t$).
By Lemma~\ref{lemma:cont_g_ell}-(ii), we conclude that a.s., for a.e. $u\geq 0$,
$$
\lim_{\e\to 0}\int_\cE \!\!\varphi(h_{b^\e_u}(A_{b^\e_u},e))\indiq_{\{\tau^\e_{u}< t_\e< \tau^\e_{u}+ \cl_{b^\e_u}(A_{b^\e_u},e)\}} 
\nn^{\e,\delta}(\dr e)=\int_\cE \!\!\varphi(h_{b_u}(A_{b_u},e)) \indiq_{\{\tau_{u}< t< \tau_{u}+ \cl_{b_u}(A_{b_u},e)\}}
\nn_*^\delta(\dr e).
$$
But $K:=\sup_{\e\in (0,1]} \nn^{\e,\delta}(\cE)<\infty$ by~\eqref{aan}. Hence
by dominated convergence, for all $A>0$,
$$
\lim_{\e\to 0} \E\Big[\varphi(R^\e_{t_\e})\indiq_{\{L^\e_{t_\e}\in \JJ_\e\cap[0,A],\ell(e^\e_{L^\e_{t_\e}})>\delta\}}\Big]
=\E\Big[\varphi(R_t)\indiq_{\{L_{t}\in \JJ\cap[0,A],\ell(e_{L_{t}})>\delta\}}\Big].
$$
To complete the step, 
it only remains to observe that $\lim_{A\to \infty}\limsup_{\e\to 0} I_{\e}^{\delta,A}=0$,
where
\begin{align*}
I_{\e}^{\delta,A}:=& \PP(L^\e_{t_\e}\in \JJ_\e\cap(A,\infty],\ell(e^\e_{L^\e_{t_\e}})>\delta)
+\PP(L_{t}\in \JJ\cap(A,\infty),\ell(e_{L_{t}})>\delta).
\end{align*}
Since $I_{\e}^{\delta,A}\leq  \PP(L^\e_{t_\e}>A)+\PP(L_{t}>A)$, this follows from Step~2.4.

\vip
{\it Step 2.6.} We next prove that if $\beta=*$, then $\lim_{\e\to 0} \PP(L^\e_{t_\e} \in \JJ_\e)=1$.
\vip
Since $t\notin S$, we have $\PP(L_t \in \JJ)=1$ (see the very beginning of Step~2), whence
$$
\Delta_\e:=| \PP(L^\e_{t_\e} \in \JJ_\e)-1|=| \PP(L^\e_{t_\e} \in \JJ_\e)-\PP(L_t \in \JJ)|\leq \Delta_\e^{\delta,1}
+ \Delta_\e^{\delta,2},
$$
where
\begin{gather*}
 \Delta_\e^{\delta,1}:=|\PP(L^\e_{t_\e} \in \JJ_\e, \ell(e_{L^\e_{t_\e}}) \leq \delta)-\PP(L_{t} \in \JJ, \ell(e_{L_{t}}) 
\leq \delta)|, \\
\Delta_\e^{\delta,2}:=|\PP(L^\e_{t_\e} \in \JJ_\e, \ell(e_{L^\e_{t_\e}}) > \delta)-\PP(L_{t} \in \JJ, \ell(e_{L_{t}}) 
> \delta)|.
\end{gather*}
The conclusion follows, since we 
know from Step~2.5 (with $\varphi=1$) that for each $\delta>0$, $\lim_{\e\to 0} \Delta_\e^{\delta,2}=0$ 
and from Step~2.3
that $\lim_{\delta\to 0}\limsup_{\e\to 0} \Delta_\e^{\delta,1}=0$.

\vip
{\it Step 2.7.} We conclude the proof when $\beta=*$.
We write 
$$
\Gamma_\e:=|\E[\varphi(R^\e_{t_\e})]-\E[\varphi(R_{t})]|\leq \Gamma_\e^{1,\delta}+ 2  
||\varphi||_\infty\Gamma_\e^{2,\delta},
$$ 
where (recall that $L_t\in\JJ$ a.s. because $t\notin S$)
\begin{gather*}
\Gamma_\e^{1,\delta}=\Big|\E[\varphi(R^\e_{t_\e})\indiq_{\{L^\e_{t_\e}\in \JJ_\e,\ell(e^\e_{L^\e_{t_\e}})>\delta\}}]
-\E[\varphi(R_{t})\indiq_{\{L_t\in\JJ, \ell(e)>\delta\}}]\Big|,\\
\Gamma_\e^{2,\delta}=\PP(L^\e_{t_\e}\notin \JJ_\e)+\PP(L^\e_{t_\e}\in \JJ_\e,\ell(e^\e_{L^\e_{t_\e}})\leq\delta)
+\PP(L_{t}\in \JJ,\ell(e_{L_{t}})\leq\delta).
\end{gather*}
We have  $\lim_{\e\to 0}\Gamma_\e^{1,\delta}=0$ for each $\delta>0$ by Step~2.5 and 
$\lim_{\delta\to 0} \limsup_{\e\to 0} \Gamma_\e^{2,\delta}=0$ by Steps~2.6 and~2.3.
\vip

\vip
{\it Step 2.8.} We now prove that when $\beta\in (0,\alpha/2)$, 
$$
\lim_{\delta\to 0} \limsup_{\e\to 0} \PP(L^\e_{t_\e}\in \JJ_\e, |e^\e_{L^\e_{t_\e}}(0)|\leq \delta)=0.
$$

We write $\PP(L^\e_{t_\e}\in \JJ_\e, |e^\e_{L^\e_{t_\e}}(0)|\leq \delta)\leq \Theta_\e^{1,\delta,a,A}+\Theta_\e^{2,a}
+\Theta_\e^{3,A}$,
where
\begin{gather*}
\Theta_\e^{1,\delta,a,A}= \PP(L^\e_{t_\e}\in \JJ_\e\cap[0,A], \ell(e^\e_{L^\e_{t_\e}})>a, |e^\e_{L^\e_{t_\e}}(0)|\leq \delta),\\
\Theta_\e^{2,a}=\PP(L^\e_{t_\e}\in \JJ_\e, \ell(e^\e_{L^\e_{t_\e}})\leq a)
\quad \text{and} \quad \Theta_\e^{3,A}=\PP(L^\e_{t_\e}>A).
\end{gather*}
But $\lim_{a\to 0} \limsup_{\e\to 0} \Theta_\e^{2,a}=0$  by Step~2.3 and
$\lim_{A\to \infty} \limsup_{\e\to 0} \Theta_\e^{3,A}=0$ by Step~2.4. Hence we only have to check that
for each $a>0$, each $A>0$, $\lim_{\delta\to 0} \limsup_{\e\to 0} \Theta_\e^{1,\delta,a,A}=0$.
But
\begin{align*}
\Theta_\e^{1,\delta,a,A}\leq& \PP(\Pi_\e([0,A]\times\{e \in \cE : |e(0)|\leq \delta, \ell(e)>a\})>0)\\
=& 1 - \exp(-A \nn^{\e}(\{e \in \cE : |e(0)|\leq \delta, \ell(e)>a\})).
\end{align*}
The conclusion follows from~\eqref{j6s}.

\vip

{\it Step 2.9.} We conclude when  $\beta\in (0,\alpha/2)$. Proceeding exactly as in Step 2.5
(recall that when  $\beta\in (0,\alpha/2)$, $\nn^{\e,\delta}(\dr e)=\indiq_{\{|e(0)|>\delta\}}\nn^\e(\dr e)$),
one can check that for any $\delta>0$,
$$
\lim_{\e\to 0} \E\Big[\varphi(R^\e_{t_\e})\indiq_{\{L^\e_{t_\e}\in \JJ_\e,|e^\e_{L^\e_{t_\e}}(0)|>\delta\}}\Big]= 
\E\Big[\varphi(R_{t})\indiq_{\{L_{t}\in \JJ,|e_{L_{t}}(0)|>\delta\}}\Big].
$$
Replacing systematically $\ell(e^\e_{L^\e_{t_\e}})>\delta$ by $|e^\e_{L^\e_{t_\e}}(0)|>\delta$,
$\ell(e^\e_{L^\e_{t_\e}})\leq \delta$ by $|e^\e_{L^\e_{t_\e}}(0)|\leq\delta$, $\ell(e_{L_{t}})>\delta$ by 
$|e_{L_{t}}(0)|>\delta$ and $\ell(e_{L_{t}})\leq \delta$ by $|e_{L_{t}}(0)|\leq\delta$, we 
prove as in Step~2.6 that $\lim_{\e\to 0} \PP(L^\e_{t_\e}\in \JJ_\e)=1$ and
conclude as in Step~2.7.
\end{proof}

We now extend the previous result to any starting point in $\cDd$.

\begin{lemma}\label{prop:conv_semi_g}
Adopt the same assumptions and notations as in Proposition~\ref{thm:conv_markov}.
Let $x \in \cDd$, $t\geq 0$ and, for any $\e\in(0,1]$, let $x_\e \in \cDd$
be such that $x_\e\to x$.
Then for any continuous function $\varphi:\cDd \to \mathbb{R}$, 
it holds that $\mathcal{P}^\e_{t} \varphi(x_\e) \to \mathcal{P}_t \varphi(x)$ as $\e\to0$.
\end{lemma}

\begin{proof} We recall that $\tell(r)=\inf\{t>0 : r(t)\notin \Dd\}$ for all $r\in \DD(\R_+,\R^d)$.
The case $t=0$ being obvious, we suppose that $t>0$. We classically may assume that
$\varphi$ is Lipschitz continuous.
\vip

{\it Step 1.} The set $S_1=\{t\geq  0 : \QQ_x(\Delta X^*_t \neq 0)>0\}$ is Lebesgue-null. Indeed, 
write $S_1=\cup_{n\geq 1} S_1^n$, where $S_1^n=\{t\geq 0 : \QQ_x(\Delta X^*_t \neq 0)\geq 1/n\}$,
recall that for any $y \in \DD(\R_+,\cDd)$, the set $j(y)=\{t\geq 0 : \Delta y(t)\neq 0\}$ is at most countable,
and notice that
$$
\text{Leb}(S_1^n)=\int_0^\infty \indiq_{\{\QQ_x(\Delta X^*_t \neq 0)\geq1/n\}}\dr t\leq
n\int_0^\infty \QQ_x(\Delta X^*_t \neq 0) \dr t=n\QQ_x(\text{Leb}(j(X^*)))=0.
$$
Moreover, the set $S_2=\{t\geq 0 : \QQ_x(\tell(X^*)=t)>0\}$ is at most countable,
because any probability measure
on $\R_+$ (here $\QQ_x(\tell(X^*)\in \dr t)$) has at most a countable number of atoms.
Exactly as in Step~1 of the proof of Lemma~\ref{prop:conv_semi}, we may assume that 
$t \notin S:=S_1\cup S_2$.

\vip

{\it Step 2.} We consider the free process 
$(Z^\e_t=\int_0^t\int_{\R^d\times \HH} u \rM_\e(\dr s,\dr u,\dr v))_{t\geq 0}$ defined in~\eqref{freep}. 
We recall that for $(R^\e_t)_{t\geq 0}$ the solution of~\eqref{eq:def_r_epsilon} starting from $x_\e \in \Dd$, 
we have $\tell(R^\e)=\tell(x_\e+Z^\e)$ and $R^\e_t=x_\e+Z^\e_t$ for all $t\in [0,\tell(R^\e))$ and
$R^\e_{\tell(R^\e)}=\Lambda(x_\e+Z^\e_{\tell(Z^\e)-},x_\e+Z^\e_{\tell(Z^\e)})$.
\vip

{\it Step 3.}
We show that the pure jump Lévy process $(Z^\e_t)_{t\geq 0}$ converges, for the local uniform topology 
(whence for the $\JS$-topology), 
to the ISP$_{\alpha,0}$. This stronger convergence will be useful in Section~\ref{stech}.
The Lévy measure of $(Z^\e_t)_{t\geq 0}$ is 
$\e^{-1}\rF_\e(z)\dr z$. Recalling Notation~\ref{fgh}, we find
$$
\e^{-1}\rF_\e(z)=\e^{-1-d/\alpha}\int_0^\infty e^{-r} \rF(\e^{-1/\alpha}r^{-1} z) r^{-d} \dr r.
$$
Using Assumption~\ref{assump_equi}, one easily checks that, since $\kappa_\rF=1/\Gamma(\alpha+1)$
and setting $M=C_\rF \Gamma(\alpha+1)$,
$$
\e^{-1}\rF_\e(z) \leq M |z|^{-d-\alpha} \quad \text{and} \quad \e^{-1}\rF_\e(z)\sim |z|^{-d-\alpha}
\quad \text{as $\e \to 0$.}
$$
We now introduce $\varphi_\e(z)= |z|^{d+\alpha} \e^{-1} \rF_\e(z)$, which is bounded by $M$ and
tends to $1$ as $\e\to 0$. We consider a Poisson measure $\rK$ on $\R_+\times[0,M]\times \R^d$
with intensity $\dr s \dr u |z|^{-d-\alpha}\dr z$, as well as
\begin{align*}
U^\e_t=& \int_0^t\int_0^M\int_{\{|z|<1\}} z \indiq_{\{u\leq \varphi_\e(z)\}}\tilde \rK(\dr s,\dr u,\dr z)+
 \int_0^t\int_0^M\int_{\{|z|\geq 1\}} z \indiq_{\{u\leq \varphi_\e(z)\}}\rK(\dr s,\dr u,\dr z),\\
Z_t=& \int_0^t\int_0^M\int_{\{|z|<1\}} z \indiq_{\{u\leq 1\}}\tilde \rK(\dr s,\dr u,\dr z)+
 \int_0^t\int_0^M\int_{\{|z|\geq 1\}} z \indiq_{\{u\leq 1\}}\rK(\dr s,\dr u,\dr z).
\end{align*}
It holds that $(U^\e_t)_{t\geq 0}$ has the same law as $(Z^\e_t)_{t\geq 0}$ and that $(Z_t)_{t\geq 0}$
is an ISP$_{\alpha,0}$. We now prove that $\Delta_{\e,T}=\sup_{[0,T]}|U^\e_t-Z_t|\to 0$ in probability and this will
complete the step. We write $\Delta_{\e,T}\leq \Delta_{\e,T}^1+\Delta_{\e,T}^2$, where
\begin{align*}
\Delta_{\e,T}^1 =& \sup_{[0,T]} \Big|\int_0^t\int_0^M\int_{\{|z|<1\}} z (\indiq_{\{u\leq \varphi_\e(z)\}}-\indiq_{\{u\leq1\}})
\tilde \rK(\dr s,\dr u,\dr z)\Big|,\\
\Delta_{\e,T}^2 =& \int_0^t\int_0^M\int_{\{|z|\geq 1\}} |z| |\indiq_{\{u\leq \varphi_\e(z)\}}-\indiq_{\{u\leq1\}}|
\rK(\dr s,\dr u,\dr z).
\end{align*}
First, By Doob's inequality, we have
\begin{align*}
\E[(\Delta_{\e,T}^1)^2]\leq& 4 T \int_{\{|z|<1\}}\int_0^M |z|^2 |\indiq_{\{u\leq \varphi_\e(z)\}}-\indiq_{\{u\leq1\}} |
\frac{\dr u \dr z}{|z|^{d+\alpha}}
= 4T \int_{\{|z|<1\}} \frac{ |\varphi_\e(z)-1|\dr z}{|z|^{d+\alpha-2}},
\end{align*}
which tends to $0$ by dominated convergence. Next, we have $\PP(\Delta_{\e,T}^2 \neq 0)\leq \PP(\rK(A_{\e,T})>0)$, 
where 
$A_{\e,T}=\{(t,u,z)\in [0,T]\times [0,M]\times \{|z|\geq 1\} : u \in [\varphi_\e(z)\land 1, \varphi_\e(z)\lor 1]\}$.
As a consequence, $\PP(\Delta_{\e,T}^2 \neq 0)\leq 1-e^{-\lambda_{\e,T}}$, with
$$
\lambda_{\e,T}=T \int_{\{|z|\geq 1\}}\int_0^M \indiq_{\{u\in[\varphi_\e(z)\land 1, \varphi_\e(z)\lor 1]\}}
\frac{\dr u \dr z}{|z|^{d+\alpha}}=T \int_{\{|z|\geq1\}} \frac{ |\varphi_\e(z)-1|\dr z}{|z|^{d+\alpha}},
$$
which tends to $0$ by dominated convergence. Thus $\PP(\Delta_{\e,T}^2 \neq 0)\to 0$.

\vip
{\it Step 4.} We show that the law of $(X^*_t\indiq_{\{t<\tell(X^*)\}},X^*_{\tell(X^*)},\tell(X^*))$
under $\QQ^\e_{x_\e}$ converges to the law of $(X^*_t\indiq_{\{t<\tell(X^*)\}},X^*_{\tell(X^*)},\tell(X^*))$ 
under $\QQ_{x}$. 

\vip
Recall Definition~\ref{dfr2}: let $(Z_t)_{t\geq 0}$ be an $\alpha$-stable process issued from $0$, 
set $R_t=x+Z_t$ for all 
$t\in[0,\tell(x+Z))$ and $R_{\tell(Z)}=\Lambda(x+Z_{\tell(x+Z)-},x+Z_{\tell(x+Z)})$. Then $\tell(R)=\tell(x+Z)$ and
$(R_{t\land\tell(R)})_{t\geq 0}$ has the same law as $(X^*_{t\land\tell(X^*)})_{t\geq 0}$ under $\QQ_x$.
Recalling Step~2, it suffices to show that $((x_\e+Z^\e_t)\indiq_{\{t<\tell(x_\e+Z^\e)\}},
\Lambda(x_\e+Z^\e_{\tell(x_\e+Z^\e)-},x_\e+Z^\e_{\tell(x_\e+Z^\e)}),
\tell(x_\e+Z^\e))$ goes in law
to $((x+Z_t)\indiq_{\{t<\tell(x+Z)\}},\Lambda(x+Z_{\tell(x+Z)-},x+Z_{\tell(x+Z)}),\tell(x+Z))$. This follows from Step~3  and
Lemma~\ref{lemma:cont_g_ell}-(iii).

\vip

{\it Step 5.} Using the Markov property at time $\tell(X^*)$, we have
\begin{align*}
\mathcal{P}^\e_{t} \varphi(x_\e)= &\QQ^\e_{x_\e}[\indiq_{\{t<\tell(X^*)\}}\varphi(X^*_t)] +
\QQ^\e_{x_\e}[\indiq_{\{t\geq \tell(X^*)\}}\cP^\e_{t-\tell(X^*)}\varphi(X^*_{\tell(X^*)})]\\
=&\QQ^\e_{x_\e}[\Psi_\e(X^*_{t}\indiq_{\{t< \tell(X^*)\}},X^*_{\tell(X^*)},\tell(X^*))],
\end{align*}
where for $(y,z,s)\in (\cDd\cup\{0\})\times\pDd\times\R_+$, 
$\Psi_\e(y,z,s)=\indiq_{\{t<s\}}\varphi(y) + \indiq_{\{t\geq s\}}\cP^\e_{t-s}\varphi(z)$.
Similarly, with $\Psi(y,z,s)=\indiq_{\{t<s\}}\varphi(y) + \indiq_{\{t\geq s\}}\cP_{t-s}\varphi(z)$,
$$
\cP_t\varphi(x)=\QQ_x[\Psi(X^*_t\indiq_{\{t< \tell(X^*)\}},X^*_{\tell(X^*)},\tell(X^*))].
$$
By Step~4 and 
Skorokhod's representation theorem, we may consider $(Y_\e, H_\e, \rho_\e)$ (resp. $(Y,H,\rho)$)
distributed as $(X^*_t\indiq_{\{t< \tell(X^*)\}},X^*_{\tell(X^*)},\tell(X^*))$ under $\QQ^\e_{x_\e}$ (resp. under $\QQ_x$)
such that almost surely, $\lim_{\e\to 0}(Y_\e,H_\e,\rho_\e)=(Y,H, \rho)$.
Since $\PP(\rho=t)=0$ (because $t \notin S_2$), we a.s. have $\indiq_{\{t<\rho_\e\}}\to \indiq_{\{t<\rho\}}$
and $\indiq_{\{t\geq \rho_\e\}}\to \indiq_{\{t\geq\rho\}}$ as $\e\to 0$. 
Moreover, we deduce from Lemma~\ref{prop:conv_semi}
that $\indiq_{\{t\geq \rho_\e\}}\cP^\e_{t-\rho_\e}\varphi(H_\e)\to \indiq_{\{t\geq \rho\}}\cP_{t-\rho}\varphi(H)$ a.s.
as $\e\to 0$. All in all, $\Psi_\e(Y_\e,H_\e,\rho_\e)\to \Psi(Y,H,\rho)$ a.s. and, since
$\Psi_\e$ is bounded by $||\varphi||_\infty$, we conclude by dominated convergence that
$$
\mathcal{P}^\e_{t} \varphi(x_\e)=\E[\Psi_\e(Y_\e,H_\e,\rho_\e)]\to \E[ \Psi(Y,H,\rho)]=\cP_t\varphi(x).
$$
\vskip-0.8cm
\end{proof}

\vip
We are now ready to give the

\begin{proof}[Proof of Proposition~\ref{thm:conv_markov}]
We consider $x_\e \in \cDd$ such that $x_\e\to x \in \cDd$,
consider a process $(R^\e_{t})_{t\geq 0}$ with law $\QQ^\e_{x_\e}$ and 
a process $(R_{t})_{t\geq 0}$ with law $\QQ_x$. Thanks to the tightness proved in Lemma~\ref{ret},
it suffices to show that the finite-dimensional distributions of $(R_t^\e)_{t\geq0}$ converge
to those of $(R_t)_{t\geq0}$: for each $n\geq 1$,  any $0 \leq t_1 < \ldots < t_n$ and any
continuous bounded functions $\varphi_1, \cdots, \varphi_n:\closure{\Dd}\to \mathbb{R}$,
\begin{equation}\label{eq:conv_reccur}
\E\bigg[\prod_{k=1}^n\varphi_k(R_{t_k}^\e)\bigg] \longrightarrow \E\bigg[\prod_{k=1}^n\varphi_k(R_{t_k})\bigg] 
\quad \text{as }\e\to0.
\end{equation} 
We work by induction on $n$.
When $n=1$, \eqref{eq:conv_reccur} 
follows from Lemma~\ref{prop:conv_semi_g}. Assume next that~\eqref{eq:conv_reccur} holds true
for some $n\geq 1$ and consider $0 \leq t_1 < \ldots < t_{n+1}$ and some continuous bounded functions
$\varphi_1, \cdots, \varphi_{n+1}:\closure{\Dd}\to \mathbb{R}$. By the Markov property 
applied at time $t_n$, we get
\[
\E\bigg[\prod_{k=1}^{n+1}\varphi_k(R_{t_k}^\e)\bigg] = 
\E\bigg[\mathcal{P}_{u}^\e \varphi_{n+1}(R_{t_n}^\e)\prod_{k=1}^n\varphi_k(R_{t_k}^\e)\bigg],
\]
where $u = t_{n+1} - t_n > 0$. A similar equality holds for $(R_t)_{t\geq0}$ and we write
$$
\bigg|\E\bigg[\prod_{k=1}^{n+1}\varphi_k(R_{t_k}^\e)\bigg] -  \E\bigg[\prod_{k=1}^{n+1} \varphi_k(R_{t_k})\bigg]
\bigg|  \leq A_\e+B_\e,
$$
where
\begin{align*}
A_\e=&\bigg|\E\bigg[\Big(\mathcal{P}_{u}^\e \varphi_{n+1}(R_{t_n}^\e) -\mathcal{P}_u\varphi_{n+1}(R_{t_n}^\e)\Big)
\prod_{k=1}^n \varphi_k(R_{t_k}^\e)\bigg]\bigg|, \\
B_\e=&\bigg|\E\bigg[\mathcal{P}_{u} \varphi_{n+1}(R_{t_n}^\e)\prod_{k=1}^n\varphi_k(R_{t_k}^\e)\bigg] - 
\mathbb{E}_x\bigg[\mathcal{P}_{u} \varphi_{n+1}(R_{t_n})\prod_{k=1}^n\varphi_k(R_{t_k})\bigg]\bigg|.
\end{align*}
Since the limiting process $(R_t)_{t\geq0}$ is Feller, see Theorem~\ref{mr2}, 
the function $\mathcal{P}_u \varphi_{n+1}$ 
is continuous so that we can use the induction hypothesis and apply~\eqref{eq:conv_reccur} 
with the functions 
$\varphi_1, \cdots \varphi_{n-1}$ and $\varphi_n\mathcal{P}_u \varphi_{n+1}$, which implies that $B_\e \to 0$ 
as $\e\to0$. Next, setting $C = \prod_{k=1}^n ||\varphi_k||_{\infty}$, we write
\[
 A_\e \leq C \E\Big[\big|\mathcal{P}_{u}^\e \varphi_{n+1}(R_{t_n}^\e) - 
\mathcal{P}_{u} \varphi_{n+1}(R_{t_n}^\e)\big|\Big].
\]
By Lemma~\ref{prop:conv_semi_g}, it holds that 
$R_{t_n}^\e$ converges in law to $R_{t_n}$ as $\e \to 0$ and therefore, by Skorokhod's representation 
theorem, there exist a family $\bar R_{t_n}^\e$ (with same law as $R_{t_n}^\e$) a.s. converging to some
$\bar R_{t_n}$ (with same law as $R_{t_n}$). Thus
\begin{align*}
A_\e\leq & C  \E\Big[\big|\mathcal{P}_{u}^\e \varphi_{n+1}(\bar R_{t_n}^\e) - 
\mathcal{P}_{u} \varphi_{n+1}(\bar R_{t_n}^\e)\big|\Big]\\
\leq & C  \E\Big[\big|\mathcal{P}_{u}^\e \varphi_{n+1}(\bar R_{t_n}^\e) - 
\mathcal{P}_{u} \varphi_{n+1}(\bar R_{t_n})\big|\Big]+  \E\Big[\big|\mathcal{P}_{u} \varphi_{n+1}(\bar R_{t_n}^\e) - 
\mathcal{P}_{u} \varphi_{n+1}(\bar R_{t_n})\big|\Big].
\end{align*}
The first term tends to $0$ as $\e\to 0$ by Lemma~\ref{prop:conv_semi_g} and dominated convergence,
as well as the second one by continuity of $\mathcal{P}_{u} \varphi_{n+1}$ and dominated convergence.
\end{proof}

\section{Convergence of the excursion measures and related estimates}\label{stech}

In this section, we first introduce some notations related to the discrete excursion measures.
Then we establish some estimates on random walks in the half-space. We next prove 
Proposition~\ref{lemma:estimates} under Assumption~\ref{assump:moments}-(b) and then under 
Assumption~\ref{assump:moments}-(a) separately, the case of (a) being more delicate.
Finally, we prove Lemma~\ref{tmok}.

\subsection{Notation}\label{notex}

For $w\in \DD(\R_+,\R^d)$, we set as usual
$$
\ell(w)=\inf\{t>0 : w(t) \notin \HH\} \quad \text{and} \quad M(w)=\sup_{t\in [0,\ell(w)]}|w(t)|.
$$
We also recall some notations introduced before, see Notation~\ref{neps} and Proposition~\ref{lemma:estimates}.
With $\zeta=1/2$ if $\beta=*$ and $\zeta=\beta/\alpha$ if $\beta\in (0,\alpha/2)$, we have, for all $\e\in (0,1]$,
\begin{equation}\label{newt1}
\nn^\e = \chi_\rG \e^{-\zeta} \cL(Y^\e_{t\land \ell(Y^\e)}),
\end{equation}
where $Y_t^\varepsilon = O_\varepsilon + \int_0^t\int_{\mathbb{R}^d} u \mathrm{K}_\varepsilon(\dr s, \dr u)$,
with $O_\varepsilon \sim \mathrm{G}_\varepsilon(v) \dr v$ independent of the Poisson measure $\mathrm{K}_\varepsilon$ on 
$\mathbb{R}_+ \times \mathbb{R}^d$ with intensity $\varepsilon^{-1}\dr s \mathrm{F}_\varepsilon(u) \dr u$.
Recall that $\rF_\e$ and $\rG_\e$ were introduced in Notation~\ref{fgh}.
Let us observe at once that since $(Y^\e_t)_{t\geq 0}$ has the same law as $(\e^{1/\alpha}Y^1_{t/\e})_{t\geq 0}$,
it holds that
\begin{equation}\label{newt2}
\nn^\e = \e^{-\zeta} (\Phi_\e\# \nn^1), \quad \text{where} \quad \Phi_\e(w)=(\e^{1/\alpha}w(t/\e))_{t\geq 0}. 
\end{equation}
For $w\in  \DD(\R_+,\R^d)$, we have $\ell(\Phi_\e(w))=\e\ell(w)$ and $M(\Phi_\e(w))=\e^{1/\alpha}M(w)$.
For $x\in \HH$, we introduce $Y^{\e,x}_t=x+ \int_0^t\int_{\mathbb{R}^d} u \mathrm{K}_\varepsilon(\dr s, \dr u)$.
Then $(Y^{\e,x}_t)_{t\geq 0}$ has the same law as $(\e^{1/\alpha}Y^{1,\e^{-1/\alpha}x}_{t/\e})_{t\geq 0}$,
whence in particular, for all $x \in \HH$, all $\e\in (0,1]$,
\begin{equation}\label{newt3}
\ell(Y^{\e,x})\stackrel{(d)}=\e \ell(Y^{1,\e^{-1/\alpha}x}) \quad \text{and}\quad M(Y^{\e,x})\stackrel{(d)}=
\e^{1/\alpha} M(Y^{1,\e^{-1/\alpha}x}).
\end{equation}
By~\eqref{newt1}, for any measurable $\phi:\cE\to \R_+$, since $\rG_\e(y)=\e^{-d/\alpha}\rG_1(\e^{-1/\alpha}y)$,
\begin{align}
\int_\cE \phi(e)  \nn^{\e} (\dr e) =& \chi_{\mathrm{G}}\e^{-\zeta-d/\alpha}\int_{\HH}
\mathbb{E}\Big[\phi\big(\big(Y_{t\wedge \ell(Y^{\e, y})}^{\e, y}\big)_{t\geq 0}\big)\Big]
\mathrm{G}_1(\e^{-1/\alpha}y) \dr y \label{newt4}\\
=& \chi_\rG\e^{-\zeta}\int_{\HH}
\mathbb{E}\Big[\phi\big(\big(Y_{t\wedge \ell(Y^{\e, \e^{1/\alpha }x})}^{\e, \e^{1/\alpha }x}\big)_{t\geq0}\big)\Big]
\mathrm{G}_1(x) \dr x. \label{newt5}
\end{align}
Finally, we will use that for all $x\in \HH$, for $(Z_t)_{t\geq 0}$
an ISP$_{\alpha,x}$ under $\PP_x$,
\begin{equation}\label{eer2}
\PP_x\Big(\inf_{t\in[0,\ell(Z))} d(Z_t,\HH^c)>0, Z_{\ell(Z)} \in \bar \HH^c \Big)=1.
\end{equation}
This is nothing but~\eqref{eer} when $\Dd=\HH$. By~\eqref{eer2} and the Markov property of $\nn_*$, 
see Lemma~\ref{mark}, we deduce that for all $\eta>0$,
\begin{equation}\label{eer3}
\text{for $\nn_*$-a.e. $e\in \cE$ such that $\ell(e)>\eta$,}\quad 
\inf_{t\in[\eta,\ell(e))} d(e(t),\HH^c)>0 \quad \text{and} \quad e(\ell(e)) \in \bar \HH^c.
\end{equation}

\subsection{Estimates on random walks}

Recall that $\rF_1$ denotes the law of $EU$ where $E \sim \mathrm{Exp}(1)$ and $U \sim \mathrm{F}(v)\dr v$,
see Notation~\ref{fgh}.

\begin{lemma}\label{prop:estimate_rw}
Grant Assumption~\ref{assump_equi} with $\kappa_F=1/\Gamma(\alpha+1)$. 
Let $(\mathbf{S}_n)_{n\geq0}$ be a random walk with incremental law 
$\mathrm{F}_1$ starting at $0$.
There is $\mathcal{V} : \mathbb{R}_+ \to \R_+$ such that for all $x\in \HH$,
\begin{align}\label{eq:sim_return}
\mathbb{P}(\check{\ell}(x+\mathbf{S}) > n) \sim \mathcal{V}(x_1)n^{-1/2} \quad \text{as }n\to\infty.
\end{align}
where $\check{\ell}(x+\mathbf{S}) = \min\{n\geq 1 : (x +\mathbf{S}_n) \notin \HH\}$.
Moreover, there is a constant $C$ such that for any $x \in\HH$, any $n\geq 0$, any $y>0$,
setting $\check{M}(x+\mathbf{S})=\max_{i=1,\dots,\check{\ell}(x+\mathbf{S})} |x+\mathbf{S}_i|$,
\begin{gather}
\mathbb{P}(\check{\ell}(x+\mathbf{S}) > n) \leq C\mathcal{V}(x_1)(n+1)^{-1/2} 
\quad \text{and} \quad \mathcal{V}(x_1) \leq C(1 + x_1^{\alpha/2}),\label{eq:maj_return}\\
\mathbb{P}(\check{M}(x+\mathbf{S}) > y) \leq C(1+|x|^{\alpha/2})y^{-\alpha/2}.\label{eq:maj_unif2}
\end{gather}
\end{lemma}

These estimates are more or less well-known, except maybe~\eqref{eq:maj_unif2}.

\begin{remark} We will use~\eqref{eq:maj_unif2} only under Assumption~\ref{assump:moments}-(b).
Its proof relies on the fact that
\begin{equation}\label{tododo}
\text{there is $C>0$ such that for all $y>0$,}\quad
\mathbb{P}(\check{M}(\mathbf{S}) > y)
\leq  C y^{-\alpha/2}.
\end{equation}
This has been proved for a one-dimensional random walk by Doney~\cite[Corollary 3]{doney1985conditional},
and his arguments can be extended to the multidimensional case. We will check~\eqref{tododo}
while proving Proposition~\ref{lemma:estimates} under Assumption~\ref{assump:moments}-(a), namely just after
the proof of Lemma~\ref{tic2}.
\end{remark}

\begin{proof}[Proof of Lemma~\ref{prop:estimate_rw}]
We set $U_n = \mathbf{S}_n \cdot \be_1$ which is a one-dimensional symmetric random walk. 
We define its {\it decreasing ladder heights and epochs} $(\gamma_n, H_n)_{n\in\mathbb{N}}$ by setting 
$H_0 = 0$, $\gamma_0=0$ and for any $n \geq 1$, $\gamma_n = \min\{k > \gamma_{n-1} : U_k < H_{n-1}\}$ and 
$H_n = U_{\gamma_n}$. By the strong Markov property,
$(\gamma_n, -H_n)_{n\in\mathbb{N}}$ is a bivariate random walk, and each 
component is increasing. We now introduce the renewal function $\mathcal{W}:(0,\infty)\to [1,\infty)$ 
of $(H_n)_{n\geq0}$ defined by $ \mathcal{W}(y) = \mathbb{E}[T_y ]$,
where $T_y = \min\{n\geq 1 : y + H_n < 0\}$. Then~\eqref{eq:sim_return}, with $\mathcal{V}(y)
=a\mathcal{W}(y)$ for some
constant $a>0$, is a rather classical fact, 
see for instance Denisov-Wachtel~\cite[Theorem 1]{denisov_wachtel_exact}. The fact that there is no slowly 
varying function in the asymptotics follows from the symmetry of the random walk, see for 
instance~\cite{BB23}.
The first bound in~\eqref{eq:maj_return} follows again from~\cite[Theorem 1]{denisov_wachtel_exact}. 
The second bound  in~\eqref{eq:maj_return} follows from Doney~\cite[Lemma 7]{MR2892956}, which tells us that 
$\mathcal{V}(y) \sim cy^{\alpha /2}$ as $y\to \infty$, and from the fact that 
$\mathcal{V}$ locally bounded on $\R_+$ (because it is nondecreasing).

\vip
We finally show~\eqref{eq:maj_unif2}. Note that we may (and will) 
assume that $y>2|x|$. It holds that 
$\check{\ell}(x+\mathbf{S})=\inf\{k\geq 1 : x+U_k<0\} = \gamma_{T_{x_1}}$. We set
$Z_k= \max_{i \in\{\gamma_{k-1},\ldots, \gamma_{k}\}}|\mathbf{S}_i-\mathbf{S}_{\gamma_{k-1}}|$ and write
\begin{gather*}
\check{M}(x+\mathbf{S}) = \max_{k \in \{1, \ldots, T_{x_1}\}} \max_{i \in\{\gamma_{k-1},\ldots, \gamma_{k}\}}|x + \mathbf{S}_i|\leq
|x|+ \max_{k \in \{1, \ldots, T_{x_1}\}} |\mathbf{S}_{\gamma_{k-1}}| +  \max_{k \in \{1, \ldots, T_{x_1}\}}Z_k.
\end{gather*}
Since  $\max_{k \in \{1, \ldots, T_{x_1}\}} |\mathbf{S}_{\gamma_{k-1}}|\leq 
\sum_{k=1}^{T_{x_1}}|\mathbf{S}_{\gamma_{k}}-\mathbf{S}_{\gamma_{k-1}}| \leq \sum_{k=1}^{T_{x_1}}  Z_k$, we find
$$
\check{M}(x+\mathbf{S}) \leq |x| + 2\sum_{k=1}^{T_{x_1}}  Z_k.
$$
Thus, since  $y>2|x|$,
$$
\PP(\check{M}(x+\mathbf{S})>y)\leq \PP\Big(\sum_{k=1}^{T_{x_1}}  Z_k \geq \frac{y-|x|}2\Big)
\leq  \PP\Big(\sum_{k=1}^{T_{x_1}}  Z_k \geq \frac{y}4\Big)\leq
\PP\Big(\sum_{k=1}^{T_{x_1}}  \Big(Z_k \land \frac y 4 \Big) \geq \frac{y}4\Big).
$$
The trick used in the last inequality is borrowed from Denisov-Wachtel~\cite{denisov_wachtel_exact}.
Thus 
$$
\PP(\check{M}(x+\mathbf{S})>y)\leq \E\Big[\sum_{k=1}^{T_{x_1}}  \Big(\frac{4Z_k}y \land 1 \Big)\Big]
= \mathbb{E}[T_{x_1}] \E\Big[\frac{4Z_1}y\land 1\Big]
$$
by Wald's identity: the sequence $(Z_k)_{k\geq 1}$ is i.i.d., with $Z_k$ being $\cF_{\gamma_k}$-measurable and
$T_{x_1}$ is an $(\cF_{\gamma_k})_{k\geq 0}$-stopping time, because it is a hitting time of $(H_k=U_{\gamma_k})_{k\geq 0}$.
But we have seen that $\mathbb{E}[T_{x_1}]=\mathcal{W}(x_1)\leq C(1+x_1^{\alpha/2})\leq C(1+|x|^{\alpha/2})$.
Moreover, recalling that $Z_1=\max_{i \in\{0,\dots,\gamma_1\}}|\mathbf{S}_i|$ and that 
$\gamma_1=\check{\ell}(\mathbf{S})$,
we see that $Z_1\stackrel{(d)}=\check{M}(\mathbf{S})$. Thus  $\PP(Z_1>z)\leq C z^{-\alpha/2}$ by~\eqref{tododo} and
$$
\E\Big[\frac{4Z_1}y\land 1\Big]=\int_0^1 \PP\Big(Z_1 \geq \frac{yz}4\Big) 
\dr z \leq C \int_0^1 \frac{\dr z}{(yz)^{\alpha/2}}\leq \frac C {y^{\alpha/2}}.
$$
We have shown that $\PP(\check{M}(x+\mathbf{S})\geq y)\leq C (1+|x|^{\alpha/2})y^{-\alpha/2}$, which was our goal.
\end{proof}

We now deduce similar estimates for continuous-time random walks.

\begin{lemma}\label{newt}
Grant Assumption~\ref{assump_equi} and recall that $(Y^{\e,x}_t)_{t\geq 0}$ was introduced in 
Subsection~\ref{notex}. There is a constant $C$ such that for all $\e\in (0,1]$, all $x\in\HH$, all $t > 0$,
all $y>0$,
\begin{gather}
\mathbb{P}(\ell(Y^{\e,x}) > t) \leq C(\e^{1/2}+x_1^{\alpha/2})t^{-1/2}, \label{newe1}\\
\mathbb{P}(M(Y^{\e,x})> y) \leq C (\e^{1/2} +|x|^{\alpha/2}) y^{-\alpha/2}. \label{newe2}
\end{gather}
For all $x\in \HH$,
\begin{equation}\label{newe3}
\mathbb{P}(\ell(Y^{1,x}) > t) \sim \mathcal{V}(x_1) t^{-1/2} \quad \text{as $t\to \infty$.}
\end{equation}
\end{lemma}

\begin{proof} By~\eqref{newt3}, it suffices to show~\eqref{newe1}-\eqref{newe2} when $\e=1$.
But we can write $Y^{1,x}_t=x+\mathbf{S}_{N_t}$, where $(\mathbf{S}_n)_{n\geq0}$ is as in Lemma~\ref{prop:estimate_rw}
and is independent of a Poisson process $(N_t)_{t\geq0}$ of parameter~$1$. We also observe that
$\{\ell(Y^{1,x})>t\}=\{\check{\ell}(x+\mathbf{S}) > N_t\}$ and that $M(Y^{1,x})=\check{M}(x+\mathbf{S})$.
Thus~\eqref{newe2} (with $\e=1$) immediately follows from~\eqref{eq:maj_unif2}, while~\eqref{eq:maj_return}
tells us that 
\begin{align*}
\mathbb{P}(\ell(Y^{1,x}) > t)\leq C (1 +x_1^{\alpha /2})\mathbb{E}\big[(1+N_t)^{-1/2}\big] 
\leq C(1 +x_1^{\alpha /2})t^{-1/2},
\end{align*}
because $\mathbb{E}[(1+N_t)^{-1/2}]\leq \E[(1+N_t)^{-1}]^{1/2}
=[e^{-t}\sum_{k\geq 0} \frac{t^{k}}{(k+1)!}]^{1/2}=[\frac{1-e^{-t}}t]^{1/2}\leq \frac 1{t^{1/2}}$.
This proves~\eqref{newe1} (when $\e=1$). Finally, for any $a \in (0,1)$,
\[
\mathbb{P}(\check{\ell}(x+\mathbf{S}) > t(1+a), \: |N_t / t -1| \leq a) 
\leq \mathbb{P}(\ell(Y^{1,x})> t, \: |N_t / t -1| \leq a) 
\leq \mathbb{P}(\check{\ell}(x+\mathbf{S}) > t(1-a)).
\]
By Bienaymé-Tchebychev's inequality, it holds that $\mathbb{P}(|N_t/t -1| > a) \leq a^{-2}t^{-1}$, from which 
$t^{1/2}\mathbb{P}(|N_t/t -1| > a) \to 0$ as $t\to\infty$. 
We thus get from~\eqref{eq:sim_return} that
\[
\mathcal{V}(x_1)(1+a)^{-1/2} \leq \liminf_{t\to\infty}t^{1/2}\mathbb{P}(\ell(Y^{1,x})> t) \leq 
\limsup_{t\to\infty}t^{1/2}\mathbb{P}(\ell(Y^{1,x}) > t) \leq \mathcal{V}(x_1)(1-a)^{-1/2}.
\]
Letting $a\to0$ completes the proof of~\eqref{newe3}.
\end{proof}

\subsection{Convergence of the excursion measure under Assumption~\ref{assump:moments}-(b)}

We start with an easy lemma. Recall that $\ell_r(w)=\inf\{t>0 : w(t)\notin B_d(r\be_1,r)\}$.

\begin{lemma}\label{lemma:conv_stop_sortie}
Grant Assumption~\ref{assump_equi} with $\kappa_\rF=1/\Gamma(\alpha+1)$
and recall that for $x\in  \HH$, the process $(Y^{\e,x}_t)_{t\geq 0}$ was introduced in Subsection~\ref{notex}.
Let $(Z_t^x)_{t\geq 0}$ be an ISP$_{\alpha,x}$.
\vip
(a) We have the following convergence in law in $\mathbb{D}(\mathbb{R}_+, \mathbb{R}^d)$ 
endowed with $\bm{\JJ}_1$-topology:
\[
\big(Y_{t\wedge \ell(Y^{\e,x})}^{\e, x}\big)_{t\geq0} \longrightarrow \big(Z_{t\wedge \ell(Z^x)}^x\big)_{t\geq0} 
\quad \text{as }\e\to0
\]

(b) For any $r >0$, $\ell_r(Y^{\e,x})$ converges in law to $\ell_r(Z^x)$ as $\e\to0$.
\end{lemma}

\begin{proof}
As in Step~3 of the proof of Lemma~\ref{prop:conv_semi_g}, it holds that
that $(Y^{\e,x}_t)_{t\geq 0}$ converges in law to $(Z_t^x)_{t\geq 0}$ for the local uniform topology.
By Skorokhod's representation theorem, we may assume that the convergence holds a.s.
From~\eqref{eer2}, it holds that a.s., $\inf_{t\in[0,\ell(Z^x))}d(Z_t^x, \HH^c) > 0$ 
and $Z^x_{\ell(Z^x)} \notin \closure{\HH}$. Thus for $\e$ small enough, we have 
$\inf_{t\in[0,\ell(Z^x))}d(Y_t^{\e,x}, \HH^c) > 0$ and $Y^{\e,x}_{\ell(Z^x)} \notin \closure{\HH}$, implying that 
$\ell(Y^{\e,x})=\ell(Z^x)$. Point (a) follows. The proof of (b) is similar.
\end{proof}

We can now give the

\begin{proof}[Proof of Proposition~\ref{lemma:estimates} under Assumption~\ref{assump:moments}-(b)]
We recall that 
$\chi_{\mathrm{G}} = 1 /(2\kappa_{\mathrm{G}}\Gamma(\beta + 1))$ with
$\kappa_{\mathrm{G}}$ introduced in Assumption~\ref{assump:moments}-(b).
\vip
{\it Step 1.} Here we prove that
\begin{equation}\label{eq:equi_G}
\mathrm{G}_1(x) \sim 2\Gamma(\beta + 1)\kappa_{\mathrm{G}}|x|^{-d- \beta} \quad \text{as }|x|\to \infty
\end{equation}
and that there is $C>0$ such that for all $x \in \HH$,
\begin{equation}\label{eq:bound_G1}
\rG_1(x)\leq  C (|x|^{1 - d}\wedge |x|^{-\beta - d}).
\end{equation}
We recall from Notation~\ref{fgh} that $\rG_1$ is the law of $EW$ where $E \sim \mathrm{Exp}(1)$ and 
$W \sim \mathrm{G}_+(v) \dr v=2\rG(v)\indiq_{\{v \in \HH\}}\dr v$. An easy computation shows that for any $x\in \HH$,
\[
 \mathrm{G}_1(x) = 2\int_0^\infty t^{-d}e^{-t}\mathrm{G}(x/t) \dr t.
\]
Recalling Assumption~\ref{assump:moments}-(b), $\mathrm{G}(x/t) \sim \kappa_{\mathrm{G}}t^{\beta + d}|x|^{-\beta-d}$
as $|x|\to \infty$, for any $t > 0$, and there is a constant $C$ such 
$\mathrm{G}(x/t) \leq C t^{\beta+d}|x|^{-\beta - d}$ for any $x\in \R^d$, any $t>0$.
Then~\eqref{eq:equi_G} follows by dominated convergence, and we also have
$\rG_1(x)\leq C |x|^{-d- \beta}$. Assumption~\ref{assump:moments}-(b) also requires that
$G(v) \leq C (1+ |v|)^{-\beta - d}\leq C (1\land|v|^{-\beta-d})$, so that 
$\rG(x/t)\leq C(1 \land |x|^{-\beta-d}t^{\beta+d})$ and thus 
$$
\rG_1(x)\leq
C |x|^{-\beta - d} \int_0^{|x|}t^{\beta}e^{-t} \dr t + C \int_{|x|}^\infty t^{-d}e^{-t} \dr t
\leq  C |x|^{-\beta - d} \int_0^{|x|}t^{\beta} \dr t + C \int_{|x|}^\infty t^{-d} \dr t = C |x|^{1-d}.
$$

\textit{Step 2.} We fix $\delta>0$ and $\phi:\cE\to\R$ bounded and continuous 
for the $\JS$-topology and we prove~\eqref{j3s}. By~\eqref{newt4} and since 
$\nn^{\e,\delta}(\dr e)=\indiq_{\{|e(0)|>\delta\}}\nn^\e(\dr e)$, we have 
\begin{equation}\label{ntruc}
\int_\cE \phi(e)  \nn^{\e,\delta} (\dr e) = \chi_{\mathrm{G}}\e^{-\beta/\alpha-d/\alpha}\int_{\{|y|>\delta\}}
\mathbb{E}\Big[\phi\big(\big(Y_{t\wedge \ell(Y^{\e, y})}^{\e, y}\big)_{t\geq 0}\big)\Big]
\mathrm{G}_1(\e^{-1/\alpha}y) \dr y.
\end{equation}
By Lemma~\ref{lemma:conv_stop_sortie}, for each $y \in \HH$,
\[
\mathbb{E}\Big[\phi\big(\big(Y_{t\wedge \ell(Z^{\e, y})}^{\e, y}\big)_{t\geq0}\big)\Big] 
\longrightarrow\mathbb{E}_y\Big[\phi\big(\big(Z_{t \wedge \ell(Z)}\big)_{t\geq0}\big)\Big] \quad \text{as }\e\to0.
\]
Moreover, it follows from Step~1 that for any $y\in\HH$, $\chi_\rG\e^{-(\beta + d) / \alpha}\mathrm{G}_1(\e^{-1/\alpha}y) 
\to |y|^{-d-\beta}$ as $\e\to0$. We also get from~\eqref{eq:bound_G1} that 
$\e^{-(\beta + d) / \alpha}\mathrm{G}_1(\e^{-1/\alpha}y) \leq C |y|^{-d-\beta}$, which is integrable on $\{|y| > \delta\}$. 
We thus can use the dominated convergence theorem to get
\[
\lim_{\e\to 0}\int_\cE \phi(e)  \nn^{\e,\delta} (\dr e)=\int_{\{|y| > \delta\}} 
\mathbb{E}_y\Big[\phi\big(\big(Z_{t \wedge \ell(Z)}\big)_{t\geq0}\big)\Big]|y|^{-d-\beta}\dr y = \int_\cE \phi(e)  
\nn_\beta^{\delta} (\dr e),
\]
recall that $\nn_\beta$ was defined in~\eqref{nnb} and that $\nn_\beta^\delta(\dr e)=
\indiq_{\{|e(0)|>\delta\}}\nn_\beta(\dr e)$, see~\eqref{nd2}. 
Note that we have not yet used the fact that $\beta < \alpha / 2$.

\vip
{\it Step 3.} We next show~\eqref{aan}: we have $\nn_\beta^\delta(\cE)=\int_{\{|y|>\delta\}}|y|^{-d-\beta}\dr y<\infty$,
while using~\eqref{ntruc}, we find 
$\nn^{\e,\delta}(\cE)= \chi_\rG \e^{-(\beta+d)/\alpha}\int_{\{|y|>\delta\}} \rG_1(\e^{-1/\alpha}y)\dr y
\leq C\int_{\{|y|>\delta\}}|y|^{-d-\beta}\dr y\leq C$.

\vip
\textit{Step 4.} We now prove~\eqref{j1s}. Fix $\theta\in(0,\beta/\alpha)$. 
We have to show that $I_\e$ is bounded, where
\begin{align*}
I_\e:=&\int_{\cE} [M(e)\land 1 + \ell(e)\land[\ell(e)]^\theta] \nn^\e(\dr e)\\
=&\chi_\rG\e^{-(\beta+d)/\alpha}\int_\HH \E\Big[M(Y^{\e,y}\big))\land 1
+\ell(Y^{\e, y})\land[\ell(Y^{\e, y})]^\theta\Big]  \rG_1(\e^{-1/\alpha}y)\dr y.
\end{align*}
by~\eqref{newt4}. We deduce from~\eqref{newe2} that
$$
\E\Big[M(Y^{\e, y})\land 1\Big]
=\int_0^1\PP(M(Y^{\e, y})>z)\dr z\leq C (\e^{1/2}+|y|^{\alpha/2})\land 1.
$$
Consider a r.v. $L>0$ satisfying $\PP(L>t)\leq a t^{-1/2}$ for some $a>0$. If $a>1$, 
since $\theta\in (0,1/2)$,
$$
\E[L\land L^\theta]=\int_0^\infty \PP(L\land L^\theta>t)\dr t \leq a^{2\theta}
+ \int_{a^{2\theta}}^\infty \PP(L>t^{1/\theta})\dr t\leq a^{2\theta}+a\int_{a^{2\theta}}^\infty \frac{\dr t}{t^{1/(2\theta)}} 
=  C a^{2\theta}.
$$
If $a\in (0,1]$, we have
$$
\E[L\land L^\theta] \leq \int_0^1 \PP(L>t)\dr t + 
\int_1^\infty \PP(L>t^{1/\theta})\dr t \leq \int_0^1 \frac{a\dr t}{t^{1/2}}+\int_0^1 \frac{a\dr t}{t^{1/(2\theta)}}
= Ca.
$$
Thus $\E[L \land L^\theta]\leq C(a \land a^{2\theta})$ in any case. From this and~\eqref{newe1}, we conclude that
$$
\E\Big[\ell(Z^{\e, y})\land[\ell(Z^{\e, y})]^\theta\Big] \leq C[ (\e^{1/2}+|y|^{\alpha/2}) \land
(\e^{1/2}+|y|^{\alpha/2})^{2\theta}].
$$
All in all (and since $a\land 1 \leq a \land a^{2\theta}$ for all $a>0$), we have proved that
\begin{align}
I_\e\leq& C \e^{-(\beta+d)/\alpha}\int_\HH[ (\e^{1/2}+|y|^{\alpha/2}) \land
(\e^{1/2}+|y|^{\alpha/2})^{2\theta}] \rG_1(\e^{-1/\alpha}y)\dr y \notag\\
\leq & C\e^{-(\beta+d)/\alpha}\int_\HH[ (\e^{1/2}+|y|^{\alpha/2}) \land
(\e^{1/2}+|y|^{\alpha/2})^{2\theta}] \frac{\dr y}{|\e^{-1/\alpha}y|^{d-1}\lor |\e^{-1/\alpha}y|^{\beta + d}} \label{ssd}
\end{align}
by~\eqref{eq:bound_G1}. We now write $I_\e\leq I_{\e,1}+I_{\e,2}+I_{\e,3}$, where $I_{\e,1}$ (resp. $I_{\e,2}$, $I_{\e,3}$) 
stands for the integral on $\{|y|\leq \e^{1/\alpha}\}$ (resp.  $\{\e^{1/\alpha}<|y|\leq 1\}$, $\{|y|>1\}$). First,
$$
I_{\e,1}\leq \frac{C}{\e^{(\beta+d)/\alpha}}\int_{\{|y|\leq \e^{1/\alpha}\}}(\e^{1/2}+|y|^{\alpha/2})\frac{\dr y}
{\e^{-1/\alpha}|y|^{d-1}}
\leq \frac{C}{\e^{(\beta+d)/\alpha}}\int_{\{|y|\leq \e^{1/\alpha}\}}\e^{1/2}\frac{\dr y}
{|\e^{-1/\alpha}y|^{d-1}}.
$$
Since $\int_{\{|y|\leq \e^{1/\alpha}\}}\frac{\dr y}{|y|^{d-1}}=C \e^{1/\alpha}$, 
we find $I_{\e,1}\leq C \e^{1/2-\beta/\alpha} \leq C$.
Next, 
$$
I_{\e,2}\leq \frac{C}{\e^{(\beta+d)/\alpha}}\int_{\{\e^{1/\alpha}<|y|\leq 1\}}(\e^{1/2}+|y|^{\alpha/2})
\frac{\dr y}{|\e^{-1/\alpha}y|^{\beta + d}}
\leq \frac{C}{\e^{(\beta+d)/\alpha}}\int_{\{\e^{1/\alpha}<|y|\leq 1\}} |y|^{\alpha/2}\frac{\dr y}{|\e^{-1/\alpha}y|^{\beta + d}}.
$$
Since $\int_{\{|y|\leq 1\}}|y|^{\alpha/2-\beta-d}\dr y = C$ (because $\alpha/2-\beta>0$),
we conclude that $I_{\e,2}\leq C$. Finally,
$$
I_{\e,3}\leq \frac{C}{\e^{(\beta+d)/\alpha}}\int_{\{|y|\geq 1\}}(\e^{1/2}+|y|^{\alpha/2})^{2\theta}
\frac{\dr y}{|\e^{-1/\alpha}y|^{\beta + d}}
\leq \frac{C}{\e^{(\beta+d)/\alpha}}\int_{\{|y|\geq 1\}} |y|^{\alpha\theta}\frac{\dr y}{|\e^{-1/\alpha}y|^{\beta + d}}.
$$
Since $\int_{\{|y|\geq 1\}}|y|^{\alpha\theta-\beta-d}\dr y = C$ (because $\alpha\theta-\beta<0$, 
since $\theta<\beta/\alpha$), we have $I_{\e,3}\leq C$.

\vip
\textit{Step 5.} We next prove~\eqref{j2s}. Exactly as in Step~4 (with $\theta=0$), see~\eqref{ssd}, we find
\begin{align*}
J_{\e,\delta}:=&\int_{\cE} [M(e)\land 1 + \ell(e)\land 1] (\nn^\e-\nn^{\e,\delta})(\dr e)\\
\leq & \frac{C}{\e^{(\beta+d)/\alpha}}\int_{\{|y|\leq \delta\}} (\e^{1/2}+|y|^{\alpha/2}) 
\frac{\dr y}{|\e^{-1/\alpha}y|^{d-1}\wedge |\e^{-1/\alpha}y|^{\beta + d}}.
\end{align*}
For $\e$ small enough, we write $J_{\e,\delta}\leq J_{\e,1}+J_{\e,\delta,2}$, where $J_{\e,1}$ (resp. $J_{\e,\delta,2}$)
stands for the integral on $\{|y|\leq \e^{1/\alpha}\}$ (resp.  $\{\e^{1/\alpha}<|y|\leq \delta\}$).
We have $J_{\e,1}=I_{\e,1}\leq  C \e^{1/2-\beta/\alpha}$, so that $\lim_{\e\to 0} J_{\e,1}=0$. We also have
$$
J_{\e,\delta,2}\leq \frac{C}{\e^{(\beta+d)/\alpha}}\int_{\{\e^{1/\alpha}<|y|\leq \delta\}}\!\!(\e^{1/2}+|y|^{\alpha/2})
\frac{\dr y}{|\e^{-1/\alpha}y|^{\beta + d}}
\leq \frac{C}{\e^{(\beta+d)/\alpha}}\int_{\{\e^{1/\alpha}<|y|\leq \delta\}}\!\! 
|y|^{\alpha/2}\frac{\dr y}{|\e^{-1/\alpha}y|^{\beta + d}}.
$$
but $\int_{\{|y|\leq \delta\}}|y|^{\alpha/2-\beta-d}\dr y = C \delta^{\alpha/2-\beta}$,
whence $J_{\e,\delta,2}\leq C\delta^{\alpha/2-\beta}$ and $\lim_{\delta\to 0} \limsup_{\e\to 0} J_{\e,\delta,2}=0$.

\vip
\textit{Step 6.} We prove~\eqref{j6s}. For $\delta  >0$ and $a>0$, we write, since $z>a$ implies 
$z\land 1\geq a\land 1$,
\begin{align*}
 \nn^{\e}(\{e\in \cE : |e(0)|\leq \delta, \ell(e)>a\})\leq &
(\nn^{\e} - \nn^{\e,\delta})(\{e\in \cE : \ell(e)\land 1 \geq  a\land 1\}) \\
\leq& \frac{1}{a\land 1}
\int_\cE (\ell(e)\land 1)(\nn^\e - \nn^{\e, \delta})(\dr e).
\end{align*}
The result then follows from~\eqref{j2s}.

\vip

\textit{Step 7.} It only remains to check~\eqref{j5s}. By~\eqref{newt4}, we have
$$
\nn^\e(\ell_r>\eta)=\chi_\rG\e^{-(\beta+d)/\alpha}\int_\HH \PP(\ell_r(Y^{\e,y})>\eta) \rG_1(\e^{-1/\alpha}y)\dr y.
$$
Thanks to Lemma~\ref{lemma:conv_stop_sortie}, for any $y\in B_d(r\be_1, r)$, it holds that 
$\ell_r(Y^{\e,y})$ converges in law to $\ell_r(Z^y)$ where $(Z_t^y)_{t\geq0}$ is an $\mathrm{ISP}_{\alpha, y}$. 
Moreover, $\lim_{\e\to 0}\chi_\rG\e^{-(\beta + d) / \alpha}\mathrm{G}_1(\e^{-1/\alpha}y) =|y|^{-d-\beta}$ by~\eqref{eq:equi_G}.
Combining Fatou's lemma and Portmanteau's theorem, we get
\[
\liminf_{\e\to0}\nn^\e(\ell_r(e) > \eta) \geq \int_{B_d(r\be_1, r)}\mathbb{P}(\ell_r(Z^{y}) > \eta)|y|^{-d-\beta} \dr y.
\]
Since $\PP(\ell_r(Z^{y}) >0)=1$ for all $y\in B_d(r\be_1, r)$, we get by monotone convergence
\[
 \liminf_{\eta\to0}\liminf_{\e\to0}\nn^\e(\ell_r(e) > \eta) \geq \int_{B_d(r\be_1, r)}|y|^{-d-\beta} \dr y = \infty.
\]
\vskip-0.95cm
\end{proof}

\subsection{Convergence of the excursion measure under Assumption~\ref{assump:moments}-(a)}

We are now able to specify the value of the constant $\chi_{\mathrm{G}}$ in Proposition~\ref{lemma:estimates} 
when $\beta=*$.

\begin{definition}\label{def:chig}
Grant Assumption~\ref{assump_equi} and Assumption~\ref{assump:moments}-(a). Let $\mathcal{V}$ be the 
function from Lemma~\ref{prop:estimate_rw}. We define, recalling that $\rG_1=\cL(UW)$
with $U\sim \mathrm{Exp}(1)$ and $W\sim 2\rG(v)\indiq_{\{v \in \HH\}}\dr v$,
\[
\chi_{\mathrm{G}} = \frac{\nn_*(\ell > 1)}{\int_{\HH}\mathcal{V}(x_1) \mathrm{G}_1(\dr x)}.
\]
\end{definition}

Let us check that $I:=\int_{\HH}\mathcal{V}(x_1) \mathrm{G}_1(\dr x)< \infty$: recalling Notation~\ref{fgh},
we have $I=\E[\mathcal{V}(EW_1)]$, where $E \sim \mathrm{Exp}(1)$ and $W \sim \mathrm{G}_+(v) \dr v$ are
independent. Thus $I\leq C\E[1+|EW_1|^{\alpha/2}]$ by~\eqref{eq:maj_return}. This quantity is 
finite since $\E[|W|^{\alpha/2}]<\infty$ under Assumption~\ref{assump:moments}-(a).

\vip

The main goal of this subsection is to prove that $\nn^\e\to \nn_*$ under Assumption~\ref{assump:moments}-(a).
Informally, starting from~\eqref{newt5} and using 
Lemma~\ref{lemma:conv_stop_sortie}-(a), we should have
$$
\int_\cE\phi(e)\nn^\e(\dr e)\simeq \chi_\rG\e^{-1/2}\int_{\HH}
\mathbb{E}\Big[\phi\big(\big(Z_{t\wedge \ell(Z^{\e^{1/\alpha }x})}^{\e^{1/\alpha }x}\big)_{t\geq0}\big)\Big]
\mathrm{G}_1(x) \dr x.
$$
Admitting~\eqref{newformula}, we would have
$\mathbb{E}[\phi((Z_{t\wedge \ell(Z^{\e^{1/\alpha }x})}^{\e^{1/\alpha }x})_{t\geq0})]\simeq a_* \e^{1/2}x_1^{\alpha/2}\nn_*(\phi)$, whence
$$
\int_\cE\phi(e)\nn^\e(\dr e)\simeq \kappa \nn_*(\phi),
\quad \text{where}\quad \kappa=\chi_\rG a_* \int_\HH x_1^{\alpha/2}\mathrm{G}_1(x) \dr x,
$$
and it might be possible to show that $\kappa=1$. However, we are far from being able to establish the first
approximate equality. We are thus led to reproduce in a multidimensional setting the arguments of 
Doney~\cite{doney1985conditional}, initiated by Bolthausen~\cite{bolthausen1976functional}, 
regarding the convergence in law of conditioned random walks.
The first step consists in showing that $\nn^\e(\ell>\delta)\to \nn_*(\ell>\delta)$.

\begin{lemma}\label{tic0}
Grant Assumptions~\ref{assump_equi} and~\ref{assump:moments}-(a). For any $\delta>0$,
$$
\lim_{\e \to 0}\nn^\e(\ell>\delta) = \nn_*(\ell>\delta)= \delta^{-1/2} \nn_*(\ell>1)
\quad \text{as $\e\to 0$}.
$$
Moreover, there is $C>0$ such that $\nn^\e(\ell>\delta) \leq C \delta^{-1/2}$ 
for all $\e\in (0,1]$, all $\delta>0$.
Finally, 
$$
\nn^1(\ell>t)\sim t^{-1/2}\nn_*(\ell>1) \quad \text{as $t\to \infty$}.
$$
\end{lemma}

\begin{proof} By~\eqref{newt5}, we have
$$
\nn^\e(\ell>\delta)= \chi_\rG\e^{-1/2}\int_{\HH}
\mathbb{P}(\ell(Y^{\e, \e^{1/\alpha}x})>\delta)\mathrm{G}_1(x) \dr x.
$$
Using next~\eqref{newe1}, we find
$$
\nn^\e(\ell>\delta) \leq C \delta^{-1/2} \int_\HH (1+x_1^{\alpha/2})\mathrm{G}_1(x) \dr x \leq C \delta^{-1/2},
$$
since $\rG$ (and thus $\rG_1$) has a moment of order $\alpha/2$ by Assumption~\ref{assump:moments}-(a). Moreover,
\eqref{newt3} and~\eqref{newe3} tell us that 
$\PP(\ell(Y^{\e, \e^{1/\alpha}x})>\delta)=\PP(\ell(Y^{1,x})>\delta/\e) \sim \mathcal{V}(x_1) (\e/\delta)^{1/2}$
as $\e\to 0$. By dominated convergence, we end with
$$
\nn^\e(\ell>\delta) \sim \chi_\rG \delta^{-1/2} \int_\HH \mathcal{V}(x_1)\dr x= \delta^{-1/2} \nn_*(\ell>1)
$$
by definition of $\chi_\rG$. It holds that $\delta^{-1/2} \nn_*(\ell>1)=\nn_*(\ell>\delta)$ by Lemma~\ref{qdist}.
Finally, \eqref{newt2} (with $\e=1/t$) implies that $\nn^1(\ell>t)=t^{-1/2}\nn^\e(\ell>1)\sim t^{-1/2} 
\nn_*(\ell>1)$ as $t\to \infty$.
\end{proof}

The next step is to show that $\nn^\e(\cdot | \ell>\delta)\to \nn_*(\cdot|\ell>\delta)$.
To mimic the arguments of Doney~\cite{doney1985conditional}, we need the following result,
that will allow us to express $\nn^\e(\cdot | \ell>\delta)$ as a piece of the path of some process
$(U^\e_t)_{t\geq 0}$. Doney's work is concerned with a one-dimensional random walk, while we have to work 
a little more, because $d\geq 2$ and under $\nn^\e$,
the walk does not start from $0$, but from a random variable $O$ distributed according to $\rG_1$. 
The assumption that $\rG$ has a moment
of order $\alpha/2$ is crucial in the following results.

\begin{lemma}\label{tic1} Grant Assumption~\ref{assump_equi} and Assumption~\ref{assump:moments}-(a).
There exists a family of processes $((U^\e_t)_{t\geq 0}$, $\e \in (0,1])$, with the following properties.
\vip
(i) $(U^\e_t)_{t\geq 0}$ converges in law to the ISP$_{\alpha,0}$ for the local uniform topology as $\e\to0$.
\vip
(ii) Fix $\e \in (0,1]$ and set $\tau^\e_0=0$ and, for $n\geq 0$,
$$
\rho_{n}^\e=\inf\{t>\tau_{n}^\e : \Delta U^\e_t\neq 0\} \quad \text{and} \quad
\tau_{n+1}^\e=\inf\{t>\rho_n^\e : U^\e_t-U^\e_{\tau_n^\e} \notin \HH\}.
$$
The family $((V^{\e,n}_t)_{t\geq 0}=(U^\e_{(\rho^\e_n+t)\land \tau^\e_{n+1}}-U^\e_{\rho_n^\e-}), n\geq 0)$ is i.i.d. 
with common law $\chi_\rG^{-1} \e^{1/2}\nn^\varepsilon$.
\vip
(iii) For all $T>0$, it holds that (convention : $\rho^\e_{-1}=0$)
\begin{equation}\label{ult1}
\sup\{\rho_n^\e-\tau_n^\e : n\geq 0, \rho_{n-1}^\e \leq T\}\to 0 \quad \text{in probability as $\e\to 0$.}
\end{equation}

(iv) Set $I^\e_t=\inf_{s\in [0,t]}U^\e_t\cdot\be_1$. For all $\e\in (0,1]$, all $n\geq 0$,
\begin{equation}\label{ult2}
I^\e_{\tau_{n+1}^\e-}=I^\e_{\rho^\e_n-}=I^\e_{\tau^\e_n}=U^\e_{\tau_n^\e}\cdot \be_1=U^\e_{\rho_n^\e-}\cdot\be_1.
\end{equation}
\end{lemma}

\begin{proof}
We recall that $\rH_1$ is a probability density on $\R^d\times\HH$
with marginals $\rF_1$ and $\rG_1$, see Notation~\ref{fgh}. We consider
a Poisson measure $\rM=\sum_{s\in \JJ}\delta_{(s,u_s,v_s)}$ on 
$[0,\infty)\times\R^d\times\HH$ with intensity $\dr s \rH_1(u,v)\dr u \dr v$ and introduce
$(\cF_t=\sigma(\{\rM(A),A\in \cB([0,t]\times\R^d\times\HH)\})_{t\geq 0}$.
Then for 
$$
S_t=\int_0^t \int_{\R^d\times\HH} u \rM(\dr s,\dr u,\dr v),
$$ 
$(\e^{1/\alpha}S_{t/\e})_{t\geq 0}$ 
has the same law as $(Y^{\e,0}_t)_{t\geq 0}$ and thus converges in law
to an ISP$_{\alpha,0}$ for the local uniform topology as in Step~3 of 
the proof of Lemma~\ref{prop:conv_semi_g}. We next consider
$$
U_t=\int_0^t \int_{\R^d\times\HH}\!\!\! (u\indiq_{\{X_\sm>0\}}+v\indiq_{\{X_\sm=0\}}) \rM(\dr s,\dr u,\dr v),
\;\; \text{where}\;\; X_t=U_t\cdot\be_1-\inf_{s\in[0,t]} U_s\cdot\be_1.
$$
This path-dependent S.D.E. obviously has a unique solution $(U_t)_{t\geq 0}$, 
since $\rM$ has a finite number of jumps on each finite time interval.
\vip
{\it Step 1.} We introduce $\tau_0=0$ and, for $n\geq 0$,  we set 
$$
\rho_{n}=\inf\{t>\tau_{n} : \Delta U_t\neq 0\}
\quad \text{and} \quad  \tau_{n+1}=\inf\{t>\rho_n :  U_t-U_{\tau_n} \notin \HH\}.
$$
Here we check that for all $n\geq 0$, $\rho_{n}=\inf\{t>\tau_{n} : X_t>0\}$ and
$\tau_{n+1}=\inf\{t>\rho_n : X_t=0\}$. We set $U_{1,t}=U_t\cdot\be_1$ and
$I_t=\inf_{s\in[0,t]} U_{1,s}$, whence $X_t=U_{1,t}-I_t$. We only give the first steps.
\vip
\noindent $\bullet$ For $t\in [0,\rho_0)$, we have $U_t=0$ and $I_t=0$, whence $X_t=0$, 
so that $U_{\rho_0}=v_{\rho_0}\in \HH$, i.e. $U_{1,\rho_0}>0$. Thus $I_{\rho_0}=0$ and $X_{\rho_0}>0$.
Consequently, $\rho_0=\inf\{t>0 : X_t>0\}$.
\vip
\noindent $\bullet$ For $t\in [\rho_0,\tau_1)$, we have $U_t=U_t-U_{\tau_0}\in\HH$,
i.e. $U_{1,t}>0$. Thus $I_t=0$ and $X_t>0$ for all $t\in [\rho_0,\tau_1)$.
Moreover, $U_{\tau_1}\notin\HH$, i.e.  $U_{1,\tau_1}\leq 0$, so that $I_{\tau_1}=U_{1,\tau_1}$ and thus $X_{\tau_1}=0$.
Consequently, $\tau_{1}=\inf\{t>\rho_0 : X_t=0\}$.
\vip
\noindent $\bullet$ Since $X_{\tau_1}=0$ and since $\rho_1$ is the first jump instant of $U$ after $\tau_1$,
we have $X_t=0$ for all $t\in [\tau_1,\rho_1)$ and $U_{\rho_1}=U_{\tau_1}+v_{\rho_1}$. 
Since $v_{\rho_1} \in \HH$, we have $U_{1,\rho_1}>U_{1,\tau_1}$, whence 
$I_{\rho_1}=I_{\tau_1}=U_{1,\tau_1}$, so that $X_{\rho_1}=v_{\rho_1}\cdot \be_1>0$. Thus $\rho_1=\inf\{t>\tau_1 : X_t>0\}$.
\vip
\noindent $\bullet$ For $t\in [\rho_1,\tau_2)$, $U_t-U_{\tau_1} \in \HH$, i.e. $U_{1,t}>U_{1,\tau_1}$,
so that $I_t=I_{\rho_1}=U_{1,\tau_1}$, whence $X_t>0$. 
Moreover, $U_{\tau_2}-U_{\tau_1}\notin\HH$, i.e.  $U_{1,\tau_2}\leq U_{1,\tau_1}$, 
so that $I_{\tau_2}=U_{1,\tau_2}$ and thus $X_{\tau_2}=0$.
As a consequence, $\tau_{2}=\inf\{t>\rho_1 : X_t=0\}$.
\vip

{\it Step 2.} We set $V^n_t=U_{(\rho_n+t)\land \tau_{n+1}}-U_{\tau_n}$ and prove that
$((V^n_t)_{t\geq 0}, n\geq 0)$ is i.i.d. with common law $\chi_\rG^{-1}\nn_1$.
It suffices that for each $n\geq 0$, $(V^n_t)_{t\geq 0}$ is independent of $\cF_{\tau_n}$ and 
$\chi_\rG^{-1}\nn_1$-distributed
We let $\rK(\dr s,\dr u)=\sum_{s\in \JJ}\delta_{(s,u_s)}$
which is Poisson with intensity $\dr s \rF_1(u)\dr u$.
\vip
Since $X_{\rho_n-}=X_{\tau_n}=0$ and since $X_{s-}>0$
for all $s\in (\rho_n,\tau_{n+1}]$ by Step 1, 
$$
V^n_t=U_{\rho_n}-U_{\tau_n}+\int_{\rho_n+}^{(\rho_n+t)\land\tau_{n+1}}\int_{\R^d\times\HH} u \rM(\dr s,\dr u,\dr w)
=v_{\rho_n} + \int_{\rho_n+}^{(\rho_n+t)\land\tau_{n+1}}\int_{\R^d\times\HH} u \rK(\dr s,\dr u).
$$
Introducing $\rK^n=\sum_{s\in \JJ,s>\rho_n}\delta_{(s-\rho_n,u_s)}$, 
which is a Poisson measure with intensity $\dr s \rF_1(u)\dr u$ independent of $\cF_{\rho_n}$, we get
$$
V^n_t= v_{\rho_n} + \int_0^{t\land (\tau_{n+1}-\rho_n)}\int_{\R^d} u \rK^n(\dr s,\dr u).
$$
Thus $(V^n_t)_{t\geq 0}$ is independent of $\cF_{\rho_n-}$ and thus of $\cF_{\tau_n}$. Setting 
$Y^n_t=v_{\rho_n}+\int_0^t \int_{\R^d} u\rK^n(\dr s,\dr u)$, we have $\tau_{n+1}-\rho_n=\ell(Y^n)$, because 
$Y^n_t=U_{(\rho_n+t)\land \tau_{n+1}}-U_{\tau_n} \in \HH$ for all $t\in [0,\tau_{n+1}-\rho_n)$ and
$Y^n_{\tau_{n+1}-\rho_n}=U_{\tau_{n+1}}-U_{\tau_n} \notin \HH$ by definition of $\tau_{n+1}$.
Thus $(V^n_t)_{t\ge 0}=(Y^n_{t\land \ell(Y^n)})_{t\geq 0}$. 
Since $\rho_n = \inf\{t > \tau_n : \Delta U_t \neq 0\}$, 
$v_{\rho_n}$ is $\rG_1(v)\dr v$-distributed and independent of $\rK^n$ (because it is
$\cF_{\rho_n}$-measurable) and we conclude that $(V^n_t)_{t\geq 0}$ is 
$\chi_\rG^{-1}\nn_1$-distributed, recall Subsection~\ref{notex}.

\vip
{\it Step 3.} We now introduce $U^\e_t=\e^{1/\alpha}U_{t/\e}$ and verify (ii). For $(\tau_n^\e,\rho_n^\e)_{n\geq 0}$
as in the statement, we have $\tau_n^\e=\e\tau_n$ and $\rho_n^\e=\e\rho_n$. Indeed, $\tau_0^\e=0=\e\tau_0$
and if  $\tau_n^\e=\e\tau_n$ for some $n\geq 0$, then
$$
\rho^\e_{n}=\!\inf\{t> \tau_{n}^\e : \Delta U^\e_t\neq 0\}=\!\inf\{t> \e \tau_{n} : \Delta U_{t/\e}\neq 0\}
=\!\e\inf\{s>\tau_{n} : \Delta U_{s}\neq 0\}=\e\rho_{n},
$$
see Step 1. Similarly,
$$
\tau^\e_{n+1}=\!\inf\{t> \rho_n^\e : U^\e_t-U^\e_{\tau_n^\e}\notin \HH\}=\!\inf\{t>\e\rho_n : U_{t/\e}-U_{\tau_n}\notin \HH\}
=\!\e\inf\{s>\rho_n : U_{s}-U_{\tau_n}\notin \HH\},
$$
which equals $\e\tau_{n+1}$. 
Since $U^\e_{\rho_n^\e-}=U^\e_{\tau_n^\e}$, we get
$(U^\e_{(\rho^\e_n+t)\land \tau^\e_{n+1}}-U^\e_{\rho_n^\e-})_{t\geq 0}
=(\e^{1/\alpha} V^n_{t/\e})_{t\geq 0}$ for each $n\geq 0$, so that (ii) follows from Step 2 and~\eqref{newt2}.

\vip
{\it Step 4.} We next show (i). Since $(\e^{1/\alpha}S_{t/\e})_{t\geq 0}$ goes in law
to the ISP$_{\alpha,0}$, it suffices that 
$\Delta^\e_T=\sup_{t\in [0,T]}|U^\e_t-\e^{1/\alpha}S_{t/\e}|\to 0$ in probability for any $T>0$. But 
$$
\Delta^\e_T=\e^{1/\alpha}\sup_{[0,T/\e]}|S_t-U_t| 
\leq \e^{1/\alpha}\int_0^{T/\e}\int_{\R^d\times \HH} (|u|+|v|)\indiq_{\{X_\sm=0\}} \rM(\dr s,\dr u,\dr w).
$$
We set $M_t=\sum_{n\geq 0} \indiq_{\{\rho_n\leq t\}}$ and write, recalling Step~1,
$$
\Delta^\e_T \leq \e^{1/\alpha} \sum_{k=0}^{M_{T/\e}} (|u_{\rho_k}|+|v_{\rho_k}|)=
(\e^{1/2} M_{T/\e})^{2/\alpha} \Gamma_{M_{T/\e}},\quad \text{where}\quad 
\Gamma_n=\frac{1}{n^{2/\alpha}} \sum_{k=0}^n(|u_{\rho_k}|+|v_{\rho_k}|).
$$
To complete the proof, it suffices to check that $\lim_{\e\to 0} \Gamma_{M_{T/\e}}=0$ in probability and that
$\lim_{a\to \infty} \limsup_{\e\to 0} \PP(\e^{1/2} M_{T/\e}>a)=0$.
\vip
The sequence $(u_{\rho_k},v_{\rho_k})_{k\geq 0}$ is i.i.d. with common law $\rH_1$ (because for each $n\geq 0$, 
$\rho_n$ is the first jump time of $\rM$ after  $\tau_n$, which is an $(\mathcal{F}_t)_{t\geq0}$ stopping time). 
By Assumptions~\ref{assump_equi} 
and~\ref{assump:moments}-(a),
$$
\E[|u_{\rho_1}|^{\alpha/2}+|v_{\rho_1}|^{\alpha/2}]=\int_{\R^d} |z|^{\alpha/2}\rF_1(z)\dr z
+\int_\HH |z|^{\alpha/2}\rG_1(z)\dr z<\infty.
$$ 
We conclude e.g. from~\cite[Proposition 33]{bethencourt2022fractional} that
$\Gamma_n\to 0$ in probability as $n\to\infty$. Since $\rho_n<\infty$ for all $n\geq0$, 
we have $\lim_{\e\to 0} M_{T/\e}=\infty$ a.s., whence $\lim_{\e\to 0} \Gamma_{M_{T/\e}}=0$ in probability. Next, 
$$
\{\e^{1/2}M_{T/\e}>a\}=\{\rho_{\lfloor \e^{-1/2}a \rfloor}<T/\e\}\subset \cap_{k=0,\cdots,\lfloor \e^{-1/2}a \rfloor-1} 
\{\tau_{k+1}-\rho_k<T/\e\},
$$
because for all $n> k\geq 0$, $\rho_n\geq \rho_{k+1}-\rho_k \geq \tau_{k+1}-\rho_k$.
The sequence $(\tau_{n+1}-\rho_n)_{n\geq 0}$ is i.i.d. and $\PP(\tau_1-\rho_0>t)=\chi_\rG^{-1}\nn^1(\ell>t)$ by Step~2.
Thus  $\PP(\tau_1-\rho_0>t)\sim c t^{-1/2}$ by Lemma~\ref{tic0}, where $c=\chi_\rG^{-1}\nn_*(\ell>1)$. Hence
$$
\PP(\e^{1/2}M_{T/\e}>a)\leq (1-\PP(\tau_1-\rho_0>T/\e))^{\lfloor \e^{-1/2}a \rfloor}\sim (1-cT^{-1/2}\e^{1/2})^{\e^{-1/2}a} \to 
\exp(-cT^{-1/2}a)
$$
as $\e\to 0$. Thus $\lim_{a\to \infty}\limsup_{\e\to 0} \PP(\e^{1/2}M_{T/\e}>a)=0$ as desired.

\vip
{\it Step 5.} We now prove (iii). Set $\Delta_{\e,T}=\sup\{\rho_n^\e-\tau_n^\e : n\geq 0, \rho_{n-1}^\e \leq T\}$.
With the notation of Step 4, we have $\Delta_{\e,T}=\e \sup\{\rho_n-\tau_n : n\geq 0, n \leq M_{T/\e}+1\}$.
Thus for any $\eta>0$, any $a>0$,
$$
\PP(\Delta_{\e,T}>\eta)\leq \PP(\e^{1/2}M_{T/\e}>a)+\PP\Big(\sup_{n=0,\dots,\lfloor a\e^{-1/2}\rfloor+1}(\rho_n-\tau_n)>\eta 
\e^{-1}\Big).
$$
For each $n\geq 0$, $\rho_n-\tau_n$
is $\mathrm{Exp}(1)$-distributed. Thus for any $a>0$,
$$
\limsup_{\e\to 0}\PP(\Delta_{\e,T}>\eta)\leq \limsup_{\e\to 0}\Big[\PP(\e^{1/2}M_{T/\e}>a)+
(a\e^{-1/2}+1) e^{-\eta/\e}\Big]=  \limsup_{\e\to 0}\PP(\e^{1/2}M_{T/\e}>a).
$$
Since $\lim_{a\to \infty} \limsup_{\e\to 0}\PP(\e^{1/2}M_{T/\e}>a)=0$ as already seen, the result follows.

\vip

{\it Step 6.} For (iv), it suffices to study the case without $\e$.
By definition of $\rho_n$, we clearly have $U_{\rho_n-}=U_{\tau_{n}}$ and $I_{\rho_n-}=I_{\tau_n}$.
Since $\tau_{n}=\inf\{t>\rho_{n-1} : X_t=0\}$ with $X_t=U_{1,t}-I_t$ by Step~1, we have 
$U_{1,\tau_n}=I_{\tau_n}$. Finally, since $X_t\geq 0$ during $(\tau_n,\tau_{n+1})$, we have $I_{\tau_{n+1}-}=I_{\tau_n}$.
\end{proof}

We can now show that $\nn^\e(\cdot |\ell>\delta)\to\nn_*(\cdot |\ell>\delta)$, 
as well as some other useful convergences.

\begin{lemma}\label{tic11}
Grant Assumptions~\ref{assump_equi} and~\ref{assump:moments}-(a).
For all $\delta>0$, the set $\DD(\R_+,\R^d)$ being endowed with the $\JS$-topology,
$$
\nn^\e(\cdot|\ell>\delta) \to \nn_*(\cdot|\ell>\delta) \quad \text{as $\e\to 0$.}
$$
Moreover, for all $\delta>0$, all $y>0$, 
$$\lim_{\e\to 0}\nn^\e(M>y|\ell>\delta)=\nn_*(M>y|\ell>\delta)\quad \text{and}\quad
\lim_{\e\to 0}\nn^\e(\ell>\delta|M>y)=\nn_*(\ell>\delta|M>y).
$$
\end{lemma}

\begin{proof} We fix $\delta>0$, $y>0$ and consider the objects introduced in Lemma~\ref{tic1}.
By Skorokhod's representation theorem, we may assume 
that $(U^\e_t)_{t\geq 0}$ a.s. converges to 
some ISP$_{\alpha,0}$ $(Z_t)_{t\geq 0}$ for the local uniform topology.
\vip

{\it Step 1.} For $\e\in (0,1]$ fixed, we introduce 
$$
\sigma_\e=\inf\{n\geq 0 : \ell(V^{\e,n})>\delta\} \quad \text{and}\quad 
\sigma_\e'=\inf\{n\geq 0 : M(V^{\e,n})>y\}.
$$
By Lemma~\ref{tic1}-(ii), $\brv_\e=(V^{\e,\sigma_\e}_t)_{t\geq 0}\sim \nn^\e(\cdot | \ell>\delta)$ and
$\brv_\e'=(V^{\e,\sigma_\e'}_t)_{t\geq 0}\sim \nn^\e(\cdot | M>y)$. In other words, if setting 
$g_\e=\rho^\e_{\sigma_\e}$, $d_\e=\tau^\e_{\sigma_\e+1}$, $g_\e'=\rho^\e_{\sigma_\e'}$
and $d_\e'=\tau^\e_{\sigma_\e'+1}$,
$$
\brv_\e=(U^\e_{(g_\e+t)\land d_\e}-U^\e_{g_\e-})_{t\geq 0}\sim \nn^\e(\cdot | \ell>\delta) 
\;\;\; \text{and} \;\;\;
\brv_\e'=(U^\e_{(g_\e'+t)\land d_\e'}-U^\e_{g_\e'-})_{t\geq 0}\sim \nn^\e(\cdot | M>y).
$$

{\it Step 2.} Recall that $\nn_*$ was defined through an ISP$_{\alpha,0}$ $(Z_t)_{t\geq 0}$ and the corresponding
Poisson measure of excursions $\Pi_*=\sum_{u \in \JJ} \delta_{(u,e_u)}$: this Poisson measure has for intensity
$\dr u \nn_*(\dr e)$, see the paragraph around~\eqref{et}.
We set
$$
\sigma=\inf\{u \in \JJ : \ell(e_u)>\delta\} \quad \text{and}\quad 
\sigma'=\inf\{u \in \JJ : M(e_u)>y\}.
$$
Due to the properties of Poisson measures, 
$\brv=e_{\sigma}\sim \nn_*(\cdot|\ell>\delta)$ and $\brv'=e_{\sigma'}\sim \nn_*(\cdot|M>y)$. 
By~\eqref{et}, this rewrites, setting $g=\gamma_{\sigma-}$, $d=\gamma_{\sigma}$, 
$g'=\gamma_{\sigma'-}$ and $d'=\gamma_{\sigma'}$,
$$
\brv=(Z_{(g+t)\land d}-Z_g)_{t\geq 0} \sim \nn_*(\cdot | \ell>\delta) \quad \text{and}\quad 
\brv'=(Z_{(g'+t)\land d'}-Z_{g'})_{t\geq 0} \sim \nn_*(\cdot | M>y).
$$

{\it Step 3.} We now recall a few properties of the stable process $(Z_t)_{t\geq 0}$. We set 
$Z_{1,t}=Z_t\cdot \be_1$ and $I_t=\inf_{s\in [0,t]} Z_{1,s}$. Recall that $(\xi_t)_{t\geq 0}$ is the local
time of $(Z_{1,t}-I_t)_{t\geq 0}$ and that $(\gamma_u)_{u\geq 0}$ is its right-continuous generalized inverse.
\vip
(a) Almost surely, for all $u \in \JJ$, $(Z_t)_{t\geq 0}$ is continuous at $\gamma_{u-}$. 
\vip
{\it Indeed, we know from Doney~\cite[Lemma~2 point~6]{doney1985conditional} that a.s.,
for all $u\in \JJ$, $(Z_{1,t})_{t\geq 0}$ is continuous at $\gamma_{u-}$, and the result follows since a.s.,
$\{s\geq 0 : \Delta Z_t\neq 0\}=\{s\geq 0 : \Delta Z_{1,t}\neq 0\}$.}

\vip
(b) Almost surely, for all $u \in \JJ$, $\ell(e_u)\neq \delta$ and $M(e_u)\neq y$.
\vip 
{\it This follows from the fact that $\nn_*(\ell=\delta)=\nn_*(M=y)=0$ by Lemma~\ref{qdist}.}
\vip

(c) Almost surely, for all $u\in \JJ$, $e(\ell(e_u)-) \in \HH$ and for  all $s\in (0,\ell(e_u))$, it holds that
$e_u(s) \in \HH$ and $e_u(s-) \in \HH$.
\vip
{\it For $s\in(0, \ell(e_u))$, we have $e_u(s) \in \HH$ by definition of $\ell$ and
$e_u(s-) \in \HH$ by~\eqref{eer3}, which also implies that $e(\ell(e_u)-) \in \HH$.}
\vip
(d)  Almost surely, for all $t>s\geq 0$, if $I_{t-}=I_s$, then $u:=\xi_t=\xi_s \in \JJ$ and 
$\gamma_{u-}\leq s<t\leq \gamma_{u}$. 
\vip
{\it Since $(\xi_r)_{r\geq 0}$ increases only when $(I_r)_{r\geq 0}$ decreases, we have $\xi_t=\xi_s$. By definition
of $(\gamma_v)_{v\geq 0}$, this implies that for $u=\xi_t$, we have $\gamma_{u-}\leq s <t\leq \gamma_u$, whence
$u \in \JJ$.}
\vip
(e) Almost surely, for all $t>s\geq 0$, if $I_s=Z_{1,s-}\land Z_{1,s}$ and $\xi_{t}=\xi_s$,
then $\xi_s \in \JJ$ and $s=\gamma_{\xi_s-}$. 
\vip
{\it As in (d), we have $u:=\xi_s \in \JJ$ and $\gamma_{u-}\leq s<t\leq \gamma_u$. If $s>\gamma_{u-}$, 
then by (c) and~\eqref{et},
$Z_{1,s}\land Z_{1,s-}>Z_{1,\gamma_{u-}}$, implying that $I_s<Z_{1,s}\land Z_{1,s-}$.}
\vip
(f) Almost surely, for all $u \in \JJ$, all $s\in [0,\gamma_{u-})$, $I_s>I_{\gamma_{u-}}$. 
\vip
{\it By Bertoin~\cite[Lemma~1 p 218]{bertoin1996levy}, $(H_v:=-I_{\gamma_v})_{v\geq 0}$ is a stable subordinator 
and is thus
strictly increasing. Now for $u \in \JJ$ and $s<\gamma_{u-}$, we have $\xi_s<u$, implying that $H_{\xi_s}<H_{u-}$,
i.e. $I_{\gamma_{\xi_s}}>I_{\gamma_{u-}}$. Since $s\leq \gamma_{\xi_s}$, the conclusion follows.}
\vip
(g) Almost surely, for all $t\geq 0$, if $I_{t-}=Z_{1,t-}$, then $\Delta Z_t=0$.
\vip
{\it As in (a), it suffices to show that $\Delta Z_{1,t}=0$. If first $\Delta Z_{1,t}>0$, then
$I_t= I_{t-}=Z_{1,t-}\land Z_{1,t}$ and $Z_{1,t}>Z_{1,t-}$. By right continuity, there is $s>t$ such that 
$I_s=I_t$, whence $\xi_t \in \JJ$ and $t=\gamma_{\xi_t-}$ by (e). Thus $\Delta Z_t=0$ by (a). 
If next $\Delta Z_{1,t}<0$, then we also have $\xi_t \in \JJ$ (because else, $\Delta Z_t=0$
exactly as in the proof of~Theorem~\ref{mr1}, see the last paragraph before Step~2) and $t>\gamma_{\xi_t-}$
(because else, $\Delta Z_t=0$ by (a)). Thus $t\in(\gamma_{\xi_t-},\gamma_{\xi_t}]$, so that
$Z_{1,t-}> Z_{1,\gamma_{\xi_t-}}$ by~\eqref{eer3}, which contradicts the fact that 
$I_{t-}=Z_{1,t-}$.
}

\vip

{\it Step 4.} Here we show that $\lim_{\e \to 0} g_\e=g$ a.s. We set $U^\e_{1,t}=U^\e_t\cdot \be_1$
and $I^\e_t=\inf_{s\in [0,t]} U^\e_{1,s}$.
\vip
{\it Step 4.1.} We have $\Gamma=\limsup_{\e\to 0} g_\e \leq g$. Indeed, consider a (random) sequence $\e_k \to 0$
such that $\lim_k g_{\e_k}=\Gamma$. Recall that $d>g+\delta$, introduce $m=(g+d)/2$
and $n_k=\sup\{{n\geq 0} : \rho^{\e_k}_n\leq m\}$, as well as $\bg_k=\rho^{\e_k}_{n_k}$ and
$\bd_k=\tau^{\e_k}_{n_k+1}$. Up to extraction, we may assume that $\bg_k\to \bg \in [0,m]$ and
that $\bd_k \to \bd \in [\bg,\infty]$.

\vip
(i) We have $\bg \leq g$. Else, $\bg_k>g$ for all $k$ large enough, whence 
$U^{\e_k}_{1,g}\geq I^{\e_k}_{\bg_k}=U^{\e_k}_{1,\bg_k-}$ by~\eqref{ult2}. But
$\lim_k U^{\e_k}_{1,g}=Z_{1,g}$ and, up to extraction, $\lim_kU^{\e_k}_{1,\bg_k-}=Z_{1,\bg}$ or $Z_{1,\bg-}$ by 
Lemma~\ref{extra}. Thus $Z_{1,g} \geq Z_{1,\bg}$ or $Z_{1,g}\geq Z_{1,\bg-}$. By Step~3-(c), we have
$Z_{1,t}>Z_{1,g}$ and $Z_{1,t-}>Z_{1,g}$ for all $t\in (g,d)$. Thus $\bg \notin (g,d)$. Since $\bg\leq m<d$,
we have $\bg \leq g$.
\vip
(ii) We have $\bg \geq g$. By definition of $n_k$, 
$\rho_{n_k+1}^{\e_k} >m$, which implies that $\tau_{n_k+1}^{\e_k} >(g+m)/2=:a$
for all $k$ large enough, see~\eqref{ult1}. We have $\bg_k<a$ for $k$ large enough by (i).
Hence by~\eqref{ult2}, $I^{\e_k}_a=I^{\e_k}_{\bg_k-}$.
We have $I^{\e_k}_{a}\to I_a$ and, by Lemma~\ref{extra}, up to extraction, $I^{\e_k}_{\bg_k-}\to I_{\bg}$ or $I_{\bg-}$.
Thus  $I_a = I_{\bg}$ or $I_a = I_{\bg-}$.
Since $(I_t)_{t\geq 0}$ is nonincreasing, we conclude that $I_{a}=I_{\bg}$ in any case.
But $a\in (g,d)$, whence $I_{a}=I_g$ and we end with $I_{g}=I_{\bg}$.
This is not possible if $\bg<g$, see Step~3-(f).
\vip
(iii) We have $\bd \geq d$. Else, $\bd \in [g,d)$. We have already seen in (ii) that $\bd_k=\tau_{n_k+1}^{\e_k} >a$
for $k$ large enough, so that $\bd \in [a,d)\subset (g,d)$. By definition of $\bd_k$, we have 
$U^{\e_k}_{1,\bd_k}<U^{\e_k}_{1,\bg_k-}$. Using Lemma~\ref{extra} and that $g$ is a continuity point of 
$(Z_t)_{t\geq 0}$ by Step~3-(a), we deduce that either $Z_{1,\bd}\leq Z_{1,g}$ or $Z_{1,\bd-}\leq Z_{1,g}$.
This is not possible if $\bd\in (g,d)$ by Step~3-(c).
\vip
(iv) As a consequence, $\ell(V^{\e_k,n_k})=\bd_k-\bg_k\to \bd-\bg\geq d-g>\delta$ by definition of $g$ and $d$.
Thus for all $k$ large enough, $\sigma_{\e_k}\leq n_k$, whence $g_{\e_k}=\rho^{\e_k}_{\sigma_{\e_k}}
\leq \rho^{\e_k}_{n_k}=\bg_k \to g$. Hence $\Gamma=\lim_k g_{\e_k}\leq g$.
\vip

{\it Step 4.2.} We have $\theta=\liminf_{\e \to 0} g_\e \geq g$. Observe that $\theta$ is finite by Step~4.1
and consider a (random) sequence $\e_k \to 0$ such that $\lim_k g_{\e_k}=\theta$.  
By~\eqref{ult2} and since 
$$
d_{\e_k}=\tau^{\e_k}_{\sigma_{\e_k}+1}>\rho^{\e_k}_{\sigma_{\e_k}}+\delta=g_{\e_k}+\delta, 
$$
we have $I^{\e_k}_{g_{\e_k}+\delta}=I^{\e_k}_{g_{\e_k}-}$.
By Lemma~\ref{extra}, up to extraction, we have $I^{\e_k}_{g_{\e_k}+\delta}\to I_{\theta+\delta}$ or $I_{(\theta+\delta)-}$
and $I^{\e_k}_{g_{\e_k}-}\to I_{\theta}$ or $I_{\theta-}$. By monotony, we conclude that
$I_{(\theta+\delta)-}=I_\theta$ in any case. Step~3-(d) then implies that for $u=\xi_\theta$, we have
$u\in \JJ$, and $\gamma_{u-}\leq \theta$ and $\ell(e_u)=\gamma_{u}-\gamma_{u-}\geq \delta$, 
whence $\ell(e_u)> \delta$ by Step~3-(b). Thus $\sigma\leq u$, which implies 
that $g=\gamma_{\sigma-}\leq \gamma_{u-}\leq \theta$.

\vip

{\it Step 5.} We next show that $\lim_{\e \to 0} g_\e'=g'$ a.s.
\vip

{\it Step 5.1.} We have $\Gamma'=\limsup_{\e\to 0} g_\e' \leq g'$. Indeed, consider a (random) sequence $\e_k \to 0$
such that $\lim_k g_{\e_k}'=\Gamma'$. We recall that $d'>g'$ and we set
$m=(g'+d')/2$. We introduce $n_k=\sup\{n\geq 0 : \rho_n^{\e_k} \leq m\}$ and set $\bg'_k=\rho^{\e_k}_{n_k}$
and $\bd'_k=\tau^{\e_k}_{n_k+1}$. Exactly as in Step 4.1, we may assume that $\lim_k \bg'_k=\bg' \in [0,m]$
and that $\lim_k \bd'_k=\bd' \in [\bg',\infty]$ and we can show that $\bg'=g'$ and that $\bd'\geq d'$.
By definition of $g',d'$, there is $t\in (g',d']$ such that $|Z_{t}-Z_{g'}|>y$. By Step~3-(a) and
since $\bg'_k\to g'$, we know that $U^{\e_k}_{\bg'_k-}\to Z_g$.

\vip
(i) If first $t \in (g',\bd')$, then $t \in (\bg_k',\bd'_k)$ for all $k$ large enough. Since
$U^{\e_k}_t\to Z_t$, we conclude that for all $k$ large enough, $|U^{\e_k}_t-U^{\e_k}_{\bg'_k-}|>y$, implying that
$M(V^{\e_k,n_k})>y$ and thus that $\sigma'_k\leq n_k$, whence $g_k'\leq \bg_k'\to g'$.
Hence $\Gamma'=\lim_k g_{\e_k}'\leq g'$.
\vip
(ii) If next $t=\bd'$, then $t=d'$ (because $t\leq d'\leq \bd'$). We now prove that $\bd'_k=t$
for $k$ large enough. This will imply that $U^{\e_k}_{\bd_k'}=U^{\e_k}_{t}\to Z_{t}$, so that
$M(V^{\e_k,n_k})\geq |U^{\e_k}_{\bd_k'}-U^{\e_k}_{\bg'_k-}|>y$ for $k$ large enough and thus that $\Gamma'\leq g'$ as in (i).
The following three points show that $\bd'_k=t$ for $k$ large enough.
\vip
\noindent $\bullet$ We have $\bd'_k=\rho^{\e_k}_{n_k+1}\geq (g'+m)/2=:a$ for $k$ large enough
exactly as in Step~4.1-(ii).
\vip
\noindent $\bullet$ We have $\inf_{s\in [a,t-g')}d(Z_{g'+s}-Z_{g'},\HH)>0$ by~\eqref{eer3} (the length of the
excursion starting at $g'$ is $d'-g'=t-g'$), so that for $k$ large enough, for all $s\in [a,t)$,
$U^{\e_k}_s-U^{\e_k}_{\bg_k'-} \in \HH$ (recall that $(U^\e_s)_{s\geq 0}$ converges locally uniformly to $(Z_s)_{s\geq 0}$). 
\vip
\noindent $\bullet$ We have $Z_t-Z_{g'} \notin \bar \HH$ by~\eqref{eer3}, so that for 
$k$ large enough, $U^{\e_k}_t-U^{\e_k}_{\bg_k'-} \notin \HH$.

\vip
{\it Step 5.2.} We have $\theta'=\liminf_{\e\to 0} g_\e' \geq g'$. Observe that $\theta'<\infty$ by Step~5.1 and 
consider a (random) sequence $\e_k \to 0$
such that $\lim_k g_{\e_k}'=\theta'$. By definition of $g'_\e$, there is $s_k>0$ such that 
$I^{\e_k}_{(g'_{\e_k}+s_k)-}=I^{\e_k}_{g'_{\e_k}-}=U^{\e_k}_{1,g'_{\e_k}-}$, see~\eqref{ult2}, and 
$|U^{\e_k}_{g'_{\e_k}+s_k}-U^{\e_k}_{g'_{\e_k}-}|>y$. Up to extraction, 
we may assume
that $\lim_ks_k=s_0 \in [0,\infty]$.
\vip
(i) It holds that $s_0<\infty$, else, we would have $I_{\infty}=I_{\theta'}$, which does a.s. not occur.
\vip
(ii) Since  $|U^{\e_k}_{g'_{\e_k}+s_k}-U^{\e_k}_{g'_{\e_k}-}|>y$ for all $k$, we infer from Lemma~\ref{extra} that either
$|Z_{\theta'+s_0}-Z_{\theta'}|\geq y$ or $|Z_{(\theta'+s_0)-}-Z_{\theta'}|\geq y$ 
or $|Z_{\theta'+s_0}-Z_{\theta'-}|\geq y$ or $|Z_{(\theta'+s_0)-}-Z_{\theta'-}|\geq y$.
\vip
(iii) Exactly as in Step~4.2, we have $\xi_{\theta'+s_0}=\xi_{\theta'}$. 
\vip
(iv) We have $s_0>0$, because
else, (ii) would necessarily give us $Z_{\theta'}=\lim_k U^{\e_k}_{g'_{\e_k}+s_k}$, $Z_{\theta'-}=\lim_k U^{\e_k}_{g'_{\e_k}-}$
and $|Z_{\theta'}-Z_{\theta'-}|\geq y$. Moreover, since $I^{\e_k}_{g'_{\e_k}-}=U^{\e_k}_{1,g'_{\e_k}-}$, we find
$I_{\theta'-}= Z_{1,\theta'-}$. By Step~3-(g), this implies that $\Delta Z_{\theta'}=0$,
which contradicts the fact that  $|Z_{\theta'}-Z_{\theta'-}|\geq y$.
\vip
(v) Since $I^{\e_k}_{g'_{\e_k}-}=U^{\e_k}_{1,g'_{\e_k}-}$, Lemma~\ref{extra} implies that either $I_{\theta'}=Z_{1,\theta'}$
or $I_{\theta'}=Z_{1,\theta'-}$ or $I_{\theta'-}=Z_{1,\theta'}$ or $I_{\theta'-}=Z_{1,\theta'-}$. In any case, we find
$I_{\theta'}= Z_{1,\theta'-}\land Z_{1,\theta'}$. This, together with (iii), (iv) and 
Step~3-(e), tells us that $u=\xi_{\theta'} \in \JJ$ and that $\theta'=\gamma_{u-}$.
By Step~3-(a), $\theta'$ is a continuity point of $(Z_t)_{t\geq 0}$, so that (ii) rewrites as
$|Z_{\gamma_{u-}+s_0}-Z_{\gamma_{u-}}|\geq y$ or $|Z_{(\gamma_{u-}+s_0)-}-Z_{\gamma_{u-}}|\geq y$. Since 
$s_0 \in (\gamma_{u-},\gamma_u]$ by (iii), this implies that $M(e_u)\geq y$, whence $M(e_u)>y$ by Step~3-(b). Thus
$\sigma'\leq u$, so that $g'=\gamma_{\sigma'-}\leq \gamma_{u-}= \theta'$.

\vip
{\it Step 6.} Since $(U^\e_t)_{t\geq 0}\to (Z_t)_{t\geq 0}$ for the $\JS$-topology and since
$g=\gamma_{\sigma-}$ and $g'=\gamma_{\sigma'-}$ are continuity points of $(Z_t)_{t\geq 0}$ by Step~3-(a), we 
deduce from Steps~4 and~5 that almost surely,  
$\brw_\e=(U^\e_{g_\e+t}-U^\e_{g_\e-})_{t\geq 0}$ goes to $\brw=(Z_{g+t}-Z_{g})_{t\geq 0}$ and
$\brw_\e'=(U^\e_{g'_\e+t}-U^\e_{g'_\e-})_{t\geq 0}$ goes to  $\brw'=(Z_{g'+t}-Z_{g'})_{t\geq 0}$
for the $\JS$-topology.
\vip
{\it Step 7.} Here we check that a.s.,
$\brv_\e=(U^\e_{(g_\e+t)\land d_\e}-U^\e_{g_\e-})_{t\geq 0}$ goes to  $\brv=(Z_{(g+t)\land d}-Z_{g})_{t\geq 0}$
for the $\JS$-topology and $M(\brv_\e)$ goes to $M(\brv)$.
By Steps~1 and~2 (and since $\nn_*(M=y)=0$), this will prove that 
$\nn^\e(\cdot|\ell>\delta)\to \nn_*(\cdot|\ell>\delta)$ and that 
$\nn^\e(M>y|\ell>\delta)\to \nn_*(M>y|\ell>\delta)$.
\vip
Recalling the notation of Step~6, we have $d_\e-g_\e=\ell(\brw_\e)=\ell(\brv_\e)$ and
$d-g=\ell(\brw)=\ell(\brv)$. We consider some (random) continuous increasing $\lambda_\e:\R_+\to\R_+$ 
such that $\lambda_\e(0)=0$, 
$\sup_{t\geq 0}|\lambda_\e(t)-t|\to 0$ and $\sup_{[0,T]} |\brx_\e(t)-\brw(t)|\to 0$ for all $T>0$, where
$\brx_\e(t)=\brw_\e(\lambda_\e(t))$. Observe that $\ell(\brw_\e)=\lambda_\e(\ell(\brx_\e))$ 
and $M(\brv_\e)=M(\brw_\e)=M(\brx_\e)$.
\vip
Almost surely, for all $\e$ small enough, $\ell(\brx_\e)=\ell(\brw)$:
this follows from the facts that 
\vip

\noindent $\bullet$ we have $\brw_\e(t)\in \HH$ for all $t\in [0,\delta]$ because 
$\ell(\brw_\e)>\delta$ by definition of $d_\e$, so that
for $\e$ small enough, $\brx_\e(t) \in \HH$ for all $t \in[0,\delta/2]$;
\vip
\noindent $\bullet$ by~\eqref{eer3}, 
$\inf_{[\delta/2,\ell(\brw))} d(\brw(t),\HH^c)>0$, so that for $\e$ small enough, $\brx_\e(t) \in \HH$ 
for $t\in [\delta/2,\ell(\brw))$;
\vip
\noindent $\bullet$ by~\eqref{eer3}, $\brw(\ell(\brw)) \notin \bar\HH$, 
so that for $\e$ small enough, $\brx_\e(\ell(\brw)) \notin \bar\HH$.
\vip
Consequently, since $M(\brv_\e)=M(\brx_\e)$, for all $\e$ small enough such that
$\ell(\brx_\e)=\ell(\brw)$,
$$
|M(\brv_\e)-M(\brw)|= \Big|\sup_{t\in [0,\ell(\brw)]} |\brx_\e(t)| - \sup_{t\in [0,\ell(\brw)]} |\brw(t)|\Big|\leq
\sup_{t\in [0,\ell(\brw)]} |\brx_\e(t)-\brw(t)|\to 0.
$$

Moreover, $\brv_\e=(\brw_\e(t\land \ell(\brw_\e))_{t\geq 0}$  goes to  $\brv=(\brw(t\land \ell(\brw)))_{t\geq 0}$
for the $\JS$-topology, because $\sup_{t\geq 0}|\lambda_\e(t)-t|\to 0$ and, still for all $\e$ small enough
so that $\ell(\brx_\e)=\ell(\brw)$,
$$
\sup_{t\in [0,\infty)}|\brv_\e(\lambda_\e(t))-\brv(t)|
=\sup_{t\in [0,\infty)}|\brx_\e(t\land\ell(\brx_\e))-\brw(t\land \ell(\brw))|=
\sup_{t\in [0,\ell(\brw)]} |\brx_\e(t)-\brw(t)|\to 0.
$$

{\it Step 8.} Finally, we verify that for $\brv_\e'=(U^\e_{(g_\e'+t)\land d_\e}-U^\e_{g_\e'-})_{t\geq 0}$ 
and $\brv'=(Z_{(g'+t)\land d}-Z_{g'})_{t\geq 0}$, we have $\ell(\brv_\e')\to \ell(\brv)$ a.s.
By Steps~1 and~2 (and since $\nn_*(\ell=\delta)=0$), this will prove that 
$\nn^\e(\ell>\delta|M>y)\to \nn_*(\ell>\delta|M>y)$.

\vip

Recalling the notation of Step~6, we have $\ell(\brv_\e')=\ell(\brw_\e')$ and $\ell(\brv')=\ell(\brw')$.
Let $\lambda_\e:\R_+\to\R_+$ be continuous, increasing, with
$\lambda_\e(0)=0$, 
$\sup_{t\geq 0}|\lambda_\e(t)-t|\to 0$ and $\sup_{[0,T]} |\brx_\e'(t)-\brw'(t)|\to 0$ for all $T>0$, where
$\brx_\e'(t)=\brw_\e'(\lambda_\e(t))$. Observe that $\ell(\brw_\e')=\lambda_\e(\ell(\brx_\e'))$.
To complete the proof, it suffices to check that a.s., for all $\e$ small enough, $\ell(\brx_\e')=\ell(\brw')$.
This can be checked as Step~7,
provided $\liminf_{\e\to 0} \ell(\brx'_\e)>0$ a.s. But on $\{\liminf_{\e\to 0} \ell(\brx'_\e)=0\}$, there is 
$\e_n\to 0$
such that $\ell(\brx'_{\e_n})\to 0$, whence $M(\brx'_{\e_n})\to 0$ (since $\brx'_{\e_n}$ converges locally 
uniformly to $\brw$ which is càd and since $\brw(0)=0$).
This is not possible, because $M(\brx'_{\e})=M(\brw_{\e}')>y$ a.s. by definition.
\end{proof}

We continue to study the convergence of $\nn^\e$ to $\nn_*$ and prove some uniform
estimates on $\nn^\e$.

\begin{lemma}\label{tic2}
Grant Assumptions~\ref{assump_equi} and~\ref{assump:moments}-(a).
\vip
(i) There is a constant $C>0$ such that $\nn^\e(M>y)\leq C y^{-\alpha/2}$ for all $\e\in (0,1]$, all $y>0$.
\vip

(ii) For all $\delta>0$, all $y>0$, as $\e\to 0$,
$$
\nn^\e(M>y,\ell\leq \delta)\to \nn_*(M>y,\ell\leq \delta).
$$
\end{lemma}

\begin{proof} 
For (i), it suffices that $C:=\sup_{\e\in (0,1]}\nn^\e(M>1)<\infty$. Indeed, by~\eqref{newt2}
and since $M(\Phi_\e(e))=\e^{1/\alpha}M(e)$,
this will imply that $\nn^1(M>\e^{-1/\alpha})=\e^{1/2}\nn^\e(M>1)\leq C\e^{1/2}$ for all $\e \in (0,1]$, 
and thus that 
$\nn^1(M>z)\leq Cz^{-\alpha/2}$ for all $z>1$. This also holds true when $z \in (0,1)$, since $\nn^1$
is a finite measure. By~\eqref{newt2} again, we will find that for all $\e\in (0,1]$, all $y>0$,
$$
\nn^\e(M>y)=\e^{-1/2}\nn^1(M>y\e^{-1/\alpha}) \leq C y^{-\alpha/2}.
$$
To prove that $\sup_{\e\in (0,1]}\nn^\e(M>1)<\infty$, we use Lemmas~\ref{tic0} and~\ref{tic11}, writing
$$
\nn^\e(M>1) = \nn^\e(\ell>1) \frac{\nn^\e(M>1|\ell>1)}{\nn^\e(\ell>1|M>1)}
\longrightarrow \nn_*(\ell>1) \frac{\nn_*(M>1|\ell>1)}
{\nn_*(\ell>1|M>1)} \quad \text{as $\e\to 0$.}
$$
Hence, $\nn^\e(M>1)$ goes to some finite constant as $\e\to 0$.
Since finally $\nn^\e(M>1)\leq \nn^\e(\cE)=\chi_\rG \e^{-1/2}$ for all $\e\in (0,1]$, we conclude that
$\sup_{\e \in (0,1]} \nn^\e(M>1)<\infty$.

\vip
We next prove (ii). Writing
\begin{align*}
\nn^\e(M>y,\ell\leq \delta)=&\nn^\e(M>y)-\nn^\e(M>y,\ell> \delta)\\
=& \nn^\e(\ell>\delta)\frac{\nn^\e(M>y|\ell>\delta)}{\nn^\e(\ell>\delta)|M>y)}
- \nn^\e(\ell>\delta)\nn^\e(M>y|\ell>\delta),
\end{align*} 
it suffices to use Lemmas~\ref{tic0} and~\ref{tic11}.
\end{proof}

We now check~\eqref{tododo}, which is a consequence of the previous lemma.

\begin{proof}[Proof of~\eqref{tododo}]
By Lemma~\ref{tic2}-(i) with $\e=1$, we know that for $(U_t)_{t\geq 0}$ a continuous-time random walk
with incremental law $\rF_1$, jumping rate $1$ and initial condition $U_0\sim\rG_1$ with a moment
of order $\alpha/2$,
there is a constant $C>0$ such that for all $y>0$, $\PP(M(U)>y)\leq C y^{-\alpha/2}$.
Consequently, for $(\mathbf{S}_n)_{n\geq 0}$ a random walk with incremental law $\rF_1$ issued from $0$,
independent of some $U_0\sim\rG_1$, $\PP(\check{M}(U_0+\mathbf{S})>y)\leq C y^{-\alpha/2}$.
But we clearly have $\check{\ell}(\mathbf{S})\leq \check{\ell}(U_0+\mathbf{S})$ (since $U_0 \in \HH$), 
so that $\check{M}(\mathbf{S})\leq \check{M}(U_0+\mathbf{S})+|U_0|$.  Thus
$$
\PP(\check{M}(\mathbf{S}))>y)\leq \PP(\check{M}(U_0+\mathbf{S})>y/2)+\PP(|U_0|>y/2)\leq C y^{-\alpha/2}.
$$
\vskip-0.8cm
\end{proof}

We can finally give the
\vip

\begin{proof}[Proof of Proposition~\ref{lemma:estimates} under Assumption~\ref{assump:moments}-(a)]
First, \eqref{aan} directly follows from the fact that $\nn_*(\ell>\delta)
+\sup_{\e\in (0,1]} \nn^\e(\ell>\delta)<\infty$
by Lemmas~\ref{qdist} and~\ref{tic0}.
\vip
The estimate~\eqref{j1s} (with $\theta \in (0,1/2)$) follows from the facts that 
$\nn^\e(\ell>\delta)\leq C \delta^{-1/2}$ and 
$\nn^\e(M>y)<Cy^{-\alpha/2}$, see Lemmas~\ref{tic0} and \ref{tic2}-(i).

\vip
Concerning~\eqref{j2s}, we first write, for $\delta \in (0,1)$,
\[
\int_{\mathcal{E}}(\ell(e)\wedge 1) (\nn^\e - \nn^{\e,\delta})(\dr e) = \int_{\mathcal{E}}(\ell(e)\wedge 1) 
\bm{1}_{\{\ell(e) \leq \delta\}}\nn^\e(\dr e) \leq \int_0^\delta \nn^\e(\ell > u) \dr u.
\]
By  Lemma~\ref{tic0}, this is smaller than $C\int_0^\delta u^{-1/2}\dr u \leq C \delta^{1/2}$, from which we conclude
that $\lim_{\delta\to 0}\sup_{\e\in (0,1]} \int_{\mathcal{E}}(\ell(e)\wedge 1) (\nn^\e - \nn^{\e,\delta})(\dr e)=0$.
Next,
\[
 \int_{\mathcal{E}}(M(e)\wedge 1) (\nn^\e - \nn^{\e,\delta})(\dr e) = \int_{\mathcal{E}}(M(e)\wedge 1) 
\bm{1}_{\{\ell(e) \leq \delta\}}\nn^\e(\dr e)
= \int_0^1\nn^\e(M > y, \ell \leq \delta) \dr y.
\]
We have $\nn^\e(M(e) > y, \ell(e) \leq \delta) \to \nn_*(M(e) > y, \ell(e) \leq \delta)$ for each $y>0$
by Lemma~\ref{tic2}-(ii), and $\sup_{\e\in(0,1)}\nn^\e(M(e) > y, \ell(e) \leq \delta) \leq C_* y^{-\alpha/2}$
by Lemma~\ref{tic2}-(i). By dominated convergence,
\[
 \lim_{\e\to0}\int_{\mathcal{E}}(M(e)\wedge 1) (\nn^\e - \nn^{\e,\delta})(\dr e) 
= \int_0^1 \nn_*(M(e) > y, \ell(e) \leq \delta) \dr y=\int_{\mathcal{E}}(M(e)\wedge 1) 
\bm{1}_{\{\ell(e) \leq \delta\}}\nn_*(\dr e).
\]
This last quantity tends to $0$ as $\delta\to 0$ by~\eqref{se1} and since $\nn_*(\ell=0)=0$.

\vip

We now check~\eqref{j3s}. We fix $\delta>0$ and  $\phi:\cE\to \R$ bounded and continuous for the $\JS$-topology.
By {Lemma~\ref{tic11},} $\nn^\e(\phi|\ell>\delta)\to \nn_*(\phi|\ell>\delta)$. Since 
$\nn^{\e,\delta}(\dr e)=\indiq_{\{\ell(e)>\delta\}}\nn^\e(\dr e)$ and 
$\nn_*^{\delta}(\dr e)=\indiq_{\{\ell(e)>\delta\}}\nn_*(\dr e)$, this rewrites as
$$
\frac{\nn^{\e,\delta}(\phi)}{\nn^\e(\ell>\delta)} \to \frac{\nn_*^\delta(\phi)}{\nn_*(\ell>\delta)}.
$$
Recalling that $\nn^\e(\ell>\delta)\to\nn_*(\ell>\delta)$
by Lemma~\ref{tic0}, we find $\nn^{\e,\delta}(\phi)\to \nn_*^\delta(\phi)$ as desired.

\vip
We next verify~\eqref{j4s}, which resembles Lemma~\ref{imp}. We fix $\delta>0$ and $\eta\in (0,\delta/2)$
and write
$$
\nn^{\e,\delta}(\ell_r<\eta)=\chi_\rG\e^{-1/2}\PP(\ell_r(Y^\e)<\eta,\ell(Y^\e)>\delta),
$$
where we recall that $Y^\e_t=O_\e+\int_0^t \int_{\R^d}u \rK_\e(\dr s,\dr u)$, see Subsection~\ref{notex}.
Applying the strong Markov property at time $\ell_r(Y^\e)$, we get
\begin{align*}
\nn^{\e,\delta}(\ell_r<\eta)=&\chi_\rG\e^{-1/2}\E\Big[\indiq_{\{\ell_r(Y^\e)<\eta \land \ell(Y^\e)\}}
g_\e(Y^\e_{\ell_r(Y^\e)},\delta-\ell_r(Y^\e))\Big]\\
\leq& \chi_\rG\e^{-1/2}\E\Big[\indiq_{\{\ell_r(Y^\e)<\eta \land \ell(Y^\e)\}}
g_\e(Y^\e_{\ell_r(Y^\e)},\delta/2)\Big],
\end{align*}
where $g_\e(x,t)=\PP(\ell(Y^{\e,x})>t)$, with $(Y^{\e,x}_t)_{t\geq 0}$ defined in Subsection~\ref{notex}. 
By~\eqref{newe1}, we find 
$$
g_\e(x,\delta/2)\leq \Big(C \delta^{-1/2}(\e^{1/2}+x_1^{\alpha/2})\Big)\land 1 
\leq C_\delta (\e^{1/2} + x_1^{\alpha/2}\land 1),
$$
for some constant $C_\delta>0$ depending on $\delta$.
Consequently, $\nn^{\e,\delta}(\ell_r<\eta)\leq C_\delta (\chi_{\rG} I_{\eta,\e} +J_{\eta,\e})$, where
$$
I_{\eta,\e}= \PP(\ell_r(Y^\e)<\eta \land \ell(Y^\e)) \leq \PP(\ell_r(Y^\e)<\ell(Y^\e))=:I_\e
$$
and
\begin{align*}
J_{\eta,\e}= &\chi_\rG \e^{-1/2}\E\Big[\indiq_{\{\ell_r(Y^\e)<\eta \land \ell(Y^\e)\}}
\Big((Y^\e_{\ell_r(Y^\e)}\cdot \be_1)^{\alpha/2}\land 1\Big)\Big]\\
=& \int_{\cE} \indiq_{\{\ell_r(e)<\eta\land\ell(e)\}} 
[(e(\ell_r(e)) \cdot\be_1)^{\alpha/2}\land 1] \nn^\e(\dr e).
\end{align*}
It remains to show that $\lim_{\e\to 0}I_{\e}=0$ 
and $\lim_{\eta\to 0} \limsup_{\e\to 0} J_{\eta,\e}=0$.

\vip
Recall that $(Y^\e_t)_{t\geq 0}\stackrel{(d)}=(\e^{1/\alpha} Y^1_{t/\e})$, see  Subsection~\ref{notex}, so that
$$
(\ell_r(Y^\e),\ell(Y^\e))\stackrel{(d)}=(\e\ell_{\e^{-1/\alpha}r}(Y^1),\e\ell(Y^1)).
$$
Thus $I_{\e}= \PP(\ell_{\e^{-1/\alpha}r}(Y^1)<\ell(Y^1))$, which tends to $0$ as $\e\to 0$ 
by the monotone convergence theorem,
since $\ell_{\e^{-1/\alpha}r}(Y^1)$ a.s. increases as $\e$ decreases and since a.s., $\ell_{\e^{-1/\alpha}r}(Y^1)=
\ell(Y^1)$ for all $\e$ small enough. Indeed, 
since $(Y^1_t)_{t\geq 0}$ is a (continuous-time) random walk of which the incremental law $\rF_1$ has a density,
we a.s. have $\min_{t\in  [0,\ell(Y^1))}d(Y_t^1,\HH^c)>0$, so that $Y^1_t \in B_d(\e^{-1/\alpha}r\be_1,\e^{-1/\alpha}r)$ 
for all $t\in [0,\ell(Y^1))$ if $\e$ is small enough.
\vip
Next, recalling that $B_d(r\be_1,r)=\{x \in \R^d  : |x|^2<2rx_1\}$, when $\ell_r(e)<\ell(e)$, we have
$0\leq e(\ell_r(e))\cdot \be_1 \leq |e(\ell_r(e))|^2/(2r)$, so that, for some constant $C_r>0$
depending on $r$,
$$
J_{\eta,\e}\leq  C_r \int_{\cE} \Big(\sup_{t\in [0,\eta\land \ell(e)]} |e(t)|^\alpha \land 1 \Big)\nn^\e(\dr e)
=C_r\int_0^1 \nn^{\e}\Big(\sup_{t\in [0,\eta\land \ell(e)]} |e(t)| > y^{1/\alpha}\Big) \dr y.
$$ 
Assume for a moment that for each $\eta>0$, for a.e. $y\in(0,1)$,
\begin{equation}\label{eutaf}
\lim_{\e\to 0} \nn^{\e}\Big(\sup_{t\in [0,\eta\land \ell(e)]} |e(t)| > y^{1/\alpha}\Big)=\nn_*
\Big(\sup_{t\in [0,\eta\land \ell(e)]} |e(t)| > y^{1/\alpha}\Big).
\end{equation}
Then by dominated convergence, since $\nn^{\e}(\sup_{t\in [0,\eta\land \ell(e)]} |e(t)| > y^{1/\alpha})
\leq \nn^\e(M>y^{1/\alpha})\leq C y^{-1/2}$ by Lemma~\ref{tic2}-(i), we find that
$$
\limsup_{\e\to 0} J_{\eta,\e}\leq C_r\int_0^1 \nn_*\Big(\sup_{t\in [0,\eta\land \ell(e)]} |e(t)| > y^{1/\alpha}\Big) \dr y
= C_r \int_{\cE} \Big(\sup_{t\in [0,\eta\land \ell(e)]} |e(t)|^\alpha \land 1 \Big)\nn^\e(\dr e)
$$
This last quantity tends to $0$ as $\eta\to 0$, as shown at the end of the proof of Lemma~\ref{imp}.

\vip

To complete the proof of~\eqref{j4s}, it remains to show~\eqref{eutaf}. We write
$$
\nn^\e\Big(\sup_{s\in[0,\eta\wedge \ell(e)]}|e(s)| > y^{1/\alpha}\Big) = 
\nn^{\e,\eta}\Big(\sup_{s\in[0,\eta]}|e(s)| > y^{1/\alpha}\Big) + \nn^\e(M(e) > y^{1/\alpha}, \ell(e) \leq \eta).
$$
Thanks to Lemma~\ref{tic2}-(ii), the second term converges to $\nn_*(M(e) > y^{1/\alpha}, \ell(e) \leq \eta)$ 
as $\e\to0$. Regarding the first term, it follows from \eqref{j3s} that 
$\nn^{\e, \eta}$ converges weakly to $\nn_*^\eta$. Since $e \mapsto \sup_{[0,\eta]}|e(s)|$ is continuous 
at each $e$ such that $\Delta e(\eta)=0$, which holds true for $\nn_*$-a.e. $e\in \cE$,
we get that the law of $\sup_{s\in[0,\eta]}|e(s)|$ under $\nn^{\e,\eta}$ converges to the law of 
$\sup_{s\in[0,\eta]}|e(s)|$ under $\nn_*^{\eta}$. Hence, for a.e. $y\in(0,1)$, 
$\nn^{\e,\eta}(\sup_{s\in[0,\eta]}|e(s)| > y^{1/\alpha}) \to \nn_*^{\eta}(\sup_{s\in[0,\eta]}|e(s)| > y^{1/\alpha})$ as $\e\to0$.

\vip

Finally, we prove \eqref{j5s}. Fix $r >0$ and write, for $\delta >\eta>0$ (recall that $\ell>\ell_r$),
\[
\nn^\e(\ell_r > \eta) \geq \nn^\e(\ell_r > \eta, \ell > \delta) 
= \nn^\e(\ell > \delta) - \nn^{\e,\delta}(\ell_r \leq \eta).
\]
Using Lemma~\ref{tic0} and~\eqref{j4s}, we find
\[
\liminf_{\eta \to 0}\liminf_{\e\to0}\nn^\e(\ell_r > \eta) \geq \delta^{-1/2}\nn_*(\ell > 1).
\]
Letting $\delta\to0$, we conclude that $\liminf_{\eta \to 0}\liminf_{\e\to0}\nn^\e(\ell_r > \eta)=\infty$ as desired.
\end{proof}

\subsection{A technical lemma for the scattering process}

It remains to show Lemma \ref{tmok}.  We first show an intermediate result.

\begin{lemma}\label{lemma:tail_kinetic_ball}
Grant Assumption~\ref{assump_equi} and Assumption~\ref{assump:moments} and fix $r>0$.
Let $U$ be a $\mathrm{G}_+$-distributed random variable and let $\rP_\e$ be a 
Poisson measure on $\mathbb{R}_+ \times \mathbb{R}^d$ with intensity $\e^{-1}\dr s\mathrm{F}(v) \dr v$. 
We introduce $V_t^\e = U + \int_0^t\int_{\mathbb{R}^d}(v - V_{s-}^\e)\rP_\e(\dr s,\dr v)$ and 
$X_t^\e = \e^{1/\alpha -1}\int_0^t V_s \dr s$, as well as the stopping time
$\ell_r(X^\e)= \inf\{t > 0 : X_t^\e \notin B_d(r\be_1, r)\}$. 
There exists $\zeta \in (0,1/2]$ and $a_0 > 0$ such that
\[
\liminf_{\e\to0}\e^{-\zeta} \mathbb{P}(\ell_r(X^\e)> a_0) >0.
\]
\end{lemma}

\begin{proof}
We write $\rP_\e = \sum_{n\geq 1}\delta_{(T_n^\e, W_n)}$, where 
$(W_n)_{n\geq1}$ is i.i.d. and $\mathrm{F}$-distributed, $(E_n^\e)_{n\geq1}$ is i.i.d. 
and $\mathrm{Exp}(\e^{-1})$-distributed and $T_n^\e = \sum_{k=1}^n E_k^\e$. We have
\[
X_{T_1^\e}^\e = \e^{1/\alpha - 1} E_1^\e U \quad \text{and for }n\geq 2,\quad X_{T_n^\e}^\e 
= \e^{1/\alpha-1} E_1^\e U + \e^{1/\alpha-1}\sum_{k=2}^n E_{k}^\e W_{k-1}.
\]
We introduce the Poisson process $M^\e_t=\sum_{n\geq 1}\indiq_{\{T^\e_n\leq t\}}$ with parameter $\e^{-1}$.
Since $(X^\e_t)_{t\geq 0}$ is linear on $[T_n^\e, T_{n+1}^\e)$, $n\geq0$ (with $T_0^\e = 0$), 
and since $B_d(r\be_1,r)$ is convex, we have, for all $a>0$,
\begin{equation}\label{hh1}
\{\ell_r(X^\e)>a\} \supset \{X^\e_{T_n^\e}\in B_d(r\be_1,r) \text{ for all } n=1,\dots,M^\e_a+1\}.
\end{equation}
We next consider another (independent) sequence $(F_n^\e)_{n\geq1}$ of $\mathrm{Exp}(\e^{-1})$-distributed i.i.d.
random variables and we let $S_n^\e = \sum_{k=1}^n F_k^\e$. Then $\mathrm{K}_\e = 
\sum_{n\geq 1}\delta_{(S_n^\e, \e^{1/\alpha - 1}E_{n+1}^\e W_n)}$ is a Poisson measure on $\mathbb{R}_+ \times \mathbb{R}^d$ 
with intensity $\e^{-1} \dr s \mathrm{F}_\e(z) \dr z$, with $\rF_\e$ defined in Notation~\ref{fgh}. Thus for
\[
 Y_t^\e =\e^{1/\alpha}E_1^\e U + \int_0^t\int_{\mathbb{R}^d}z \mathrm{K}_\e(\dr s, \dr z) \quad \text{and}
\quad \ell(Y^\e)=\inf\{t>0 : Y^\e_t\notin \HH\},
\]
the process $(Y^\e_{t\land \ell(Y^\e)})_{t\geq 0}$ is $\chi_\rG^{-1}\e^{\zeta}\nn^\e$-distributed, with 
$\zeta=1/2$ under Assumption~\ref{assump:moments}-(a) and $\zeta=\beta/\alpha\in (0,1/2)$ 
under Assumption~\ref{assump:moments}-(b), see~\eqref{newt1}.
We thus deduce from~\eqref{j5s} that $\lim_{a\to 0} \liminf_{\e\to 0} \e^{-\zeta}\PP(\ell_r(Y^\e)>a)=\infty$, so that
for some $a_0>0$,
$$
p:=\liminf_{\e\to 0} \e^{-\zeta}\PP(\ell_r(Y^\e)>3a_0)>0.
$$
We now introduce the Poisson process $O^\e_t=\sum_{n\geq 1}\indiq_{\{S^\e_n\leq t\}}$ with parameter $\e^{-1}$ and 
observe that $Y^\e_{S_n^\e}=X^\e_{T_{n+1}^\e}$ for all $n\geq 0$, whence
\begin{align*}
\{\ell_r(Y^\e)>3a_0\}=&\{Y^\e_{S_{n}^\e} \in B_d(r\be_1,r) \text{ for all } n=0,\dots,O^\e_{3a_0}\}\\
=&\{X^\e_{T_{n}^\e} \in B_d(r\be_1,r) \text{ for all } n=1,\dots,O^\e_{3a_0}+1\}.
\end{align*}
Recalling~\eqref{hh1}, we see that 
$$
\PP(\ell_r(X^\e)>a_0)\geq\PP(\ell_r(Y^\e)>3a_0)-\PP(O^\e_{3a_0}< M^\e_{a_0}).
$$
Hence $\liminf_{\e\to 0} \e^{-\zeta}\PP(\ell_r(X^\e)>a_0) \geq p-\limsup_{\e\to 0}\e^{-\zeta}\PP(O^\e_{3a_0}< M^\e_{a_0})
=p$, because
\begin{align*}
\PP(O^\e_{3a_0}< M^\e_{a_0})=\PP(e^{M^\e_{a_0}-O^\e_{3a_0}}>1)\leq\E[e^{M^\e_{a_0}-O^\e_{3a_0}}]=
\exp\Big(\frac{a_0}\e [e-1+ 3(e^{-1}-1)] \Big)
\end{align*}
and because $e-1+3(e^{-1}-1) <0$. We used that $(M^\e_t)_{t\geq 0}$ and  $(O^\e_t)_{t\geq 0}$
are independent Poisson processes with parameter $\e^{-1}$.
\end{proof}

\begin{proof}[Proof of Lemma~\ref{tmok}]
We let to the reader the care to recall Definition~\ref{dfsp}. We only recall here that $\lambda(x,v,s)$ 
is defined by~\eqref{lambda} and that $T_1^\e = \lambda(\bX_0^\e,\e^{(1-\alpha)/\alpha}\bV_0^\e,E^\e_1)$ and that 
for any $n\geq 1$, $T_{n+1}^\e = T_n^\e + \lambda(\bX^\e_{T_n^\e},\e^{(1-\alpha)/\alpha}\bV^\e_{T_n^\e},E^\e_{n+1})$. 
Finally, $N_t^\e = \sum_{n\geq 1}\indiq_{\{T_n^\e \leq t\}}$.

\vip
\textit{Step 1.} It suffices that for all $t>0$, $\max_{n=1,\dots,\lfloor t/\e\rfloor}|T_n^\e-n\e|\to 0$
in probability as $\e\to0$.

\vip

Indeed, this implies that for all $t>0$, $\sup_{[0,t]} |\e N^\e_s-s|\to 0$ in probability, because
$$
\text{for all $\eta \in (0,t)$, all $\e \in (0,\eta/2)$,}\quad
\Big\{\sup_{[0,t]} |\e N^\e_s-s|>\eta\Big\}\subset \Big\{\max_{n=1,\dots,\lfloor 2t/\e\rfloor}|T_n^\e-n\e|>\eta/2\Big\}.
$$
Indeed, if there is $s\in [0,t]$ such that $\e N^\e_s-s>\eta$, then with $n=\lfloor \frac{s+\eta}\e
\rfloor \in \{0,\dots,\lfloor \frac{2t}\e\rfloor\}$, we have $T_n^\e\leq s \leq (n+1)\e-\eta < n\e -\eta/2$,
while if there is $s\in [0,t]$ such that $\e N^\e_s-s<-\eta$ (which implies that $s>\eta$), 
then with $n=\lfloor \frac{s-\eta}\e\rfloor\in \{0,\dots,\lfloor \frac{t}\e\rfloor\}$, 
we have $T_n^\e\geq s \geq \e n+\eta> n\e+\eta/2$.

\vip

This also implies that $\max_{n=1,\dots,N^\e_t}|T_n^\e-n\e|\to 0$ in probability, because for all $\eta>0$,
$$
\PP\Big(\max_{n=1,\dots,N^\e_t}|T_n^\e-n\e|>\eta\Big)\leq \PP\Big(N^\e_t > \frac {2t}\e\Big)+
\PP\Big(\max_{n=1,\dots,\lfloor 2t/\e\rfloor}|T_n^\e-n\e|>\eta\Big).
$$
By the above discussion, the first term tends to $0$, while the second one tends to $0$ by assumption.

\vip

{\it Step 2.} Since $\lambda(x,v,s)=s$ when $x+vs\in \Dd$, it holds that $T^\e_{n+1}-T^\e_n=E^\e_{n+1}$ 
when $\bX^\e_{T^\e_{n+1}}\in \Dd$. Since moreover we always have $\lambda(x,v,s)\leq s$,
$$
S^\e_n-R^\e_n\leq T^\e_n \leq S^\e_n, \quad \text{where} \quad S^\e_n=\sum_{k=1}^n E^\e_k\quad\text{and}
\quad R^\e_n=\sum_{k=1}^n E^\e_k\indiq_{\{\bX^\e_{T^\e_{k}} \in \pDd\}},
$$
whence
$$
\max_{n=1,\dots,\lfloor t/\e\rfloor}|T_n^\e-n\e| \leq \max_{n=1,\dots,\lfloor t/\e\rfloor}|S_n^\e-n\e| + R^\e_{\lfloor t/\e\rfloor}.
$$
Since $S^\e_n-n\e$ is a martingale, we deduce from Doob's inequality that
$$
\E\Big[\max_{n=1,\dots,\lfloor t/\e\rfloor}|S^\e_n-n\e|^2\Big]\leq 4 \text{Var } S^\e_{\lfloor t/\e\rfloor}=4\lfloor t/\e\rfloor
\text{Var } E^\e_1 = 4\lfloor t/\e\rfloor \e^2 \to 0.
$$

{\it Step 3.} 
It only remains to show that $R^\e_{\lfloor t/\e\rfloor}\to 0$ in probability as $\e\to 0$. We fix $t>0$ and $\eta>0$
and we consider
$\zeta \in (0,1/2]$ as Lemma~\ref{lemma:tail_kinetic_ball}. We write
$\PP(R^\e_{\lfloor t/\e\rfloor}>\eta)\leq I_{\e}+J_{\e}$, where
$$
I_{\e}=\PP\Big(\max_{n=1,\dots,\lfloor t/\e\rfloor} E^\e_n \geq \e^{(\zeta+1)/2}\Big), \quad 
J_{\e}=\PP(m^\e_t > \eta \e^{-(\zeta+1)/2})
\quad \hbox{and}\quad
m^\e_t=\sum_{k=1}^{\lfloor t/\e\rfloor}\indiq_{\{\bX^\e_{T^\e_{k}} \in \pDd\}}.
$$

First, $I_\e=1-\PP(E^\e_1< \e^{(\zeta+1)/2})^{\lfloor t/\e\rfloor} = 1-(1-e^{-\e^{-(1-\zeta)/2}})^{\lfloor t/\e\rfloor}$ 
tends to $0$ as $\e\to 0$. 
\vip
Next, we let $\tau^\e_1=\inf\{T^\e_k : k\geq 1, \bX^\e_{T^\e_k}\in \pDd\}$
and $\tau^\e_{n+1}=\inf\{T^\e_k : k\geq 1, T^\e_k> \tau^\e_n, \bX^\e_{T^\e_k}\in \pDd\}$ for $n\geq 1$.
Thanks to Remark~\ref{imp4}, there exists $r > 0$ such that $B_d(x + r\bn_x, r) \subset \Dd$ for all $x\in\pDd$
so that, if we let 
$\gamma_{n+1}^\e = \inf\{t > 0 : \bX_{\tau_n^\e + t}^\e \notin B_d(\bX_{\tau_n^\e}^\e + r\bn_{\bX_{\tau_n^\e}^\e}, r)\}$, 
then $\gamma_{n+1}^\e \leq \tau_{n+1}^\e - \tau_n^\e$ for all $n\geq 1$. Thus 
$\tau^\e_n \geq \sum_{k=2}^n \gamma^\e_k$ for all $n\geq 2$. 
\vip
By the strong Markov property of 
$(\bX^\e_t, \bV_t^\e)_{t\geq0}$ and by the rotational invariance of $\mathrm{F}$ 
and $\mathrm{G}$, the sequence $(\gamma_n^\e)_{n\geq 2}$ is i.i.d. and $\gamma_2^\e$ has the same law
as $\ell_r(X^\e)$ in Lemma~\ref{lemma:tail_kinetic_ball}. 
\vip
All in all, setting $n_{\e}=\lfloor \eta \e^{-(\zeta+1)/2}\rfloor$,
$$
J_{\e}=\PP(\tau^\e_{n_{\e}}<t) \leq  
\mathbb{P}\Big(\sum_{k=2}^{n_{\e}} \gamma_k^\e < t\Big)
\leq e^t \E\Big[\exp\Big(-\sum_{k=2}^{n_{\e}} \gamma_k^\e\Big)\Big]
=e^t \big( \E[e^{-\ell_r(X^\e)}] \big)^{n_{\e}-1}.
$$
By Lemma~\ref{lemma:tail_kinetic_ball}, there are $q>0$ and $a_0>0$ 
so that $\PP(\ell_r(X^\e)>a_0) \geq q\e^{\zeta}$ for all $\e\in (0,1]$ small enough, implying that
$$
\E[e^{-\ell_r(X^\e)}]\leq \E[e^{-a_0\indiq_{\{\ell_r(X^\e)>a_0\}}}]=1 - (1-e^{-a_0})\PP(\ell_r(X^\e)>a_0)
\leq 1 - c \e^{\zeta},
$$
where $c=(1-e^{-a_0})q>0$. We thus find
$J_{\e}\leq e^t (1-c\e^{\zeta})^{n_{\e}-1}$,
which tends to $0$ as $\e\to 0$.
\end{proof}

\appendix

\section{Geometric lemmas and inequalities}\label{ageo}

In this section, we check that the cutoff function $\Lambda$ is continuous,
we build some regular families of isometries,
we establish some parameterization lemmas of constant use and we prove the geometric 
inequalities stated in Proposition~\ref{tyvmmm}.

\subsection{Continuity of the cutoff function}

Here we check the following result.

\begin{lemma}\label{Lambdacon}
Assume that $\Dd$ is open, bounded and strictly convex. The function $\Lambda : \cDd\times \R^d\to \cDd$ 
defined in~\eqref{Lambda} is continuous.
\end{lemma}
 
\begin{proof} Let $V:\R^d \to \R_+$ be defined by $V(x)=d(x,\cDd)$. The function $V$ is continuous and convex.
For $(y,z)\in \cDd\times \R^d$, we define $v_{y,z}:[0,1]\to \R_+$ by $v_{y,z}(t)=V(y+t(z-y))$. Observing that
$v_{y,z}$ is continuous and convex and that $v_{y,z}(0)=0$ (because $y \in \cDd$), we may define
$t_{y,z}=\max\{t\in [0,1] : v_{y,z}(t)=0\}$. It then holds that $\Lambda(y,z)=y+t_{y,z}(z-y)$.
\vip
We consider a sequence $(y_k,z_k) \in \cDd\times \R^d$ converging to some $(y,z)\in \cDd\times\R^d$,
and we have to show that $\lim_k\Lambda(y_k,z_k)=\Lambda(y,z)$. Since the sequence $(\Lambda(y_k,z_k))_{k\geq 1}$
is valued in $\cDd$ which is compact, it suffices to show that $\Lambda(y,z)$ is its 
only accumulation point. We thus may assume that $\lim_{k\to \infty} \Lambda(y_k,z_k)$ exists.
Extracting a subsequence, we may moreover assume that $t_{y_k,z_k}$ converges to some $t\in [0,1]$.
Finally, using again subsequences, we may assume either (a) for all $k\geq 1$, $z_k \in \cDd$ or
(b) for all $k\geq 1$, $z_k \notin \cDd$.

\vip

{\it Case (a).} In this case, we have $z=\lim_k z_k \in \cDd$, whence $\Lambda(y_k,z_k)=z_k\to z=\Lambda(y,z)$.

\vip

{\it Case (b).} We always have $t\leq t_{y,z}$. Indeed, this follows from the fact that $v_{y,z}(t)=0$, 
because $v_{y,z}(t)=V(y+t(z-y))=\lim_k V(y_k+t_{y_k,z_k}(z_k-y_k))=0$.

\vip

If first $t= t_{y,z}$, which means that $\lim_k t_{y_k,z_k}=t_{y,z}$, then
$\Lambda(y_k,z_k)=y_k+t_{y_k,z_k}(z_k-y_k)$ tends to $y+t_{y,z}(z-y)=\Lambda(y,z)$ as $k\to \infty$.

\vip

If next $t< t_{y,z}$, then $v_{y,z}(s)=0$ (implying $y+s(z-y)\in \cDd$)
for all $s\in (t,t_{y,z})$. Moreover, $y+s(z-y)=\lim_k (y_k+s(z_k-y_k)) \in \Dd^c$ for
all $s \in (t,t_{y,z})$, because $\Dd^c$ is closed and because 
$y_k+s(z_k-y_k) \in \Dd^c$ for all $k$ large enough so that $s>t_{y_k,z_k}$.
Thus $\pDd$ contains the line segment $\{y+s(z-y) : s \in (t,t_{y,z})\}$, 
which is not possible by strict convexity of $\Dd$, except if $y=z$. In such a case,
$\Lambda(y_k,z_k)=y_k+t_{y_k,z_k}(z_k-y_k)$ tends to $y+t(z-y)=y=\Lambda(y,z)$.
\end{proof}

\subsection{Some regular families of isometries}

We recall that for $y\in \pDd$, $\cI_y$ stands for the set of linear isometries
sending $\be_1$ to $\bn_y$.

\begin{lemma}\label{locglob}
Suppose Assumption~\ref{as}. 
\vip
(i) If $d=2$, there exists a family $(A_y)_{y \in \pDd}$ such that
$A_y \in \cI_y$ for all $y \in \pDd$ and such that $y \mapsto A_y$ is globally Lipschitz continuous
on $\pDd$.
\vip
(ii) If $d\geq 3$, for any $z \in \pDd$, there exists a family $(A_y^z)_{y \in \pDd}$ such that
$A_y^z \in \cI_y$ for all $y \in \pDd$ and such that $y \mapsto A_y^z$ is locally Lipschitz continuous
on $\pDd\setminus\{z\}$.
\end{lemma}

\begin{proof}
When $d=2$, we define $A_y\in \cI_y$ 
by $A_y \be_1=\bn_y$ and $A_y \be_2=-(\bn_y\cdot \be_2) \be_1+ (\bn_y\cdot \be_1)\be_2$. 
The map $y\mapsto A_y$ has the same regularity as $y\mapsto \bn_y$,
{\it i.e.} $C^2$ under Assumption~\ref{as}.

\vip
When $d\geq 3$, we fix $z \in \pDd$ and consider a $C^3$-diffeomorphism $f:\Sp_{d-1} \to \pDd$ 
(e.g. the restriction 
to $\Sp_{d-1}$ of $\Phi^{-1}$ defined in Lemma~\ref{parafac}). We set $p=f^{-1}(z)$
and consider a $C^\infty$-diffeomorphism $g:\R^{d-1}\to \Sp_{d-1}\setminus\{p\}$.
Thus $h=f \circ g:\R^{d-1}\to \pDd\setminus \{z\} $ is a $C^3$-diffeomorphism.
We consider the canonical basis $(\be_2',\dots,\be'_{d})$ of $\R^{d-1}$, and define, for
$y \in \pDd\setminus  \{z\}$ and $k=2,\dots,d$,
$\bbf_k(y)=d_{h^{-1}(y)} h (\be_k')$. For each $y \in \pDd\setminus  \{z\}$,
the family $(\bbf_2(y),\dots,\bbf_{d}(y))$ is a basis of the tangent space $T_y$ at $y$ to $\pDd$
and for each $k=2,\dots,d$, the map $y\to \bbf_k(y)$ is of class $C^{2}$ on $\pDd \setminus\{z\}$.
 We now use the Gram-Schmidt procedure
to get, for each $y \in \pDd\setminus \{z\}$, an orthonormal basis $(\bbg_2(y),\dots,\bbg_{d}(y))$ of $T_y$.
For each $k=2,\dots,d$, the map $y\to \bbg_k(y)$ is still of class $C^{2}$  on $\pDd \setminus\{z\}$ (because 
$(\bbf_2(y),\dots,\bbf_{d}(y))$ are locally uniformly linearly independent for $y \in \pDd\setminus
 \{z\}$). 
We define
$A_y^z$ by $A_y^z\be_1=\bn_y$ and $A_y^z \be_k= \bbg_k(y)$ for $k=2,\dots,d$. We have $A_y^z\in \cI_y$ for each
$y \in \pDd\setminus\{z\}$, and $y\mapsto A_y^z$ is of class $C^{2}$ on $\pDd\setminus \{z\}$.
Finally, we choose $A_z^z \in \cI_z$ arbitrarily.
\end{proof}

\subsection{Some parameterization lemmas}

The following lemma may be standard, but we found no precise reference.

\begin{lemma}\label{parafac}
Grant Assumption~\ref{as}.
There is a $C^3$-diffeomorphism $\Phi : \R^d \to \R^d$, with both  $D \Phi$ and $D \Phi^{-1}$ bounded on $\R^d$,
such that $\Phi(\Dd)={B}_d(0,1)$. 
Consequently, there is $\kappa>0$
such that for all $x,y \in \R^d$, $\kappa^{-1} |x-y| \leq |\Phi(x)-\Phi(y)|\leq \kappa |x-y|$.
\end{lemma}

\begin{proof} 
We assume without loss of generality that $0\in \Dd$ and split the proof into three steps.
\vip
{\it Step 1.} Let $r:\Sp_{d-1}\to (0,\infty)$ be uniquely defined, for $\sigma \in \Sp_{d-1}$, by 
the fact that $r(\sigma)\sigma \in \pDd$. We then have $\Dd=\{u\sigma : \sigma \in \Sp_{d-1}, u\in [0,r(\sigma))\}$.
Let us show that $r$ is of class $C^3$. 
\vip
We fix $\sigma_0 \in \Sp_{d-1}$
and set $x_0=r(\sigma_0)\sigma_0 \in \pDd$. Since $\pDd$ is of class $C^3$, there exists a neighborhood $U$ of 
$x_0$ in $\R^d$ and a $C^3$ function $F:U \to \R$ such that $\nabla F$ does not vanish and
$\pDd \cap U = \{u \in U : F(u) = 0\}$.
Moreover, $\nabla F(x_0)$ is colinear to $\bn_{x_0}$. 
Let $\psi(t,\sigma)=F(t\sigma)$, which is defined and $C^3$ on a neighborhood $V$ of $(r(\sigma_0),\sigma_0)$ in 
$(0,\infty)\times \Sp_{d-1}$. We have $\psi(r(\sigma_0),\sigma_0)=F(x_0)=0$ and 
$\partial_t \psi(r(\sigma_0),\sigma_0)=\sigma_0\cdot \nabla F(x_0)\neq 0$ (since $0\in\pDd$,
since $\nabla F(x_0)$ is colinear to $\bn_{x_0}$ and since for all $a\in \Dd$, all $y \in \pDd$, 
$(a-y)\cdot \bn_y>0$ by convexity of $\Dd$).
By the implicit function theorem, there is a $C^3$ function $g$, defined on a neighborhood $W$ of $\sigma_0$
in $\Sp_{d-1}$, such that $\psi(g(\sigma),\sigma)=0$ (whence $g(\sigma)\sigma \in \pDd$) for all $\sigma \in W$. 
We have $g(\sigma)=r(\sigma)$
for all $\sigma \in W$, so that $r$ is $C^3$ on $W$.
\vip

{\it Step 2.} We introduce  
$$
r_0=\min_{\Sp_{d-1}}r, \quad r_1=\max_{\Sp_{d-1}}r, \quad \eta=\frac{r_0\land 1}2, \quad 
m_0=\frac{r_0-\eta}{1-\eta} \quad \text{and}\quad m_1=\frac{r_1-\eta}{1-\eta}. 
$$
Consider some smooth
$\varphi:\R_+\to[0,1]$ such that $\varphi(t)=0$ for $t\in [0,\eta/2]\cup[2,\infty)$ and such that
$\int_0^1 \varphi(t)\dr t=1-\eta$. We set 
\begin{gather*}
q(t,m)=\int_0^t (1+\varphi(s)(m-1))\dr s \geq 0 \quad
\text{ for $t\in \R_+$, $m\in [m_0,m_1]$,}\\
\bmm(\rho)=\frac{\rho-\eta}{1-\eta}\in [m_0,m_1]\quad 
\text{ for $\rho\in[r_0,r_1]$.}
\end{gather*}
We have the following properties.
\vip
\noindent (i) $q(t,m)=t$ for all $t\in [0,\eta/2]$, all $m\in [m_0,m_1]$, 
\vip
\noindent (ii) $q(1,m)=(1-\eta)m+\eta$ for all $m \in [m_0,m_1]$,
\vip
\noindent (iii) $q(1,\bmm(\rho))=\rho$ for all $\rho \in [r_0,r_1]$, 
\vip
\noindent (iv) $\partial_t q(t,m)=1+\varphi(t)(m-1)\in [\eta,1+m_1]$ for all $t\in \R_+$, all 
$m\in [m_0,m_1]$,
\vip
\noindent (v) $\partial_m q(t,m)=\int_0^t \varphi(s) \dr s$, 
so that $\partial_mq(t,m)$ is bounded on $\R_+\times [m_0,m_1]$,
\vip
\noindent (vi) for all $t\geq 0$, all $m \in [m_0,m_1]$, $\eta t\leq q(t,m) \leq (1+m_1) t$.
\vip
Points (i), (ii), (iii) and (v) are straightforward. To show that $1+\varphi(t)(m-1)\geq \eta$ 
in (iv), we use that $m_0\geq \eta$.  Point (vi) uses that by (i) and (iv), for every $t\geq 0$,
every $m\in [m_0,m_1]$,
$$
q(t,m)=q(0,m)+ \int_0^t \partial_t q(s,m)\dr s \in [\eta t, (1+m_1)t].
$$

{\it Step 3.} We now show that $H(x)=q(|x|,\bmm(r(\frac{x}{|x|})))\frac{x}{|x|}$ 
is a $C^3$-diffeomorphism from $\R^d$ into $\R^d$, with $DH$ and $D H^{-1}$ bounded
on $\R^d$, and that $H(B_d(0,1))=\Dd$. Setting $\Phi=H^{-1}$ will complete the proof.

\vip
By Step 1, $\Dd=\{u \sigma, \sigma \in \Sp_{d-1}, u \in [0,r(\sigma))\}$.
For each $\sigma \in \Sp_{d-1}$, $t \mapsto q(t,\bmm(r(\sigma)))$ continuously increases
from $0$ (at $t=0$) to $q(1,\bmm(r(\sigma)))=r(\sigma)$ (at $t=1$) to $\infty$ (at $t=\infty$). 
Thus  $H(B_d(0,1))=\Dd$ and $H$ is a bijection from $\R^d$ into $\R^d$.

\vip
Since $H(x)=x$ for all $x\in B_d(0,\eta/2)$, we only have to show that $H$ is $C^3$
on $\R^d\setminus B_d(0,\eta/4)$, with $DH$ bounded from above and from below 
on $\R^d\setminus B_d(0,\eta/4)$. 

\vip
First,
$H$ is of class $C^3$ on  $\R^d\setminus B_d(0,\eta/4)$, because $q$ is smooth on $\R_+\times [m_0,m_1]$,
$x\mapsto \frac x{|x|}$ is smooth on $\R^d\setminus B_d(0,\eta/4)$, $r$ is $C^3$ on $\Sp_{d-1}$ and $\bmm$
is smooth on $[m_0,m_1]$.

\vip
For $x=t\sigma \in \R^d$ with $\sigma \in \Sp_{d-1}$ and $t\geq \eta/4$, and for $y \in \R^d$, we write
$y=a \sigma+b \tau$, with $\tau\in \Sp_{d-1}$ orthogonal to $\sigma$ and $a^2+b^2=|y|^2$. 
We find that 
$$
DH(t\sigma)(y)=[a \rho_1(t,\sigma)+b\rho_2(t,\sigma,\tau)]\sigma+b\rho_3(t,\sigma)\tau,
$$
where $\rho_1(t,\sigma)=\partial_t q(t,\bmm(r(\sigma)))$,
$$
\rho_2(t,\sigma,\tau)=\frac{\partial_mq(t,\bmm(r(\sigma))
\bmm'(r(\sigma))Dr(\sigma)(\tau)}t \quad \text{and} \quad \rho_3(t,\sigma)=\frac{q(t,\bmm(r(\sigma)))}{t}.
$$
Since $\rho_1$, $\rho_2$, $\rho_3$ are bounded (for $t> \eta/4$) by Steps 1 and 2, we conclude that
$|DH(t\sigma)(y)|\leq C|y|$, showing that $DH$ is bounded on $\R^d\setminus B_d(0,\eta/4)$.

\vip

Next, we recall that $\rho_1\geq \eta$, that $\rho_3\geq \eta$, and call
$A=\sup_{t\geq \eta/4,\sigma \in \Sp_{d-1}, \tau \in \Sp_{d-1}}|\rho_2(t,\sigma,\tau)|$.
If first $|b|\leq \frac{|a|\eta}{2A}$, then
$$
|DH(t\sigma)(y)|\geq |a \rho_1(t,\sigma)+b\rho_2(t,\sigma,\tau)| \geq |a|\eta - |b|A \geq \frac{|a|\eta}2=  
\frac{|a|\eta}2\lor (A|b|).
$$
If next $b\geq \frac{|a|\eta}{2A}$, then 
$$
|DH(t\sigma)(y)|\geq |b\rho_3(t,\sigma)| \geq \eta|b|=(\eta|b|) \lor \frac{\eta^2|a|}{2A}).
$$ 
Hence there is $c>0$ such that in any case, $|DH(t\sigma)(y)|\geq c(|a|\lor|b|)\geq \frac c 2 |y|$.
This shows that $DH$ is uniformly bounded from below on $\R^d\setminus B_d(0,\eta/4)$.
\end{proof}

The next lemma relies more deeply on Assumption~\ref{as}, see Remark~\ref{asp}.

\begin{lemma}\label{paradur}
Grant Assumption~\ref{as}.
For $x\in \pDd$ and $A\in\cI_x$, let $\Dd_{x,A}={\{y \in \R^d : h_x(A,y)\in\Dd\}}$.
There are $\e_1>0$, $\eta>0$, $C>0$ such that the following properties hold true.
\vip
(i) For all $x\in \pDd$, all $A\in \cI_x$,
there is a $C^3$ function $\psi_{x,A}: B_{d-1}(0,\e_1)\to \R_+$ such that 
$\psi_{x,A}(0)=0$, $\nabla \psi_{x,A}(0)=0$ and
\begin{equation}\label{casp2}
\Dd_{x,A}\cap B_d(0,\e_1)=\{u\in B_d(0,\e_1) : u_1>\psi_{x,A}(u_2,\dots,u_d)\},
\end{equation}
such that ${\rm Hess}\; \psi_{x,A}(v)\geq \eta I_{d-1}$ and 
$|D^2 \psi_{x,A}(v)|+|D^3\psi_{x,A}(v)|\leq C$ for all $v \in B_{d-1}(0,\e_1)$.

\vip

(ii) For all $x\in \pDd$, all $A\in\cI_x$, all  $u \in B_{d-1}(0,\e_1)$,
$$
\frac12 |u_1-\psi_{x,A}(u_2,\dots,u_d)|\leq d(u,\pDd_{x,A})\leq |u_1-\psi_{x,A}(u_2,\dots,u_d)|.
$$

(iii) For all $x,x' \in \pDd$, all $A\in \cI_x$, all $A' \in \cI_{x'}$, setting 
$\rho_{x,x',A,A'}=|x-x'|+||A-A'||$,
\begin{equation}\label{tbd2}
|\psi_{x,A}(v)-\psi_{x',A'}(v)|\leq C \rho_{x,x',A,A'} |v|^2
\qquad \hbox{for all $v\in B_{d-1}(0,\e_1)$.}
\end{equation}
\end{lemma}

\begin{proof}
We first prove (i) with $\e_1=\e_0$ defined in Remark~\ref{asp}.
Observe that~\eqref{casp} and~\eqref{casp2} are equivalent.
Hence by Remark~\ref{asp}, a function $\psi_{x,A}$ satisfying all the requirements of (i) exists for
a some $A \in \cI_x$.
For another $B\in\cI_x$ and $y\in \R^d$, $y \in \Dd_{x,B}$ if and only if $A^{-1}B y \in \Dd_{x,A}$, 
so that
$\psi_{x,B}$ defined on $B_{d-1}(0,\e_0)$ by
$\psi_{x,B}(u):=\psi_{x,A}(A^{-1}Bu)$ (here we identify $u \in \R^{d-1}$ to $\sum_{i=2}^d u_i\be_{i+1} \in \be_1^\perp$ 
and observe that $A^{-1}B$ is an isometry from $\be_1^\perp$ into itself) is suitable.
\vip
Let us check (ii), for $\e_1\in (0,\e_0]$ small enough.
For $u=(u_1,\dots,u_d) \in \R^d$, we use the shortened notation $u_0=(u_2,\dots,u_d)$.
For $u \in B_{d}(0,\e_0)$, it holds that $v:=(\psi_{x,A}(u_0),u_0) \in \pDd_{x,A}$, so that
$d(u,\pDd_{x,A})\leq |u-v|=|u_1-\psi_{x,A}(u_0)|$. We next consider $\vartheta>0$ such that 
$\sup_{x \in \pDd,  A \in \cI_x } \sup_{u \in B_{d-1}(0,\vartheta)}
|\nabla \psi_{x,A} (u)|\leq 1$ and set $\e_1=\frac{\vartheta\land \e_0}2$. We fix $x \in \pDd$, $A \in \cI_x$ 
and $u \in B_d(0,\e_1)$
and show that for all $v \in \pDd_{x,A}$, $|u-v| \geq \frac12 |u_1-\psi_{x,A}(u_0)|$.

\vip
\noindent $\bullet$ If $v \notin B_d(0,2\e_1)$, then  $|u-v|\geq |v|-|u|\geq \e_1 \geq 
\frac12(|u_1|+|\psi_{x,A}(u_0)|)\geq \frac12|u_1-\psi_{x,A}(u_0)|$. 
We used that $|u_1|\leq \e_1$  and $|\psi_{x,A}(u_0)|\leq \e_1$
(since $\psi_{x,A}(0)=0$ and $u_0 \in B_{d-1}(0,\e_1)$, on which $|\nabla \psi_{x,A}|$ is bounded by $1$).

\vip
\noindent $\bullet$ If $v \in B_d(0,2\e_1)$, 
we write, using that $v_1=\psi_{x,A}(v_0)$ (because $v\in B_d(0,\e_0)\cap \pDd_{x,A}$),
\begin{align*}
|u\!-\!v| \!\geq \!\frac12(|u_1-\psi_{x,A}(v_0)|+|u_0-v_0|)\geq  \frac12(|u_1-\psi_{x,A}(u_0)|
-|\psi_{x,A}(v_0)-\psi_{x,A}(u_0)| +|u_0-v_0|).
\end{align*}
But $|\psi_{x,A}(v_0)-\psi_{x,A}(u_0)|\leq |u_0-v_0|$ because $u_0,v_0 \in B_{d-1}(0, \vartheta)$.
Hence $|u-v| \geq \frac12|u_1-\psi_{x,A}(u_0)|$.

\vip

Let us finally verify (iii) for $\e_1\in (0,\e_0]$ small enough.
By (i), there is $C>0$ such that for all $x \in \pDd$, all $A\in \cI_x$, all $v\in B_{d-1}(0,\e_0)$,
\begin{equation}\label{hf}
|D^2\psi_{x,A}(v)|+|D^3\psi_{x,A}(v)|\leq C, \quad |D \psi_{x,A}(v)|\leq C |v| \quad \hbox{and}\quad 
|\psi_{x,A}(v)|\leq C |v|^2.
\end{equation}
It is thus enough to show that there
are $\kappa \in (0,\e_0/2)$, $\e_1 \in (0,\e_0/2)$ and $C>0$ such that $|\psi_{x,A}(v)-\psi_{x',A'}(v)|\leq 
C \rho_{x,x',A,A'}|v|^2$
when $|x-x'|\leq \kappa$ and $v\in B_{d-1}(0,\e_1)$: 
when $|x-x'|>\kappa$ (and $v\in B_{d-1}(0,\e_1)$), we simply write
$$
|\psi_{x,A}(v)-\psi_{x',A'}(v)|\leq 2C|v|^2\leq \frac{2C}{\kappa} \rho_{x,x',A,A'} |v|^2.
$$ 
Thus, let us fix $\kappa \in (0,\e_0/2)$ 
and $\e_1 \in (0,\e_0/2)$ to be chosen later.

\vip
Fix $x,x'\in \pDd$ such that $|x-x'|\leq \kappa$, as well as $A \in \cI_x$ and $A'\in \cI_{x'}$
and observe that $y=A^{-1}(x'-x)\in \pDd_{x,A}\cap B_{d}(0,\e_0)$ (because $h_x(A,y)=x'\in \pDd$ and
$|y|=|x'-x|\leq \kappa< \e_0$).
By~\eqref{casp2}, there  is $a\in B_{d-1}(0,\e_0)$ such that $y=(\psi_{x,A}(a),a_1,\dots,a_{d-1})$.
Let us observe that
$$
|x-x'|=|y|\geq |a|.
$$
The inward unit normal vector $\sigma$ to $\pDd_{x,A}$ at $y$ is
$$
\sigma=\frac{\be_1-\sum_{i=1}^{d-1} \partial_i\psi_{x,A}(a) \be_{i+1}}{r} \quad \hbox{where}\quad  r
=\sqrt{1+|\nabla \psi_{x,A}(a)|^2}
\quad \hbox{and we have}\quad \sigma= \Theta \be_1,
$$
with $\Theta:=A'A^{-1}$.
Moreover, the tangent space $T$ to $\pDd_{x,A}$ at $y$ is given by
$$
T=\Big\{u \in \R^d : u_1 = \psi_{x,A}(a)+\sum_{i=1}^{d-1} \partial_i\psi_{x,A}(a) (u_{i+1}-a_i)\Big\}.
$$
We now fix $v\in B_{d-1}(0,\e_1)$
and use the notation $\bv=\sum_{i=1}^{d-1}v_i \be_{i+1}\in B_d(0,\e_1)$. 
\begin{figure}[ht]\label{fig1}
\centerline{\fbox{\begin{minipage}{0.9\textwidth}
\caption{\small{Here $O=h_x^{-1}(A,x)$ and $y=h_x^{-1}(A,x')$.}}
\centerline{\resizebox{8cm}{!}{\includegraphics{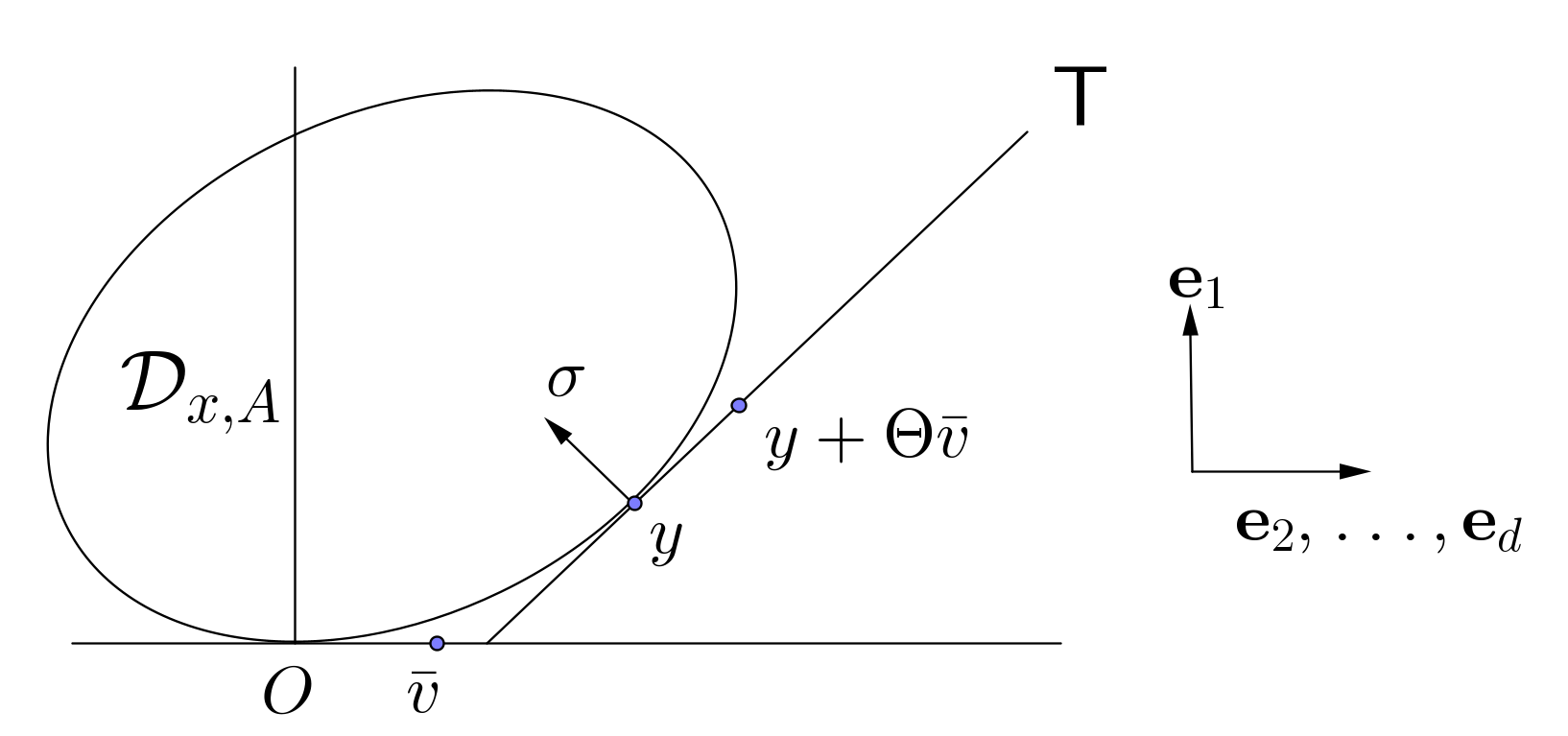}}}
\end{minipage}}}
\end{figure}
Note that $y+\Theta \bv \in T$, because 
$\Theta \bv \cdot \sigma = \Theta \bv \cdot \Theta \be_1 = \bv \cdot \be_1=0$.
Then, the value of $\gamma:=\psi_{x',A'}(v)\geq 0$ is determined by
the following fact: it is the smallest positive $\gamma$ such that $z:=y+\Theta \bv + \gamma \sigma \in \pDd_{x,A}$.
Observe that $\gamma= \psi_{x',A'}(v)\leq C |v|^2$ by~\eqref{hf}.
Note that we have
\[
 |z|\leq |y|+|\bv|+\gamma\leq |x-x'|+|v|+C|v|^2\leq \kappa+\e_1+C\e_1^2,
\]
and therefore, $z \in B_d(0,\e_0)$ if $\kappa$ and $\e_1$ are small enough. Thus, the condition 
$z\in \pDd_{x,A}$ can be rewritten as $z\cdot \be_1 = \psi_{x,A}(z\cdot \be_2,\dots,z\cdot \be_d)$.
In other words, since $y\cdot \be_1=\psi_{x,A}(a)$ and $z\cdot \be_i=a_{i-1}$ for $i=2,\dots,d$ and since
$\sigma\cdot \be_1=r^{-1}$ and $\sigma\cdot \be_i=-r^{-1}\partial_{i-1}\psi_{x,A}(a)$ for $i=2,\dots,d$,
\begin{equation}\label{tbr}
\psi_{x,A}(a)+\Theta \bv \cdot \be_1 + \frac{\gamma}{r} = \psi_{x,A}(b-\gamma c),
\end{equation}
where $b_i=a_i+\Theta \bv \cdot \be_{i+1}$ and $c_i=r^{-1}  \partial_{i}\psi_{x,A}(a)$ 
for $i=1,\dots,d-1$. Recalling the equation of $T$ and that $y+\Theta \bv \in T$, we get
$$
\Theta \bv \cdot \be_1= \sum_{i=1}^{d-1} \partial_i\psi_{x,A}(a) \Theta \bv \cdot \be_{i+1}
=(b-a)\cdot \nabla\psi_{x,A}(a).
$$
This formula, inserted in~\eqref{tbr}, gives us
\begin{equation}\label{star}
\frac{\gamma}{r} = \psi_{x,A}(b-\gamma c)-\psi_{x,A}(a) - (b-a)\cdot \nabla\psi_{x,A}(a).
\end{equation}
Since $\gamma=\psi_{x',A'}(v)$, we conclude that
$$
|\psi_{x',A'}(v)-\psi_{x,A}(v)|=|r[\psi_{x,A}(b-\gamma c)-\psi_{x,A}(a) - 
(b-a)\cdot \nabla\psi_{x,A}(a)]- \psi_{x,A}(v)|\leq I+J+K,
$$
where
\begin{align*}
I=&r|\psi_{x,A}(b-\gamma c)- \psi_{x,A}(b)|,\\
J=&(r-1)|\psi_{x,A}(b)-\psi_{x,A}(a) - (b-a)\cdot \nabla\psi_{x,A}(a)|,\\
K=& |\psi_{x,A}(b)-\psi_{x,A}(a) - (b-a)\cdot \nabla\psi_{x,A}(a)-\psi_{x,A}(v)  |.
\end{align*}

Observe that $1\leq r =\sqrt{1+|\nabla \psi_{x,A}(a) |^2}\leq 1+|\nabla \psi_{x,A}(a) |^2 \leq 1+C |a|^2$ 
by~\eqref{hf}, whence $1\leq  r \leq 1+ C|x-x'|^2\leq 1+C\kappa^2$, so that 
$r \in [1,2]$ if $\kappa$ is small enough. 
Since $\gamma\leq C|v|^2$ and $|c|\leq |\nabla\psi_{x,A}(a)|\leq C |a| \leq C|x-x'|$,
$$
I \leq 2||D\psi_{x,A}||_\infty \gamma |c| \leq C |x-x'| |v|^2.
$$
Next, since $r-1\leq  C|x-x'|^2 \leq C|x-x'|$ and $|b-a|=|\bar v|= |v|$,
$$
J \leq C |x-x'| ||D^2\psi_{x,A}||_\infty |b-a|^2 \leq C  |x-x'| |v|^2.
$$
To treat $K$, we use
the Taylor formula to write, recalling that $\psi_{x,A}(0)=0$ and $\nabla\psi_{x,A}(0)=0$,
\begin{gather*}
\psi_{x,A}(b)-\psi_{x,A}(a) - (b-a)\cdot \nabla\psi_{x,A}(a)=(b-a)\cdot M(b-a)\quad  
\hbox{and} \quad\psi_{x,A}(v)= v \cdot M' v,\\
\hbox{where} \quad M=\int_0^1 \mathrm{Hess}\,\psi_{x,A}(a+t(b-a)) (1-t)\dr t
\quad  \hbox{and} \quad
M'=\int_0^1 \mathrm{Hess}\,\psi_{x,A}(t v) (1-t)\dr t.
\end{gather*}
Consequently,
$$
K\leq ||M|| |b-a-v|(|b-a|+|v|) + ||M-M'|| |v|^2  \leq C |b-a-v| |v| + ||M-M'|| |v|^2,
$$
since $||M|| \leq ||D^2\psi_{x,A}||_\infty$ and $|b-a|= |v|$. Moreover, 
$$
|b-a-v|= |\Theta \bv - \bv|\leq ||\Theta-I|||\bv|= ||A-A'|| |v|,
$$
because $||\Theta-I||=||A'A^{-1}-I || = ||A'-A||$. Finally,
$$
||M-M'|| \leq ||D^3\psi_{x,A}||_\infty \sup_{t\in [0,1]} |a+t(b-a-v)| \leq C (|a|+|b-a-v|),
$$
whence $||M-M'||\leq C(|x-x'|+||A-A'||)$. All in all,
$K \leq C(|x-x'|+||A-A'||)|v|^2$.
\end{proof}

\subsection{Proof of the geometric inequalities}

Our goal is to prove Proposition~\ref{tyvmmm}. We recall that for $x\in \pDd$, $A\in \cI_x$ and $y,z \in \R^d$,
$h_x(A,y)=x+Ay$, $\delta(y)=d(y,\pDd)$ and
$\bg_x(A,y,z)=\Lambda(h_x(A,y),h_x(A,z))$. When $h_x(A,y)\in \Dd$ and $h_x(A,z)\notin \Dd$,
$\bg_x(A,y,z)=[h_x(A,y),h_x(A,z)]\cap \pDd$.
We start with a consequence
of the Thales theorem.

\begin{lemma}\label{thales}
Consider a $C^1$ open convex domain $\Sigma\subset \R^d$. For all $y \in \Sigma$, all 
$z \in \R^d\setminus \Sigma$, setting $a=[y,z]\cap \partial \Sigma$, we have
$$
|a-z|\leq \frac{|y-z| d(z,\partial \Sigma)}{d(y, \partial \Sigma)}.
$$
\end{lemma}

\begin{proof}
Introduce the tangent space $P_a$ to $\partial \Sigma $ at $a$, denote by $b$ (resp. $c$) 
the orthogonal projection
of $y$ (resp. $z$) on $P_a$. By the Thales theorem, we have
$$
\frac{|a-z|}{|y-a|}=\frac{|c-z|}{|y-b|},\quad \text{{\it i.e.}}\quad |a-z|=\frac{|y-a||c-z|}{|y-b|}.
$$
\begin{figure}[ht]\label{fig2}
\centerline{\fbox{\begin{minipage}{0.9\textwidth}
\caption{\small{Illustration of the proof of Lemma~\ref{thales}.}}
\centerline{\resizebox{6cm}{!}{\includegraphics{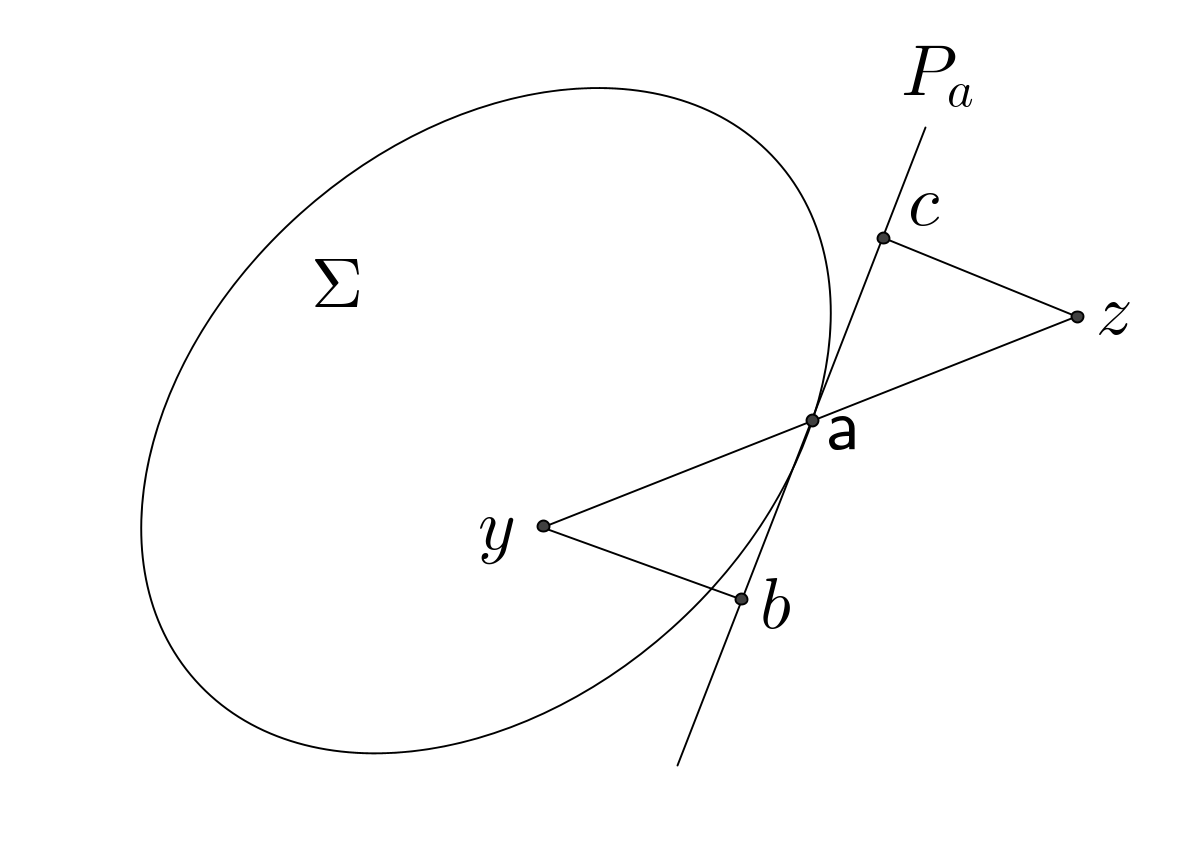}}}
\end{minipage}}}
\end{figure}
But $|y-a|\leq |y-z|$, $|c-z|\leq d(z,\partial \Sigma)$ and $|y-b|\geq d(y,\partial \Sigma)$.
\end{proof}

The next lemma is the main difficulty of the section.

\begin{lemma}\label{letop}
Grant Assumption~\ref{as}. There is a constant $C\in (0,\infty)$ such that
for all $x,x' \in \pDd$, all $A \in \cI_x$, $A' \in \cI_{x'}$, setting $\rho_{x,x',A,A'}=|x-x'|+||A-A'||$,
\vip
(a) for all $y,z \in \R^d$ such that both $h_x(A,y)$ and $h_{x'}(A',y)$ belong to $\Dd$ while $h_x(A,z)\notin \Dd$
and  $h_{x'}(A',z)\notin \Dd$, if $|\bg_x(A,y,z)-h_x(A,y)|\leq |\bg_{x'}(A',y,z)-h_{x'}(A',y)|$, then 
$$
\Big||\bg_x(A,y,z)-\bg_{x'}(A',y,z)|-|x-x'|\Big| \leq C \rho_{x,x',A,A'}\Big(|y|\land 1+
|z|\land 1 + \frac{|y|(|y-z|\land 1)}{\delta(h_x(A,y))}\Big),
$$

(b) for all $y,z \in \R^d$ such that $h_x(A,y)$, $h_{x'}(A',y)$ and $h_{x'}(A',z)$ belong to $\Dd$ 
but $h_x(A,z)\notin \Dd$,
$$
\Big||\bg_x(A,y,z)-h_{x'}(A',z)|-|x-x'|\Big| \leq C \rho_{x,x',A,A'}\Big(|y|\land 1+
|z|\land 1 + \frac{|y|(|y-z|\land 1)}{\delta(h_x(A,y))}\Big).
$$
\end{lemma}

\begin{proof}
Let  $D$ be the diameter of $\Dd$ and let $K=D\lor 1$.
We fix $x,x'\in \pDd$, $A \in \cI_x$ and $A'\in \cI_{x'}$.
We introduce $\Dd_{x,A}=\{u \in \R^d : h_x(A,u)\in \Dd\}$ and  $\Dd_{x',A'}=\{u \in \R^d : h_{x'}(A',u)\in \Dd\}$
and we recall that for all $u \in \R^d$, $\delta(h_x(A,u))=d(u,\pDd_{x,A})$.

\vip

{\it Step 1.} Here we show that it is sufficient to prove that for all $y \in \Dd_{x,A}\cap \Dd_{x',A'}$, 
all $v\in \cDd_{x',A'}\setminus \Dd_{x,A}$, for some constant $M$ depending only on $\Dd$, setting
$a(x,y,v)=[y,v]\cap \pDd_{x,A}$,
\begin{equation}\label{goal}
|a(x,y,v)-v| \leq M \rho_{x,x',A,A'} \frac{|y||y-v|}{d(y,\pDd_{x,A})}.
\end{equation}
\begin{figure}[ht]\label{fig3}
\centerline{\fbox{\begin{minipage}{0.9\textwidth}
\caption{\small{Illustration of~\eqref{goal}.}}
\centerline{\resizebox{9cm}{!}{\includegraphics{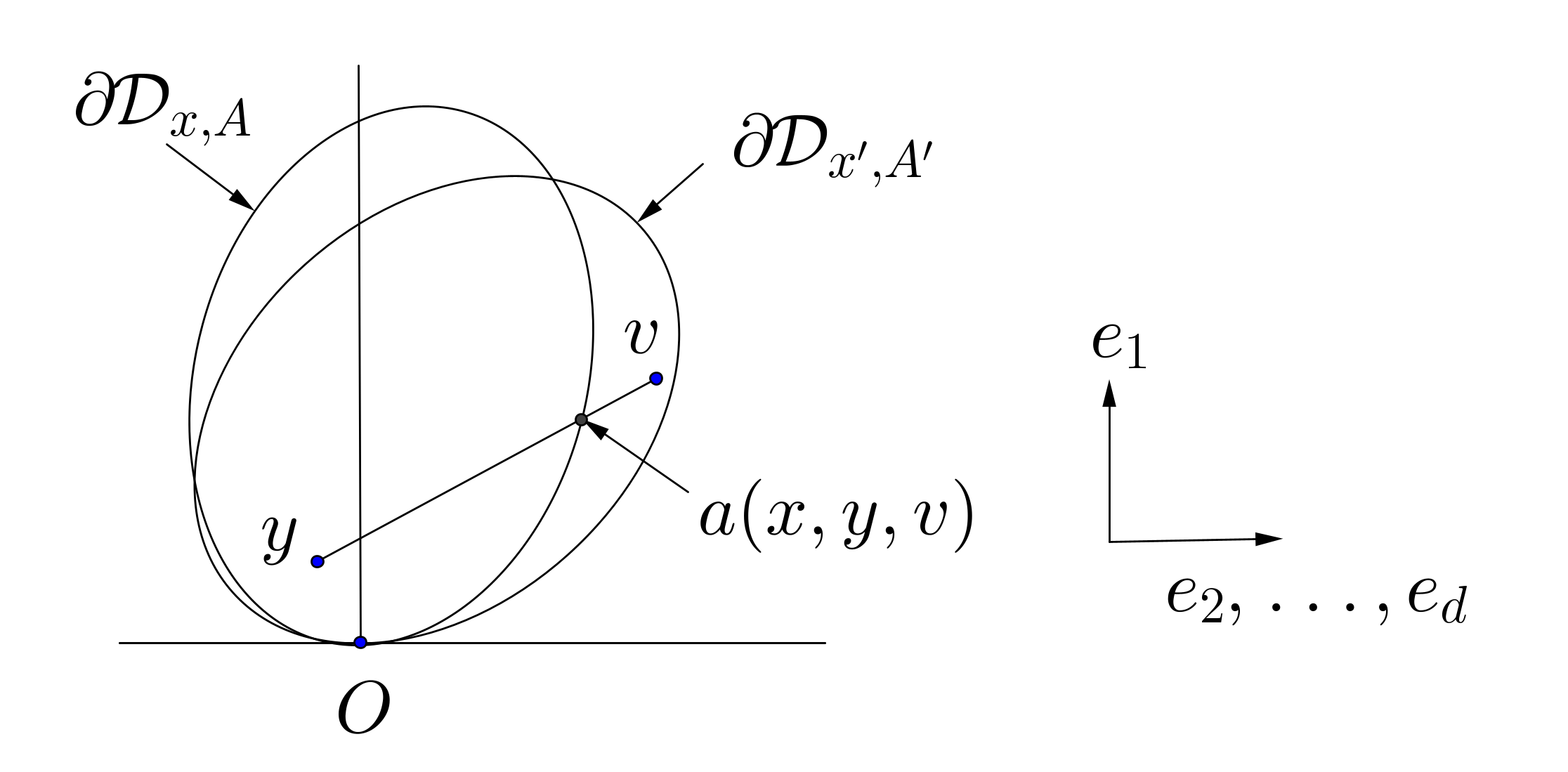}}}
\end{minipage}}}
\end{figure}

Let us prove that this implies (a). For $y,z$ as in the statement, we have $y \in \Dd_{x,A}\cap\Dd_{x',A'}$ 
and $z \in \Dd_{x,A}^c\cap  \Dd_{x',A'}^c$. We have
$\bg_x(A,y,z)=h_x(A,a(x,y,z))=x+Aa(x,y,z)$ and $\bg_{x'}(A',y,z)=x'+A'a(x',y,z)$, so that
\begin{align*}
\Big||\bg_x(A,y,z)-\bg_{x'}(A',y,z)|-|x-x'|\Big| \leq& |Aa(x,y,z)-A'a(x',y,z)|
\leq \Delta_1+\Delta_2,
\end{align*} 
where $\Delta_1=||A'-A|| |a(x,y,z)|$ and $\Delta_2=|A'a(x,y,z)-A'a(x',y,z)|$.
Since $|a(x,y,z)|\leq (|y|+|z|)\land D$ 
(because $a(x,y,z)\in [y,z]$, because  $0\in \Dd_{x,A}$ and $a(x,y,z)\in \pDd_{x,A}$), we find
$$
\Delta_1 \leq ||A-A'|| [(|y|+|z|)\land D]\leq  K \rho_{x,x',A,A'} (|y|\land 1 + |z|\land 1).
$$
Next, notice that the condition $|\bg_x(A,y,z)-h_x(A,y)|\leq |\bg_{x'}(A',y)-h_{x'}(A',y)|$ rewrites
\begin{equation}\label{hyp}
|a(x,y,z)-y| \leq |a(x',y,z)-y|,
\end{equation}
since $\bg_x(A,y,z)=h_x(A,a(x,y,z))$ and since $h_x(A,\cdot)$ is an isometry. 
Introduce now $v=a(x',y,z) \in \pDd_{x',A'}$, which satisfies $a(x',y,v)=v=a(x',y,z)$. 
Since moreover $v \notin \Dd_{x,A}$ by~\eqref{hyp}, \eqref{goal} tells us that
$$
\Delta_2=|a(x,y,z)-a(x',y,z)|=|a(x,y,z)-v|\leq M \rho_{x,x',A,A'} \frac{|y||y-v|}{\delta(h_x(A,y))}.
$$
The conclusion follows, since $|y-v|\leq|z-y|\land D$, because
$v\in [y,z]$ and $v,y \in \cDd_{x',A'}$.
\vip

Let us next show that~\eqref{goal} implies (b). For $y,z$ as in the statement, we have 
$y \in \Dd_{x,A}\cap\Dd_{x',A'}$ 
and $z \in \Dd_{x',A'}\setminus \Dd_{x,A}$. We have
$\bg_x(A,y,z)=x+Aa(x,y,z)$ and $h_{x'}(A',z)=x'+A'z$, so that
\begin{align*}
\Big||\bg_x(A,y,z)-h_{x'}(A',z)|-|x-x'|\Big| \leq& |Aa(x,y,z)-A'z| \leq \Delta_1+\Delta_3,
\end{align*} 
where  $\Delta_1=||A'-A|| |a(x,y,z)|$ and $\Delta_3=|A'a(x,y,z)-A'z|$. Of course,
$\Delta_1$ is controlled as previously, while~\eqref{goal} gives
(we have $|y-z|=|y-z|\land D$ because $y,z\in \cDd_{x',A'}$)
$$
\Delta_3=|a(x,y,z)-z|\leq M \rho_{x,x',A,A'} \frac{|y||y-z|}{\delta(h_x(A,y))}
\leq M \rho_{x,x',A,A'} \frac{|y|(|y-z|\land D)}{\delta(h_x(A,y))}.
$$

{\it Step 2.} There is a constant $C$ such that for any $\e\in (0,1)$, any $y \in \Dd_{x,A} \cap \Dd_{x',A'}$ 
such that $|y|\geq \e$ and any 
$v\in \cDd_{x',A'}\setminus \Dd_{x,A}$, \eqref{goal} holds true with the constant $M=\frac C \e$.

\vip

Indeed, by Lemma~\ref{thales} with the convex set $\Sigma=\Dd_{x,A}$, with $y \in \Dd_{x,A}$ and with 
$v \in D_{x, A}^c$,
$$
|a(x,y,v)-v|\leq \frac{|y-v| d(v,\pDd_{x,A})}{d(y,\pDd_{x,A})} 
\leq \frac 1 \e \frac{|y||y-v| d(v,\pDd_{x,A})}{d(y,\pDd_{x,A})}
\leq \frac C \e \rho_{x,x',A,A'} \frac{|y||y-v|}{d(y,\pDd_{x,A})},
$$
where we finally used Lemma~\ref{paradurf}-(ii): since $v\in \cDd_{x',A'}\setminus \Dd_{x,A}$,
we have $d(v,\pDd_{x,A})\leq C \rho_{x,x',A,A'}$.

\vip

{\it Step 3.} We next show that there exist $\e_2>0$ and 
$h:(0,\e_2]\to (0,\infty)$ with $\lim_{\e\to 0} h(\e)=0$
such that for all $\e\in(0,\e_2]$,
there is a constant $C_\e$ such that for all $y \in \Dd_{x,A} \cap \Dd_{x',A'}$ 
such that $|y|\leq \e$ and all $v\in \cDd_{x',A'}\setminus \Dd_{x,A}$ such that $|v|\geq h(\e)$, 
\eqref{goal} holds true with the constant $M=C_\e$.

\vip

By Assumption~\ref{as}, there exist $a>\gamma>0$ such that 
$$
U_a:=B_d(\be_1/a,1/a) \subset \Dd_{x,A}\cap\Dd_{x',A'}\subset \Dd_{x,A}\cup\Dd_{x',A'}
\subset B_d(\be_1/\gamma,1/\gamma)=:U_\gamma,
$$
and we e.g. assume that $a=1$. We fix $y \in \Dd_{x,A} \cap \Dd_{x',A'}$ 
such that $|y|\leq \e$ and $v\in \cDd_{x',A'}\setminus \Dd_{x,A}$ such that $|v|\geq \e'$, 
where $\e'=h(\e) =\frac{8\e^{1/2}}{\gamma^{1/2}}$.
For all $\e>0$ small enough, we have
\begin{gather*}
\frac{\gamma (\e')^2}2 - \e \geq \frac{\gamma (\e')^2}4, \quad
\frac{\gamma^4 (\e')^4}{256}-2\e^2\geq \frac{\gamma^4 (\e')^4}{300} \quad \text{and}\quad
\e'-\e\geq  \frac{\gamma^2 (\e')^2}{16}.
\end{gather*}
Actually the two first inequalities hold true for all $\e\in (0,1]$, and
the second one uses that $8^4/256-2>8^4/300$ and that $\gamma>1$.
We choose $\e_2>0$ such that $\frac{\gamma^2 (\e')^2}{16}\leq 1$ and the above inequalities hold true
for all $\e \in (0,\e_2]$.

\vip

{\it Step 3.1.} We have $v_1\geq \gamma (\e')^2/2$, because $v \in \cDd_{x',A'}\subset \closure{U}_\gamma$, whence 
$|v-\be_1/\gamma|^2\leq 1/\gamma^2$, {\it i.e.} $2v_1/\gamma \geq |v|^2\geq (\e')^2$. 
Moreover, $|v-y|\leq$diam$(\Dd_{x',A'}) \leq 2/\gamma$, so that
$$
\frac{v_1-y_1}{|v-y|} \geq \frac{\gamma (\e')^2/2-\e}{2/\gamma}\geq \frac{\gamma^2(\e')^2}{8}.
$$

{\it Step 3.2.} Let $p=y + \frac{\gamma^2 (\e')^2}{16}\frac{v-y}{|v-y|}$. We show here that
$p \in [y,v]\cap U_1$ and
$d(p, \partial U_1) \geq \frac{\gamma^4(\e')^4}{600}$.
\vip
We write $\frac{v-y}{|v-y|}=\sin\theta \be_1+ \cos \theta \be_0$, with $\be_0 \in \Sp_{d-1}$ orthogonal to
$\be_1$ and $\sin \theta \geq \frac{\gamma^2(\e')^2}{8}$ by Step~3.1.
We also write $y=y_1 \be_1+y_0 \be_0'$,
for some $\be_0' \in \Sp_{d-1}$ orthogonal to $\be_1$, so that
\begin{align*}
|p-\be_1|^2=&\Big(y_1+  \frac{\gamma^2 (\e')^2}{16} \sin\theta-1\Big)^2
+\Big|y_0 \be_0'+ \frac{\gamma^2 (\e')^2}{16} \cos \theta \be_0\Big|^2\\
\leq & y_1^2 +2y_1 \Big(\frac{\gamma^2 (\e')^2}{16} \sin\theta-1\Big) +
\Big(\frac{\gamma^2 (\e')^2}{16} \sin\theta-1\Big)^2
+ 2 y_0^2+2\Big(\frac{\gamma^2 (\e')^2}{16} \cos \theta
\Big)^2.
\end{align*}
The second term is nonpositive and we can bound $y_1^2+2y_0^2\leq 2|y|^2\leq 2\e^2$, whence
\begin{align*}
|p-\be_1|^2\leq  2\e^2+  1 -2 \frac{\gamma^2 (\e')^2}{16}\sin \theta + 
\frac{\gamma^4 (\e')^4}{256}+ \frac{\gamma^4(\e')^4}{128}.
\end{align*}
Recalling that $\sin \theta \geq \frac{\gamma^2(\e')^2}{8}$, we end with 
$$
|p-\be_1|^2\leq 1+2\e^2
-\frac{\gamma^4 (\e')^4}{256}\leq 1 - \frac{\gamma^4 (\e')^4}{300}.
$$
This implies that $p \in U_1$ and 
$d(p,\partial U_1)=1-|p-\be_1|\geq \frac12(1-|p-\be_1|^2)\geq \frac{\gamma^4 (\e')^4}{600}$. 

\vip

Finally, we have $p \in [y,v]$ because $\frac{\gamma^2 (\e')^2}{16|v-y|}\in [0,1]$, recall that
we are in the situation where $|v|\geq \e'$ and $|y|\leq \e$, whence
$|v-y|\geq \e'-\e\geq  \frac{\gamma^2 (\e')^2}{16}.$

\vip

{\it Step 3.3.} By Step 3.2 and since $U_1\subset \Dd_{x,A}$,
we have $a(x,y,v)=a(x,p,v)$, so that by Lemma~\ref{thales},
$$
|a(x,y,v)-v|\leq \frac{|p-v|d(v,\pDd_{x,A})}{d(p,\pDd_{x,A})}\leq \frac{|y-v|d(v,\pDd_{x,A})}{d(p,\partial U_1)}.
$$
By Lemma~\ref{paradurf}-(ii), we have $d(v,\pDd_{x,A})\leq C \rho_{x,x',A,A'}$, 
because $v\in \cDd_{x',A'}\setminus \Dd_{x,A}$.
By Step~3.2,
$$
d(p,\partial U_1) \geq \frac{\gamma^4(\e')^4}{600}
\geq  \frac{\gamma^4(\e')^4}{600} \frac{d(y,\pDd_{x,A})}{|y|}, 
$$
because $d(y,\pDd_{x,A})\leq |y|$ since $0 \in \pDd_{x,A}$.
All this shows that, as desired,
$$
|a(x,y,v)-v|\leq \frac{600C}{\gamma^4(\e')^4} \rho_{x,x',A,A'}\frac{|y||y-v|}{d(y,\pDd_{x,A})}.
$$

{\it Step 4.} We now fix $\e_4 \in (0,1)$ such that $\e_4\leq \e_3$ and $h(\e_4)\leq  \e_3$, with 
\begin{equation}\label{e4}
\e_3=\e_1 \land \e_2 \land 1 \land \frac{\eta}
{24 C \max_{x,x' \in \pDd,  A \in \cI_x, A'\in \cI_{x'}} \rho_{x,x',A,A'}}
\end{equation}
where $\e_1$,  $\eta$ and $C$ were defined in Lemma~\ref{paradur} ($C$ is the constant 
appearing in~\eqref{tbd2}) and $\e_2$ was defined in Step 3.
By Steps 2 and 3, we know that~\eqref{goal} holds true for any $y \in \Dd_{x,A} \cap \Dd_{x',A'}$ and 
$v\in \cDd_{x',A'}\setminus \Dd_{x,A}$ such that $|y|\geq \e_4$ or ($|y| \leq \e_4$ and $|v|\geq h(\e_4)$),
and it only remains to study the case where $|y|<\e_4$ and $|v|<h(\e_4)$, 
which implies that  $|y|<\e_3$ and $|v|<\e_3$. Since $\e_3\leq \e_1$, 
we can use Lemma~\ref{paradur}. We introduce  $y_0=(y_2,\dots,y_d)$ and $v_0=(v_2,\dots,v_d)$.

\vip

{\it Step 4.1.} If first $|v_0|^2\leq 2|y|$, then  we use Lemma~\ref{thales} as in Step 2 to write
$$
|a(x,y,v)-v|\leq \frac{|y-v| d(v,\pDd_{x,A})}{d(y,\pDd_{x,A})}.
$$
Since  $v\in \cDd_{x',A'}\setminus \Dd_{x,A}$,
we have $v_1 \in [\psi_{x',A'}(v_0),\psi_{x,A}(v_0)]$ by~\eqref{casp2}, so that by 
Lemma~\ref{paradur},
$$
d(v,\pDd_{x,A})\leq |v_1-\psi_{x,A}(v_0)|\leq |\psi_{x',A'}(v_0)-\psi_{x,A}(v_0)|
\leq C \rho_{x,x',A,A'} |v_0|^2\leq 2C \rho_{x,x',A,A'}|y|,
$$
whence 
$$
|a(x,y,v)-v|\leq 2C  \rho_{x,x',A,A'}\frac{|y||y-v|}{d(y,\pDd_{x,A})}.
$$

{\it Step 4.2.} If next $|v_0|^2 > 2|y|$, we will show that 
\begin{equation}\label{todo}
a(x,y,v)=t y + (1-t) v\quad \text{for some } t \in \Big(0,\frac{12C\rho_{x,x',A,A'}\e_3}{\eta}\Big].
\end{equation}
Using moreover that $|y|=|y-0|\geq d(y,\pDd_{x,A})$ because $0\in \pDd_{x,A}$, this will imply~\eqref{goal}, since
$$
|a(x,y,v)-v|=t|y-v|\leq \frac{12C\rho_{x,x',A,A'}\e_3}{\eta}|y-v|
\leq \frac{12C\e_3}{\eta} \rho_{x,x',A,A'} \frac{|y-v||y|}
{d(y,\pDd_{x,A})}.
$$

To prove~\eqref{todo}, it is enough to verify that, setting $s=\frac{12C\rho_{x,x',A,A'}\e_3}{\eta}\in (0,1/2)$
(by~\eqref{e4}),
$$
s y_1+(1-s)v_1- \psi_{x,A}(sy_0+(1-s)v_0)\geq 0.
$$
Indeed, $a(x,y,v)=t y + (1-t) v$, with $t \in (0,1)$ determined by the fact that $a(x,y,v)\in \pDd_{x,A}$,
{\it i.e.} $t y_1 + (1-t) v_1 - \psi_{x,A}(ty_0+(1-t)v_0)=0$.

\vip
But, as in Step 4.1, $v_1 \in [\psi_{x',A'}(v_0),\psi_{x,A}(v_0)]$. Moreover,
$y_1 \geq \psi_{x',A'}(y_0)$ because $y \in \Dd_{x',A'}$.
It thus suffices to check that
\begin{equation}\label{retodo}
r:=s \psi_{x',A'}(y_0)+(1-s)\psi_{x',A'}(v_0)- \psi_{x,A}(sy_0+(1-s)v_0)\geq 0.
\end{equation}

By~\eqref{tbd2}, we have 
\begin{align*}
r\geq& s \psi_{x',A'}(y_0)+(1-s)\psi_{x',A'}(v_0)- \psi_{x',A'}(sy_0+(1-s)v_0)
-C \rho_{x,x',A,A'} |sy_0+(1-s)v_0|^2.
\end{align*}
Since $|y_0|<\e_3\leq 1$ and $|v_0|<\e_3$ and since $2|y_0|<|v_0|^2$,
we have 
$$
|sy_0+(1-s)v_0|^2\leq \e_3 (|y_0|+|v_0|)\leq  \e_3 (|v_0|^2/2+|v_0|)\leq \frac 32 \e_3 |v_0|.
$$
Since Hess $\psi_{x',A'}\geq \eta I_{d-1}$, we classically have
$$
s \psi_{x',A'}(y_0)+(1-s)\psi_{x',A'}(v_0)- \psi_{x',A'}(sy_0+(1-s)v_0) \geq \frac \eta 2 s(1-s)|v_0-y_0|.
$$ 
Moreover, $|v_0-y_0|\geq |v_0|-|y_0| \geq |v_0| - \frac{|v_0|^2}2\geq \frac{|v_0|}{2}$.
All in all,
$$
r \geq \frac\eta 4  s(1-s) |v_0| - \frac 32  C \rho_{x,x',A,A'} \e_3 |v_0| \geq
\frac\eta 8  s |v_0| - \frac 32  C \rho_{x,x',A,A'} \e_3 |v_0|,
$$
because $s \in [0,1/2]$ as already seen. This last quantity equals $0$ by definition of $s$.
\end{proof}

We can finally handle the 

\begin{proof}[Proof of Proposition~\ref{tyvmmm}]
We fix $x,x' \in \pDd$ and $A \in \cI_x$, $A' \in \cI_{x'}$.

\vip
For (i), assume that $h_x(A,y)\in\Dd$, $h_{x'}(A',y) \in \Dd$, $h_x(A,z)\notin \Dd$ and 
$h_{x'}(A',z) \notin \Dd$. Since we either have 
$|\bg_x(A,y,z)-h_x(A,y)|\leq |\bg_{x'}(A',y,z)-h_{x'}(A',y)|$ or the converse inequality, 
we deduce from Lemma~\ref{letop}-(a) that
\begin{align*}
\Big||\bg_x(A,y,z)\!-\!\bg_{x'}(A',y,z)|\!-\!|x-x'|\Big| \leq& C \rho_{x,x',A,A'}\Big(|y|\land 1+
|z|\land 1 + \frac{|y|(|y-z|\land 1)}{\delta(h_x(A,y))\land \delta(h_{x'}(A',y))}\Big)\\
\leq & 2C \rho_{x,x',A,A'}\Big(|z|\land 1 + \frac{|y|(|y-z|\land 1)}{\delta(h_x(A,y))\land \delta(h_{x'}(A',y))}\Big).
\end{align*}
For the second inequality, note that 
$|y|\geq \delta(h_x(A,y))\land \delta(h_{x'}(A',y))$ (since $|y|={|h_x(A,y)-x|}\geq \delta(h_x(A,y))$
and $|y|=|h_{x'}(A',y)-x'|\geq \delta(h_{x'}(A',y))$ because $x,x'\in \pDd$) and write
$$
|y|\land 1 \leq |z|\land 1 + |y-z|\land 1\leq |z|\land 1+ 
\frac{|y|(|y-z|\land 1)}{\delta(h_x(A,y))\land \delta(h_{x'}(A',y))}.
$$ 

For (ii), assume that $h_x(A,y)\in\Dd$, $h_{x'}(A',y) \in \Dd$, $h_x(A,z)\notin \Dd$ and 
$h_{x'}(A',z) \in \Dd$. By Lemma~\ref{letop}-(b),
\begin{align*}
\Big||\bg_x(A,y,z)\!-\!h_{x'}(A',z)|\!-\!|x-x'|\Big| \leq& C \rho_{x,x',A,A'}\Big(|y|\land 1+
|z|\land 1 + \frac{|y|(|y-z|\land 1)}{\delta(h_x(A,y))}\Big)\\
\leq & 2C \rho_{x,x',A,A'}\Big(|z|\land 1 + \frac{|y|(|y-z|\land 1)}{\delta(h_x(A,y))}\Big).
\end{align*}
We used that $|y|\geq \delta(h_x(A,y))$, whence $|y|\land 1 \leq |z|\land 1 + |y-z|\land 1\leq |z|\land 1+ 
\frac{|y|(|y-z|\land 1)}{\delta(h_x(A,y))}$.

\vip

For (iii), we assume that $h_x(A,z)\notin \Dd$ and $h_{x'}(A',z) \notin \Dd$. Consider $\e>0$
small enough so that $h_x(A,\e \be_1) \in \Dd$ and $h_{x'}(A',\e\be_1) \in \Dd$. Then we can apply (i)
to find
$$
\Big||\bg_x(A,\e\be_1,z)-\bg_{x'}(A',\e \be_1,z)|-|x-x'|\Big| \leq C \rho_{x,x',A,A'}
\Big(|z|\land 1 + \frac{|\e\be_1|(|\e\be_1-z|\land 1)}
{\delta(h_x(A,\e \be_1)) \land \delta(h_{x'}(A',\e \be_1))}\Big).
$$ 
But if $\e>0$ is small enough, we have $\delta(h_x(A,\e \be_1))=d(x+\e \bn_x,\pDd)=\e$, recall that
$A \be_1= \bn_x$, and $\delta(h_{x'}(A',\e \be_1))=\e$. We end with
$$
\Big||\bg_x(A,\e\be_1,z)-\bg_{x'}(A',\e \be_1,z)|-|x-x'|\Big| \leq C \rho_{x,x',A,A'}
\Big(|z|\land 1 + |\e\be_1-z|\land 1) \leq  2C\rho_{x,x',A,A'}.
$$ 
Letting $\e\to 0$, we find 
$||\bg_x(A,0,z)-\bg_{x'}(A',0,z)|-|x-x'|| \leq 2C \rho_{x,x',A,A'}$ as desired.

\vip
Point (iv) is obvious, since $\bg_x(A,0,z)$, $\bg_{x'}(A',0,z)$, $x$, $x'$ belong to 
$\cDd$, which is bounded.
\end{proof}

We conclude this appendix with the following strange observation: we could treat the
flat case $\Dd=\HH$ (with some small difficulties since $\HH$ is not bounded, but with many
huge simplifications), but we could not treat, with our method, 
general convex domains that are not strongly convex.

\begin{remark}\label{strange}
(a) The following version of Lemma~\ref{letop}-(a) holds true in the flat case where $\Dd=\HH$:
for all $x,x' \in \partial \HH$, all $A \in \cI_x$, $A' \in \cI_{x'}$, 
all $y,z \in \R^d$ such that $y \in \HH$ and $z \notin \closure{\HH}$, 
(so that $h_x(A,y) \in \HH$, $h_{x'}(A',y)\in \HH$, $h_x(A,z) \notin \HH$ and  $h_{x'}(A',z) \notin \HH$),
$$
\Big||\bg_x(A,y,z)-\bg_{x'}(A',y,z)|-|x-x'|\Big| \leq ||A-A'||(|y|+|z|).
$$

(b) Consider a (non-strongly) convex open bounded domain $\Dd$ of $\R^2$ such that 
$$
\Dd \cap B_2(0,1)=\{u=(u_1,u_2) \in B_2(0,1) : u_1>u_2^4\}.
$$
There does not exist any constant $M \in (0,\infty)$ such that for all $x,x' \in \pDd$, all $A \in \cI_x$,
all $A' \in \cI_{x'}$, all $y,z \in \R^2$ such that  
$h_x(A,y)\in\Dd$, $h_{x'}(A',z') \in \Dd$, $h_x(A,z)\notin \Dd$ and $h_{x'}(A',z) \notin \Dd$,
$$
|\bg_x(A,y,z)-\bg_{x'}(A',y,z)|-|x-x'| \!\leq \! M(|x-x'|+||A-A'||)\Big(|y|+|z|
+ \frac{|y||y-z|}{\delta(h_x(A,y))\land\delta(h_{x'}(A',y)) }\Big).
$$
\end{remark}

\begin{proof}
For point (a), it suffices to note that $\bg_x(A,y,z)=x+A \bg_0(I,y,z)$, where $I$ is the identity matrix, 
and the result follows from
the fact that $|\bg_0(I,y,z)|\leq |y|+|z|$, because $\bg_0(I,y,z)=[y,z]\cap\partial \HH \in [y,z]$.
\vip

For point (b), we choose $x=(0,0)$, $x'=(a^4,a)$, $y=(\alpha,0)$ and $z=(\alpha,-\alpha^{1/4})$,
with $a>0$ small and $\alpha=\frac{a^4(6a^2-4a^3+a^4)}{\sqrt{1+16a^2}}>0$. We have $x,x',z \in \pDd$.
We naturally choose $A=I \in \cI_x$ (because $\bn_x=\be_1$) and $A'\in \cI_{x'}$ defined by
$A'(u_1,u_2)=u_1\sigma+u_2\tau$, where 
$$
\tau=\frac{(4a^3,1)}{\sqrt{1+16 a^6}} \quad \text{and} \quad \sigma=\frac{(1,-4a^3)}{\sqrt{1+16 a^6}}
$$
are the tangent and normal unit vectors to $\pDd$ at $x'$ (we have $\sigma=\bn_{x'}$). 
For all $u=(u_1,u_2) \in \R^2$, we have
$h_x(u)=u$ and $h_{x'}(u)=x'+u_1\sigma+u_2\tau$.
\begin{figure}[ht]\label{fig4}
\centerline{\fbox{\begin{minipage}{0.9\textwidth}
\caption{\small{Illustration of Remark~\ref{strange}-(b).}}
\centerline{\resizebox{13cm}{!}{\includegraphics{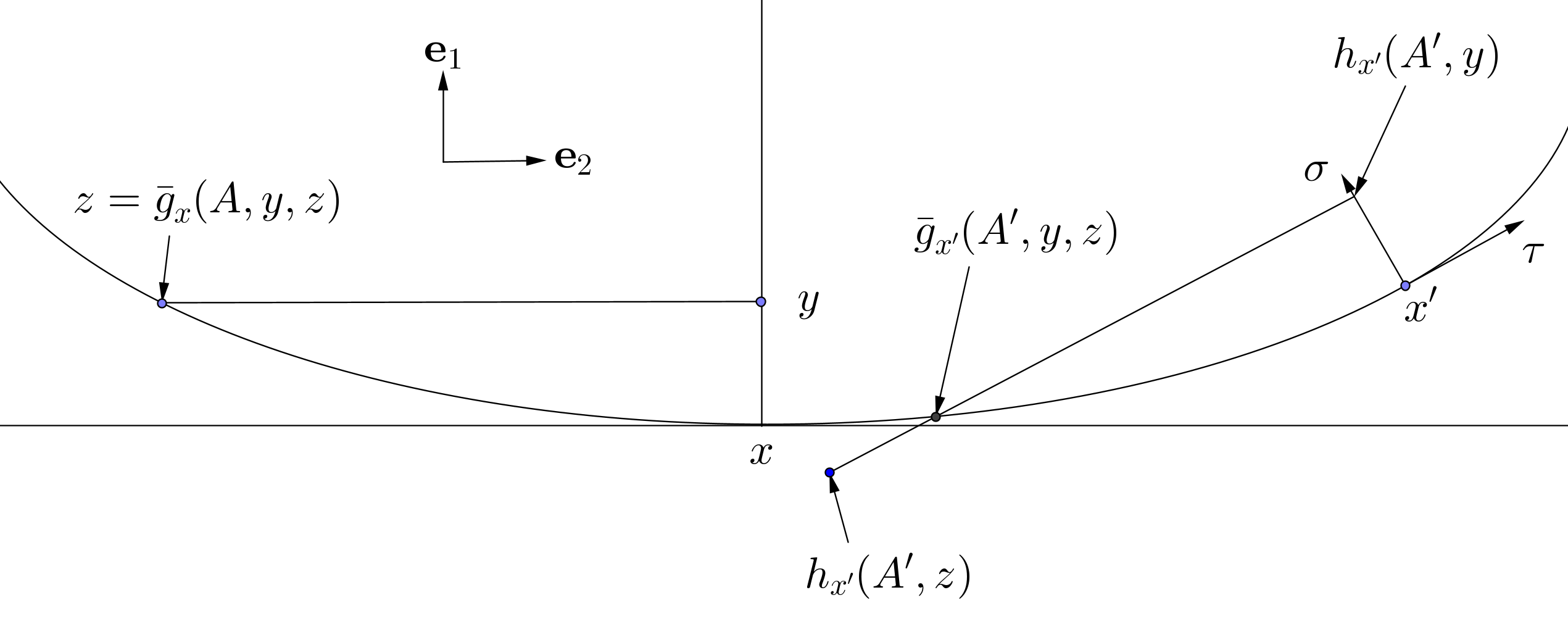}}}
\end{minipage}}}
\end{figure}
Moreover, since $a>0$ is small,  $y=h_x(A,y) \in \Dd$ and $h_{x'}(A,y) =x'+\alpha \sigma  \in \Dd$.
Observe also that $h_{x'}(A,z)=x'+\alpha \sigma - \alpha^{1/4}\tau.$ 
\vip

We have
$\bg_x(A,y,z)=z$, and it holds that $\bg_{x'}(A',y,z)=((a-a^2)^4,a-a^2)$:
it suffices that there exists $\theta \in (0,1)$ such that 
$((a-a^2)^4,a-a^2)=h_{x'}(A',y)+\theta(h_{x'}(A',z)-h_{x'}(A',y))$,
{\it i.e.}  $((a-a^2)^4,a-a^2)=x'+\alpha\sigma-\theta\alpha^{1/4}\tau$, which rewrites
$$
(a-a^2)^4=a^4+\frac\alpha{\sqrt{1+16 a^6}}-\frac{4\theta \alpha^{1/4}a^3}{\sqrt{1+16 a^6}}
\quad \text{and} \quad a-a^2=a-\frac{4 \alpha a^3}{\sqrt{1+16 a^6}}-\frac{\theta \alpha^{1/4}}{\sqrt{1+16 a^6}}.
$$
A  computation shows that
(the second equality uses the definition of $\alpha$)
$$
\theta=\frac{a^2\sqrt{1+16a^6} - 4 \alpha a^3}{\alpha^{1/4}}=
\frac{\alpha + \sqrt{1+16 a^6}(a^4-(a-a^2)^4)}{4\alpha^{1/4}a^3}
$$
is suitable, and indeed $\theta \in (0,1)$ if $a>0$ is small because $\alpha\sim 6a^6$ as $a\to0$.

\vip

One then observes that, as $a\to 0$,
$$
|\bg_x(A,y,z)-\bg_{x'}(A',y,z)|\geq |a-a^2+\alpha^{1/4}|=a+6^{1/4}a^{3/2}+o(a^{3/2}),
$$
while $|x-x'|=\sqrt{a^8+a^2}=a + o(a^6)$, so that
$$
|\bg_x(A,y,z)-\bg_{x'}(A',y,z)|-|x-x'| \geq 6^{1/4}a^{3/2}+o(a^{3/2}).
$$

Next, for all $a>0$ small, we have 
$$
||A-A'||\leq |\sigma - \be_1|+|\tau-\be_2|\leq  2\Big(1-\frac1{\sqrt{1+16a^2}}
+\frac{4a^3}{\sqrt{1+16a^2}}\Big) 
\leq 2 (\sqrt{1+16a^6}-1+4a^3),
$$
so that $||A-A'||\leq 8a^3+o(a^3)$. Moreover,
$|y|=\alpha\leq 6a^6$, $|z|\leq 2\alpha^{1/4} \leq 4 a^{3/2}$ and $|y-z|\leq |y|+|z|\leq 10 a^{3/2}$. 
Also, we have $\delta(h_x(A,y))=\delta(h_{x'}( A', y))=|y|$.
Thus, for $a>0$ small,
\begin{align*}
(|x-x'|+||A-A'||)\Big(|y|+|z|+ \frac{|y||y-z|}{\delta(h_x(A,y))\land\delta(h_{x'}(A',y))}\Big) \leq &
(a+o(a))(14a^{3/2}+o(a^{3/2}))\\
=&14a^{5/2}+o(a^{5/2}).
\end{align*}

Hence if there would exist $M$ as in the statement, we would have, for all $a>0$ small,
$6^{1/4}a^{3/2}+o(a^{3/2})\leq M(14a^{5/2}+o(a^{5/2}))$,
which is not possible.
\end{proof}

\section{About test functions}\label{tf}

The goal of this appendix is to prove Remarks~\ref{dnonvide} and~\ref{pdi}, as well as the following
result.

\begin{remark}\label{lcont}
Recall that $\cL\varphi(x)$ was  introduced in Definition~\ref{opbd} for $\varphi \in C(\cDd)\cap C^2(\Dd)$ and 
$x \in \Dd$. It is well-defined and continuous on $\Dd$. Moreover, we have 
\begin{equation}\label{newg}
\cL\varphi(x)= \int_{\R^d} [\varphi(\Lambda(x,x+z))-\varphi(x) - z \cdot \nabla \varphi(x) \indiq_{\{|z|<a\}}]
\frac{\dr z}{|z|^{d+\alpha}},
\qquad x\in\Dd,
\end{equation}
for any $a \in (0,\infty)$ and we can choose $a=0$ when $\alpha \in (0,1)$ and $a=\infty$ when $\alpha\in (1,2)$.
\end{remark}

\begin{proof} 
We recall that $\Lambda$ is continuous by Lemma~\ref{Lambdacon} and write
$\mathcal{L}\varphi = \mathcal{L}_1\varphi + \mathcal{L}_2\varphi + \mathcal{L}_3\varphi$, where
\begin{gather*}
\cL_1\varphi(x)=\int_{\{|z|\geq 1\}}  [\varphi(\Lambda(x,x+z))-\varphi(x)]\frac{\dr z}{|z|^{d+\alpha}},\\
\cL_2\varphi(x)=\int_{\{|z|<1\}} [\varphi(\Lambda(x,x+z))-\varphi(x)-(\Lambda(x,x+z)-x)\cdot \nabla\varphi(x)]
\frac{\dr z}{|z|^{d+\alpha}},\\
\cL_3\varphi(x)= \int_{\{|z|<1\}} (\Lambda(x,x+z)-x-z)\cdot \nabla \varphi(x)
\frac{\dr z}{|z|^{d+\alpha}}.
\end{gather*}
Since $\varphi\in C(\cDd)$, $\mathcal{L}_1\varphi$ is
well-defined and continuous on $\cDd$.
For $\varepsilon > 0$, we introduce the set
$\Dd_\varepsilon = \{x \in \Dd, \: d(x, \pDd) \geq \varepsilon\}$.
Since $\varphi\in C^2(\Dd)$, there is a constant $C_\e>0$ such that 
$|\varphi(y) -\varphi(x) - (y-x)\cdot\nabla\varphi(x)| \leq C_\varepsilon |y-x|^2$  for any $x,y\in \Dd_\varepsilon$.
Therefore, for any $x \in \Dd_\varepsilon$ and any $|z| \leq 1$, we have
\[
 \big|\varphi(\Lambda(x,x+z))-\varphi(x)-(\Lambda(x,x+z)-x)\cdot \nabla\varphi(x)\big| 
\leq C_{\varepsilon/2}|z|^2 \indiq_{\{|z| \leq \varepsilon/2\}} + \tilde{C}_\varepsilon \indiq_{\{\varepsilon/2 \leq |z| \leq  1\}},
\]
where $\tilde{C}_\varepsilon = 2\sup_{x\in \cDd} |\varphi(x)| + \sup_{x\in\Dd_\varepsilon}|\nabla\varphi(x)|$. 
We used that $|\Lambda(x,x+z) - x| \leq |z| \leq 1$. This bound is integrable with respect to 
$|z|^{-d-\alpha} \dr z$ so that we conclude by dominated convergence 
that $\mathcal{L}_2 \varphi$ is well-defined and continuous on $\Dd_\varepsilon$, for any $\varepsilon > 0$. 
Finally
observe that  $\Lambda(x, x+z)=x+z$ if $|z| \leq d(x, \pDd)$, so that for all $x \in \Dd_\e$, all $|z|\leq 1$,
$$
|(\Lambda(x,x+z)-x-z)\cdot \nabla \varphi(x)|\leq \indiq_{\{|z|>\e\}} \sup_{x\in \cDd} |\varphi(x)|.
$$
This bound being integrable with respect to $|z|^{-d-\alpha} \dr z$, we conclude
that $\cL_3\varphi$ is well-defined and continuous on $\Dd_\e$ for all $\e>0$.

\vip
Since $\int_{\mathbb{R}^d}z\indiq_{\{a < |z| < 1\}}|z|^{-d-\alpha} =\int_{\mathbb{R}^d}z\indiq_{\{1 < |z| < b\}}|z|^{-d-\alpha} = 0$ 
for any $a<1<b$, \eqref{newg} holds for any $a \in(0,\infty)$. 
When $\alpha \in(1, 2)$, $\int_{\{|z| > 1\}}|z|^{1-d-\alpha}\dr z < \infty$ 
and $\int_{\{|z| > 1\}}z|z|^{d-\alpha}\dr z = 0$ so that~\eqref{newg} holds with 
$a = \infty$. When $\alpha \in(0,1)$, we have 
$\int_{\{|z| < 1\}}|z|^{1-d-\alpha}\dr z < \infty$ and $\int_{\{|z| < 1\}}z|z|^{d-\alpha}\dr z = 0$ 
so that~\eqref{newg} holds with $a = 0$.
\end{proof}

\begin{proof}[Proof of Remark~\ref{dnonvide}]
We start with (a): we fix $\alpha \in (0,1)$, $\e \in (0,1-\alpha]$ 
and assume that $\varphi \in C^2(\Dd)\cap C^{\alpha+\e}(\cDd)$,
so that for all $x\in\Dd$, all $z \in \R^d$, using that $\Lambda(x,x+z)\in [x,x+z]$,
$$
|\varphi(\Lambda(x,x+z))-\varphi(x)|\leq C (|\Lambda(x,x+z)-x|^{\alpha+\e}\land 1)\leq C (|z|^{\alpha+\e}\land 1).
$$
Using~\eqref{newg} with $a=0$, we conclude that $\cL\varphi$ is bounded on $\Dd$ as desired.
\vip
We carry on with (b): we fix $\alpha \in [1,2)$, $\e \in (0,2-\alpha]$ 
and assume that $\varphi \in C^2(\Dd)\cap C^{\alpha+\e}(\cDd)$, with $\nabla \varphi(x)\cdot \bn_x=0$ for 
all $x\in \pDd$. We have to prove that $\cL\varphi$ is bounded on $\Dd$.
We use~\eqref{newg} with $a=\e_0/2$, $\e_0\in (0,1)$ being defined in Remark~\ref{asp}:
we write
$\cL\varphi(x)=\sum_{k=1}^4\cL_k\varphi(x)$, where,
introducing $\tx \in \pDd$ such that $d(x,\pDd)=|x-\tx|$,
\begin{gather*}
\cL_1\varphi(x)=\int_{\{|z|\geq \e_0/2\}}  [\varphi(\Lambda(x,x+z))-\varphi(x)]\frac{\dr z}{|z|^{d+\alpha}},\\
\cL_2\varphi(x)=\int_{\{|z|<\e_0/2\}} [\varphi(\Lambda(x,x+z))-\varphi(x)-(\Lambda(x,x+z)-x)\cdot \nabla\varphi(x)]
\frac{\dr z}{|z|^{d+\alpha}},\\
\cL_3\varphi(x)=\int_{\{|z|<\e_0/2\}} (\Lambda(x,x+z)-x-z)\cdot (\nabla \varphi(x)-\nabla \varphi(\tx))
\frac{\dr z}{|z|^{d+\alpha}},\\
\cL_4\varphi(x)=\int_{\{|z|< \e_0/2\}} (\Lambda(x,x+z)-x-z)\cdot \nabla \varphi(\tx)
\frac{\dr z}{|z|^{d+\alpha}},
\end{gather*}

Since $\varphi$ is bounded, $\cL_1\varphi$ is obviously bounded. 

\vip

Since $\varphi \in C^{\alpha+\e}(\cDd)$,
there is $C>0$ such that for all $x\in\Dd$, all $z \in \R^d$,
$$
|\varphi(\Lambda(x,x+z))-\varphi(x)-(\Lambda(x,x+z)-x)\cdot \nabla\varphi(x)|
\leq C |\Lambda(x,x+z)-x|^{\alpha+\e}\leq C |z|^{\alpha+\e},
$$
because $\Lambda(x,x+z)\in [x,x+z]$. Hence $\cL_2\varphi$ is bounded. 

\vip

Using that 
$|\nabla \varphi(x)-\nabla \varphi(\tx)|\leq C |x-\tx|^{\alpha+\e-1}$, that 
$|\Lambda(x,x+z)-x-z|\leq z$ and that $\Lambda(x,x+z)=x+z$
if $|z|<d(x,\pDd)=|x-\tx|$, we see that when $\alpha \in (1,2)$,
$$
|\cL_3\varphi(x)| \leq C |x-\tx|^{\alpha+\e-1} \int_{\{|z|\geq |x-\tx|\}}
\frac{|z|\dr z}{|z|^{d+\alpha}} = C  |x-\tx|^{\alpha+\e-1} |x-\tx|^{1-\alpha}=C|x-\tx|^\e,
$$
which is bounded since $\Dd$ is bounded. When $\alpha=1$, we write
$$
|\cL_3\varphi(x)| \leq C |x-\tx|^{\e} \int_{\{|x-\tx|\leq |z|\leq \e_0/2\}}
\frac{|z|\dr z}{|z|^{d+1}} \leq C  |x-\tx|^{\e} |\log |x-\tx||\leq C.
$$

We now study $\cL_4$. If first $|x-\tx|\geq \e_0/2$, 
then $x+z\in \Dd$ for all $z\in B_d(0,\e_0/2)$, whence $\cL_4\varphi(x)=0$.
We next treat the case where $|x-\tx|<\e_0/2$. Using a suitable isometry, we may assume that
$\tx=0$ and $\bn_\tx=\be_1$ (whence $0 \in \pDd$, 
$\Dd\subset \HH$ and $\nabla \varphi(\tx)\cdot\be_1=0$). Necessarily, $x-\tx$ is colinear to 
$\bn_\tx$ so that $x=x_1\be_1$, with $x_1\in (0,\e_0/2)$ (because $0<|x-\tx|<\e_0/2$).
Using that $\Lambda(x,x+z)=x+z$ if $x+z\in \Dd$, write
$\cL_4\varphi(x)=\cL_{41}\varphi(x)+\cL_{42}\varphi(x)+\cL_{43}\varphi(x)$,
where
\begin{gather*}
\cL_{41}\varphi(x)=
\int_{\{|z|< \e_0/2\}} \indiq_{\{x+z \in \HH\setminus \Dd\}}(\Lambda(x,x+z)-x-z)\cdot \nabla \varphi(\tx)
\frac{\dr z}{|z|^{d+\alpha}},\\
\cL_{42}\varphi(x)=\int_{\{|z|< \e_0/2\}} \indiq_{\{x+z \notin \HH\}}(\Lambda(x,x+z)-\Gamma(x,z))\cdot \nabla \varphi(\tx)
\frac{\dr z}{|z|^{d+\alpha}},\\
\cL_{43}\varphi(x)=\int_{\{|z|< \e_0/2\}} \indiq_{\{x+z \notin \HH\}}(\Gamma(x,z)-x-z)\cdot \nabla \varphi(\tx)
\frac{\dr z}{|z|^{d+\alpha}},
\end{gather*}
and where, for $z =(z_1,\dots,z_d) \in \R^d$ such that $x+z \notin \HH$, we have set 
$$
\Gamma(x,z)=-\frac{x_1}{z_1}\sum_{i=2}^d z_i \be_i.
$$
By Remark~\ref{asp} (with $x=0$ and $A=I$),
$\Dd\cap B_d(0,\e_0)=\{u \in B_d(0,\e_0) : u_1>\psi(u_2,\dots,u_d)\}$, 
for some $C^3$ convex function $\psi:B_{d-1}(0,\e_0)\to\R_+$
such that $\psi(0)=0$, $\nabla\psi(0)=0$, and $||D^2\psi||_\infty\leq C$ (with $C$ not depending on $\tx$).
This implies that for $|z|<\e_0/2$ (recall that $x=x_1\be_1$ with $x_1\in (0,\e_0/2)$), $x+z\in \HH\setminus \Dd$
if and only if $x_1+z_1 \in (0,\psi(z_2,\dots,z_d)]$. Thus, bounding roughly $|\Lambda(x,x+z)-x-z|\leq |z|$,
writing $z=(z_1,v)$ and using that $|z|\geq|v|$,
$$
|\cL_{41}\varphi(x)|\leq ||\nabla\varphi||_\infty \int_{B_{d-1}(0,\e_0/2)}\!\!\dr v\int_{-x_1}^{\psi(v)-x_1}\!\!
|z|^{1-d-\alpha} \dr z_1 \leq ||\nabla\varphi||_\infty  \int_{B_{d-1}(0,\e_0/2)} \!\!   |v|^{1-d-\alpha}\psi(v)\dr v.
$$
Since $\psi(v)\leq C|v|^2$ and since $3-d-\alpha>-(d-1)$, we conclude that $\cL_{41}\varphi$ is bounded.
\vip

Since $\nabla\varphi(\tx)\cdot \be_1=0$, 
we may write $\nabla\varphi(\tx)=\sum_{i=2}^d \rho_i \be_i$. 
Moreover, we also have $\nabla\varphi(\tx)\cdot x=0$  so that
$(\Gamma(x,z)-x-z)\cdot \nabla \varphi(\tx)=-(\frac{x_1}{z_1}+1)\sum_{i=2}^d\rho_i z_i$,
and we find that
$$
\cL_{43}\varphi(x)=-\sum_{i=2}^d\rho_i \int_{\{z \in B_d(0,\e_0/2) \: : \: x+z \notin \HH\}}\Big(\frac{x_1}{z_1}+1\Big) 
z_i \frac{\dr z}{|z|^{d+\alpha}} =  0
$$ 
by a symmetry argument,
using the substitution $(z_1,z_2,\dots,z_d)\mapsto (z_1,-z_2,\dots,-z_d)$, which leaves invariant the
set $\{z \in B_d(0,\e_0/2) : x+z \notin \HH\}$ (recall that $x=x_1\be_1$).

\vip

Next, assume for a moment that for some constant $K>0$,
\begin{equation}\label{ttch} 
\text{for all $z\in \R^d$ such that $x+z \notin \HH$,}\quad
|\Gamma(x,z)-\Lambda(x,x+z)| \leq K \frac{x_1^2}{|z_1|^3} \sum_{i=1}^d |z_i|^3.
\end{equation}
Recalling that $x+z \notin \HH$ if and only if $x_1+z_1\leq 0$ and that $\e_0\in (0,1)$,
\begin{align*}
|\cL_{42}\varphi(x)|\leq& K||\nabla \varphi||_\infty x_1^2 \int_{-1}^{-x_1} \dr z_1 \int_{B_{d-1}(0,1)} 
\Big(1+\frac{|z_2|^3+\cdot+|z_d|^3}{|z_1|^3} \Big) \frac{\dr z_2\cdots \dr z_d}{|z|^{d+\alpha}} \\
\leq &K' x_1^2 \int_{-1}^{-x_1} \dr z_1 \int_{B_{d-1}(0,1)} 
\Big(1+\frac{|z_2|^3}{|z_1|^3} \Big) \frac{\dr z_2\cdots \dr z_d}{|z|^{d+\alpha}}\\
\leq & K' x_1^2 \int_{x_1}^{1} \dr z_1 \int_0^1 \Big(1+\frac{z_2^3}{z_1^3} \Big) 
\frac{\dr z_2}{(z_1+z_2)^{2+\alpha}}\\
\leq &  K' x_1^2 \int_{x_1}^{1} \dr z_1 \Big[\int_0^{z_1} \Big(1+\frac{z_2^3}{z_1^3} \Big) 
\frac{\dr z_2}{z_1^{2+\alpha}} +\int_{z_1}^1 \Big(1+\frac{z_2^3}{z_1^3} \Big) 
\frac{\dr z_2}{z_2^{2+\alpha}}\Big],
\end{align*} 
for some constant $K'$, allowed to vary from line to line.
Using that $\alpha \in [1,2)$, we conclude that
\begin{align*}
|\cL_{42}\varphi(x)|\leq K' x_1^2 \int_{x_1}^{1} (z_1^{-1-\alpha} + z_1^{-3})\dr z_1 \leq 
 K'  x_1^2 \int_{x_1}^{1} z_1^{-3}\dr z_1 \leq K'.
\end{align*}

It remains to check~\eqref{ttch}. We set $\gamma=\Gamma(x,z)$ and observe that 
$\gamma=x+\theta_1 z$, where
$$
\theta_1= \frac {x_1}{|z_1|}.
$$ 
Indeed, recall that $z_1<-x_1<0$ (since $x+z \notin \HH$), that $x=x_1\be_1$, and note that $x+\theta_1 z = x_1\be_1-\frac{x_1}{z_1}z=-\frac{x_1}{z_1}\sum_{i=2}^d z_i\be_i$.
We then set $\lambda=\Lambda(x,x+z)=x+\theta z$, with $\theta \in (0,1)$ determined by the fact that 
$\lambda \in \pDd$. It holds that $\theta\leq \theta_1$: it suffices to prove that $\gamma \notin \Dd$,
which is obvious since $\gamma \notin \HH$ (because $\gamma_1=0$).
We also have $\theta\geq \theta_0$, where,
for $C$ the constant such that $\psi(v)\leq C |v|^2$,
$$
\theta_0=\Big(1- \frac{C x_1 |z|^2}{|z_1|^2}\Big) \theta_1.
$$ 
Indeed, it suffices to treat the case where $\theta_0>0$ and to check that
$\lambda^0:=x+\theta_0 z \in \Dd$, {\it i.e.} that
$\lambda^0_1\geq \psi(\lambda^0_2,\dots,\lambda^0_d)$. To this end, it suffices to note that
$$
\lambda^0_1=x_1+\frac {x_1}{|z_1|}\Big(1- \frac{Cx_1 |z|^2}{|z_1|^2}\Big) z_1
= \frac{C x_1^2|z|^2}{|z_1|^2},
$$
while, recalling that $x=x_1\be_1$ and that $\theta_0 \leq \theta_1$ (since we are in the case $\theta_0> 0$),
$$
\psi(\lambda^0_2,\dots,\lambda^0_d)\leq C\sum_{i=2}^d (\lambda_i^0)^2
=C \theta_0^2\sum_{i=2}^d z_i^2
\leq C\theta_0^2|z|^2 \leq C \theta_1^2|z|^2 =
C\frac {x_1^2|z|^2}{z_1^2}.
$$
We have shown that $\Gamma(x,z)=x+\theta_1 z$ and $\Lambda(x,x+z)=x+\theta z$ with $\theta\in [\theta_0,\theta_1]$,
so that
$$
|\Gamma(x,z)-\Lambda(x,x+z)|\leq  (\theta_1-\theta_0)||z|=\frac{Cx_1^2|z|^3}{|z_1|^3},
$$
from which~\eqref{ttch} follows.
\end{proof}

\begin{proof}[Proof of Remark~\ref{pdi}]
Point (a) relies on a rather noticeable observation, see Step 8, 
while points (b) and (c) are not difficult once (a) is checked (because $\Dd=B_d(0,1)$).
\vip

{\it Step 1.} Let $\beta \in (0,\alpha/2)$ and $\varphi \in C^{\alpha/2}(\cDd)$ be radially symmetric
and satisfy $\varphi|_{\pDd}=0$. Then 
$\varphi \in H_\beta$ if and only if $\cH_\beta \varphi(\be_1)=0$, where
$$
\cH_\beta\varphi(\be_1):=\int_\Dd \varphi(y) \frac{\dr y}{|y-\be_1|^{d+\beta}}.
$$
Indeed, using that $\varphi|_{\pDd}=0$ and the symmetry of $\varphi$ and of the domain $\Dd=B_d(0,1)$
we see that
$$
\text{for all $x \in \pDd$, all $\e>0$,} \quad
\cH_{\beta,\e}\varphi(x)=\cH_{\beta,\e}\varphi(\be_1)=\int_\Dd \varphi(y) \indiq_{\{|y-\be_1|\geq \e\}}
\frac{\dr y}{|y-\be_1|^{d+\beta}},
$$
which converges to $\cH_\beta\varphi(\be_1)$, since $|\varphi(y)|=|\varphi(y)-\varphi(\be_1)|\leq
C|y-\be_1|^{\alpha/2}$ and since $\beta \in (0,\alpha/2)$.

\vip

{\it Step 2.} Let $\varphi_1\in C(\cDd)$ such that
$\varphi_1(x)=[d(x,\Dd^c)]^{\alpha/2}$ when $d(x,\Dd^c)\leq 1/2$
and let $\varphi_2 \in C_c(\Dd)$. Then $\varphi_0=\varphi_1+\varphi_2 \notin H_*$.
Indeed, for any $x \in \pDd$ and for $h=(h_1,0)\in  B_2(\be_1,1)$ with $h_1>0$ small enough, we have,
since $\bn_{\be_1}=-\be_1$,
$$
\cH_{*}\varphi_0(x,h)=\cH_{*}\varphi_0(\be_1,h)=h_1^{-\alpha/2}[d((1-h_1)\be_1,D^c)]^{\alpha/2}=1.
$$
Hence $\liminf_{\e\to 0} \sup_{h \in B_2(0,\e)\cap  B_2(\be_1, 1)} \cH_{*}\varphi_0(x,h)\geq 1\neq 0$.
\vip

{\it Step 3.} For any $\varphi \in C^1(\cDd)$, we have $\varphi \in H_*$, because
$|\cH_{*}\varphi(x, h)| \leq |h|^{1-\alpha/2}||\nabla \varphi||_\infty$ for all $x \in \pDd$, all $h \in B_2(\be_1,1)$.

\vip

{\it Step 4.} Point (b) follows from  Step~1, since $\varphi_1-a\varphi_2 \in C^{\alpha/2}(\cDd)$ for any $a>0$,
and since choosing $a=\frac{\cH_\beta\varphi_1(\be_1)}{\cH_\beta\varphi_2(\be_1)}$, we find
$\cH_\beta (\varphi_1-a\varphi_2)(\be_1)=0$.

\vip

{\it Step 5.}  Assuming that (a) holds true, we show (c) when $\beta \in (0,\alpha/2)$ and $\beta'=*$.
Consider some nonnegative radially symmetric 
$\varphi_1 \in C^2(\Dd)$ such that
$\varphi_1(x)=[d(x,\Dd^c)]^{\alpha/2}$ as soon as $d(x,\Dd^c)\leq 1/2$.
Then $\varphi_1 \in D_\alpha$ by (a) and  $\varphi_1 \in C^{\alpha/2}(\cDd)$.
Let now $\varphi_2 \in C^2_c(\Dd)$ be radially symmetric, 
nonnegative and positive on $B_d(0,1/2)$. By Remark~\ref{dnonvide}, $\varphi_2 \in D_\alpha$. 
Thus $\varphi_0=\varphi_1 - a \varphi_2 \in D_\alpha $ for all $a>0$.
By (b), there is $a>0$ such that $\varphi_0\in H_\beta$. 
But $\varphi_0 \notin H_*$ by Step~2 and since $\varphi_0 = \varphi_1$ near the boundary.

\vip

{\it Step 6.} We now show (c) when $\beta = *$ and $\beta' \in (0,\alpha/2)$. We consider some nonnegative 
radially symmetric $\varphi \in C^2(\cDd)$ such that
$\varphi(x)=[d(x,\Dd^c)]^2$ as soon as $d(x,\Dd^c)\leq1/2$.
Then $\nabla \varphi(x)=0$ for all $x \in \pDd$, so that $\varphi \in D_\alpha$ by Remark~\ref{dnonvide}. 
Moreover, $\varphi \in C^1(\cDd)$, whence $\varphi \in H_*$ by Step 3.
But $\cH_{\beta'}\varphi(\be_1)>0$, so that $\varphi \notin H_{\beta'}$ by Step 1.

\vip

{\it Step 7.} We next prove (c) when $0<\beta<\beta'<\alpha/2$ (a similar argument works
when ${0<\beta'<\beta<\alpha/2}$). Consider some radially symmetric probability density $\varphi_1\in C^2_c(\Dd)$ 
such that $\cH_\beta\varphi_1(\be_1)>1$. Such a function $\varphi_1$
is easily built, since $\int_\Dd |y-\be_1|^{-d-\beta}\dr y =\infty$. By the Jensen inequality,
$\cH_{\beta'}\varphi_1(\be_1)\geq [\cH_{\beta}\varphi_1(\be_1)]^{(d+\beta')/(d+\beta)}>\cH_{\beta}\varphi_1(\be_1)$.
Let $\eta=\cH_{\beta'}\varphi_1(\be_1)-\cH_{\beta}\varphi_1(\be_1)$ and let $\varphi_2 \in C^2_c(\Dd)$ 
be radially symmetric and so that 
$|\cH_{\beta}\varphi_2(\be_1)-1|+|\cH_{\beta'}\varphi_2(\be_1)-1|\leq 
\frac 1 2 \land \frac \eta {4\cH_\beta\varphi_1(\be_1)}$. 
Such a function
is easily built, since $\int_{\Dd} |y-\be_1|^{-d-\beta}\delta_0(\dr y)
=\int_{\Dd} |y-\be_1|^{-d-\beta'}\delta_0(\dr y)=1$. Consider now 
$a=\frac{\cH_{\beta}\varphi_1(\be_1)}{\cH_\beta \varphi_2(\be_1)}$, 
so that $\cH_\beta \varphi_0(\be_1)=0$, where
$\varphi_0=\varphi_1-a\varphi_2 \in H_\beta$ by Step~1. Also, $\varphi_0 \in C^2_c(\Dd)$,
so that $\varphi_0 \in D_\alpha$ by Remark~\ref{dnonvide}. But $\cH_\beta \varphi_2(\be_1)\geq 1/2$, so that
$a\leq 2\cH_{\beta}\varphi_1(\be_1)$, and $|\cH_{\beta}\varphi_2(\be_1)-
\cH_{\beta'}\varphi_2(\be_1)|\leq \frac{\eta}{4\cH_\beta\varphi_1(\be_1)}$, whence
$$
\cH_{\beta'}\varphi_0(\be_1)=\eta + \cH_\beta \varphi_1(\be_1)-a\cH_{\beta'}\varphi_2(\be_1)
=\eta+a(\cH_{\beta}\varphi_2(\be_1)-
\cH_{\beta'}\varphi_2(\be_1))
\geq \eta-\frac {a\eta} {4 \cH_\beta\varphi_1(\be_1)}
\geq \frac \eta 2.
$$
Thus $\varphi_0 \notin H_{\beta'}$ by Step~1.

\vip

{\it Step 8.} For $\psi \in C^2((0,\infty))\cap C([0,\infty))$ and $a \in (0,\infty)$, we set 
$$
\cK\psi(a)= \int_\R [\psi((a+b)_+)-\psi(a)- \psi'(a)b\indiq_{\{|b|<1\}}] \frac{\dr b}{|b|^{1+\alpha}}.
$$
Setting $\psi_*(a)=a^{\alpha/2}$, we have $\cK\psi_*(a)=0$ for all $a \in (0,\infty)$.
\vip
When $\alpha \in (0,1)$, we have 
$\cK\psi_*(a)= \int_\R [(a+b)_+^{\alpha/2}-a^{\alpha/2}] \frac{\dr b}{|b|^{1+\alpha}}=a^{-\alpha/2} A$, where
$$
A= \int_\R [(1+b)_+^{\alpha/2}-1] \frac{\dr b}{|b|^{1+\alpha}}=A_1+A_2-\frac1\alpha,
$$
$A_1$, $A_2$, $-\frac1\alpha$ corresponding to the integrals on $(0,\infty)$, $(-1,0)$, $(-\infty,-1)$. 
The values of $A_1$ and $A_2$ are given in~\cite[Proof of Lemma~9, page~143]{f2013} (with 
$\beta=\frac\alpha2$). We find $A=0$.
\vip
When $\alpha \in (1,2)$, we have $\cK\psi_*(a)= \int_\R [(a+b)_+^{\alpha/2}-a^{\alpha/2}
-\frac\alpha2 a^{\alpha/2-1}b] \frac{\dr b}{|b|^{1+\alpha}}=a^{-\alpha/2} B$, where
$$
B= \int_\R [(1+b)_+^{\alpha/2}-1-\frac\alpha2b] \frac{\dr b}{|b|^{1+\alpha}}=B_1+B_3+\frac{\alpha}{2(\alpha-1)}-\frac{1}
\alpha,
$$
$B_1$, $B_3$, $\frac{\alpha}{2(\alpha-1)}-\frac{1}\alpha$ corresponding to the integrals on 
$(0,\infty)$, $(-1,0)$, $(-\infty,-1)$. The values of $B_1$ and $B_3$ are given 
in~\cite[Proof of Lemma~9, pages~144-145]{f2013} (with $\beta=\frac\alpha2$): we find $B=0$.
\vip
When $\alpha=1$, $\cK\psi_*(a)= \int_\R [(a+b)_+^{1/2}-a^{1/2}
- \frac 1{2a^{1/2}}b\indiq_{\{ |b|<1 \}}] \frac{\dr b}{|b|^{2}}=a^{-1/2} C$, with
$$
C= \int_\R [(1+b)_+^{1/2}-1-\frac1 2 b \indiq_{\{|b|<1/a\}}] \frac{\dr b}{|b|^{2}}= 
\int_\R [(1+b)_+^{1/2}-1-\frac1 2 b \indiq_{\{|b|<1\}}] \frac{\dr b}{|b|^{2}}=
C_1+C_2+C_3+C_4,
$$
corresponding to the integral on 
$(1,\infty)$, $(0,1)$, $(-1,0)$ and $(-\infty,-1)$.
One can show that
\begin{gather*}
C_1=\sqrt 2 +\log(1+\sqrt 2)-1,\;\; C_2=\frac 32 -\sqrt 2 -\log(1+\sqrt 2)+ \log 2,\;\;
C_3=\frac 12 - \log 2, \;\; C_4=-1,
\end{gather*}
so that $C=0$.
\vip

{\it Step 9.} It remains to prove (a). We consider $\varphi\in C^2(\Dd)$  such that
$\varphi(x)=[d(x,\Dd^c)]^{\alpha/2}$ as soon as $d(x,\Dd^c)\leq 1/2$, {\it i.e.}  
$\varphi(x)=(1-|x|)^{\alpha/2}$ as soon as $|x|\geq 1/2$. We  
only study the case $\alpha \in (1,2)$, the cases $\alpha\in (0,1)$ and $\alpha=1$ 
being treated similarly, with some variations.
We have to show that $\cL\varphi$ is bounded on $\Dd$.
Since $\cL\varphi \in C(\Dd)$ by Remark~\ref{lcont}, 
it suffices that $\cL\varphi$ is bounded on
$K=\{x \in \Dd : d(x,D^c)\leq 1/4\}$ and, by symmetry, on $K_1=\{x_1\be_1 : x_1  \in [3/4,1)\}$.
For $x \in K_1$, we have $\varphi(x)=(1-x_1)^{\alpha/2}$,
$\nabla \varphi(x)=-\frac{\alpha}{2}(1-x_1)^{\alpha/2-1}\be_1$
and $d(x+z,\Dd^c)\leq 1/2$ if $|z|< 1/4$, so that, using~\eqref{newg} with $a=1/4$,
$$
\cL\varphi(x)=\int_{\R^d} [\varphi(\Lambda(x,x+z))-\varphi(x)-z\cdot \nabla \varphi(x)\indiq_{\{|z|<1/4\}}] 
\frac{\dr z}{|z|^{d+\alpha}}= \cL_1\varphi(x)+ \cL_2\varphi(x)+ \cL_3\varphi(x),
$$
with
\begin{gather*}
\cL_1\varphi(x)=\int_{\{|z|\geq 1/4\}} [\varphi(\Lambda(x,x+z)))-\varphi(x)] \frac{\dr z}{|z|^{d+\alpha}},\\
\cL_2\varphi(x)=\int_{\{|z|< 1/4\}} [(1-x_1-z_1)_+^{\alpha/2} - (1-x_1)^{\alpha/2}+\frac\alpha2 (1-x_1)^{\alpha/2-1} z_1]  
\frac{\dr z}{|z|^{d+\alpha}},\\
\cL_3\varphi(x)=\int_{\{|z|< 1/4\}} [d(\Lambda(x,x+z),\Dd^c)^{\alpha/2}- (1-x_1-z_1)_+^{\alpha/2}]  
\frac{\dr z}{|z|^{d+\alpha}}.
\end{gather*}
First, $\cL_1\varphi$ is bounded because $\varphi$ is bounded. Next,
$\cL_2\varphi(x)=\cL_{21}\varphi(x)-\cL_{22}\varphi(x)$, with
\begin{gather*}
\cL_{21}\varphi(x)=\int_{\R^d} [(1-x_1-z_1)_+^{\alpha/2} - (1-x_1)^{\alpha/2}+\frac\alpha2 (1-x_1)^{\alpha/2-1} z_1
\indiq_{\{|z|<1/4\}}]
\frac{\dr z}{|z|^{d+\alpha}},\\
\cL_{22}\varphi(x)=\int_{\{|z|\geq 1/4\}} [(1-x_1-z_1)_+^{\alpha/2} - (1-x_1)^{\alpha/2}]
\frac{\dr z}{|z|^{d+\alpha}}.
\end{gather*}
Integrating in $z_2,\dots,z_d$, setting $y=1-x_1>0$ and using the substitution $u=-z_1$, we find that 
for some constant $c>0$,
\begin{align*}
\cL_{21}\varphi(x)=&c\int_{\R} [(1-x_1-z_1)_+^{\alpha/2} - (1-x_1)^{\alpha/2}+\frac\alpha2 (1-x_1)^{\alpha/2-1} z_1
\indiq_{\{|z_1|<1/4\}}]
\frac{\dr z_1}{|z_1|^{1+\alpha}}\\
=&c\int_{\R} [(y+u)_+^{\alpha/2} - y^{\alpha/2}-\frac\alpha 2 y^{\alpha/2-1} u \indiq_{\{|u|<1/4\}}]\frac{\dr u}{|u|^{1+\alpha}}
=0
\end{align*}
by Step 7. Moreover, $\cL_{22}\varphi(x)$ is bounded, because
$|(1-x_1-z_1)_+^{\alpha/2} - (1-x_1)^{\alpha/2}|\leq |z_1|^{\alpha/2}$, which is integrable against $|z|^{-\alpha-d}$.
\vip

Set $\HH_1=\{y \in \R^d : y_1 <1\}$, which contains $\Dd=B_d(0,1)$.
Observe that if $x+z\notin \HH_1$, then $d(\Lambda(x,x+z),\Dd^c)=0$ and $(1-x_1-z_1)_+=d(x+z,\HH_1^c)=0$.
We thus may write $\cL_3\varphi(x)=-\cL_{31}\varphi(x)-\cL_{32}\varphi(x)$, where
\begin{gather*}
\cL_{31}\varphi(x)=\int_{\{|z|<1/4\}} \indiq_{\{x+z \in \HH_1\setminus \Dd\}} 
[(1-x_1-z_1)^{\alpha/2}-d(\Lambda(x,x+z), \Dd^c )^{\alpha/2}] \frac{\dr z}{|z|^{d+\alpha}},\\
\cL_{32}\varphi(x)=\int_{\{|z|<1/4\}} \indiq_{\{x+z \in \Dd\}}
[(1-x_1-z_1)^{\alpha/2}-d(\Lambda(x,x+z), \Dd^c )^{\alpha/2}]  \frac{\dr z}{|z|^{d+\alpha}}.
\end{gather*}
If $x+z\in \HH_1\setminus \Dd$, $d(\Lambda(x,x+z),\Dd^c)=0$, so that
$$
\cL_{31}\varphi(x)=\int_{\{|z|<1/4\}} \indiq_{\{x+z \in \HH_1\setminus B_d(0,1)\}} 
(1-x_1-z_1)^{\alpha/2} \frac{\dr z}{|z|^{d+\alpha}}.
$$
Recalling that $x = x_1\be_1$ and using the notation $z=(z_1,v)$, so that $x+z=(x_1+z_1,v)$, 
we have $1-|v|^2\leq x_1+z_1< 1$ as soon as $|z|<1/4$ and $x+z \in \HH_1\setminus B_d(0,1)$: indeed, 
$|x+z|^2\geq 1$ tells us that $(x_1+z_1)^2+|v|^2\geq 1$, while $x_1+z_1<1$ since $x+z \in \HH_1$
and $x_1+z_1\geq 0$ since $x_1\geq3/4$ and $|z|<1/4$, so that $x_1+z_1 \geq (x_1+z_1)^2\geq 1-|v|^2$.
Since finally $|z|\geq|v|$,
$$
0\leq \cL_{31}\varphi(x)\leq \int_{B_{d-1}(0,1/4)} \frac{\dr v}{|v|^{d+\alpha}}
\int_{1-x_1-|v|^2}^{1-x_1} \!\!\!(1-x_1-z_1)^{\alpha/2} \dr z_1
=\frac1{\alpha/2+1} \int_{B_{d-1}(0,1/4)} \frac{|v|^{2+\alpha}\dr v}
{|v|^{d+\alpha}},
$$
and $\cL_{31}\varphi$ is bounded, because $2+\alpha-d-\alpha>1-d$. 

\vip

To treat $\cL_{32}$, we observe that $x+z \in B_d(0,1)$, whence
$$
\Delta(x,z):=(1-x_1-z_1)^{\alpha/2}-d(\Lambda(x,x+z),D^c)^{\alpha/2}=(1-x_1-z_1)^{\alpha/2}-(1-|x+z|)^{\alpha/2}\geq 0, 
$$ 
Moreover, $x+z \in B_d(0,1)$ and $|z|<1/4$ imply that $x_1+z_1\in (1/2,1)$ (because $x_1 \in [3/4,1)$)
and that, still with the notation $z=(z_1,v)$, so that $x+z=(x_1+z_1,v)$,
\begin{align*}
\Delta(x,z)\leq& (1-x_1-z_1)^{\alpha/2-1}[|x+z|-x_1-z_1] \leq (1-x_1-z_1)^{\alpha/2-1} |v|^2.
\end{align*}
We used that $a^{\alpha/2}-b^{\alpha/2}\leq a^{\alpha/2-1}(a-b)$ for $a \geq b\geq 0$
and that, since $\sqrt{1+u}\leq 1+ \frac u2$ for all $u\geq 0$, it holds that
$$
0\leq |x+z|-x_1-z_1=(x_1+z_1)\Big(\sqrt{1+\frac{|v|^2}{(x_1+z_1)^2}}-1 \Big)\leq 
\frac{|v|^2}{2(x_1+z_1)}\leq |v|^2.
$$ 
All this shows that
\begin{align*}
0\leq \cL_{32}\varphi(x)\leq&  \int_{\{|z|<1/4\}} \indiq_{\{1/2<x_1+z_1<1\}} 
(1-x_1-z_1)^{\alpha/2-1} |v|^2 \frac{\dr z}{|z|^{d+\alpha}}.
\end{align*}
We now set $\theta=d+\frac12+\frac\alpha4$ and use that $|z|^{-d-\alpha}\leq [\max(|v|,|z_1|)]^{-d-\alpha}\leq
|v|^{-\theta} |z_1|^{\theta-d-\alpha}$ (we have $\theta \in [0,d+\alpha]$ since $\alpha>1\geq \frac 23$) to get,
for some constant $C$ depending only on $\alpha$ and $d$,
\begin{align*}
|\cL_{32}\varphi(x)|\leq&  \int_{\{|v|<1/4\}} |v|^{2-\theta} \dr v \int_{1/2-x_1}^{1-x_1} (1-x_1-z_1)^{\alpha/2-1} 
|z_1|^{\theta-d-\alpha}\dr z_1\\
\leq& C \int_{-1}^{1-x_1} (1-x_1-z_1)^{\alpha/2-1} 
|z_1|^{\theta-d-\alpha}\dr z_1,
\end{align*}
since $2-\theta>1-d$ (because $\alpha<2$) and $x_1<1$. Substituting $z_1=-(1-x_1)u$, we find
\begin{align*}
|\cL_{32}\varphi(x)|\leq& C (1-x_1)^{\theta-d-\alpha/2} \int_{-1}^{1/(1-x_1)} (1+u)^{\alpha/2-1} 
|u|^{\theta-d-\alpha}\dr u\\
\leq & C  (1-x_1)^{\theta-d-\alpha/2} \Big[\int_{-1}^1  (1+u)^{\alpha/2-1} |u|^{\theta-d-\alpha}\dr u +
\int_1^{1/(1-x_1)} u^{\theta-1-d-\alpha/2}\dr u\Big]\\
\leq & C (1-x_1)^{\theta-d-\alpha/2} [1+ (1-x_1)^{-(\theta-d-\alpha/2)}].
\end{align*}
We finally used that $\alpha/2-1>-1$, that $\theta-d-\alpha>-1$ (since $\alpha<2$), and that 
$\theta-d-\alpha/2>0$ (since $\alpha<2$). Recalling that $x_1 \in (3/4,1)$, we conclude that $\cL_{32}\varphi$
is bounded.
\end{proof}

\section{Skorokhod topologies}\label{sko}

For $x_k,x \in \DD(\R_+,\R^d)$, we say that $x_k \to x$ for the $\JS$-topology,
see Jacod and Shiryaev~\cite[Chapter~VI]{jacod2013limit}, if
there exists a family of continuous increasing functions $\lambda_k=\R_+\to\R_+$ with $\lambda_k(0)=0$ such that
$$
\lim_k ||\lambda_k-I||_\infty=0 \quad \text{and for all $T>0$,} \quad \lim_k ||x_k\circ \lambda_k - x||_{\infty,T}=0,
$$
where $||\lambda_k-I||_\infty=\sup_{t\geq 0}|\lambda_k(t)-t|$ and 
$||x_k\circ \lambda_k - x||_{\infty,T}=\sup_{t\in [0,T]} |x_k(\lambda_k(t))-x(t)|$.

\vip

We now recall the notion of convergence in the $\MS$-topology, see Whitt~\cite[Chapter~12]{book_whitt}.  
For $x\in \DD(\R_+,\R^d)$, we define the completed graph $\Gamma_{x}$ of $x$ as
\[
 \Gamma_{x} = \big\{(t, y)\in\mathbb{R}_+ \times\mathbb{R}^d, \: y \in [x(t-), x(t)]\big\}.
\]
A parametric representation of $x$ is a continuous function $(u, r)$ mapping $\mathbb{R}_+$ 
onto $\Gamma_{x}$ such that
for any $0 \leq s \leq t$, either (i) $u(s) < u(t)$, or 
(ii) $u(s) = u(t)$ and for any $i=1, \cdots, d$, 
$|x_i(u(s) -) - r_i(s)| \leq |x_i(u(t)-) - r_i(t)|$.
Let $\Pi_{x}$ be the set of parametric  representations of $x$.

\vip
For $x_k, x \in \mathbb{D}(\mathbb{R}_+, \mathbb{R}^d)$, we say that $x_k \to x$ for the $\MS$-topology if 
there exists a sequence $(u_k, r_k) \in \Pi_{x_k}$ and $(u, r) \in \Pi_x$ such that
\[
\text{for all $T>0$,} \quad \lim_k (||u_k - u||_{\infty, T} + ||r_k -r||_{\infty, T}) = 0.
\]
The following lemma must be standard, but we found no precise reference.

\begin{lemma}\label{j1m1}
Consider a family of continuous piecewise affine elements $(x_k)_{k \geq 1}$ of $C(\R_+,\R^d)$: for some $0=t_0^k< t_1^k < \dots$ such that $\lim_n t^k_n=\infty$, we have, for all $n\geq 0$,
\begin{equation}\label{apt}
x_k(t)=x_k(t^k_n)+\frac{t-t^k_n}{t^k_{n+1}-t^k_n}(x_k(t^k_{n+1})-x_k(t^k_n))
\quad \text{for all $t\in [t^k_n,t^k_{n+1})$.} 
\end{equation}
For each $k\geq 1$, define 
$\bar x_k(t)=\sum_{n\geq 0}x_k(t^k_n)\indiq_{\{t \in [t^k_n,t^k_{n+1})\}}$ and, for $t\geq 0$, set
$m^k_t=\sum_{n\geq 1} \indiq_{\{t^k_n\leq t\}}$.
If $\bar x_k \to x \in \DD(\R_+,\R^d)$ 
for the $\JS$-topology and if for all $T>0$,
$$
\lim_k \sup_{n=0,\dots,m^k_T} (t^k_{n+1}-t^k_n)= 0,
$$
then $x_k\to x$ for the $\MS$-topology.
\end{lemma}

\begin{proof}
Since the $\JS$-topology is stronger than the $\MS$-topology, we have $\bar{x}_k \to x$ for 
the $\MS$-topology.
Hence it suffices to show that there exist $(u_k, r_k) \in \Pi_{x_k}$ and
$(\bar{u}_k, \bar{r}_k) \in \Pi_{\bar{x}_k}$ such that for all $T>0$, 
$\lim_k (||u_k - \bar u_k||_{\infty, T} + ||r_k -\bar r_k||_{\infty, T}) = 0$.

\vip

Since $x_k$ is continuous, we naturally 
define $(u_k,r_k) \in \Pi_{x_k}$ by setting $(u_k(t), r_k(t)) = (t, x_k(t))$.

\vip

We then define $(\bar{u}_k, \bar{r}_k)$:
for all $n \geq 0$, we set 
$$
A^k_n= k |x_k(t^k_{n+1})- x_k(t^k_{n})|+2
\quad \text{and} \quad
s^k_n=t^k_n+\frac{t^k_{n+1}-t^k_n}{A^k_n} \in (t^k_n,t^k_{n+1}),
$$
and, for $t \in [t^k_n,t^k_{n+1})$,
$$
(\bar u_k(t),\bar r_k(t))= \begin{cases}
\Big(t^k_n+A^k_n(t-t^k_n), x_k(t^k_n)\Big) & \text{if $t \in [t^k_n,s^k_n]$},\\[10pt]
\Big(t^k_{n+1},x_k(t^k_n)+\frac{t-s^k_n}{t^k_{n+1}-s^k_n}(x_k(t^k_{n+1})-x_k(t^k_n)\Big) 
& \text{if $t \in (s^k_n,t^k_{n+1})$}.
\end{cases}
$$
Let us check that $(\bar{u}_k, \bar{r}_k) \in \Pi_{\bar x_k}$. 
One easily checks that $(\bar{u}_k, \bar{r}_k)$ is a continuous bijection from $\R_+$ onto 
$$
\Gamma_{\bar x_k}= \cup_{n\geq 0} \Big[\Big([t^k_n,t^k_{n+1})\times \{x_k(t_n^k)\}\Big)\cup
\Big(\{t^k_{n+1}\}\times[x_k(t^k_n), x_k(t^k_{n+1})) ]\}\Big)\Big].
$$
Next for $0\leq s < t$, if $\bar u_k(s)=\bar u_k(t)$, then there is $n\geq 0$ such that
$s^k_n \leq s < t \leq t^k_{n+1}$, so that $\bar u_k(s)=\bar u_k(t)=t^k_{n+1}$, whence
$\bar x_{k}(\bar u_k(s)-)=x_k(t^k_{n})$. Thus for all $i=1,\dots,d$,
\begin{align*}
|\bar x_{k,i}(\bar u_k(s)-)- \bar r_{k,i}(s)|=&\Big|\frac{s-s^k_n}{t^k_{n+1}-s^k_n}\Big|
|x_{k,i}(t^k_{n+1})-x_{k,i}(t^k_n)|\\
\leq& \Big|\frac{t-s^k_n}{t^k_{n+1}-s^k_n}\Big|
|x_{k,i}(t^k_{n+1})-x_{k,i}(t^k_n)|\\
=& |\bar x_{k,i}(\bar u_k(t)-)- \bar r_{k,i}(t)|.
\end{align*}
For each 
$n\geq 0$, each $t \in  [t^k_n,t^k_{n+1})$, both $\bar u_k(t)$ and $u_k(t)$ belong to $[t^k_n,t^k_{n+1}]$,
so that
$$
\text{for all $T>0$,} \quad ||u_k - \bar{u}_k||_{\infty, T} \leq \sup_{n=0,\dots,m^k_T} (t^k_{n+1}-t^k_n),
$$
which tends to $0$ by assumption. 

\vip 
We now show that $||r_k-\bar r_k||_\infty\leq 1/k$ and this well complete the proof.
Fix $t\geq 0$ and let $n\geq 0$ such that $t \in [t^k_n,t^k_{n+1})$. Recall the definition of $A^k_n$,
that $r_k=x_k$ and~\eqref{apt}. If first $t\in  [t^k_n,s^k_n]$, then
$$
|r_k(t)-\bar r_k(t)|=\Big|\frac{t-t^k_n}{t^k_{n+1}-t^k_n}(x_k(t^k_{n+1})-x_k(t^k_n))\Big|
\leq \Big|\frac{s^k_n-t^k_n}{t^k_{n+1}-t^k_n}(x_k(t^k_{n+1})-x_k(t^k_n))\Big|\leq \frac 1k,
$$
while if  $t\in  [s^k_n,t^k_{n+1})$, then 
$$
|r_k(t)\!-\!\bar r_k(t)|\!=\!\Big|\Big(\frac{t-t^k_n}{t^k_{n+1}-t^k_n}-\frac{t-s^k_n}{t^k_{n+1}-s^k_n}\Big)(x_k(t^k_{n+1})
-x_k(t^k_n))\Big|\!\leq \! \Big|\frac{s^k_n-t^k_n}{t^k_{n+1}-t^k_n}(x_k(t^k_{n+1})\!-\!x_k(t^k_n))\Big|\!\leq\! \frac 1k.
$$
We used that the function $t\mapsto \frac{t-t^k_n}{t^k_{n+1}-t^k_n}-\frac{t-s^k_n}{t^k_{n+1}-s^k_n}$ is
decreasing on $[s^k_n,t^k_{n+1}]$, equals $\frac{s^k_n-t^k_n}{t^k_{n+1}-t^k_n}$ when $t=s^k_n$
and equals $0$ when $t=t^k_{n+1}$. 
\end{proof}

\begin{lemma}\label{compj1}
Consider a family $(x_k)_{k \geq 1}$ of elements of $\DD(\R_+,\R^d)$ such that $x_k\to x\in \DD(\R_+,\R^d)$
for the $\JS$-topology, and a family $(\rho_k)_{k \geq 1}$ of nondecreasing elements of  $\DD(\R_+,\R_+)$
such that $\sup_{[0,T]} |\rho_k(t)-t|=0$ for all $T>0$. Then $x_k\circ \rho_k \to x$ for the $\JS$-topology.
\end{lemma}

\begin{proof}
We denote by $\mathbb{D}^\uparrow$ the set of nondecreasing càdlàg functions from 
$\mathbb{R}_+ \to \mathbb{R}_+$. We have that $\rho_k\to \rho= \mathrm{id}_{\mathbb{R}_+}$ 
in $\DD^\uparrow$ for the $\JS$-topology
and that $x_k\to x$ in $\DD(\R_+,\R^d)$ for the $\JS$-topology. 
Since $\rho$ is continuous, we conclude by Whitt~\cite[Theorem 3.1]{whitt_useful} 
that $x_k\circ \rho_k \to x$ for the $\JS$-topology.
\end{proof}

\begin{lemma}\label{convL}
Consider a family $((z_n(u)_{u\geq 0})_{n\geq 1}$ of increasing elements of $\DD(\R_+,\R)$ such that $z_n(0)=0$ and
$\lim_{u\to \infty} z_n(u)=\infty$. Assume that  $((z_n(u))_{u\geq 0})_{n\geq 1}$
converges to some strictly increasing $(z(u))_{u\geq 0} \in \DD(\R_+,\R)$ for the $\JS$-topology, with 
$\lim_{u\to\infty} z(u)=\infty$.
For $t\geq 0$, set $y(t)=\inf\{u\geq 0 : z(u)>t\}$ and $y_n(t)=\inf\{u\geq 0 : z_n(u)>t\}$. 
\vip
(i) For all $T>0$, $\lim_n\sup_{[0,T]} |y_n(t)-y(t)| = 0$.
\vip
(ii) If $\lim_n t_n=t>0$, and $z(y(t)-)<t<z(y(t))$, then $\lim_n \Delta z_n(y_n(t_n))=\Delta z(y(t))$.
\end{lemma}

\begin{proof}
For (i), it suffices, by Dini's theorem and since $y$ is continuous (because $z$ is increasing)
to show that $\lim_n y_n(t)=y(t)$ for all fixed $t> 0$. Let us e.g. show that $\limsup_n y_n(t)\leq y(t)$,
the proof that $\liminf_n y_n(t)\geq y(t)$ being similar (but slightly easier). Fix $\e > 0$ 
and $u$ a continuity point of $z$ such that $y(t+\e)<u$. This implies that $t + \e \leq z(u)$, so that (since 
$\lim_n z_n(u)=z(u)$ because $u$ is a continuity point of $z$) $t < z_n(u)$ for all $n$ large enough.
Consequently, $y_n(t)\leq u$ for all $n$ large enough, whence $\limsup_n y_n(t)\leq u$. Since we can choose
$u$ arbitrarily close to $y(t+\e)$ (because the set of jump times of $z$ is at most countable),
we conclude that $\limsup_n y_n(t)\leq y(t+\e)$. It only remains to let
$\e\to 0$, using that $y$ is continuous.

\vip

For (ii), we e.g. show that $z_n(y_n(t_n))\to z(y(t))$, the proof that  $z_n(y_n(t_n)-)\to z(y(t)-)$ being similar.
Then sequence $z_n(y_n(t_n))$ being compact in $\R_+$ (since $z_n$ is locally uniformly bounded
and since  $y_n(t_n)\to y(t)$ by (i)), it suffices to show that the only adherent point of 
$z_n(y_n(t_n))$ is $z(y(t))$.
\vip
We first show that if $s_n\to s$, then the only possible 
adherent points of $z_n(s_n)$ are $z(s-)$ and $z(s)$. To this end, consider a time change $\lambda_n$ 
such that $||\lambda_n-I||_\infty\to 0$
and $||z_n - z \circ \lambda_n||_{\infty,T}\to 0$ (with e.g. $T=s+1$), write $z_n(s_n)=z_n(s_n)-z(\lambda_n(s_n))
+z(\lambda_n(s_n))$, use that $z_n(s_n)-z(\lambda_n(s_n))\to 0$ and that the only possible
adherent points of $z(\lambda_n(s_n))$ are $z(s-)$ and $z(s)$.
\vip

Since $y_n(t_n)\to y(t)$ by (i), we deduce that the only possible adherent points of 
$z_n(y_n(t_n))$ are $z(y(t)-)$ and $z(y(t))$.
But by definition, we have $z_n(y(t_n))\geq t_n \to t$. Since $z(y(t)-)<t$,
we conclude that the only adherent point of $z_n(y_n(t_n))$ is $z(y(t))$.
\end{proof}

Finally, we check the following easy fact.

\begin{lemma}\label{extra}
Assume that $x_n,x\in \DD(\R_+,\R^d)$ and $t_n,t\geq 0$ are such that $x_n\to x$ for the $\JS$-topology and
$t_n\to t$. We take the convention that $x_n(0-)=x_n(0)$ and $x(0-)=x(0)$.
\vip
(a) There exists $n_k\to \infty$ such that $x_{n_k}(t_{n_k})\to x(t)$ or
$x_{n_k}(t_{n_k})\to x(t-)$.
\vip
(b) There exists $n_k\to \infty$ such that $x_{n_k}(t_{n_k}-)\to x(t)$ or $x_{n_k}(t_{n_k}-)\to x(t-)$.
\end{lemma}

\begin{proof}
We consider $\lambda_n=\R_+\to\R_+$ continuous increasing, with $\lambda_n(0)=0$, 
$\lim_n ||\lambda_n-I||_\infty=0$ and $\lim_n ||y_n-x||_{\infty,T}=0$ for all $T>0$, where $y_n=x_n\circ \lambda_n$.
We set $s_n=\lambda_n^{-1}(t_n)$, so that $x_n(t_n)=y_n(s_n)$ and $x_n(t_n-)=y_n(s_n-)$. We have $\lim_n s_n=t$.
We can find a subsequence $n_k$ such that either (i) $s_{n_k}$ is nonincreasing or (ii) $s_{n_k}$ is strictly 
increasing. In the first case, $\lim x_{n_k}(t_{n_k})=x(t)$, because (for $k$ large enough)
$$
|x_{n_k}(t_{n_k})-x(t)|=|y_{n_k}(s_{n_k})-x(t)|\leq ||y_{n_k}-x||_{\infty,t+1}+|x(s_{n_k})-x(t)|\to 0
$$
since $x$ is càd. In the second case,  $\lim x_{n_k}(t_{n_k})=x(t-)$, because (for $k$ large enough)
$$
|x_{n_k}(t_{n_k})-x(t)|=|y_{n_k}(s_{n_k})-x(t)|\leq ||y_{n_k}-x||_{\infty,t+1}+|x(s_{n_k})-x(t-)|\to 0.
$$
This proves (a) and one can check (b) similarly.
\end{proof}

\section{From the scattering process to the scattering P.D.E.}\label{deredp}

Here we show the link between the scattering process introduced in Definition~\ref{dfsp}
and the scattering P.D.E.~\eqref{rescale_scattering_eq}.
We first provide a notion of weak solutions to~\eqref{rescale_scattering_eq}, 
see Jabir and Profeta~\cite[Theorem~4.2.1]{jp} and Bernou and Fournier~\cite[Definition 4]{bernou_fournier}
for similar considerations.
We introduce the sets $F_+=\{(x,v) \in \pDd\times\R^d :  v\cdot\bn_x>0\}$ and 
$F_-=\{(x,v) \in\pDd\times\R^d : v\cdot\bn_x<0\}$.

\begin{definition}
We say that a family $(f^\e_t)_{t\geq 0}$ of probability measures on $\cDd\times\R^d$ is a weak solution
to~\eqref{rescale_scattering_eq} if there exist some measures $\nu_+$ on
$\R_+\times F_+$ and $\nu_-$ on $\R_+\times F_-$ such that 
$\nu_+([0,T]\times F_+)+\nu_-([0,T]\times F_-)<\infty$ for all $T>0$ and 
such that for all $\varphi \in C^\infty_c(\R_+\times\cDd\times\R^d)$,
\begin{align}\label{wfs}
0=&\int_{\cDd\times\R^d} \varphi(0,x, v) f^\e_0(\dr x,\dr v) + \int_0^\infty\int_{\cDd\times \R^d}
\partial_t\varphi(t,x,v)f^\e_t(\dr x,\dr v) \dr t \notag \\
&+\e^{\frac{1-\alpha}\alpha}
\int_0^\infty\int_{\cDd\times \R^d} v\cdot \nabla_x\varphi(t,x,v)f^\e_t(\dr x,\dr v) \dr t \notag\\
&+\frac1\e \int_0^\infty\int_{\cDd\times\R^d} \Big(\int_{\R^d}\varphi(t,x,w)\rF(w)\dr w - \varphi(t,x,v)  
\Big)f^\e_t(\dr x,\dr v) \dr t\notag \\
&+\int_{\R_+\times F_+} \varphi(t,x,v)\nu_+(\dr t,\dr x,\dr v) 
-\int_{\R_+\times F_-} \varphi(t,x,v)\nu_-(\dr t,\dr x,\dr v)
\end{align}
and
\begin{equation}\label{wfs2}
\nu_+(\dr t,\dr x,\dr v)=2\rG(v)\indiq_{\{v\cdot\bn_x>0\}} \nu_-(\dr t,\dr x,\R^d)\dr v.
\end{equation}
\end{definition}

The following remark explains this definition.

\begin{remark}
If $(f^\e_t)_{t\geq 0}$ is a smooth weak solution to~\eqref{rescale_scattering_eq}, then it
solves~\eqref{rescale_scattering_eq}.
\end{remark}

\begin{proof} We assume that $f^\e_t$ has a smooth density that we still denote by $f^\e_t$.
An integration by parts shows that the first line of~\eqref{wfs} equals
$$
-\int_0^\infty\int_{\cDd\times\R^d}\varphi(t,x,v)\partial_tf^\e_t(x,v) \dr x\dr v \dr t.
$$
The Green formula implies that the second line of~\eqref{wfs} equals
\begin{align*}
&-\e^{\frac{1-\alpha}\alpha}\int_0^\infty\int_{\cDd\times \R^d} \varphi(t,x,v)
v\cdot \nabla_x f^\e_t(x,v) \dr x \dr v \dr t \\
&\hskip3cm- \e^{\frac{1-\alpha}\alpha}\int_0^\infty\int_{\pDd\times \R^d} 
\varphi(t,x,v)
(v\cdot \bn_x) f^\e_t(x,v) \dr x \dr v \dr t.
\end{align*}
Finally, the third line of~\eqref{wfs} equals
$$
\frac1\e \int_0^\infty\int_{\cDd\times\R^d}\varphi(t,x,v) 
\Big(\rF(v)\int_{\R^d}f^\e_t(x,w)\dr w - f^\e_t(x,v)  
\Big) \dr x \dr v \dr t.
$$
All in all, since~\eqref{wfs} holds true for any $\varphi\in C^\infty_c(\R_+\times\Dd\times\R^d)$
(with $\Dd$ open so that $\varphi$ vanishes on $\pDd$), we find
that for all $(t, x, v )\in \R_+\times\Dd\times\R^d$,
$$
\partial_t f^\e_t(x,v)+\e^{\frac{1-\alpha}\alpha}v\cdot \nabla_x f^\e_t(x,v)=\frac1\e
\Big(\rF(v)\int_{\R^d}f^\e_t(x,w)\dr w - f^\e_t(x,v)\Big).
$$
Hence $(f^\e_t)_{t\geq 0}$ solves the first line of~\eqref{rescale_scattering_eq}. Moreover, once this is seen,
\eqref{wfs} rewrites, for all $\varphi \in  C^\infty_c(\R_+\times\cDd\times\R^d)$,
\begin{align*}
&-\e^{\frac{1-\alpha}\alpha}\int_0^\infty\!\!\int_{\pDd\times \R^d}\! 
\varphi(t,x,v) (v\cdot \bn_x) f^\e_t(x,v) \dr x \dr v \dr t \\
&+ 
\int_{\R_+\times F_+} \varphi(t,x,v)\nu_+(\dr t,\dr x,\dr v) -\int_{\R_+\times F_-} \varphi(t,x,v)\nu_-(\dr t,\dr x,\dr v)
=0
\end{align*}
Thus $\e^{\frac{1-\alpha}{\alpha}}(v\cdot \bn_x) f^\e_t(x,v)\dr x\dr v\dr t= (\nu_+-\nu_-)(\dr t,\dr x,\dr v)$,
so that necessarily
\begin{gather*}
\nu_+(\dr t,\dr x,\dr v)=\e^{\frac{1-\alpha}{\alpha}}(v\cdot \bn_x) f^\e_t(x,v)\indiq_{\{v\cdot \bn_x>0\}}
\dr x\dr v\dr t,\\
\nu_-(\dr t,\dr x,\dr v)=\e^{\frac{1-\alpha}{\alpha}}|v\cdot \bn_x| f^\e_t(x,v)\indiq_{\{v\cdot \bn_x<0\}}
\dr x\dr v\dr t. 
\end{gather*}
Hence~\eqref{wfs2} tells us that for all $t\geq 0$, all $x \in \pDd$, all $v \in \R^d$, 
$$
(v\cdot \bn_x) f^\e_t(x,v)\indiq_{\{v\cdot \bn_x>0\}}=2\rG(v) \int_{w\cdot \bn_x<0} |w\cdot \bn_x|f^\e_t(x,w)\dr w,
$$
so that $(f^\e_t)_{t\geq 0}$ solves the second line of~\eqref{rescale_scattering_eq}.
\end{proof}

We end the paper with the

\begin{proof}[Proof of Remark~\ref{edp}] We fix $(x_0,v_0)\in \bE$ (recall~\eqref{E}), 
consider a measurable family
$(A_y)_{y \in \pDd}$ such that $A_y  \in \cI_y$ for each $y \in \pDd$ and divide the proof into 3 steps.
\vip

{\it Step 1.} Consider a Poisson measure $\rN_\e =\sum_{k\geq 1}\delta_{(S^\e_k,Y^\e_k)}$ 
on $\mathbb{R}_+ \times \mathbb{R}^d$ with intensity measure 
$\e^{-1}\dr t \mathrm{F}(\dr v)$, independent of a collection of i.i.d. $\rG_+$-distributed 
random variables $(W_n)_{n \geq 1}$, recall that $\rG_+(v)=2\rG(v)\indiq_{\{v\cdot\be_1>0\}}$.
Consider the càdlàg process $(\bX_t^\e, \bV_t^\e)_{t\geq0}$ solving
\begin{equation}\label{rea}
\left\lbrace
\begin{aligned}
\bX_t^\e & = x_0  + \e^{\frac{1-\alpha}\alpha}\int_0^t \bV_s^\e \dr s, \\
\bV_t^\e & = v_0  + \int_0^t\int_{\mathbb{R}^d}\Big(z - \bV_{s-}^\e\Big)\rN_\e (\dr s, \dr z) 
+ \sum_{n\geq 1}\Big(A_{\bX_{\tau_n^\e}^\e} W_n
- \bV_{\tau_n^\e-}^\e\Big)\indiq_{\{\tau_n \leq t\}}, \\
\tau_0^\e & = 0 \quad \text{and} 
\quad \tau_{n+1}^\e = \inf\{t> \tau_n^\e, \: \bX_t^\e \in \partial\mathcal{D}\}.
\end{aligned}
\right.
\end{equation}
The intensity of the Poisson measure $\rN_\e$
being finite on $[0,T]\times\R^d$ for all $T>0$, \eqref{rea} 
has a pathwise unique solution, which is an $\e$-scattering process starting from $(x,v)$, see 
Definition~\ref{dfsp}.
\vip
Indeed, \eqref{rea} tells us that the position process $\bX^\e$ moves according to its velocity 
$\e^{(1-\alpha)/\alpha}\bV^\e$; that the velocity process $\bV^\e$ is refreshed at rate
$\e^{-1}$, its new value being chosen according to $\rF$; and that when the position process reaches the boundary 
$\partial \mathcal{D}$ (at some time $\tau_n^\e$), it is restarted with a $\rG_{\bX_{\tau_n^\e}^\e}$-distributed
velocity (recall that for $y\in \pDd$ and $W\sim \rG_+(v)\dr v$, $A_yW\sim \rG_y(v)\dr v$).

\vip

The process built in Definition~\ref{dfsp} has exactly the same dynamics.

\vip

{\it Step 2.} Here we observe that $\E[M^\e_T]<\infty$ for all $T>0$, where
$M^\e_T=\sum_{n=1}^\infty \indiq_{\{\tau^\e_n\leq T\}}$. Since $M^\e_T\leq \bM^\e_T$, where $\bM^\e_T$ stands 
for the total number of jumps of
$(\bV^\e_t)_{t\geq 0}$ during $[0,T]$, this follows from Remark~\ref{scwd}.

\vip

{\it  Step 3.} For $\varphi\in C^\infty_c(\R_+\times \cDd\times\R^d)$ and for $T>0$, 
we deduce from~\eqref{rea} that
\begin{align*}
\varphi(T,\bX^\e_T,\bV^\e_T)=&\varphi(0,x_0,v_0)+\int_0^T \Big(\partial_t\varphi(t,\bX^\e_t,\bV^\e_t) 
+ \e^{\frac{1-\alpha}\alpha}\bV^\e_t\cdot \nabla_x\varphi(t,\bX^\e_t,\bV^\e_t)\Big) \dr t\\
&+\int_0^T\int_{\R^d} \Big(\varphi(t,\bX^\e_t,z)-\varphi(t,\bX^\e_t,\bV^\e_{t-})\Big) 
\rN_\e (\dr t,\dr z)\\
&+\sum_{n\geq 1} \Big(\varphi(\tau_n^\e,\bX^\e_{\tau_n^\e},A_{\bX^\e_{\tau_n^\e}}W_n)
-\varphi(\tau_n^\e,\bX^\e_{\tau_n^\e},\bV^\e_{\tau_n^\e-})\Big)\indiq_{\{\tau^\e_n \leq T\}}.
\end{align*}
Taking expectations and choosing $T$ such that Supp $\varphi\subset[0,T]\times\cDd\times\R^d$,
we get 
\begin{align*}
0=&\varphi(0,x_0,v_0)+\E\Big[\int_0^\infty \Big(\partial_t\varphi(t,\bX^\e_t,\bV^\e_t) 
+ \e^{\frac{1-\alpha}\alpha}\bV^\e_t\cdot \nabla_x\varphi(t,\bX^\e_t,\bV^\e_t)\Big) \dr t\Big]\\
&+\frac1\e
\E\Big[\int_0^\infty \Big(\int_{\R^d}\varphi(t,\bX^\e_t,w)\rF(w)\dr w 
- \varphi(t,\bX^\e_t,\bV^\e_{t})\Big) \dr t \Big]\\
&+\E\Big[\sum_{n\geq 1} \Big(\varphi(\tau_n^\e,\bX^\e_{\tau_n^\e},A_{\bX^\e_{\tau_n^\e}}W_n)
-\varphi(\tau_n^\e,\bX^\e_{\tau_n^\e},\bV^\e_{\tau_n^\e-})\Big)\Big]
\end{align*}
Denoting by $f^\e_t(\dr x,\dr v)=\PP(\bX^\e_t\in \dr x, \bV^\e_t \in \dr v)$, we see that  
$(f^\e_t)_{t\geq 0}$ satisfies~\eqref{wfs} with $f_0^\e=\delta_{(x_0,v_0)}$ and with the measures
\begin{gather*}
\nu_+(\dr t, \dr x, \dr v)=\sum_{n \geq 1} \PP\Big(\tau_n^\e\in \dr t,\bX^\e_{\tau_n^\e}\in \dr x, 
A_{\bX^\e_{\tau_n^\e}}W_n \in \dr v \Big),\\
\nu_-(\dr t, \dr x, \dr v)=\sum_{n \geq 1} \PP\Big(\tau_n^\e\in \dr t,\bX^\e_{\tau_n^\e}\in \dr x,
\bV^\e_{\tau_n^\e-}\in \dr v\Big),
\end{gather*}
which are respectively carried by $\R_+\times F_+$ and $\R_+\times F_-$. By Step 2, we have
$\nu_+([0,T]\times F_+)=\nu_-([0,T]\times F_-)=\E[M^\e_T]<\infty$, and this moreover allows us to justify 
the previous computations of the present step.
It remains to check~\eqref{wfs2}. But for each $n\geq 1$, since $W_n$ is independent
of $(\tau^\e_n,\bX^\e_{\tau^\e_n},\bV^\e_{\tau^\e_n-})$ and 
$A_xW_n \sim 2\rG(v)\indiq_{\{v\cdot \bn_x>0\}}\dr v$ for each $x \in \pDd$,
\begin{align*}
\PP\Big(\tau_n^\e \in \dr t, \bX^\e_{\tau_n^\e} \in \dr x,A_{\bX^\e_{\tau_n^\e}}W_n \in \dr v\Big)
= 2\rG(v)\indiq_{\{v\cdot \bn_x>0\}}\dr v 
\PP\Big(\tau_n^\e \in \dr t, \bX^\e_{\tau_n^\e} \in \dr x, \bV^\e_{\tau^\e_n-} \in \R^d \Big).
\end{align*}
Summing this identity on $n\geq 1$ precisely gives~\eqref{wfs2}.
\end{proof}

\bibliographystyle{abbrv}
\bibliography{biblio.bib}

\end{document}